\documentclass{gsm-l}

\usepackage{color}
\usepackage{framed}
\usepackage{empheq}
\usepackage{enumerate}
\usepackage{amsthm, amssymb}
\usepackage{hyperref}
\usepackage{array}
\usepackage{mathrsfs} 
\usepackage[shortlabels]{enumitem}

\usepackage{background}
\backgroundsetup{
  position=current page.east,
  angle=-90,
  vshift=-4mm,
  hshift=2mm,
  opacity=0.8,
  scale=3,
  contents= {\rm --- $\overline{\underline{\mbox{DRAFT}}}$ --- }
}

\usepackage{chngcntr}
\counterwithin{figure}{chapter}

\newtheorem{thm}{Theorem}[chapter]
\newtheorem*{thm*}{Theorem}
\newtheorem{lem}[thm]{Lemma}
\newtheorem{cor}[thm]{Corollary}
\newtheorem{prop}[thm]{Proposition}

\theoremstyle{definition}
\newtheorem{defi}[thm]{Definition}

\theoremstyle{remark}
\newtheorem{rem}[thm]{Remark}
\newtheorem*{rem*}{Remark}
\newtheorem*{obs}{Observation}

\numberwithin{section}{chapter}
\numberwithin{equation}{chapter}

\newcommand{\average}{{\mathchoice {\kern1ex\vcenter{\hrule height.4pt
width 6pt depth0pt} \kern-9.7pt} {\kern1ex\vcenter{\hrule
height.4pt width 4.3pt depth0pt} \kern-7pt} {} {} }}
\newcommand{\ave}{\average\int}

\def\R{\mathbb{R}}

\def\N{\mathbb{N}}
\def\E{\mathbb{E}}

\newcommand{\divv}{{\rm div}}
\newcommand{\grad}{\nabla}
\newcommand{\eps}{\varepsilon}
\newcommand{\de}{\partial}

\DeclareMathOperator*{\osc}{osc}

\makeindex

\begin{document}
\frontmatter

\title{{\Huge Regularity Theory for Elliptic PDE}\\ \vspace{1cm}}

\author{Xavier Fern\'andez-Real}
\address{EPFL SB MATH, Institute of Mathematics, Station 8, CH-1015 Lausanne, Switzerland}
\email{xavier.fernandez-real@epfl.ch}

\author{Xavier Ros-Oton}
\address{Universit\"at Z\"urich, Institut f\"ur Mathematik, Winterthurerstrasse 190, 8057 Z\"urich, Switzerland, \&}
\address{ICREA, Pg. Llu\'is Companys 23, 08010 Barcelona, Spain, \&}
\address{Universitat de Barcelona, Departament de Matem\`atiques i Inform\`atica, Gran Via de les Corts Catalanes 585, 08007 Barcelona, Spain, \&}
\address{Centre de Recerca Matem\`atica, Edifici C, Campus Bellaterra, 08193 Bellaterra, Spain}
\email{xros@icrea.cat}

\date{\today}



\maketitle
	
\setcounter{page}{4}
\tableofcontents

%

\chapter*{Preface}

One of the most basic and important 
questions in PDE is that of regularity.
It is also a unifying problem in the field, since it affects all kinds of PDEs.
A~classical example is Hilbert's XIXth problem (1900), which roughly speaking asked to determine whether all solutions to uniformly elliptic {variational} PDEs are smooth.
The question was answered positively by De Giorgi and Nash in 1956 and 1957, and it is now one of the most famous and important theorems in the whole field of PDE.

The question of regularity has been a central line of research in elliptic PDE since the mid-20th century, with extremely important contributions by Nirenberg, Caffarelli, Krylov, Evans, Figalli, and many others.
Their works have enormously influenced many areas of Mathematics linked one way or another with PDE, including: Harmonic Analysis, Calculus of Variations, Differential Geometry, Geometric Measure Theory, Continuum and Fluid Mechanics, Probability Theory, Mathematical Physics, and Computational and Applied Mathematics.

This text emerged from two PhD courses on elliptic PDE given by the second author at the University of Z\"urich in 2017 and 2019.
It aims to provide a self-contained introduction to the regularity theory for elliptic PDE, focusing on the main ideas rather than proving all results in their greatest generality.
The book can be seen as a bridge between an elementary PDE course and more advanced textbooks such as \cite{GT} or \cite{CC}.
Moreover, we believe that the present selection of results and techniques complements nicely other books on elliptic PDE such as \cite{Evans}, \cite{HL}, and \cite{K}, as well as the recent book \cite{ACM}.
For example, we give a different proof of the Schauder estimates (due to L. Simon) which is not contained in other textbooks; we prove some basic results for fully nonlinear equations that are not covered in \cite{CC}; and we also include a detailed study of the obstacle problem, often left to more specialized textbooks such as \cite{Fri} or \cite{PSU}.
Furthermore, at the end of Chapters 3, 4, and 5 we provide a review of some recent results and open problems.

We would like to thank Alessio Figalli, Thomas Kappeler, Alexis Michelat, Joaquim Serra, and Wei Wang, for several comments and suggestions on this book.

Finally, we acknowledge the support received from the following funding agencies:
X.F. was supported by the European Research Council under the Grant Agreement No. 721675 ``Regularity and Stability in Partial Differential Equations (RSPDE)'',  by the Swiss National Science Foundation (SNF grants 
200021\_182565 and PZ00P2\_208930), and by the Swiss State Secretariat for Education, Research and lnnovation (SERI) under contract number M822.00034;
X.R. was supported by the European Research Council under the Grant Agreement No. 801867 ``Regularity and singularities in elliptic PDE (EllipticPDE)'', by the Swiss National Science Foundation (SNF grant 200021\_178795), by AEI project PID2021-125021NA-I00 (Spain), by the grant RED2018-102650-T funded by MCIN/AEI/10.13039/501100011033, and by the Spanish State Research Agency through the Mar\'ia de Maeztu Program for Centers and Units of Excellence in R{\&}D (CEX2020-001084-M).

\aufm{Z\"urich, 2020}


\mainmatter
%

\chapter{Overview and Preliminaries}
\label{ch.0}

A beautiful result in Complex Analysis states that because the real part $u(x, y)$ of any holomorphic function satisfies 
\[u_{xx}+u_{yy}=0,\]
it must be real analytic. 
Moreover, the oscillation of $u$ in any given domain controls \emph{all} the derivatives in any (compactly contained) subdomain.

In higher dimensions, the same phenomenon occurs for solutions to
\begin{equation} \label{eq.LapD}
\Delta u =0\quad \textrm{in}\quad \Omega\subset\R^n.
\end{equation}
These are \emph{harmonic functions}, and \eqref{eq.LapD} is the simplest elliptic partial differential equation (PDE).
Any solution to this equation is smooth (real analytic), and satisfies 
\[\|u\|_{C^k(Q)}\leq C_{k,Q}\|u\|_{L^\infty(\Omega)}\qquad \textrm{for all}\quad k=1,2,3,...\]
for any compact subdomain $Q\subset\subset \Omega$.
That is, all derivatives are controlled by the supremum of $u$.

Here, and throughout the book, $\Omega$ is any bounded domain of $\R^n$.

\vspace{2mm}

\begin{center} 
 \fbox{
\begin{minipage}{0.85\textwidth}
\vspace{1.5mm}

\noindent \ $\bullet$\ \emph{Regularity for Laplace's equation}:
\vspace{2mm}
\[\qquad\Delta u=0\quad\textrm{in}\quad \Omega\subset\R^n\qquad \Longrightarrow\qquad u\ \textrm{is}\ C^\infty\ \textrm{inside}\ \Omega.\quad\]

\vspace{2mm}

\end{minipage}
}

\end{center}
\vspace{2mm}

This kind of regularization property is common in elliptic PDEs  and is the topic of the present book.

\vspace{3mm}

One can give three different kinds of explanations for this phenomenon:

\vspace{1mm}

\begin{itemize}
\item[(a)] \underline{\smash{Integral representation of solutions}}: Poisson kernels, fundamental solutions, etc.

\vspace{2mm}

\item[(b)] \underline{\smash{Energy considerations}}: Harmonic functions are local minimizers of the Dirichlet energy
\[\mathcal E(u):=\int_\Omega |\nabla u|^2\,dx\]
(i.e., if we change $u$ to $w$ in $\tilde\Omega\subset\Omega$, then $\mathcal E(w)\geq\mathcal E(u)$).

\vspace{2mm}

\item[(c)] \underline{\smash{Comparison principle}}: A harmonic function cannot have any interior maximum point (maximum principle).
\end{itemize}

\vspace{2mm}

These three approaches are extremely useful in different contexts, as well as in the development of the regularity theory for \emph{nonlinear} elliptic PDEs.

\vspace{3mm}

The structure of the book is as follows:

\vspace{2mm}

$\star$ First, in {\bf Chapter \ref{ch.1}} we will study \emph{linear} elliptic PDEs
\[\sum_{i,j=1}^n a_{ij}(x)\partial_{ij}u= f(x)\quad \textrm{in}\quad \Omega\subset\R^n\]
and 
\[\sum_{i,j=1}^n \partial_i\bigl(a_{ij}(x)\partial_{j}u\bigr)= f(x)\quad \textrm{in}\quad \Omega\subset\R^n,\]
where the coefficients $a_{ij}(x)$ and the right-hand side $f(x)$ satisfy appropriate regularity assumptions.
In the simplest case, $(a_{ij})_{i,j}\equiv \textrm{Id}$, we have
\[\Delta u=f(x)\quad \textrm{in}\quad \Omega\subset\R^n.\]
The type of result we want to prove is: ``$u$ \emph{is two derivatives more regular than} $f$''.

\vspace{2mm}

$\star$ Then, in {\bf Chapter \ref{ch.2}} we will turn our attention to \emph{nonlinear variational PDEs}:
\[\textrm{minimizers of}\quad \mathcal E(u):=\int_\Omega L(\nabla u)dx,\quad L\ \textrm{smooth and uniformly convex}.\]
The regularity for such kind of nonlinear PDEs was Hilbert's XIXth problem (1900).

\vspace{2mm}

$\star$ In {\bf Chapter \ref{ch.3}} we will study nonlinear elliptic PDEs in their most general form
\[F(D^2u,\nabla u,u,x)=0 \quad \textrm{in}\quad \Omega\subset\R^n,\]
or simply
\[F(D^2u)=0 \quad \textrm{in}\quad \Omega\subset\R^n.\]
These are called \emph{fully nonlinear elliptic equations}, and in general they do \emph{not} have a variational formulation in terms of an energy functional.

\vspace{2mm}

$\star$ In {\bf Chapter \ref{ch.4}} we will study the \emph{obstacle problem}, a constrained minimization problem:
\[\textrm{minimize}\qquad \int_\Omega |\nabla u|^2dx,\qquad \textrm{among functions}\ u\geq\varphi\ \textrm{in}\ \Omega,\]
where $\varphi$ is a given smooth ``obstacle''.
This is the simplest and most important elliptic \emph{free boundary problem}.
Moreover, it can be seen as a nonlinear PDE of the type \ $\min\{-\Delta u,\,u-\varphi\}=0$ in $\Omega$.

\vspace{3mm}

As we will see, in each of these contexts we will use mainly: (b) energy considerations, or (c) maximum principle.

At the end of the book, we have also included four appendices to complement the theory from the main chapters.

\section{Preliminaries: Sobolev and H\"older spaces}
\label{sec.hs}

We next give a quick review on $L^p$, Sobolev, and H\"older spaces, stating the results that will be used later in the book. 

\subsection*{$L^p$ spaces}
Given $\Omega\subset\R^n$ and $1\leq p<\infty$, the space $L^p(\Omega)$ is the set
\[L^p(\Omega):=\left\{u\textrm{ measurable in }\Omega\,:\, \int_\Omega |u|^pdx<\infty\right\}.\]
It is a Banach space, with the norm $\|u\|_{L^p(\Omega)}:=(\int_\Omega |u|^p)^{1/p}$.

When $p=\infty$, the space $L^\infty(\Omega)$ is the set of bounded functions (up to sets of measure zero), with the norm $\|u\|_{L^\infty(\Omega)}:=\textrm{esssup}_\Omega|u|$.

A well-known result in this setting is the Lebesgue differentiation theorem (see, for example, \cite{EG92}).

\begin{thm}\label{Lebesgue}\index{Lebesgue differentiation theorem}
If $u\in L^1(\Omega)$, then for almost every $x\in \Omega$ we have
\[\lim_{r\to0}\ave_{B_r(x)}\bigl|u(x)-u(y)\bigr|dy=0.\]
When this holds at a point $x\in \Omega$, we say that $x$ is a Lebesgue point of $u$.
\end{thm}

Here, and throughout the book, $\ave_A$ denotes the average $\frac{1}{|A|}\int_A$, where $A\subset\R^n$ is any set of finite and positive measure.\index{Average integral}

A useful consequence of this result is the following.

\begin{cor}\label{ch0-ae}
Assume $u\in L^1(\Omega)$, and
\[\int_\Omega uv\,dx=0\qquad \textrm{for all }\, v\in C^\infty_c(\Omega).\]
Then, $u=0$ a.e. in $\Omega$.
\end{cor}

\subsection*{Integration by parts}

A fundamental identity in the study of PDEs is the following.

\begin{thm}[Integration by parts]\label{ch0-int-parts}\index{Integration by parts}
Assume $\Omega\subset\R^n$ is any bounded $C^1$ domain\footnote{We refer to the Notation section (page \pageref{domainnotation}) for the definition of $C^1$ domains.}.
Then, for any $u, v\in C^1(\overline\Omega)$ we have
\begin{equation}\label{int-parts}
\int_\Omega \partial_iu\,v\,dx= -\int_\Omega u\,\partial_iv\,dx+\int_{\partial\Omega} uv\,\nu_i\,dS,
\end{equation}
where $\nu$ is the unit (outward) normal vector to $\partial\Omega$, and $i=1,2,...,n$.
\end{thm}

Notice that, as an immediate consequence, we find the divergence theorem, as well as Green's first identity
\[\int_\Omega \grad u\cdot\grad v\,dx= -\int_\Omega u\,\Delta v\,dx+\int_{\partial\Omega} u\,\frac{\partial v}{\partial \nu}\,dS.\]

The regularity requirements of Theorem \ref{ch0-int-parts} can be relaxed.
For instance, the domain $\Omega$ need only be Lipschitz, while only $u,v\in H^1(\Omega)$ is necessary in \eqref{int-parts} --- where $H^1$ is a Sobolev space, defined below.

\subsection*{Sobolev spaces}

\index{Sobolev space}

Given any domain $\Omega\subset\R^n$ and $1\leq p\leq\infty$, the Sobolev spaces $W^{1,p}(\Omega)$ consist of all functions whose (weak) derivatives are in $L^p(\Omega)$, namely
\[W^{1,p}(\Omega):=\left\{u\in L^p(\Omega)\,:\, \partial_i u\in L^p(\Omega)\,\textrm{ for }\,i=1,...,n\right\}.\]
We refer to the excellent books \cite{Evans,Brezis} for the definition of weak derivatives and a detailed exposition on Sobolev spaces.

A few useful properties of Sobolev spaces are the following (see \cite{Evans}):

\begin{enumerate}[leftmargin=*,label={\bf (S\arabic*)}]
\setlength\itemsep{2mm}
\item \label{it.S1}  The spaces $W^{1,p}(\Omega)$ are complete.
\item \label{it.S2}   The inclusion $W^{1,p}(\Omega)\subset L^p(\Omega)$ is compact.
\item \label{it.S3}   The space $H^1(\Omega):=W^{1,2}(\Omega)$ is a Hilbert space with the scalar product 
\[(u,v)_{H^1(\Omega)}=\int_\Omega uv+\int_\Omega \nabla u\cdot\nabla v.\]
\item \label{it.S4}  Any bounded sequence $\{u_k\}$ in the Hilbert space $H^1(\Omega)$ contains a weakly convergent subsequence $\{u_{k_j}\}$, that is, there exists $u\in H^1(\Omega)$ such that 
 \begin{equation}\label{ch0-weak-conv}
 \qquad\qquad(u_{k_j},v)_{H^1(\Omega)}\to (u,v)_{H^1(\Omega)} \quad \textrm{for all}\ v\in H^1(\Omega).
 \end{equation}
 In addition, such $u$ will satisfy 
  \begin{equation}\label{ch0-weak-conv2}
  \|u\|_{H^1(\Omega)}\leq \liminf_{j\to\infty}\|u_{k_j}\|_{H^1(\Omega)},
 \end{equation}
 and since $H^1(\Omega)$ is compactly embedded in $L^2(\Omega)$ one has
  \begin{equation}\label{ch0-weak-conv3}
  \|u\|_{L^2(\Omega)}= \lim_{j\to\infty}\|u_{k_j}\|_{L^2(\Omega)}.
 \end{equation}

\item \label{it.S5} Let $\Omega$ be any bounded Lipschitz domain, and $1\leq p\le\infty$.
Then, there is a continuous (and compact for $p>1$) trace operator from $W^{1,p}(\Omega)$ to $L^p(\partial\Omega)$. 
For $C^0$ functions, such trace operator is simply $u\mapsto u|_{\partial\Omega}$. 

Because of this, for any function $u \in H^1(\Omega)$ we will still denote by $u|_{\partial\Omega}$ its trace on $\partial\Omega$.

\item \label{it.S7}  For $1\le p < \infty$, $C^\infty(\Omega)$ functions are dense in $W^{1,p}(\Omega)$. Moreover, if $\Omega$ is bounded and Lipschitz, $C^\infty(\overline{\Omega})$ functions are dense in $W^{1,p}(\Omega)$.

\item \label{it.S6} For $1\le p < \infty$, we define the space $W^{1,p}_0(\Omega)$ as the closure of $C^\infty_c(\Omega)$ in $W^{1,p}(\Omega)$. Similarly, we denote $H_0^1(\Omega) := W^{1,2}_0(\Omega)$. When $\Omega$ is bounded and Lipschitz, it is the space of functions $u\in W^{1,p}(\Omega)$ such that $u|_{\partial\Omega}=0$.

%

\item \label{it.S9}  If $u\in W^{1,p}(\Omega)$, $1\leq p\leq \infty$, then for any subdomain $K\subset\subset \Omega$ we have
 \[\left\|\frac{u(x+h)-u(x)}{|h|}\right\|_{L^p(K)}\leq C\left\|\nabla u\right\|_{L^p(\Omega)}\]
 for all $h\in B_\delta$, with $\delta>0$ small enough.
 
 Conversely, if $u\in L^p(\Omega)$, $1<p\leq \infty$, and 
 \[\left\|\frac{u(x+h)-u(x)}{|h|}\right\|_{L^p(K)}\leq C\]
 for every $h\in B_\delta$, then $u\in W^{1,p}(K)$ and $\left\|\nabla u\right\|_{L^p(\Omega)}\leq C$.
 (However, this property fails when $p=1$.)

\item \label{it.S10}  Given any function $u$, define $u^+=\max\{u,0\}$ and $u^-=\max\{-u,0\}$, so that $u=u^+-u^-$.
Then, for any $u\in W^{1,p}(\Omega)$ we have $u^+,u^-\in W^{1,p}(\Omega)$, and $\nabla u=\nabla u^+-\nabla u^-$ a.e. in $\Omega$.

In particular, the gradient of Sobolev functions vanishes almost everywhere on level sets, $\nabla u(x) = 0$ for a.e. $x\in \{u = 0\}$.  
\end{enumerate}

An important inequality in this context is the following.

\begin{thm}[Sobolev inequality]\label{ch0-Sob}\index{Sobolev inequality}
If $p<n$, then
\[
\left(\int_{\R^n} |u|^{p_*}dx\right)^{1/p_*} \leq C \left(\int_{\R^n}|\grad u|^pdx \right)^{1/p}, \qquad \frac{1}{p_*}=\frac1p-\frac1n,
\]
for some constant $C$ depending only on $n$ and $p$.
In particular, we have a continuous inclusion $W^{1,p}(\R^n)\subset L^{p_*}(\R^n)$.
\end{thm}

Notice that, as $p\uparrow n$ we have $p_*\to\infty$.
In the limiting case $p=n$, however, it is \emph{not} true that $W^{1,n}$ functions are bounded. 
This can be seen by taking, for example, $u(x)=\log \log \left(1+\frac{1}{|x|}\right)\in W^{1,n}(B_1)$.
Still, in case $p>n$, the following occurs.

\begin{thm}[Morrey inequality]\index{Morrey inequality}
If $p>n$, then
\[
\sup_{x\neq y}\frac{\bigl|u(x)-u(y)\bigr|}{|x-y|^\alpha} \leq C \left(\int_{\R^n}|\grad u|^pdx \right)^{1/p}, \qquad \alpha=1-\frac{n}{p},
\]
for some constant $C$ depending only on $n$ and $p$.
\end{thm}

In particular, when $p>n$ any function in $W^{1,p}$ is continuous (after possibly being redefined on a set of measure 0).

Finally, we will also use the following inequalities in bounded domains.

\begin{thm}[Poincar\'e inequality]\label{ch0-Poinc}\index{Poincar\'e inequality}
Let $\Omega\subset\R^n$ be any bounded Lipschitz domain, and let $p\in [1, \infty)$.
Then, for any $u\in W^{1, p}(\Omega)$ we have
\[
\int_{\Omega} |u-u_\Omega|^pdx \leq C_{\Omega, p} \int_\Omega|\grad u|^pdx,
\]
where $u_\Omega:=\ave_\Omega u$, and
\[
\int_{\Omega} |u|^pdx \leq C_{\Omega,p}' \left(\int_\Omega|\grad u|^pdx +\int_{\partial\Omega}\bigl|u|_{\partial\Omega}\bigr|^pd\sigma\right).
\]
The constants $C_{\Omega,p}$ and $C_{\Omega,p}'$ depend only on $n$, $p$, and $\Omega$.
\end{thm}

\subsection*{H\"older spaces}
\label{ssec.HolderSpaces}
\index{H\"older space}

Given $\alpha\in(0,1)$, the H\"older space $C^{0,\alpha}(\overline\Omega)$ is the set of continuous functions $u\in C(\overline\Omega)$ such that the H\"older semi-norm is finite,
\[ [u]_{C^{0,\alpha}(\overline{\Omega})} := \sup_{\substack{x,y\in\overline\Omega\\x\neq y}}\frac{\bigl|u(x)-u(y)\bigr|}{|x-y|^\alpha}<\infty.\]\index{H\"older semi-norm}
The H\"older norm is
\[\|u\|_{C^{0,\alpha}(\overline\Omega)}:=\|u\|_{L^\infty(\Omega)}+[u]_{C^{0,\alpha}(\overline{\Omega})}.\]\index{H\"older norm}
When $\alpha=1$, this is the usual space of Lipschitz continuous functions.

More generally, given $k\in\mathbb N$ and $\alpha\in(0,1)$, the space $C^{k,\alpha}(\overline\Omega)$ is the set of functions $u\in C^k(\overline\Omega)$ such that the following norm is finite
\[
\|u\|_{C^{k,\alpha}(\overline\Omega)} =  \|u\|_{C^k(\overline\Omega)} + [D^k u]_{C^{0,\alpha}(\overline{\Omega})},
\]
where 
\[
\|u\|_{C^k(\overline\Omega)} := \sum_{j=1}^k\|D^ju\|_{L^\infty(\Omega)}.
\]
Notice that this yields the inclusions
\[C^0\supset C^{0,\alpha} \supset \textrm{Lip} \supset C^1 \supset C^{1,\alpha}\supset ...\supset C^\infty.\]
We will often write $\|u\|_{C^{k,\alpha}(\Omega)}$ instead of $\|u\|_{C^{k,\alpha}(\overline\Omega)}$.

Finally, it is sometimes convenient to use the following notation.
When $\beta>0$ is \emph{not} an integer, we define $C^\beta(\overline\Omega):=C^{k,\alpha}(\overline\Omega)$, where $\beta=k+\alpha$, $k\in\mathbb N$, $\alpha\in(0,1)$.

\vspace{2mm}

There are many properties or alternative definitions of H\"older spaces that will be used throughout the book.
They are valid for all $\alpha\in(0,1)$, and are proved in Appendix~\ref{app.A}.

\begin{enumerate}[leftmargin=*,label={\bf (H\arabic*)}]
\setlength\itemsep{2mm}
\item \label{it.H1} Assume
\[{\rm osc}_{B_r(x)}u\leq C_{\circ}r^\alpha\qquad \textrm{for all }\, B_r(x)\subset\overline{B_1},\]
where ${\rm osc}_A u:=\sup_A u - \inf_A u$.

Then, $u\in C^{0,\alpha}(\overline{B_1})$ and $[u]_{C^{0,\alpha}(\overline{B_1})}\leq CC_{\circ}$, with $C$ depending only on $n,\alpha$.

\item \label{it.H2} Let $u_{x,r}:=\ave_{B_r(x)}u$. Assume
\[\|u-u_{x,r}\|_{L^\infty(B_r(x))}\leq C_{\circ}r^\alpha\qquad \textrm{for all }\, B_r(x)\subset\overline{B_1}.\]

Then, $u\in C^{0,\alpha}(\overline{B_1})$ and $[u]_{C^{0,\alpha}(\overline{B_1})}\leq CC_{\circ}$, with $C$ depending only on $n,\alpha$.

\item \label{it.H3} Let $u_{x,r}:=\ave_{B_r(x)}u$. Assume
\[\left(\ave_{B_r(x)}|u-u_{x,r}|^2\right)^{1/2}\leq C_{\circ}r^\alpha\qquad \textrm{for all }\, B_r(x)\subset\overline{B_1}.\]
Then, $u\in C^{0,\alpha}(\overline{B_1})$ and $[u]_{C^{0,\alpha}(\overline{B_1})}\leq CC_{\circ}$, with $C$ depending only on $n,\alpha$.

\item \label{it.H4}  Assume that for every $x$ there is a constant $C_x$ such that
\[\|u-C_x\|_{L^\infty(B_r(x))}\leq C_{\circ}r^\alpha\qquad \textrm{for all }\, B_r(x)\subset\overline{B_1}.\]
Then, $u\in C^{0,\alpha}(\overline{B_1})$ and $[u]_{C^{0,\alpha}(\overline{B_1})}\leq CC_{\circ}$, with $C$ depending only on $n,\alpha$.

Assume that for every $x$ there is a linear function $\ell_x(y)=a_x+b_x\cdot(y-x)$ such that
\[\|u-\ell_x\|_{L^\infty(B_r(x))}\leq C_{\circ}r^{1+\alpha}\qquad \textrm{for all }\, B_r(x)\subset\overline{B_1}.\]
Then, $u\in C^{1,\alpha}(\overline{B_1})$ and $[Du]_{C^{0,\alpha}(\overline{B_1})}\leq CC_{\circ}$, with $C$ depending only on $n,\alpha$.

Assume that for every $x$ there is a quadratic polynomial $P_x(y)$ such that
\[\|u-P_x\|_{L^\infty(B_r(x))}\leq C_{\circ}r^{2+\alpha}\qquad \textrm{for all }\, B_r(x)\subset\overline{B_1}.\]
Then, $u\in C^{2,\alpha}(\overline{B_1})$ and $[D^2u]_{C^{0,\alpha}(\overline{B_1})}\leq CC_{\circ}$, with $C$ depending only on $n,\alpha$.

\item \label{it.H5} Let $\rho_\circ\in (0,1)$.
Assume that, for every $x\in B_{1/2}$, there exists a sequence of quadratic polynomials, $(P_k)_{k\in \N}$, such that
\begin{equation}
\label{ch0-H_QP}
\|u-P_k\|_{L^\infty(B_{\rho_\circ^k}(x))}\leq C_{\circ}\rho_\circ^{k(2+\alpha)}\qquad\textrm{for all }\, k\in \N.
\end{equation}
Then, $u\in C^{2,\alpha}(B_{1/2})$ and $[D^2u]_{C^{0,\alpha}(B_{1/2})}\le C C_{\circ}$, with $C$ depending only on $n$, $\alpha$, and $\rho_\circ$. 

\item \label{it.H6} Assume that $\alpha\in(0,1)$, $\|u\|_{L^\infty(B_1)}\le C_{\circ}$, and 
\[
\sup_{\substack{x\in B_1\\ x\pm h\in\overline{B_1}}}\frac{\bigl|u(x+h)+u(x-h)-2u(x)\bigr|}{|h|^\alpha}\leq C_{\circ}.
\]
Then, $u\in C^{0,\alpha}(\overline{B_1})$ and $\|u\|_{C^{0,\alpha}(\overline{B_1})}\leq CC_{\circ}$, with $C$ depending only on $n,\alpha$.

Assume that $\alpha\in (0, 1)$, $\|u\|_{L^\infty(B_1)}\le C_{\circ}$, and 
\[\sup_{\substack{x\in B_1\\ x\pm h\in\overline{B_1}}}\frac{\bigl|u(x+h)+u(x-h)-2u(x)\bigr|}{|h|^{1+\alpha}}\leq C_{\circ}.
\]
Then, $u\in C^{1,\alpha}(\overline{B_1})$ and $\|u\|_{C^{1,\alpha}(\overline{B_1})}\leq CC_{\circ}$, with $C$ depending only on $n,\alpha$.

However, such property fails when $\alpha=0$.

\item \label{it.H7} Assume that $\alpha\in (0,1]$, $\|u\|_{L^\infty(B_1)}\le C_{\circ}$, and that for every $h\in B_1$ we have
\begin{equation}
\label{eq.H7}
\left\|\frac{u(x+h)-u(x)}{|h|^\alpha}\right\|_{C^\beta(B_{1-|h|})}\leq C_{\circ},
\end{equation}
with $C_{\circ}$ independent of $h$.
Assume in addition that $\alpha+\beta$ is not an integer.
Then, $u\in C^{\alpha+\beta}(\overline{B_1})$ and $\|u\|_{C^{\alpha+\beta}(\overline{B_1})}\leq CC_{\circ}$, with $C$ depending only on $n,\alpha,\beta$.

However, such property fails when $\alpha+\beta$ is an integer.

\item \label{it.H8}  Assume that $u_i\to u_0$ uniformly in $\overline\Omega\subset \R^n$, and that $\|u_i\|_{C^{k,\alpha}(\overline\Omega)}\leq C_{\circ}$, with $\alpha\in (0, 1]$ and for some $C_{\circ}$ independent of $i$.
Then, we have that $u_0\in C^{k,\alpha}(\overline\Omega)$, and
\[\|u_0\|_{C^{k,\alpha}(\overline\Omega)}\leq C_{\circ}.\]
\end{enumerate}


 \vspace{2mm}

Finally, an important result in this context is the following particular case of the Arzel\`a--Ascoli theorem.

\begin{thm}[Arzel\`a--Ascoli]\label{ch0-AA}\index{Arzel\`a-Ascoli Theorem}
Let $\Omega\subset\R^n$, $\alpha\in (0,1)$, and let $\{f_i\}_{i\in \mathbb N}$ be any sequence of functions $f_i$ satisfying 
\[\|f_i\|_{C^{0,\alpha}(\overline\Omega)}\leq C_{\circ}.\]

Then, there exists a subsequence $f_{i_j}$ which converges uniformly to a function $f\in C^{0,\alpha}(\overline\Omega)$.
\end{thm}

More generally, this result --- combined with \ref{it.H8} --- implies that if 
\[\|u_i\|_{C^{k,\alpha}(\overline\Omega)}\leq C_{\circ},\]
with $\alpha\in(0,1)$, then a subsequence $u_{i_j}$ will converge in the $C^k(\overline\Omega)$ norm to a function $u\in C^{k,\alpha}(\overline\Omega)$.

\subsection*{Interpolation inequalities in H\"older spaces}
\index{Interpolation inequalities}
A useful tool that will be used throughout the book is the following.
For each $0\leq \gamma<\alpha<\beta\leq1$ and every $\varepsilon>0$, we have
\begin{equation}\label{ch0-interp}
\|u\|_{C^{0,\alpha}(\overline\Omega)} \leq C_{\varepsilon}\|u\|_{C^{0,\gamma}(\overline\Omega)} + \varepsilon \|u\|_{C^{0,\beta}(\overline\Omega)},
\end{equation}
where $C$ is a constant depending only on $n$ and $\varepsilon$. (When $\gamma=0$, $C^{0,\gamma}$ should be replaced by $L^\infty$.)
This follows from the interpolation inequality
\[\|u\|_{C^{0,\alpha}(\overline\Omega)} \leq \|u\|_{C^{0,\gamma}(\overline\Omega)}^t \|u\|_{C^{0,\beta}(\overline\Omega)}^{1-t} \qquad t=\frac{\beta-\alpha}{\beta-\gamma}.\]

More generally, \eqref{ch0-interp} holds for higher-order H\"older norms too.
In particular, we will use that for any $\varepsilon>0$ and $\alpha\in(0,1)$
\[
\|\nabla u\|_{L^\infty(\overline\Omega)} \leq C_{\varepsilon}\|u\|_{L^\infty(\overline\Omega)} + \varepsilon [\nabla u]_{C^{0,\alpha}(\overline\Omega)},
\]
and
\begin{equation}\label{ch0-interp2}
\|u\|_{C^{2}(\overline\Omega)}  = \|u\|_{C^{1,1}(\overline\Omega)} \leq C_{\varepsilon}\|u\|_{L^\infty(\overline\Omega)} + \varepsilon [D^2 u]_{C^{0,\alpha}(\overline\Omega)}.
\end{equation}
We refer to \cite[Lemma 6.35]{GT} for a proof of such inequalities.

%

\section{A review on the Laplace equation}

Elliptic equations are those that share some common properties with the Laplace equation. (We will be more rigorous about this in the subsequent chapters.)
Thus, we start with a quick review about the Laplace equation and harmonic functions.

The \emph{Dirichlet problem} for this equation is the following:\index{Dirichlet problem!Laplace equation}
\begin{equation}\label{DirP}
\left\{ 
\begin{array}{rcll}
\Delta u & = & 0 &\text{in } \Omega\\
u &=& g &\text{on } \partial\Omega,
\end{array}\right.
\end{equation}
where the boundary condition $g$ is given.
The domain $\Omega\subset\R^n$ is bounded and smooth (or at least Lipschitz).
The Dirichlet problem is solvable, and it has a unique solution.

A useful way to think of the Laplacian $\Delta$ is to notice that, up to a multiplicative constant, it is the only linear operator of second order which is translation invariant \emph{and} rotation invariant.
Indeed, it can be seen as an operator which measures (infinitesimally) the difference between $u$ at $x$ and the average of $u$ around $x$, in the following sense: for any $C^2$ function $w$ we have
\begin{equation} \label{Laplacian-radially}
\begin{split}
 \Delta w(x)&=\lim_{r\to0}\frac{c_n}{r^2}\left\{\ave_{B_r(x)}w(y)dy-w(x)\right\}\\ 
 &=\lim_{r\to0}\frac{c_n}{r^2}\ave_{B_r(x)}\bigl(w(y)-w(x)\bigr)dy, 
\end{split}
\end{equation}
for some positive constant $c_n$.
This can be shown, for example, by using the Taylor expansion of $w(y)$ around $x$.
Moreover, a similar formula holds with integrals in $\partial B_r(x)$ instead of $B_r(x)$.  See, for example, \cite{DV21}.

Actually, one can show by using the divergence theorem that
\begin{equation}
\label{Laplacian-radially_2}
\frac{n}{r} \frac{d}{dr}\ave_{\partial B_r(x)} w \, d\sigma = \ave_{B_r(x)}\Delta w,
\end{equation}
from which \eqref{Laplacian-radially} also follows.

\subsection*{Existence of solutions: energy methods}
\index{Energy method}
The most classical way to construct solutions of \eqref{DirP} is by ``energy methods''. 
Namely, we consider the convex functional
\[\qquad\qquad\mathcal E(u):= \frac12\int_\Omega |\nabla u|^2dx\qquad \textrm{among functions satisfying}\quad u|_{\partial\Omega}=g,\]
and then look for the function $u$ that minimizes the functional --- see Theorem~\ref{ch0-existence} below for more details about the existence of a minimizer.
Notice that such minimizer $u$ will clearly satisfy the boundary condition $u=g$ on $\partial\Omega$, so we only have to check that it will satisfy in addition $\Delta u=0$ in $\Omega$.

If $u$ is the minimizer, then $\mathcal E(u)\leq \mathcal E(u+\varepsilon v)$ for every $v\in C^\infty_c(\Omega)$.
Since, for every fixed $v$, such function in $\varepsilon$ has a minimum at $\varepsilon=0$, we have 
\[\left.\frac{d}{d\varepsilon}\right|_{\varepsilon=0}\mathcal E(u+\varepsilon v) =0.\]
Thus, 
\begin{eqnarray*}
0 & = & \left.\frac{d}{d\varepsilon}\right|_{\varepsilon=0} \mathcal E(u+\varepsilon v) 
   = \left.\frac{d}{d\varepsilon}\right|_{\varepsilon=0} \frac12\int_\Omega |\nabla u+\varepsilon v|^2 dx\\
   &=& \left.\frac{d}{d\varepsilon}\right|_{\varepsilon=0} \frac12\int_\Omega \bigl(|\nabla u|^2+2\varepsilon\nabla u\cdot\nabla v+\varepsilon^2|\nabla v|^2\bigr)dx\\
  &=& \int_\Omega \nabla u\cdot\nabla v\,dx.
\end{eqnarray*}

Hence, if $u$ is the minimizer of the functional, then 
\begin{equation}\label{weaksol-0}
 \int_\Omega \nabla u\cdot\nabla v\,dx=0\qquad \textrm{for all}\quad v\in C^\infty_c(\Omega).
\end{equation}
\emph{If $u$ is regular enough} (say, $u\in C^2$), then we can integrate by parts (Theorem~\ref{ch0-int-parts}) to find that
\[\int_\Omega \Delta u\, v\,dx=0\qquad \textrm{for all}\quad v\in C^\infty_c(\Omega).\]
Thus, using Corollary \ref{ch0-ae} we deduce that $\Delta u=0$ in $\Omega$, as wanted.

\vspace{3mm}

\begin{rem}
As mentioned above, one should prove \emph{regularity} of $u$ before integrating by parts --- a priori the minimizer $u$ will only satisfy $u\in H^1(\Omega)$.
We will prove this in Corollary \ref{ch0-smooth} below.

If no extra regularity of $u$ is available, then the above argument shows that any minimizer $u$ of $\mathcal E$ is a \emph{weak solution}, in the following sense.
\end{rem}

\begin{defi}\label{defi-weak}\index{Weak solution}
We say that $u$ is a \emph{weak solution} of the Dirichlet problem \eqref{DirP} whenever $u\in H^1(\Omega)$, $u|_{\partial\Omega}=g$, and 
\[ \int_\Omega \nabla u\cdot\nabla v\,dx=0\qquad \textrm{for all}\quad v\in H^1_0(\Omega).\]
Here, $u|_{\partial\Omega}$ is the trace of $u$ on $\partial\Omega$; recall \ref{it.S5} above.

More generally, given $f\in L^2(\Omega)$, we say that $u$ satisfies $-\Delta u=f$ in~$\Omega$ in the weak sense whenever $u\in H^1(\Omega)$ and 
\[ \int_\Omega \nabla u\cdot\nabla v\,dx=\int_\Omega fv\qquad \textrm{for all}\quad v\in H^1_0(\Omega).\]
Finally, we say that $u$ is \emph{weakly superharmonic} (resp. \emph{weakly subharmonic}) in $\Omega$, or satisfies $\Delta u \le 0$ in $\Omega$ in the weak sense (resp. $\Delta u \ge 0$ in the weak sense) if 
\[ \int_\Omega \nabla u\cdot\nabla v\,dx \ge 0 \quad \left(\text{resp. } \int_\Omega \nabla u\cdot\nabla v\,dx \le 0 \right)\quad \textrm{for all}\quad v\in H^1_0(\Omega), v\ge 0.\]
\end{defi}

Notice that, if $H^1(\Omega)\ni u_k \rightharpoonup u\in H^1(\Omega)$ weakly in $H^1$, and $L^2(\Omega)\ni f_k \rightharpoonup f\in L^2(\Omega)$ weakly in $L^2$ are such that $\Delta u_k = f_k$ in $\Omega$ in the weak sense, then $\Delta u = f$ in the weak sense as well (by taking the limits in the previous definitions). Similarly, the weak limit of weakly (sub-)superharmonic functions is (sub-)superharmonic.

We next show the following:

\begin{thm}[Existence and uniqueness of weak solutions]\label{ch0-existence} \index{Existence and uniqueness!Laplace equation}
Assume that $\Omega\subset\R^n$ is any bounded Lipschitz domain, and that 
\begin{equation} \label{ch0-nonempty-g}
\left\{w\in H^1(\Omega)\,:\, w|_{\partial\Omega}=g\right\}\neq \varnothing.
\end{equation}

Then, there exists a unique weak solution to the Dirichlet problem \eqref{DirP}.
\end{thm}

\begin{proof}
\textsc{Existence.}
Let 
\[\theta_{\circ}:=\inf\left\{\frac12\int_\Omega |\nabla w|^2dx\,:\, w\in H^1(\Omega),\ w|_{\partial\Omega}=g\right\},\]
that is, the infimum value of $\mathcal E(w)$ among all admissible functions $w$.

Let us take a sequence of functions $\{u_k\}$ such that
\begin{itemize}
\item $u_k\in H^1(\Omega)$

\item $u_k|_{\partial\Omega}=g$

\item $\mathcal E(u_k)\to \theta_{\circ}$ as $k\to\infty$.
\end{itemize} 
By the Poincar\'e inequality (Theorem \ref{ch0-Poinc} with $p = 2$), the sequence $\{u_k\}$ is uniformly bounded in $H^1(\Omega)$, and therefore a subsequence $\{u_{k_j}\}$ will converge to a certain function $u$ strongly in $L^2(\Omega)$ and weakly in $H^1(\Omega)$ (recall \eqref{ch0-weak-conv}-\eqref{ch0-weak-conv3} in \ref{it.S4}).
Moreover, by compactness of the trace operator, we will have $u_{k_j}|_{\partial\Omega}\to u|_{\partial\Omega}$ in $L^2(\partial\Omega)$, so that $u|_{\partial\Omega}=g$.
Furthermore, such function $u$ will satisfy $\mathcal E(u)\leq \liminf_{j\to\infty}\mathcal E(u_{k_j})$ (by \eqref{ch0-weak-conv2} and \eqref{ch0-weak-conv3}), and therefore it will be a minimizer of the energy functional.

Thus, we have constructed a minimizer $u$ of the energy functional $\mathcal E(u)$ satisfying the boundary condition $u|_{\partial\Omega}=g$.
By the argument above, for any minimizer $u$ we have that \eqref{weaksol-0} holds.
Since $C^\infty_c(\Omega)$ is dense in $H^1_0(\Omega)$, it follows that \eqref{weaksol-0} holds for all $v\in H^1_0(\Omega)$, and thus it is a weak solution of \eqref{DirP}.

\vspace{2mm}

\textsc{Uniqueness.}
If $u$ is any weak solution to \eqref{DirP}, then for every $v\in H^1_0(\Omega)$ we have
\begin{eqnarray*}
\mathcal E(u+v)&=&\frac12\int_\Omega|\nabla u+\nabla v|^2dx\\ &=&\frac12\int_\Omega |\nabla u|^2dx+\int_\Omega\nabla u\cdot \nabla v\,dx+\frac12\int_\Omega|\nabla v|^2dx\\
 &=&\mathcal E(u)+0+\frac12\int_\Omega|\nabla v|^2dx\geq \mathcal E(u),
\end{eqnarray*}
with strict inequality if $v\not\equiv0$.
Thus, if $u$ solves \eqref{DirP}, then it is unique.
\end{proof}

In other words, we have shown that $u$ is a weak solution of \eqref{DirP} if and only if it minimizes the functional $\mathcal E(u)$ and, moreover, the minimizer of such energy functional exists and it is unique.

\begin{rem}
 An interesting question is to determine the set of possible boundary data $g:\partial\Omega\to\R$ such that \eqref{ch0-nonempty-g} holds.
 Of course, when $\Omega$ is any bounded Lipschitz domain, and $g$ is Lipschitz, then it is easy to show that $g$ has a Lipschitz extension inside $\Omega$, and in particular \eqref{ch0-nonempty-g} holds.
 However, if $g$ is very irregular then it might happen that it is \emph{not} the trace of any $H^1(\Omega)$ function, so that \eqref{ch0-nonempty-g} fails in this case.
 It turns out that the right condition on $g$ is the following: Given any bounded Lipschitz domain $\Omega$, \eqref{ch0-nonempty-g} holds if and only if 
 \[\int_{\partial\Omega}\int_{\partial\Omega}\frac{|g(x)-g(y)|^2}{|x-y|^{n+1}}\,dx\,dy<\infty.\]
 We refer to \cite{Evans} for more details.
\end{rem}

\subsection*{Poisson kernel and fundamental solution}
\index{Poisson kernel}\index{Fundamental solution}
The unique weak solution to the Dirichlet problem in a ball is explicit:
\begin{equation*}
\left\{ 
\begin{array}{rcll}
\Delta u & = & 0 &\text{in } B_1\\
u &=& g &\text{on } \partial B_1,
\end{array}\right.\qquad \Longrightarrow\qquad u(x)=c_n\int_{\partial B_1}\frac{(1-|x|^2)g(\sigma)}{|x-\sigma|^n}\,d\sigma,
\end{equation*}
where $c_n$ is a positive dimensional constant.
By an easy rescaling argument, a similar formula holds in any ball $B_r(x_{\circ})\subset\R^n$.

Thus, we deduce that for any harmonic function $\Delta u=0$ in $\Omega$, with $B_r\subset\Omega$, we have
\begin{equation}\label{ch0-Poisson}
u(x)=\frac{c_n}{r}\int_{\partial B_r}\frac{(r^2-|x|^2)u(y)}{|x-y|^n}\,dy.
\end{equation}
By taking $x=0$, this yields the \emph{mean value property} $u(0)=\ave_{\partial B_r}u$.
Moreover, an immediate consequence of the Poisson kernel representation is the following.

\begin{cor}\label{ch0-smooth}
Let $\Omega\subset\R^n$ be any open set, and $u\in H^1(\Omega)$ be any function satisfying $\Delta u=0$ in $\Omega$ in the weak sense.
Then, $u$ is $C^\infty$ inside~$\Omega$.

Moreover, if $u$ is bounded and $\Delta u = 0$ in $B_1$ in the weak sense, then we have the estimates
\begin{equation}
\label{eq.estimatesuk}
\|u\|_{C^k(B_{1/2})} \le C_k \|u\|_{L^\infty(B_1)},
\end{equation}
for all $k\in \N$, and for some constant $C_k$ depending only on $k$ and $n$. 
\end{cor}

\begin{proof}
For any ball $B_r(x_{\circ})\subset \Omega$, we will have \eqref{ch0-Poisson}.
Thanks to such representation, it is immediate to see then that $u\in C^\infty(B_{r/2}(x_{\circ}))$ and the estimates \eqref{eq.estimatesuk} hold.
Since this can be done for any ball $B_r(x_{\circ})\subset \Omega$, we deduce that $u$ is $C^\infty$ inside~$\Omega$.
\end{proof}

\vspace{2mm}

On the other hand, we recall that the \emph{fundamental solution} for the Laplacian is given by
\begin{equation}\label{fundamental-sol}
\Phi(x):=\left\{\begin{array}{ll} \displaystyle
\frac{\kappa_n}{|x|^{n-2}} & \quad\textrm{if}\ n\ge3 \\
\displaystyle \kappa_2\log\frac{1}{|x|} & \quad\textrm{if}\ n=2,
\end{array}\right.
\end{equation}
for some explicit positive dimensional constant $\kappa_n$.
Such function satisfies $\Delta \Phi=0$ in $\R^n\setminus\{0\}$, but it is singular at $x=0$.
In fact, it satisfies 
\[-\Delta\Phi=\delta_{0}\quad \textrm{in}\quad \R^n,\]
where $\delta_{0}$ is the Dirac delta function.
In particular, we have that $w:=\Phi * f$ solves $-\Delta w=f$ in $\R^n$, for any given $f$ with appropriate decay at infinity.

\subsection*{Maximum principle}
\index{Maximum principle!Laplace equation}
The maximum principle states the following: If $\Delta u\geq0$ in $\Omega$, and $u\in C(\overline\Omega)$, then
\[\max_{\overline\Omega}u=\max_{\partial\Omega}u.\]
In particular, we also deduce the \emph{comparison principle}: if $\Delta u\geq \Delta v$ in $\Omega$, and $u\leq v$ on $\partial\Omega$, then $u\leq v$ in the whole domain $\Omega$.\index{Comparison principle!Laplace equation}

Recall that a function is said to be \emph{subharmonic} if $-\Delta u\leq0$, and \emph{superharmonic} if $-\Delta u\geq0$.\index{Subharmonic function}\index{Superharmonic function}

As shown next, the maximum principle actually holds for any weak solution $u$.

\begin{prop}\label{max-princ-weak}
Let $\Omega\subset\R^n$ be any bounded open set.
Assume that $u\in H^1(\Omega)$ satisfies, in the weak sense,
\begin{equation*}
\left\{ 
\begin{array}{rcll}
-\Delta u & \geq & 0 &\text{in } \Omega\\
u &\geq& 0 &\text{on } \partial \Omega.
\end{array}\right.
\end{equation*}
Then, $u\geq0$ in $\Omega$.
\end{prop}

\begin{proof}
Notice that $-\Delta u \geq 0$ in $\Omega$ if and only if 
\begin{equation}\label{ch0-max-princ1}
 \int_\Omega \nabla u\cdot\nabla v\,dx\geq0\qquad \textrm{for all}\quad v\geq 0,\ v\in H^1_0(\Omega).
\end{equation}
Let us consider $u^-:=\max\{-u,0\}$ and $u^+:=\max\{u,0\}$, so that $u=u^+-u^-$.
By \ref{it.S10} we have that $u^\pm\in H^1(\Omega)$ whenever $u\in H^1(\Omega)$, and thus we can choose $v=u^-\geq0$ in \eqref{ch0-max-princ1}.
Namely, using that $u^+u^-=0$ and $\nabla u=\nabla u^+-\nabla u^-$, we get 
\[0\leq \int_\Omega \nabla u\cdot\nabla u^-\,dx=-\int_\Omega |\nabla u^-|^2\,dx.\]
Since $u^-|_{\de\Omega} \equiv 0$ this implies $u^-\equiv0$ in $\Omega$, that is, $u\geq0$ in $\Omega$.
\end{proof}

A useful consequence of the maximum principle is the following.

\begin{lem} \label{lem.maxPrinciple}
 Let $u$ be any weak solution of
\begin{equation*}
\left\{ 
\begin{array}{rcll}
\Delta u & = & f &\text{in } \Omega\\
u &=& g &\text{on } \partial \Omega.
\end{array}\right.
\end{equation*}
Then, 
\[\|u\|_{L^\infty(\Omega)}\leq C\bigl(\|f\|_{L^\infty(\Omega)}+\|g\|_{L^\infty(\partial\Omega)}\bigr),\]
for a constant $C$ depending only on the diameter of $\Omega$.
\end{lem}

\begin{proof}
Let us consider the function
\[\tilde u(x):=u(x)/\left(\|f\|_{L^\infty(\Omega)}+\|g\|_{L^\infty(\partial\Omega)}\right).\]
We want to prove that $|\tilde u|\leq C$ in $\Omega$, for some constant $C$ depending only on the diameter of $\Omega$.

Notice that such function $\tilde u$ solves 
\begin{equation*}
\left\{ 
\begin{array}{rcll}
\Delta \tilde u & = & \tilde f &\text{in } \Omega\\
u &=& \tilde g &\text{on } \partial \Omega,
\end{array}\right.
\end{equation*}
with $|\tilde g|\leq 1$ and $|\tilde f|\leq 1$.

Let us choose $R$ large enough so that $B_R\supset \Omega$; after a translation, we can take $R=\frac12\textrm{diam}(\Omega)$.
In $B_R$, let us consider the function
\[w(x)=\frac{R^2-x_1^2}{2}+1.\]
Such function $w$ satisfies
\begin{equation*}
\left\{ 
\begin{array}{rcll}
\Delta w & = & -1 &\text{in } \Omega\\
w &\geq& 1 &\text{on } \partial \Omega,
\end{array}\right.
\end{equation*}
Therefore, by the comparison principle, we deduce that
\[\tilde u\leq w\quad\textrm{in}\quad \Omega.\]
Since $w\leq C$ (with $C$ depending only on $R$), we deduce that $\tilde u\leq C$ in $\Omega$.
Finally, repeating the same argument with $-\tilde u$ instead of $\tilde u$, we find that $|\tilde u|\leq C$ in $\Omega$, and thus we are done.
\end{proof}

Finally, another important result which follows from the maximum principle is the following.
Here, we say that $\Omega$ satisfies the interior ball condition whenever there exists $\rho_{\circ}>0$ such that every point on $\partial\Omega$ can be touched from inside with a ball of radius $\rho_{\circ}$ contained in $\overline\Omega$.
That is, for any $x_\circ\in \partial\Omega$ there exists $B_{\rho_\circ}(y_\circ)\subset \Omega$ with $x_\circ\in \partial B_{\rho_\circ}(y_\circ)$.

It is not difficult to see that any $C^2$ domain satisfies such condition, and also any domain which is the complement of a convex set.

\begin{lem}[Hopf Lemma]\label{Hopf}\index{Hopf Lemma}
Let $\Omega\subset \R^n$ be any domain satisfying the interior ball condition.
Let $u\in C(\overline\Omega)$ be any positive harmonic function in $\Omega\cap B_2$, with $u\geq0$ on $\partial\Omega \cap B_2$.

Then, $u\geq c_{\circ}d$ in $\Omega\cap B_1$ for some $c_{\circ}>0$, where $d(x):={\rm dist}(x,\Omega^c)$.
\end{lem}

\begin{proof}
Since $u$ is positive and continuous in $\Omega\cap B_2$, we have that $u\geq c_1>0$ in $\{d\geq \rho_{\circ}/2\}\cap B_{3/2}$ for some $c_1>0$.

Let us consider the solution of $\Delta w=0$ in $B_{\rho_{\circ}}\setminus B_{\rho_{\circ}/2}$, with $w=0$ on $\partial B_{\rho_{\circ}}$ and $w=1$ on $\partial B_{\rho_{\circ}/2}$.
Such function $w$ is explicit --- it is simply a truncated and rescaled version of the fundamental solution $\Phi$ in \eqref{fundamental-sol}.
In particular, it is immediate to check that $w\geq c_2(\rho_{\circ}-|x|)$ in $B_{\rho_{\circ}}$ for some $c_2>0$.

By using the function $c_1w(x_{\circ}+x)$ as a subsolution in any ball $B_{\rho_{\circ}}(x_{\circ})\subset \Omega\cap B_{3/2}$, we deduce that $u(x)\geq c_1w(x_{\circ}+x)\geq c_1c_2(\rho_{\circ}-|x-x_{\circ}|)\geq c_1c_2d$ in $B_{\rho_{\circ}}(x_\circ)$.
Setting $c_{\circ}=c_1c_2$ and using the previous inequality for every ball $B_{\rho_{\circ}}(x_{\circ})\subset \Omega\cap B_{3/2}$, the result follows.
\end{proof}

\subsection*{Mean value property and Liouville theorem}

If $u$ is harmonic in $\Omega$ (i.e., $\Delta u=0$ in $\Omega$), then
\index{Mean value property}
\begin{equation}
\label{eq.mean_value_property}
u(x)=\ave_{B_r(x)} u(y)dy\qquad \textrm{for any ball}\quad B_r(x)\subset\Omega.
\end{equation}
This is called the mean value property.

Conversely, if $u\in C^2(\Omega)$ satisfies the mean value property, then $\Delta u=0$ in $\Omega$.
This can be seen for example by using \eqref{Laplacian-radially} above.

In fact, the mean value property \eqref{eq.mean_value_property} can be used to give yet another (weak) definition of harmonic functions that only requires $u$ to be locally integrable. Similarly, it is not difficult to deduce the corresponding pro\-per\-ty arising from the definitions of weak super- and subharmonicity (see Definition~\ref{defi-weak}): 

From \eqref{Laplacian-radially_2}, if $u$ is weakly superharmonic in $\Omega$ ($\Delta u \le 0$ in $\Omega$ in the weak sense) then for all $x\in \Omega$
\begin{equation}
\label{eq.superharmonic_integral}
r\mapsto \ave_{B_r(x)} u(y)\, dy\quad\text{is monotone non-increasing for $r\in (0, {\rm dist}(x, \partial\Omega))$.}
\end{equation}
(And it is monotone non-decreasing for weakly subharmonic functions.)

Thus, we can define (weak) super- and subharmonicity for $L^1_{\rm loc}$ functions: we say that $u\in L^1_{\rm loc}(\Omega)$ is superharmonic in $\Omega$ if \eqref{eq.superharmonic_integral} holds for all $x\in\Omega$. Similarly, we say that $u\in L^1_{\rm loc}(\Omega)$ is subharmonic in $\Omega$ if the map in \eqref{eq.superharmonic_integral} is monotone non-decreasing for all $x\in\Omega$ and $r\in (0, {\rm dist}(x, \partial\Omega))$.

We now give two lemmas that will be used in Chapter~\ref{ch.4}. The first lemma says that the pointwise limit of a sequence of superharmonic uniformly bounded functions is superharmonic.

\begin{lem}
\label{lem.convergence_pointwise}
Let $\Omega\subset \R^n$, and let $\{w_n\}_{n\in \N}$ be a sequence of uniformly bounded functions $w_n: \Omega\to \R$ satisfying \eqref{eq.superharmonic_integral}, converging pointwise to some $w:\Omega \to \R$. Then $w$ satisfies \eqref{eq.superharmonic_integral}. 
\end{lem}
\begin{proof}
Let $w_\infty := w$ and let us define for $n\in \N\cup\{\infty\}$, $\varphi_{x,n}(r) := \ave_{B_r(x)} w_n$. Notice that $\varphi_{x,n}(r)$ is non-increasing in $r$ for all $n\in \N$. In particular, given $0< r_1 < r_2 < R_x$, we have that $\varphi_{x, n}(r_1) \ge \varphi_{x, n}(r_2)$ for $n\in \N$. Now we let $n\to \infty$ and use that $w_n\to w$ pointwise to deduce, by the dominated convergence theorem (notice that $w_n$ are uniformly bounded), that $\varphi_{x, \infty}(r_1) \ge \varphi_{x, \infty}(r_2)$. That is, $w_\infty = w$ satisfies \eqref{eq.superharmonic_integral}. 
\end{proof}

The second lemma shows that superharmonic functions are lower semi-continuous. 

\begin{lem}
\label{lem.lower_semi}
Let us assume that $w$ is bounded and satisfies \eqref{eq.superharmonic_integral} in $\Omega\subset \R^n$. Then, up to changing $w$ in a set of measure 0, $w$ is lower semi-continuous.
\end{lem}
\begin{proof}
The proof is standard. If we define $w_0(x) := \lim_{r\downarrow 0}\ave_{B_r(x)} w$ (which is well defined, since the average is monotone non-increasing), then $w_0 (x) = w(x) $ if $x$ is a Lebesgue point, and thus $w_0 = w$ almost everywhere in $\Omega$.  Let us now consider $x_\circ\in \Omega$, and let $x_k \to x_\circ$ as $k\to \infty$. Then, by the dominated convergence theorem we have that 
\[
\ave_{B_r(x_\circ)} w = \lim_{k\to \infty} \ave_{B_r(x_k)} w \le \liminf_{k\to \infty} w_0(x_k) 
\]
for $0 < r < \frac12 {\rm dist}(x_\circ, \partial\Omega)$. Now, by letting $r\downarrow 0$ on the left-hand side, we reach that 
\[
w_0(x_\circ) \le \liminf_{k\to \infty} w_0(x_k),
\]
that is, $w_0$ is lower semi-continuous. 
\end{proof}

On the other hand, a well-known theorem that can be deduced from the mean value property is the classification of global bounded harmonic functions.

\index{Liouville Theorem}\begin{thm}[Liouville's theorem]
\label{thm.Liouville}
Any bounded solution of $\Delta u=0$ in $\R^n$ is constant.
\end{thm}

\begin{proof}
Let $u$ be any global bounded solution of $\Delta u=0$ in $\R^n$.
Since $u$ is smooth (by Corollary \ref{ch0-smooth}), each derivative $\partial_i u$ is well-defined and is harmonic too. 
Thus, thanks to the mean-value property and the divergence theorem, for any $x \in \R^n$ and $R \ge 1$ we have
\[
|\partial_i u(x)| = \left|\frac{c_n}{R^n}\int_{B_R(x)} \partial_i u \right| = 
 \left|\frac{c_n}{R^n}\int_{\partial B_R(x)} u(y) \frac{y_i}{|y|} \,dy\right|\le \frac{C}{R^{n}}\int_{\partial B_R(x)} |u| .
\]
Thus, using that $|u|\leq M$ in $\R^n$, we find
\begin{align*}
|\partial_i u(x)| & \le \frac{c_n}{R^n} |\partial B_R(x)| M \\
& =  \frac{c_n}{R^n}  |\partial B_1|R^{n-1} M= \frac{c_nM}{R} \to 0,\quad\textrm{ as }\quad R\to \infty. 
\end{align*}
Therefore, $\partial_i u(x) = 0$ for all $x\in \R^n$, and $u$ is constant. 
\end{proof}

More generally, one can even prove a classification result for functions with polynomial growth.
%
Here, for $\gamma\in\R$, $\lfloor \gamma\rfloor$ denotes the floor function, that is, the largest integer less or equal to $\gamma$. 

\begin{prop}[Liouville's theorem with growth]
\label{cor.Liouville}
Assume that $u$ is a solution of $\Delta u=0$ in $\R^n$ satisfying $|u(x)|\leq C(1+|x|^\gamma)$ for all $x\in\R^n$, with $\gamma>0$.
Then, $u$ is a polynomial of degree at most $\lfloor \gamma\rfloor$.
\end{prop}

\begin{proof}
Let us define $u_R(x) := u(Rx)$, and notice that $\Delta u_R = 0$ in $\R^n$. From Corollary~\ref{ch0-smooth} and the growth assumption
\begin{align*}
R^k \|D^{k} u\|_{L^\infty(B_{R/2})} & = \|D^{k} u_R\|_{L^\infty(B_{1/2})} \\
& \le C_k \|u_R\|_{L^\infty(B_1)} = C_k \|u\|_{L^\infty(B_R)} \le C_k R^\gamma.
\end{align*}
In particular, if $k = \lfloor \gamma \rfloor+1$,
\[
\|D^{k} u\|_{L^\infty(B_{R/2})} \le C_k R^{\gamma-k} \to 0\quad\textrm{as}\quad R\to \infty. 
\]
That is, $D^k u \equiv 0$ in $\R^n$, and $u$ is a polynomial of degree $k-1 = \lfloor\gamma\rfloor$. 
\end{proof}

%

\subsection*{Existence of solutions: comparison principle}

We saw that one way to prove existence of solutions to the Dirichlet problem for the Laplacian is by using \emph{energy methods}.
With such approach, one proves in fact the existence of a \emph{weak solution} $u\in H^1(\Omega)$.

Now, we will see an alternative way to construct solutions: via the \emph{comparison principle}.
With this method, one can show the existence of a \emph{viscosity solution} $u\in C(\overline\Omega)$.

For the Laplace equation, these solutions (weak or viscosity) can then be proved to be $C^\infty(\Omega)$, and thus they coincide.

We start by giving the definition of sub- and superharmonicity \emph{in the viscosity sense}. 
It is important to remark that in such definition the function $u$ is only required to be \emph{continuous}.

\begin{defi}\label{ch0-viscosity}\index{Viscosity solution}
 A function $u\in C(\overline\Omega)$ is \emph{subharmonic} (in the viscosity sense) if for every function $v\in C^2$ such that $v$ touches $u$ from above at $x_{\circ}\in \Omega$ (that is, $v\geq u$ in $\Omega$ and $v(x_{\circ})=u(x_{\circ})$), we have $\Delta v(x_{\circ})\geq0$. See Figure~\ref{fig.1}.

The definition of \emph{superharmonicity} for $u\in C(\Omega)$ is analogous (touching from below and with $\Delta v(x_\circ) \le 0$).

A function $u\in C(\Omega)$ is \emph{harmonic} if it is both sub- and superharmonic in the above viscosity sense.
\end{defi}
\begin{figure}
\includegraphics[width=\textwidth]{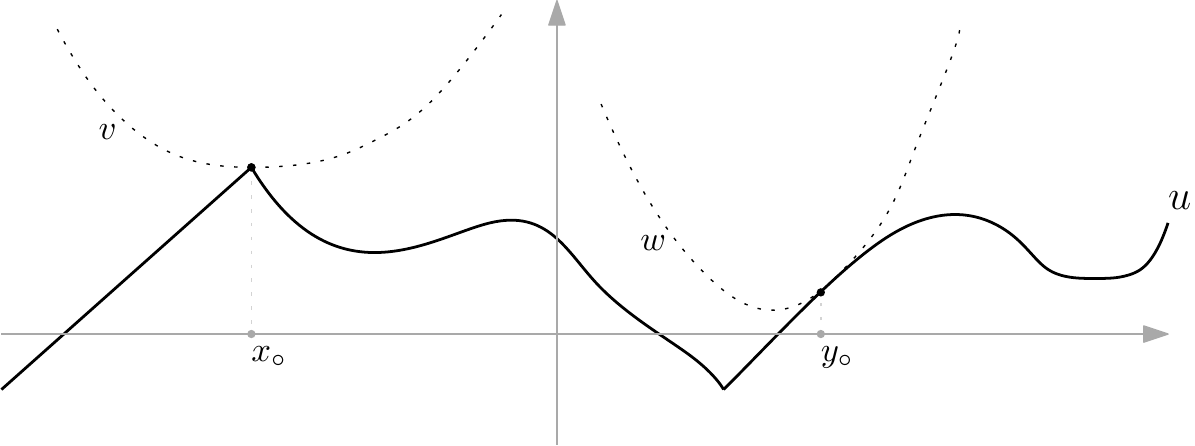}
\caption{$v$ touches $u$ from above at $x_\circ$, $w$ touches $u$ from above at $y_\circ$.}
\label{fig.1}
\end{figure}

\vspace{2mm}

This definition obviously coincides with the one we know in case $u\in C^2$.
However, it allows non-$C^2$ functions $u$, for example $u(x)=|x|$ is subharmonic and $-|x|$ is superharmonic.

A useful property of viscosity sub-/supersolutions is the following.

\begin{prop}
 The maximum of two subharmonic functions is also subharmonic.
 That is, if  $u_1,u_2\in C(\Omega)$ are subharmonic, then the function $v:=\max\{u_1,u_2\}$ is subharmonic as well. See Figure~\ref{fig.2}.
 
 Similarly, the minimum of two superharmonic functions is superharmonic.
\end{prop}

The proof follows easily from Definition \ref{ch0-viscosity} above, and it is left as an exercise to the reader.

Moreover, we also have the following:

\begin{prop}\label{max-princ-viscosity}
Let $\Omega\subset\R^n$ be a bounded domain, and assume that $u\in C(\overline\Omega)$ satisfies, in the viscosity sense,
\begin{equation*}
\left\{ 
\begin{array}{rcll}
-\Delta u & \geq & 0 &\text{in } \Omega\\
u &\geq& 0 &\text{on } \partial \Omega.
\end{array}\right.
\end{equation*}
Then, $u\geq0$ in $\Omega$.
\end{prop}

\begin{proof}
After a rescaling, we may assume $\Omega\subset B_1$.

Assume by contradiction that $u$ has a negative minimum in $\Omega$.
Then, since $u\geq0$ on $\partial \Omega$, we have $\min_{\overline \Omega} u=-\delta$, with  $\delta>0$, and the minimum is achieved in $\Omega$. 

Let us now consider $0<\varepsilon<\delta$, and $v(x):= -\kappa+\varepsilon(|x|^2-1)$, with $\kappa>0$ (that is, a sufficiently flat paraboloid).

Now, notice that $u-v>0$ on $\partial\Omega$, and that we can choose $\kappa>0$ so that $\min_{\overline \Omega} (u-v)=0$. That is, we can slide the paraboloid from below the solution $u$ until we touch it, by assumption, at an interior point.
Thus, there exists $x_\circ\in \Omega$ such that $u(x_\circ)-v(x_\circ)=\min_{\overline \Omega} (u-v)=0$.
Therefore, with such choice of $\kappa$, the function $v$ touches $u$ from below at $x_\circ\in\Omega$, and hence, by definition of viscosity solution, we must have
\[\Delta v(x_\circ)\leq0.\]
However, a direct computation gives $\Delta v\equiv 2n\varepsilon>0$ in $\Omega$, a contradiction.
\end{proof}

\begin{figure}
\includegraphics[width=\textwidth]{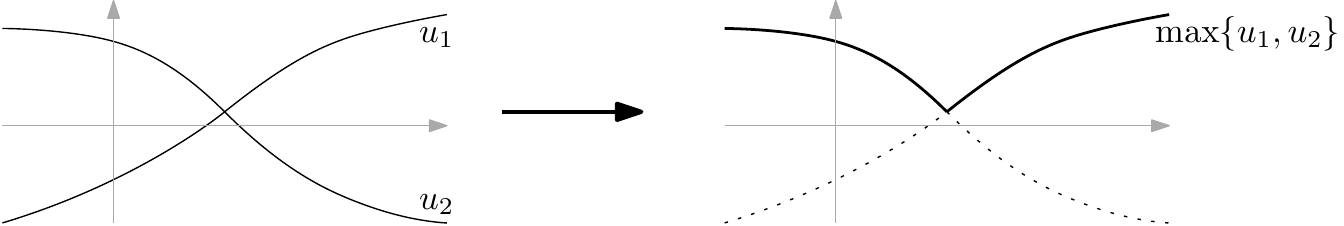}
\caption{The maximum of two functions $u_1$ and $u_2$.}
\label{fig.2}
\end{figure}

Thanks to these two propositions, the existence of a (viscosity) solution to the Dirichlet problem can be shown as follows.

Let
\[S_g:=\left\{v\in C(\overline\Omega)\,:\,v\ \textrm{is subharmonic, and}\ v\leq g\ \textrm{on}\ \partial\Omega\right\},\]
and define the pointwise supremum
\[u(x):=\sup_{v\in S_g} v(x).\]
Then, it can be shown that, if $\Omega$ is regular and $g$ is continuous, then $u\in C(\overline\Omega)$, and $\Delta u=0$ in $\Omega$, with $u=g$ on $\partial \Omega$.
This is the so-called \emph{Perron method}.
We refer to \cite{HL} for a complete description of the method in case of the Laplace operator.

In Chapter~\ref{ch.2} we will study the existence of viscosity solutions in the more general setting of fully nonlinear elliptic equations.

\subsection*{Short summary on existence of solutions}

We have \emph{two completely different ways} to construct solutions: by energy methods; or by the maximum (or comparison) principle.

In the first case, the constructed solution belongs to $H^1(\Omega)$, in the second case to $C(\overline\Omega)$.
In any case, one can then prove that $u\in C^\infty(\Omega)\cap C(\overline\Omega)$ --- as long as $\Omega$ and $g$ are regular enough --- and therefore $u$ solves the Dirichlet problem in the usual sense.

\section{Probabilistic interpretation of harmonic functions}
\label{sec.prob_interp}

To end this introductory chapter, we give a well-known probabilistic interpretation of harmonic functions.
The discussion will be mostly heuristic, just to give an intuition on the Laplace equation in terms of stochastic processes. We refer to Appendix~\ref{app.B} for further probabilistic interpretations for fully nonlinear equations and for the obstacle problem.

\vspace{2mm}

\index{Brownian motion}Recall that the \emph{Brownian motion} is a stochastic process $X_t$, $t\geq0$, satisfying the following properties:
\begin{enumerate}[(1)]
\item \label{it.1} $X_0 = 0$ almost surely.
\item \label{it.2} $X_t$ has \emph{no memory} (is independent of the past, or it has independent increments). 
\item \label{it.3} $X_t$ has \emph{stationary increments}: $X_{t+s}-X_s$ is equal in distribution to $X_{t}$.
\item \label{it.4} $X_t$ has \emph{continuous paths} ($t\mapsto X_t$ is continuous) almost surely. 
\item \label{it.5} $X_t$ is \emph{isotropic}, i.e., it is rotationally symmetric in distribution. 
\end{enumerate}
The previous properties actually determine the stochastic process $X_t$ up to a multiplicative constant.
Another important property of Brownian motion is that it is scale invariant, i.e., 
\begin{enumerate}
 \item[(6)] \label{it.6} $r^{-1}X_{r^2t}$ equals $X_t$ in distribution, for any $r>0$.
\end{enumerate}

As we will see next, there is a strong connection between the Brownian motion and the Laplace operator.

\subsection*{Expected payoff}
\index{Expected payoff}
Given a regular domain $\Omega\subset\R^n$, and a Brownian motion $X_t^x$ starting at $x$ (i.e., $X_t^x:=x+X_t$), we play the following stochastic game:
When the process $X_t^x$ hits the boundary $\partial\Omega$ for the first time we get a \emph{payoff} $g(z)$, depending on the hitting point $z\in \partial\Omega$. (See Figure~\ref{fig.3}.)

\begin{figure}
\includegraphics[scale = 1.4]{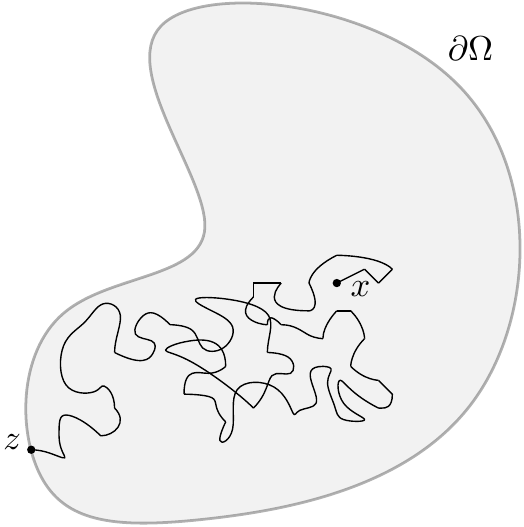}
\caption{A stochastic process $X_t^x$ defined in $\Omega$ starting at $x$ until it hits the first point on the boundary $z\in \de\Omega$.}
\label{fig.3}
\end{figure}

We then ask ourselves:
\[\textit{What is the expected payoff?}\]

To answer this question, we define 
\[
\begin{split}
&\tau:=\textrm{first hitting time of }X_t^x,\\
&u(x):=\mathbb E\left[g(X_\tau^x)\right]\qquad \textrm{(value function)}.
\end{split}
\]
The value of $u(x)$ is, by definition, the answer to the question above.
Namely, it is the expected value of $g$ at the first point where $X_t^x$ hits the boundary $\partial \Omega$.

To find $u(x)$, we try to relate it with values of $u(y)$ for $y\neq x$.
Then, we will see that this yields a \emph{PDE for} $u$, and by solving it we can find $u(x)$.

Indeed, let us consider a ball $B_r(x)\subset\Omega$, with $r>0$.
For any such ball, we know that the process $X_t^x$ will hit (before reaching  $\partial\Omega$, by property \ref{it.4}) some point on $\partial B_r(x)$, and moreover any point on $\partial B_r(x)$ will be hit with the \emph{same probability}.
This is because the process is \emph{rotationally symmetric in distribution}, \ref{it.5}.

Since the process has \emph{no memory}, \ref{it.2}, and \emph{stationary increments}, \ref{it.3}, this means that
\begin{equation}\label{ch0-brownian-mean}
u(x)=\ave_{\partial B_r(x)}u(y)dy.
\end{equation}
Heuristically, this is because when the process hits the boundary $\partial B_r(x)$ at a point $y$, it simply starts again the game from such point $y$.
But because all points $y\in \partial B_r(x)$ are reached for the first time with the same probability, then \eqref{ch0-brownian-mean} holds.

Now, since this can be done for every $x\in \Omega$ and $r>0$, we deduce that $u(x)$ satisfies the mean value property, and therefore it is harmonic, $\Delta u=0$ in $\Omega$, \eqref{Laplacian-radially}.

Moreover, since we also know that $u=g$ on $\partial \Omega$ (since when we hit the boundary we get the payoff $g$ surely), then $u$ must be the unique solution of
\[\left\{ 
\begin{array}{rcll}
\Delta u & = & 0 &\text{in } \Omega\\
u &=& g &\text{on } \partial\Omega.
\end{array}\right.\]

We refer to \cite{Law} for a nice introduction to this topic.

\subsection*{Expected hitting time}

A similar stochastic problem is the following.
Given a smooth domain $\Omega\subset \R^n$, and a Brownian motion $X_t^x$, we ask:
\index{Expected hitting time}\[\textit{What is the expected first time at which}\ X_t^x\ \textit{will hit}\ \partial\Omega\, \textrm{?}\]

To answer this question, we argue as before, using that the process must first hit the boundary of balls $B_r(x)\subset\Omega$.
Indeed, we first denote by $u(x)$ the expected hitting time that we are looking for.
Then, for any such ball we have that the process $X_t^x$ will hit (before reaching $\partial\Omega$) some point on $\partial B_r(x)$, and moreover any point $y\in\partial B_r(x)$ will be hit with the same probability.
Thus, the total expected time $u(x)$ will be the expected time it takes to hit $\partial B_r(x)$ for the first time, \emph{plus} the expected time when we start from the corresponding point $y\in \partial B_r(x)$, which is $u(y)$.
In other words, we have
\[u(x)=T(r)+\ave_{\partial B_r(x)}u(y)dy.\]
Here, $T(r)$ is the expected first time at which $X_t^x$ hits $\partial B_r(x)$ --- which clearly depends only on $r$ and $n$.

Now, using the scale-invariance property of the Brownian motion, i.e. $r^{-1}X_{r^2t}\sim X_t$, we see that $T(r)=T(1)r^2=c_1r^2$ for some constant $c_1>0$.
Thus, we have
\[u(x)=c_1r^2+\ave_{\partial B_r(x)}u(y)dy,\]
and by rearranging terms we find
\[-\frac{1}{r^2}\left\{\ave_{B_r(x)}u(y)dy-u(x)\right\}=c_1.\]
Finally, taking $r\to0$ and using \eqref{Laplacian-radially}, we deduce that $-\Delta u=c_2$, for some constant $c_2>0$.
Since we clearly have $u=0$ on $\partial\Omega$, the expected hitting time $u(x)$ is the unique solution of the problem
\[\left\{ 
\begin{array}{rcll}
-\Delta u & = & c_2 &\text{in } \Omega\\
u &=& 0 &\text{on } \partial\Omega.
\end{array}\right.\]

By considering a \emph{non-homogeneous} medium (in which it takes more time to move in some regions than others), the same argument leads to the problem with a right-hand side
\[\left\{ 
\begin{array}{rcll}
-\Delta u & = & f(x) &\text{in } \Omega\\
u &=& 0 &\text{on } \partial\Omega,
\end{array}\right.\]
with $f\geq0$.


%
%
%


\chapter{Linear elliptic PDE}
\label{ch.1}

In this chapter we will study linear elliptic PDEs of the type\index{Non-divergence-form PDE}
\begin{equation}
\label{eq.mainEPDE}
\boxed{{\rm tr} \big(A(x) D^2 u(x)\big) = \sum_{i,j=1}^n a_{ij}(x)\partial_{ij}u= f(x)\quad \textrm{in}\quad \Omega\subset\R^n,}
\end{equation}
as well as\index{Divergence-form PDE}
\begin{equation}
\label{eq.mainEPDE_div}
 \boxed{\divv\big(A(x) \nabla u(x) \big) = \sum_{i,j=1}^n \partial_i \big(a_{ij}(x)\partial_{j}u(x) \big)= f(x)\quad \textrm{in}\quad \Omega\subset\R^n.}
\end{equation}
These are elliptic PDEs in non-divergence and divergence form, respectively. 

The coefficients $(a_{ij}(x))_{ij}$ and the right-hand side $f(x)$ satisfy appropriate regularity assumptions. In addition, we will assume that the coefficient matrix $A(x) = (a_{ij}(x))_{ij}$ satisfies the {\em uniform ellipticity} condition\index{Uniform ellipticity condition}
\[
0<  \lambda\,{\rm Id}\le (a_{ij}(x))_{ij} \le\Lambda\,{\rm Id},
\]
for some ellipticity constants\index{Ellipticity constants!Linear equations} $0<\lambda \le \Lambda<\infty$. (For two matrices $A, B\in \mathcal{M}_n$, we say $A \ge B$ if the matrix $A-B$ is positive semi-definite.)

We will show that, under appropriate regularity assumptions on $A(x)$, solutions $u$ to \eqref{eq.mainEPDE} ``gain two derivatives'' with respect to $f$ and the coefficients $A(x)$. On the other hand, for the divergence-form equation, \eqref{eq.mainEPDE_div}, we expect solutions to ``gain one derivative'' with respect to the coefficients $A(x)$.

In order to do that,  we will use perturbative methods, by ``freezing'' the coefficients around a certain point and studying the constant coefficient equation first. After a change of variables, one can transform the constant coefficient equation into the most ubiquitous and simple elliptic equation: Laplace's equation, where $(a_{ij}(x))_{ij}$ is the identity. Thus, we will begin the chapter by studying properties of Laplace's equation such as Harnack's inequality and the H\"older regularity with bounded right-hand side. After that, we proceed by showing Schauder estimates for the Laplacian to continue with the main theorems of the current chapter: Schauder estimates for \eqref{eq.mainEPDE} and \eqref{eq.mainEPDE_div}.

We finish the chapter by studying equations of the type \eqref{eq.mainEPDE} and \eqref{eq.mainEPDE_div} with continuous coefficients. In this case we do not gain two (resp. one) derivatives, and instead we lose an arbitrarily small H\"older exponent of regularity.

Equations in non-divergence and divergence form will become particularly useful in Chapters~\ref{ch.2} and \ref{ch.3} in the context of nonlinear variational PDEs and fully nonlinear elliptic PDEs. 

For both equations in non-divergence and divergence form, we establish {\it a priori} estimates. That is, rather than proving that the solution is regular, we show that if the solution is regular, then one can actually estimate the norm of respectively two and one derivative higher in terms of the H\"older norms of the coefficients $(a_{ij}(x))_{ij}$ and the right-hand side $f$. This is enough for the application to nonlinear equations in  Chapters~\ref{ch.2} and~\ref{ch.3}.

When the operator is the Laplacian, thanks to the a priori estimates, and by means of an approximation argument, we show that weak solutions are in fact smooth. For more general elliptic operators, a priori estimates together with the continuity method yield the existence of regular solutions. 
We refer the reader to \cite{GT} for such an approach.

\section{Harnack's inequality}

We start this chapter with one of the most basic estimates for harmonic functions. It essentially gives a kind of ``maximum principle in quantitative form''. 

We will usually write that $u\in H^1$ is harmonic, meaning in the weak sense. Recall from the introduction, however, that as soon as a function is harmonic, it is immediately $C^\infty$.

\begin{thm}[Harnack's inequality]\index{Harnack inequality}
\label{thm.Harnack}
Assume $u\in H^1(B_1)$ is a non-negative, harmonic function in $B_1$. Then the infimum and the supremum of $u$  are comparable in $B_{1/2}$. That is,
\[
\left\{ 
\begin{array}{rcll}
\Delta u & = & 0 &\text{in } B_1\\
u &\geq& 0 &\text{in } B_1
\end{array}\right. \quad \Rightarrow\quad 
\sup_{B_{1/2}} u \le C\inf_{B_{1/2}} u,
\]
for some constant $C$ depending only on $n$.
\end{thm}
\begin{proof}
This can be proved by the mean value property. Alternatively, we can also use the Poisson kernel representation, 
\[
u(x) = c_n\int_{\partial B_1} \frac{(1-|x|^2)u(z)}{|x-z|^n} \, dz.
\]
Notice that, for any $x\in B_{1/2}$ and $z\in \de B_1$, we have $2^{-n} \le |x-z|^n \le (3/2)^{n}$ and $3/4\le 1-|x|^2\le 1$. Thus, since $u\ge 0$ in $B_1$, 
\[
C^{-1} \int_{\de B_1} u(z)\, dz \le u(x) \le C\int_{\de B_1} u(z) \, dz,\quad\textrm{for all}\quad x\in B_{1/2},
\]
for some dimensional constant $C$.  In particular, for any $x_1, x_2\in B_{1/2}$ we have that $u(x_1) \le C^2 u(x_2)$. Taking the infimum for $x_2\in B_{1/2}$ and  the supremum for $x_1\in B_{1/2}$, we reach that $\sup_{B_{1/2}} u \le \tilde C\inf_{B_{1/2}} u$, for some dimensional constant $\tilde C$, as desired. 
\end{proof}
\begin{rem}
\label{rem.harnack_rho}
This inequality says that, if $u\ge 0$ in $B_1$, then not only $u > 0$ in $B_{1/2}$ (strong maximum principle), but also we get quantitative information: $u \ge C^{-1}\sup_{B_{1/2}} u$ in $B_{1/2}$, for some constant $C$ depending only on $n$. See Figure~\ref{fig.4}.

\begin{figure}
\includegraphics[scale = 1.3]{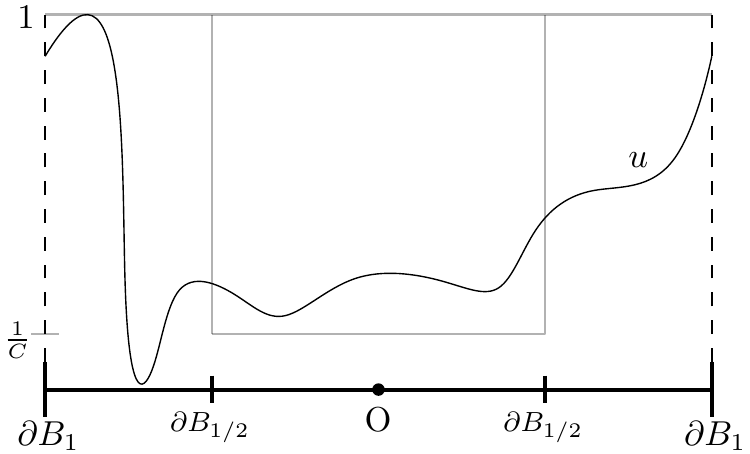}
\caption{Graphic representation of Harnack's inequality for a harmonic function $u > 0$ such that $\sup_{B_1} u = 1$. }
\label{fig.4}
\end{figure}

Notice that there is nothing special about $B_{1/2}$. We can obtain a similar inequality in $B_\rho$ with $\rho<  1$, but the constant $C$ would depend on $\rho$ as well.  Indeed, repeating the previous argument, one gets that if $\Delta u = 0$ and $u\ge 0$ in $B_1$, then 
\begin{equation}
\label{eq.Harnack_rho}
\sup_{B_\rho} u\le \frac{C}{(1-\rho)^n}\inf_{B_\rho} u,
\end{equation}
for some $C$ depending only on $n$, and where $\rho\in (0, 1)$. 
\end{rem}

From  Harnack's inequality, we deduce the oscillation decay for harmonic functions. That is, the oscillation of a harmonic function is reduced (quantitatively) in smaller domains. The oscillation in a domain $\Omega$ is defined as 
\[
\osc_{\Omega} u := \sup_{\Omega} u - \inf_{\Omega} u.
\]
We remark that the following lemma is valid for all harmonic functions, not necessarily positive.
\begin{cor}[Oscillation decay]\index{Oscillation decay}
\label{cor.osc_decay}
Let $u\in H^1(B_1)$ be a harmonic function in $B_1$, i.e. $\Delta u = 0$ in $B_1$. Then 
\[
\osc_{B_{1/2}} u \le (1-\theta)\osc_{B_1}  u
\]
for some small $\theta > 0$ depending only on $n$.
\end{cor}
\begin{proof}
Let 
\[
w(x) := u(x)- \inf_{B_1} u,
\]
which satisfies $w\ge 0$ in $B_1$ and $\osc_{B_{1/2}} w = \osc_{B_{1/2}} u$. Since $\Delta w = 0$ in $B_1$, we get by Harnack's inequality
\[
\sup_{B_{1/2}} w \le C\inf_{B_{1/2}} w,
\]
so that 
\[
\osc_{B_{1/2}} u   =\osc_{B_{1/2}} w = \sup_{B_{1/2}} w - \inf_{B_{1/2}} w \le \left(1 - \frac{1}{C}\right) \sup_{B_{1/2}} w \le \left(1 - \frac{1}{C}\right) \sup_{B_{1}} w.
\]
Now notice that $\sup_{B_1} w = \osc_{B_1} u$, and we are done. 
\end{proof}
\begin{rem}[Alternative proof of Corollary~\ref{cor.osc_decay}]
Alternatively, we can rewrite the previous proof of Corollary~\ref{cor.osc_decay} by taking advantage of the invariance of the estimate. 

Indeed, the function $u-\inf_{B_1} u$ is non-negative and harmonic. Since the estimate we want to prove is invariant under addition and multiplication by constants, we may assume that $\inf_{B_1} u = 0$ and $\sup_{B_1} u = 1$. Let $\theta := \frac{1}{C+1}$, where $C$ is the constant in Harnack's inequality, Theorem~\ref{thm.Harnack}. Now we have two options:
\begin{itemize}
\item If $\sup_{B_{1/2}} u \le 1-\theta$, we are done, 
\item If $\sup_{B_{1/2}} u \ge 1-\theta$ we use Harnack's inequality to get 
\[
\inf_{B_{1/2}} u \ge \frac{1}{C} (1-\theta)\ge\theta.
\]
\end{itemize}
In any case, we get $\osc_{B_{1/2}} u \le 1-\theta$, so we are done. 
\end{rem}

\begin{rem}
We have proved that Harnack's inequality implies the oscillation decay. This is {\em always} true, we did not use the fact that we are dealing with harmonic functions. In general, we have 
\[
\left(
\begin{array}{c}
\text{Harnack's}\\
\text{inequality}
\end{array}
\right)\Longrightarrow
\left(
\begin{array}{c}
\text{Oscillation}\\
\text{decay}
\end{array}
\right)\Longrightarrow
\left(
\begin{array}{c}
\text{H\"older}\\
\text{regularity}
\end{array}
\right)
\]
\end{rem}

Harnack's inequality and the oscillation decay are scale invariant. That is, the following corollary holds:

\begin{cor}[Rescaled versions]
\label{cor.independent_radius_Harnack}
Let $u\in H^1(B_r)$ be such that $\Delta u  = 0$ in $B_r$. Then
\begin{itemize}
\item {\rm (Harnack's inequality)} If $u \ge 0$ in $B_r$, then 
\[
\sup_{B_{r/2}} u \le C\inf_{B_{r/2}} u,
\]
for some $C$ depending only on $n$.
\item {\rm (Oscillation decay)} One has
\[
\osc_{B_{r/2}} u \le (1-\theta)\osc_{B_r} u,
\]
for some small $\theta> 0$ depending only on $n$.
\end{itemize}
\end{cor}
\begin{proof}
Define $\tilde u(x) := u(r x)$, which fulfills $\Delta \tilde u = 0$ in $B_1$ and therefore 
\[
\sup_{B_{r/2}} u  = \sup_{B_{1/2}} \tilde u \le C\inf_{B_{1/2}} \tilde u = C\inf_{B_{r/2}} u,
\]
by Theorem~\ref{thm.Harnack}. Similarly,
\[
\osc_{B_{r/2}} u = \osc_{B_{1/2}} \tilde u \le (1-\theta)\osc_{B_{1}} \tilde u = (1-\theta)\osc_{B_r}  u
\]
by Corollary~\ref{cor.osc_decay}.
\end{proof}

A standard consequence of the quantitative oscillation decay proved above  is the H\"older regularity of solutions. 

\begin{cor}[H\"older regularity]
\label{cor.Holder_regularity_1}
Let $u\in H^1(B_1)\cap L^\infty(B_1)$ be such that $\Delta u = 0$ in $B_1$. Then 
\[
\|u\|_{C^{0, \alpha}(B_{1/2})} \le C\|u\|_{L^\infty(B_1)}
\]
for some constants $\alpha > 0$ and $C$ depending only on $n$.
\end{cor}
\begin{proof}
If we denote $\tilde u := (2\|u\|_{L^\infty(B_1)})^{-1} u$, then $\tilde u \in H^1\cap L^\infty(B_1)$ fulfills $\Delta \tilde u = 0$ in $B_1$ and $\|\tilde u\|_{L^\infty(B_1)}\le \frac12$. If we show $\|\tilde u\|_{C^{0,\alpha}(B_{1/2})}\le C$, then the result will follow. 

Thus, dividing $u$ by a constant if necessary, we may assume that $\|u\|_{L^\infty(B_1)}\le \frac12$. We need to prove that
\[
|u(x) - u(y)|\le C |x-y|^\alpha\quad\textrm{for all}\quad x, y \in B_{1/2},
\]
for some small $\alpha > 0$. We do it at $y = 0$ for simplicity.

Let $x\in B_{1/2}$ and let $k\in \mathbb{N}$ be such that $x\in B_{2^{-k}}\setminus B_{2^{-k-1}}$. Then, 
\[
|u(x) - u(0)|\le \osc_{B_{2^{-k}}} u \le (1-\theta)^k \osc_{B_1} u \le (1-\theta)^k  = 2^{-\alpha k},
\]
with $\alpha = -\log_2 (1-\theta)$. (Notice that we are using Corollary~\ref{cor.independent_radius_Harnack} $k$-times, where the constant $\theta$ is independent from the radius of the oscillation decay.)

Now, since $2^{-k}\le 2|x|$, we find 
\[
|u(x)-u(0)|\le (2|x|)^\alpha \le C|x|^\alpha,
\]
as desired. See Figure~\ref{fig.5} for a graphical representation of this proof.
\begin{figure}
\includegraphics[width=\textwidth]{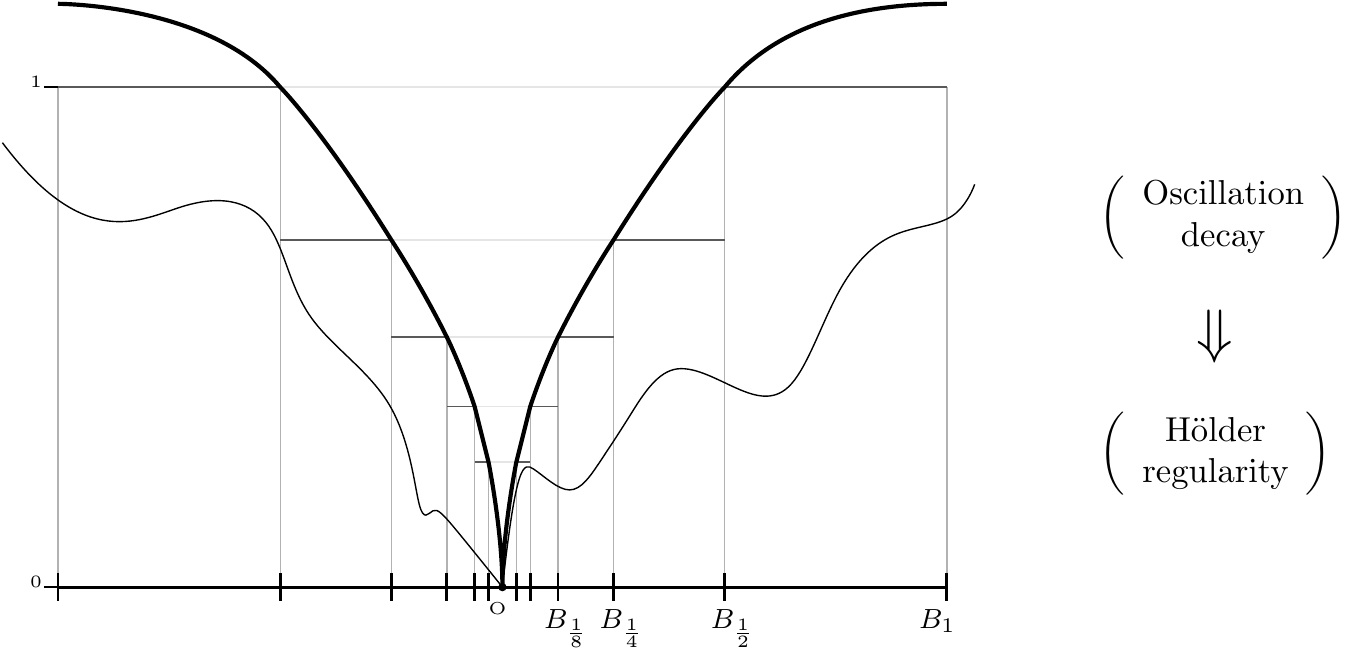}
\caption{Graphical representation of the fact that oscillation decay-type lemmas imply H\"older regularity.}
\label{fig.5}
\end{figure}
\end{proof}

Finally, another important consequence of Harnack's inequality is the Liouville theorem for non-negative harmonic functions. 

\begin{cor}
\label{cor.Liouv_pos}\index{Liouville Theorem}
Let $u$ be a non-negative harmonic function, that is, $u\ge 0$ and $\Delta u = 0$ in $\R^n$. Then, $u$ is constant. 
\end{cor}
\begin{proof}	
Let 
\[
v = u - \inf_{\R^n} u,
\]
where $\inf_{\R^n} u$ is well-defined and finite since $u \ge 0$. Then, thanks to Harnack's inequality in arbitrary balls from Corollary~\ref{cor.independent_radius_Harnack}, we get that for any $R > 0$, 
\[
\sup_{B_{R}} v \le C\inf_{B_R} v  = C\left( \inf_{B_{R}} u - \inf_{\R^n} u\right) \to 0,
\]
as $R\to \infty$. That is, $\sup_{\R^n} u = \inf_{\R^n} u$ and therefore $u$ is constant in $\R^n$.  
\end{proof}

Of course, the previous result also holds if $u \ge -M$ in $\R^n$, for some constant $M$, since then $u + M$ is non-negative and harmonic. 

\subsection*{Harnack's inequality with a right-hand side}

We can prove a Harnack inequality for equations with a right-hand side, that is, when the Laplacian is not necessarily zero, $\Delta u = f$. Again, we will be dealing with functions $u\in H^1$, so that we have to understand the equation $\Delta u = f $ in the weak sense. 

\begin{thm}\label{thm.Harnack.f}\index{Harnack's inequality}
Let $f\in L^\infty(B_1)$, and $u\in H^1(B_1)$. Then,
\[
\left\{
\begin{array}{rcll}
\Delta u & = & f &\text{in } B_1\\
u &\geq& 0 &\text{in } B_1
\end{array}\right. \quad \Rightarrow\quad 
\sup_{B_{1/2}} u \le C\left(\inf_{B_{1/2}} u+\|f\|_{L^\infty(B_1)}\right),
\]
for some $C$ depending only on $n$.
\end{thm}

\begin{proof}
We express $u$ as $u = v+w$ with 
\[
\left\{\begin{array}{rcll}
\Delta v & = & 0 &\text{in } B_1\\
v &=& u &\text{on } \partial B_1
\end{array}
\right.
\qquad\qquad
\left\{\begin{array}{rcll}
\Delta w & = & f &\text{in } B_1\\
w &=& 0 &\text{on } \partial B_1.
\end{array}
\right.
\]
Then, we have 
\[
\sup_{B_{1/2}} v \le C\inf_{B_{1/2}} v\qquad\textrm{and}\qquad \sup_{B_1} w \le C\|f\|_{L^\infty(B_1)}
\]
by Theorem~\ref{thm.Harnack} and Lemma~\ref{lem.maxPrinciple}. Thus,
\begin{align*}
\sup_{B_{1/2}} u & \le \sup_{B_{1/2}} v + C\|f\|_{L^\infty(B_1)}\\
& \le C\inf_{B_{1/2}} v +C\|f\|_{L^\infty(B_1)}\le C\left(\inf_{B_{1/2}} u +\|f\|_{L^\infty(B_1)}\right),
\end{align*}
where we are taking a larger constant if necessary.
Notice that we have also used here that $v \le u + C\|f\|_{L^\infty(B_1)}$. 
\end{proof}

Thus, as before, we also get an oscillation decay, but now involving an error term of size $\|f\|_{L^\infty}$. 

\begin{cor}
\label{cor.osc_decay_2}\index{Oscillation decay}
Let $f\in L^\infty(B_1)$ and $u\in H^1(B_1)$. If $\Delta u = f$ in $B_1$ and $f\in L^\infty(B_1)$, then 
\[
\osc_{B_{1/2}} u \le (1-\theta)\osc_{B_1} u + {2}\|f\|_{L^\infty(B_1)},
\]
for some $\theta > 0$ depending only on $n$. 
\end{cor}
\begin{proof}
The proof is the same as in the case $f \equiv 0$, see the proof of Corollary~\ref{cor.osc_decay}.
\end{proof}
\begin{rem}
Now, with the right-hand side $f = f(x)$, the equation $\Delta u  =f$ and Harnack's inequality are {\em not} invariant under rescalings in $x$. In fact, as we zoom-in, the right-hand side gets smaller!

Namely, if $\Delta u = f$ in $B_r$, then $\tilde u(x) := u(rx)$ satisfies $\Delta \tilde u(x) = r^2 f(rx)$ in $B_1$ so that
\[
\sup_{B_{1/2}} \tilde u \le C \left(\inf_{B_{1/2}} \tilde u + 2 r^2 \|f\|_{L^\infty(B_r)}\right),
\]
and therefore
\[
\sup_{B_{r/2}} u \le C \left(\inf_{B_{r/2}}  u + 2 r^2 \|f\|_{L^\infty(B_r)}\right)
\]
for some constant $C$ depending only on $n$. 
\end{rem}

Even if the previous oscillation decay contains an error depending on $f$, it is enough to show H\"older regularity of the solution.

\begin{cor}[H\"older regularity]
\label{cor.Hold_cont_f}
Let $f\in L^\infty(B_1)$ and $u\in H^1\cap L^\infty(B_1)$. If $\Delta u = f$ in $B_1$, then 
\[
\|u\|_{C^{0, \alpha}(B_{1/2})}\le C \left(\|u\|_{L^\infty(B_1)}+\|f\|_{L^\infty(B_1)}\right),
\]
for some constants $\alpha > 0$ and $C$ depending only on $n$. 
\end{cor}
\begin{proof}
If we denote $\tilde u := (2\|u\|_{L^\infty(B_1)}+2\|f\|_{L^\infty(B_1)})^{-1} u$, then $\tilde u \in H^1(B_1)\cap L^\infty(B_1)$ fulfills $\Delta \tilde u = \tilde f$ in $B_1$ with $\|\tilde u\|_{L^\infty(B_1)}\le \frac12$ and $\|\tilde f \|_{L^\infty(B_1)}\le \frac12$. If we show that $\|\tilde u\|_{C^{0,\alpha}(B_{1/2})}\le C$, then the result will follow.

Thus, after dividing $u$ by a constant if necessary, we may assume that $\|u\|_{L^\infty(B_1)}\le \frac12$ and $\|f\|_{L^\infty(B_1)}\le \frac12$.

As in Corollary~\ref{cor.Holder_regularity_1} we want to prove that $|u(x_\circ)-u(0)|\le C|x_\circ|^\alpha$ for all $x_\circ\in B_{1/2}$ and for some constant $C$ depending only on $n$. 

Let us show that it is enough to prove that 
\begin{equation}
\label{eq.osc_it}
\osc_{B_{2^{-k}}} u \le C2^{-\alpha k}\quad\textrm{for all }k\in \N,\,k\ge k_\circ,
\end{equation}
for some $\alpha > 0$, $C$, and for some fixed $k_\circ$, all three  depending only on $n$. Indeed, let $k\in \N$ be such that $x\in B_{2^{-k}}\setminus B_{2^{-k-1}}$. If $k < k_\circ$, then $|x|\ge 2^{-k_\circ-1}$ and
\[
|u(x)- u(0)|\le \osc_{B_1} u \le 1  \le 2^{\alpha(k_\circ+1)} |x|^\alpha.
\]

On the other hand, if $k \ge k_\circ$, by \eqref{eq.osc_it}
\[
|u(x)-u(0)|\le \osc_{B_{2^{-k}}} u \le C2^{-\alpha k} \le C (2|x|)^\alpha,
\]
where in the last inequality we used that $|x| \ge 2^{-k-1}$. Thus, it will be enough to show \eqref{eq.osc_it}.

Let $k\in \N$ and $k \ge k_\circ$ for some $k_\circ$ to be chosen, and define 
\[
\tilde u (x) := u(rx),\quad r = 2^{-k+1}.
\]
Then $\tilde u$ satisfies $\Delta \tilde u = r^2 f(rx)$ in $B_1$ (in fact, in $B_{2^{k-1}}$), and thus, by Corollary~\ref{cor.osc_decay_2}
\[
\osc_{B_{1/2}} \tilde u \le (1-\theta)\osc_{B_1} \tilde u + 2 r^2\|f\|_{L^\infty(B_r)}.
\]
Since $\osc_{B_1} \tilde u = \osc_{B_{2^{-k+1}}} u $ and $\|f\|_{L^\infty(B_r)}\le \frac12$, we find 
\[
\osc_{B_{2^{-k}}} u \le (1-\theta)\osc_{B_{2^{-k+1}}} u + 4^{-k+2}.
\]

Now, take $k_\circ\in \N$ large enough so that $4^{-k_\circ+1}\le\frac{\theta}{2}$. Then, 
\[
\osc_{B_{2^{-k}}} u \le (1-\theta)\osc_{B_{2^{-k+1}}} u + \frac{\theta}{2} 4^{k_\circ-k}.
\]
It is immediate to check by induction that this yields 
\[
\osc_{B_{2^{-k+1}}} u \le 2^{\alpha(k_\circ-k)},\quad\textrm{for all }k \in \N. 
\]
Indeed, the induction step follows as 
\begin{align*}
\osc_{B_{2^{-k}}} u& \le (1-\theta)2^{\alpha(k_\circ-k)}+ \frac{\theta}{2}4^{k_\circ-k} \le (1-\theta)2^{\alpha(k_\circ-k)}+ \frac{\theta}{2}2^{\alpha(k_\circ-k)} \\
& = \left(1-\frac{\theta}{2}\right)2^{\alpha(k_\circ-k)} = 2^{\alpha(k_\circ-k-1)} \qquad\textrm{if}\quad1-\frac{\theta}{2} = 2^{-\alpha}.
\end{align*}
Thus, \eqref{eq.osc_it} holds with $C = 2^{\alpha k_\circ}$.
\end{proof}

Summarizing, we have checked that \underline{\smash{Harnack's inequality}} for harmonic functions yields the \underline{\smash{H\"older regularity}} of solutions, even with a right-hand side $f \in L^\infty$. 

This is a general fact, and holds for other types of elliptic equations, too. 

\section{Schauder estimates for the Laplacian}
\label{sec.Sch_Lap}
We now want to establish sharp results for the equation 
\[
\boxed{\Delta u = f(x) \quad\textrm{in}\quad B_1}\qquad \left(\textrm{or in } \Omega\subset \R^n\right).
\]
This will serve as an introduction for the more general case of equations in non-divergence and divergence form. 

The philosophy is that the sharp results should state that ``$u$ is two derivatives more regular than $f(x)$''. 

The main known results in that directions are the following:
\begin{enumerate}[(a)]
\item \underline{\smash{Schauder estimates.}} If $f\in C^{0, \alpha}$ then $u\in C^{2, \alpha}$, for $\alpha \in (0, 1)$. 
\item \underline{\smash{Calder\'on--Zygmund estimates.}} If $f\in L^p$ then $u\in W^{2, p}$, for $p\in (1, \infty)$. 
\item When $\alpha $ is an integer, or when $p \in \{1, \infty\}$, the above results do {\em not} hold. For example, if $f\in C^0$, it is not true in general that $u \in C^2$, not even $C^{1, 1}$. (In that case, $u\in C^{1, 1-\eps}$  for all $\eps > 0$, and $u \in W^{2, p}$ for all $p < \infty$.)
\end{enumerate}

\subsection*{Two counterexamples}

Let us provide two counterexamples to show that Schauder and Calder\'on--Zygmund estimates in general do not hold for the limiting values, $\alpha = 0$ and $p = 1$ or $p = \infty$. 

We start with an example of a function $u$ whose Laplacian is bounded ($\Delta u \in L^\infty$), but whose second derivatives are {\em not} bounded ($u\notin W^{2,\infty}$). Thus, we give a counterexample to Calder\'on--Zygmund estimates for $p = \infty$. 

Let 
\[
u(x, y) = (x^2-y^2) \log(x^2+y^2) \quad\textrm{in}\quad \R^2. 
\]
Then, 
\begin{align*}
\partial_{xx} u&  = 2\log(x^2+y^2) + \frac{8x^2}{x^2+y^2} - 2\left(\frac{x^2-y^2}{x^2+y^2}\right)^2,\\
\partial_{yy} u& = -2\log(x^2+y^2) - \frac{8y^2}{x^2+y^2} + 2\left(\frac{x^2-y^2}{x^2+y^2}\right)^2,
\end{align*}
that is, both $\partial_{xx} u$ and $\partial_{yy} u$ are unbounded, and $u\notin W^{2, \infty}$. However, 
\[
\Delta u = \partial_{xx} u +\partial_{yy} u = 8\frac{x^2-y^2}{x^2+y^2} \in L^\infty(\R^2).
\]

One can modify such construction in order to make $\Delta u $ continuous and $u\notin C^{1, 1}$, thus giving a counterexample to Schauder estimates for $\alpha = 0$, by taking $u(x, y) = (x^2-y^2) \log |\log (x^2+y^2)|$. (However, recall that Schauder estimates tell us that this is \emph{not} possible if $\Delta u$ is H\"older continuous ($C^{0, \alpha}$).) 

Let us now provide a counterexample for Calder\'on--Zygmund estimates when $p = 1$. The fact that the estimate does not hold can be seen by taking smooth approximations of the Dirac delta (with constant integral) as right-hand side, so that the solution converges to the fundamental solution, which is not in $W^{2,1}$. 

Let us, however, give a specific example of a function $u$ whose Laplacian is integrable ($\Delta u \in L^1$) but whose second derivatives are not ($u\notin W^{2,1}$). 

Let 
\[
u(x, y) = \log \log \frac{1}{x^2 +y^2} = \log \log r^{-2}\quad\textrm{in}\quad \R^2,
\]
where we are using polar coordinates and denote $r^2 := x^2+y^2$. Since $u= u(r)$ and $u_r = \frac{1}{r\log r}$ we have that
\[
\Delta u = u_{rr} + \frac{1}{r} u_r = -\frac{\log r + 1}{r^2 (\log r)^2} + \frac{1}{r^2\log r} = -\frac{1}{r^2(\log r)^2} \in L^1(B_{1/2}),
\]
since $\int_{B_{1/2}} \Delta u = -2\pi \int_0^{1/2} \frac{dr}{r(\log r)^2} <\infty$. On the other hand, a direct computation gives that $\de_{xx} u$ (and $\de_{yy} u$) are not absolutely integrable around the origin, and thus $u\notin W^{2,1}$. (Alternatively, since one has the embedding $W^{2,1}(\R^2)\subset L^\infty(\R^2)$ \cite[Corollary 9.13]{Brezis} and $u\notin L^\infty$, we deduce $u\notin W^{2,1}$). 

A similar counterexample can be built in any dimension $n\ge 2$, by taking as function $u$ an appropriate primitive of $\frac{r^{1-n}}{\log r}$. 

In this book we focus our attention on proving (a) Schauder estimates, but not (b) Calder\'on--Zygmund estimates. Later in the book we will see applications of Schauder-type estimates to nonlinear equations. 

\begin{rem}[Calder\'on-Zygmund estimates for $p = 2$]
\label{rem.CZ2}
In the case $p = 2$, one can prove a priori Calder\'on-Zygmund estimates with a simple computation. 
That is, let $u, f\in C^\infty(B_1)$, be such that 
\[
\Delta u = f\quad\text{in}\quad B_1.
\]
Then,
\begin{equation}
\label{eq.calderon2}
\|u\|_{W^{2,2}(B_{1/2})} \le C \left(\|u\|_{L^{2}(B_1)} + \|f\|_{L^2(B_1)}\right)
\end{equation}
for some constant $C$ depending only on $n$. Indeed, let $w := \eta u$ for some fixed test function $\eta\in C^\infty_c(B_1)$ such that $\eta \equiv 1$ in $B_{1/2}$, $\eta \equiv 0$ in $B_1\setminus B_{3/4}$ and $\eta \ge 0$. Then, integrating by parts gives
\begin{align*}
\|D^2 u\|_{L^2(B_{1/2})} & = \sum_{i,j=1}^n \int_{B_{1/2}} |D_{ij}^2 u|^2   \le \sum_{i,j=1}^n \int_{B_{1}} |D_{ij}^2 w|^2 \\
& = - \sum_{i,j=1}^n \int_{B_{1}} (D_{iij} w) (D_j w) = \sum_{i,j=1}^n \int_{B_{1/2}} (D_{ii} w) (D_{jj} w)\\
& = \int_{B_1} (\Delta w)^2 \le C\int_{B_1} \left(u^2 + (\Delta u)^2 + |\nabla \eta|^2|\nabla u|^2\right),
\end{align*}
where in the last equality we can take $ C = C'\sup_{B_1} \left(\eta^2  + |\Delta \eta|^2\right)$ for some dimensional constant $C'$. Then, again integrating by parts twice and using $2ab \le a^2+b^2$, we get
\begin{align*}
\int_{B_1} |\nabla \eta|^2|\nabla u|^2 & = - \int_{B_1} |\nabla \eta|^2 u \Delta u + \int_{B_1} \frac12 u^2 \Delta |\nabla \eta|^2 \le \tilde C \int_{B_1} \big(u^2 + (\Delta u)^2\big),
\end{align*}
where $\tilde C =  \sup_{B_1} \big(|\nabla \eta|^2 + \Delta |\nabla\eta|^2\big)$.

This directly yields the result \eqref{eq.calderon2} for smooth functions $u$ and $f$ such that $\Delta u  = f$. Arguing as in the proof of Corollary~\ref{cor.Schauder_estimates_L} below, the same result also holds as long as $u\in H^1(B_1)$  is a weak solution to  $\Delta u = f$ in $B_1$ for $f\in L^2(B_1)$ .
\end{rem}

\subsection*{Proofs of Schauder estimates: some comments}

There are various proofs of Schauder estimates, mainly using: 
\begin{enumerate}
\item integral representation of solutions (fundamental solutions);
\item energy considerations;
\item comparison principle. 
\end{enumerate}

The most flexible approaches are (2) and (3). Here, we will see different proofs of type (3). 

The {\em common traits} in proofs of type (2)-(3) are their ``perturbative character'', that is, that by zooming in around any point the equation gets closer and closer to $\Delta u = {\rm constant}$, and thus (after subtracting a paraboloid) close to $\Delta u = 0$. Thus, the result can be proved by using the information that we have on harmonic functions.

Let us start by stating the results we want to prove in this section: Schauder estimates for the Laplacian.

\begin{thm}[Schauder estimates for the Laplacian]\index{Schauder estimates!Laplacian}
\label{thm.Schauder_estimates_L}
Let $\alpha \in (0, 1)$, and let $u\in C^{2,\alpha}(B_1)$ satisfy
\[
\Delta u = f\quad\textrm{in } B_1,
\]
with $f\in C^{0, \alpha}(B_1)$. Then
\begin{equation}
\label{eq.Schauder_estimates_L}
\|u\|_{C^{2, \alpha}(B_{1/2})} \le C\left(\|u\|_{L^\infty(B_1)}+\|f\|_{C^{0, \alpha}(B_1)}\right).
\end{equation}
The constant $C$ depends only on $\alpha$ and the dimension $n$. 
\end{thm}

We will, in general, state our estimates in balls $B_{1/2}$ and $B_1$. By means of a covering argument explained below, this allows us to obtain interior regularity estimates in general domains. 

\begin{rem}[Covering argument]\index{Covering argument}
\label{rem.covering_argument}
Let us assume that we have an estimate, like the one in \eqref{eq.Schauder_estimates_L}, but in a ball $B_{r_1}$ for some $r_1 \in (0, 1)$, which will be typically very close to zero. Namely, we know that if $\Delta u = f$ in $B_1$, then 
\begin{equation}
\label{eq.estrhorho}
\|u\|_{C^{2, \alpha}(B_{r_1})} \le C\left(\|u\|_{L^\infty(B_1)}+\|f\|_{C^{0, \alpha}(B_1)}\right).
\end{equation}
Let us suppose that we are interested in finding an estimate for a bigger ball, $B_{r_2}$ with $r_1 <r_2 \in (0, 1)$, where ${r_2}$ will be typically close to one. We do that via a ``covering argument''. (See Figure~\ref{fig.6}.)

\begin{figure}
\includegraphics[scale = 0.9]{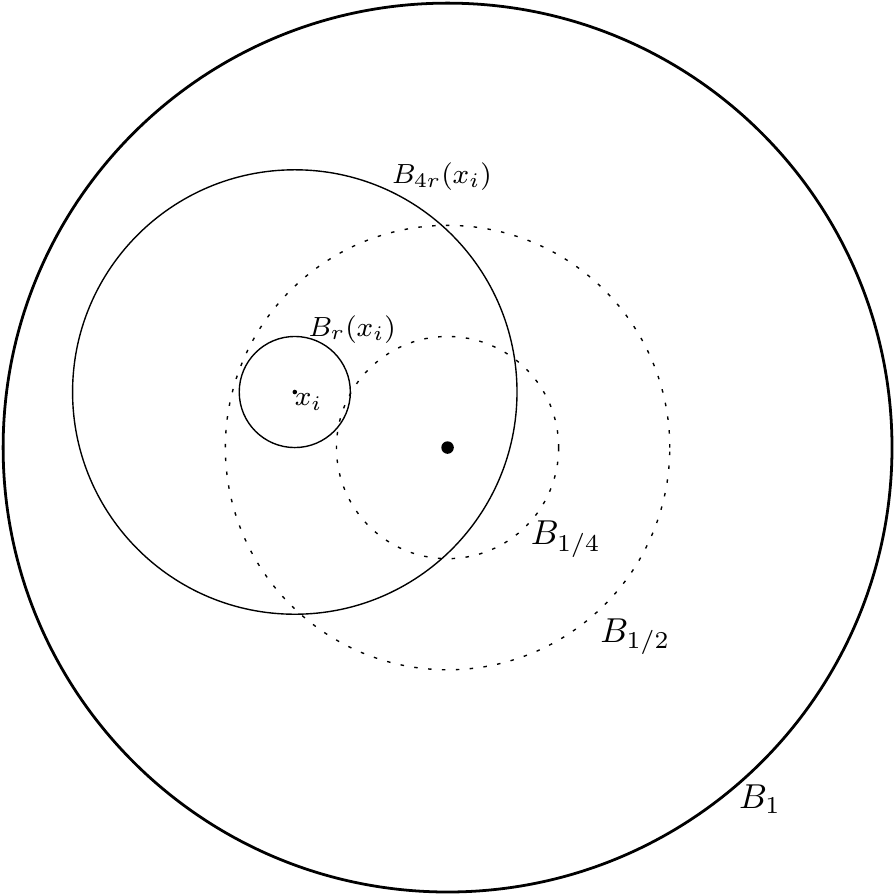}
\caption{Graphical representation of the ``covering argument'' in the case $r_1 = \frac14$, $r_2 = \frac12$, and $r = \frac18$. }
\label{fig.6}
\end{figure}

That is, let us cover the ball $B_{r_2}$ with smaller balls $B_{r}(x_i)$ such that $x_i \in B_{{r_2}}$ and $ r = (1-{r_2})r_1$. We can do so with a finite number of balls, so that $i \in \{1,\dots, N\}$, for some $N$ depending on $r_1$, ${r_2}$, and $n$. Notice that $B_{r/r_1}(x_i)\subset B_1$.

We apply our estimate \eqref{eq.estrhorho} (translated and rescaled) at each of these balls $B_{r/r_1}(x_i)$ (we can do so, because  $\Delta u = f$ in $B_{r/r_1}(x_i)\subset B_1$). Thus, we obtain a bound for $\|u\|_{C^{2, \alpha}(B_r(x_i))}$
\begin{align*}
\|u\|_{C^{2, \alpha}(B_{r}(x_i))} & \le C(r_1, {r_2})\left(\|u\|_{L^\infty(B_{r/r_1}(x_i))}+\|f\|_{C^{0, \alpha}(B_{r/r_1}(x_i))}\right)\\
& \le C(r_1, {r_2})\left(\|u\|_{L^\infty(B_{1})}+\|f\|_{C^{0, \alpha}(B_{1})}\right).
\end{align*}
Now, since $B_{r_2}$ can be covered by a finite number of these balls, we obtain 
\[
\|u\|_{C^{2, \alpha}(B_{r_2})} \le \sum_{i = 1}^n \|u\|_{C^{2, \alpha}(B_{r}(x_i))} \le  NC(r_1, {r_2})\left(\|u\|_{L^\infty(B_{1})}+\|f\|_{C^{0, \alpha}(B_{1})}\right).
\]
This is the type of bound we wanted, where the constant now also depends on $r_1$ and $r_2$. 
%
%
%
%
\end{rem}

As a consequence of the ``a priori'' estimate for the Laplacian we will show:
\begin{cor}
\label{cor.Schauder_estimates_L}
Let $u$ be any bounded weak solution to
\[
\Delta u = f\quad\textrm{in } B_1,
\]
with $f\in C^{0, \alpha}(B_1)$  for some $\alpha\in(0,1)$. Then, $u$ is in $ C^{2, \alpha}$ inside $B_1$, and the estimate \eqref{eq.Schauder_estimates_L} holds.
\end{cor}

Furthermore, iterating the previous estimate we will establish the following. 

\begin{cor}[Higher order regularity estimates]\index{Higher order Schauder estimates!Laplacian}
\label{cor.Schauder_estimates_L_HO}
Let $u$ be any bounded weak solution to
\[
\Delta u = f\quad\textrm{in } B_1,
\]
with $f\in C^{k, \alpha}(B_1)$ for some $\alpha\in(0,1)$, and $k\in\N$. Then, $u$ is in $ C^{k+2, \alpha}$ inside $B_1$ and 
\[
\|u\|_{C^{k+2, \alpha}(B_{1/2})} \le C\left(\|u\|_{L^\infty(B_1)}+\|f\|_{C^{k, \alpha}(B_1)}\right),
\]
for some constant $C$ that depends only on $k$, $\alpha$, and the dimension $n$. 
\end{cor}


In case $f\in L^\infty$, we will prove the following. 

\begin{prop}
\label{prop.Schauder_estimates_L_bounded}
Let $u$ be any solution to
\[
\Delta u = f\quad\textrm{in } B_1,
\]
with $f\in L^\infty(B_1)$. Then, $u$ is in $ C^{1, 1-\eps}$ inside $B_1$, for any $\eps > 0$, with the estimate 
\[
\|u\|_{C^{1, 1-\eps}(B_{1/2})} \le C_\eps \left(\|u\|_{L^\infty(B_1)} + \|f\|_{L^\infty(B_1)}\right)
\]
for some constant $C_\eps$ depending only on $\eps$ and $n$.
\end{prop}

We will give two different proofs of Theorem~\ref{thm.Schauder_estimates_L}. The first proof follows a method introduced by Wang in \cite{Wan} and shows the a priori estimate using a very much self-contained approach. For the second proof we use an approach {\it \`a la} Caffarelli from \cite{M, Caf89}.

Before doing so, let us observe the following:

\begin{itemize}
\item If $\Delta u = f\in L^\infty$ then $\tilde u(x) := u(rx)$ solves $\Delta \tilde u = r^2f(rx)$. In other words, if $|\Delta u |\le C$, then $|\Delta\tilde u|\le Cr^2$ (and if $r$ is small, the right-hand side becomes smaller and smaller). 
\item If $\Delta u = f\in C^{0, \alpha}$, $\tilde u(x) = u(rx) - \frac{f(0)}{2n}|x|^2$ solves $\Delta \tilde u  = r^2 (f(rx)-f(0))$, so that $|\Delta \tilde u |\le Cr^{2+\alpha}$ in $B_1$. This, by the comparison principle, means that $\tilde u$ is ``very close'' to a harmonic function. 
\end{itemize}

Let us now show that Corollary~\ref{cor.Schauder_estimates_L} holds assuming Theorem~\ref{thm.Schauder_estimates_L}. This follows by an approximation argument.

\begin{proof}[Proof of Corollary~\ref{cor.Schauder_estimates_L}]
We will deduce the result from Theorem~\ref{thm.Schauder_estimates_L}.  Let $u$ be any solution to $\Delta u = f$ in $B_1$, with $f\in C^{0,\alpha}(B_1)$, and let $\eta\in C^\infty_c (B_1)$ be any smooth function with $\eta \ge 0$ and $\int_{B_1}\eta = 1$. Let
\[
\eta_\eps (x) := \eps^{-n} \eta\left(\frac{x}{\eps}\right),
\]
which satisfies $\int_{B_{\eps}} \eta_\eps = 1$, $\eta_\eps\in C^\infty_c(B_\eps)$. Consider the convolution 
\[
u_\eps(x) := u * \eta_\eps(x) = \int_{B_\eps} u(x-y)\eta_\eps(y)\, dy,
\]
which is $C^\infty$ and satisfies 
\[
\Delta u_\eps = f*\eta_\eps =: f_\eps \quad\textrm{in}\quad B_{1-\eps}. 
\]
(Notice that for smooth functions, derivatives and convolutions commute; the same can be done for weak derivatives.) Since $u_\eps \in C^\infty$, we can use Theorem~\ref{thm.Schauder_estimates_L} to get 
\[
\|u_\eps\|_{C^{2, \alpha}(B_{1/2})} \le C\left(\|u_\eps\|_{L^\infty(B_{1-\eps})}+\|f_\eps\|_{C^{0, \alpha}(B_{1-\eps})}\right),
\]
where we are also using the covering argument in Remark~\ref{rem.covering_argument} to write it in a ball $B_{1-\eps}$ in the right-hand side. Observe now that for any $x,y\in B_{1-\eps}$
\[
|u_\eps(x)| \le \int_{B_\eps} |u(x-z)|\eta_\eps(z)\, dy \le \|u\|_{L^\infty(B_1)}\int_{B_\eps}\eta_\eps(z)\, dz = \|u\|_{L^\infty(B_1)},
\]
and 
\[
|f_\eps(x) - f_\eps(y)| \le \int_{B_\eps} |f(x-z) - f(y-z)| \eta_\eps(z)\, dz = [f]_{C^{0,\alpha}(B_1)} |x-y|^\alpha. 
\]


From here, we deduce  $\|u_\eps\|_{L^\infty(B_{1-\eps})}\le \|u\|_{L^\infty(B_1)}$ and $\|f_\eps\|_{C^{0, \alpha}(B_{1-\eps})}\le \|f\|_{C^{0, \alpha}(B_{1})}$. Thus, the sequence $u_\eps $ is uniformly bounded in $C^{2,\alpha}(B_{1/2})$,
\[
\|u_\eps\|_{C^{2, \alpha}(B_{1/2})} \le C\left(\|u\|_{L^\infty(B_{1})}+\|f\|_{C^{0, \alpha}(B_{1})}\right).
\]
Moreover, since $u$ is continuous (see Corollary~\ref{cor.Hold_cont_f}), arguing as before we get $\|u_\eps- u\|_{L^\infty(B_1)}\to 0$ as $\eps \downarrow 0$, so that $u_\eps\to u$ uniformly. We can use \ref{it.H8} from Chapter~\ref{ch.0} to deduce that $u\in C^{2,\alpha}$ and
\[
\|u\|_{C^{2, \alpha}(B_{1/2})} \le C \left(\|u\|_{L^\infty(B_1)}+\|f\|_{C^{0, \alpha}(B_1)}\right).
\]
By a covering argument (see Remark~\ref{rem.covering_argument}) we can get a similar estimate in any ball $B_\rho$ with $\rho <1$, 
\[
\|u\|_{C^{2, \alpha}(B_{\rho})} \le C_\rho\left(\|u\|_{L^\infty(B_1)}+\|f\|_{C^{0, \alpha}(B_1)}\right),
\]
where now the constant $C_\rho$ depends also on $\rho$, and in fact, blows up when $\rho\uparrow 1$. In any case, we have that $u\in C^{2,\alpha}(B_\rho)$ for any $\rho < 1$, i.e., $u$ is in $C^{2,\alpha}$ inside $B_1$.
\end{proof}

The previous proof is an example of a recurring phenomenon when proving regularity estimates for PDEs. If one can get estimates of the kind 
\[
\|u\|_{C^{2, \alpha}} \le C\left(\|u\|_{L^\infty}+\|f\|_{C^{0, \alpha}}\right),
\]
for all $C^\infty$ functions $u$, and with a constant $C$ that depends only on $\alpha$ and $n$ (but independent of $u$ and $f$), then, in general, the estimate holds as well for all solutions $u$. Thus, if one wants to prove the higher-order regularity estimates from Corollary~\ref{cor.Schauder_estimates_L_HO}, it is enough to get a priori estimates in the spirit of Theorem~\ref{thm.Schauder_estimates_L}. As a consequence, assuming that Theorem~\ref{thm.Schauder_estimates_L} holds, we can prove Corollary~\ref{cor.Schauder_estimates_L_HO}.

\begin{proof}[Proof of Corollary~\ref{cor.Schauder_estimates_L_HO}]
As mentioned above, we just need to show that for any $u\in C^\infty$ such that $\Delta u = f$, one has
\begin{equation}
\label{eq.toshowk}
\|u\|_{C^{k+2, \alpha}(B_{1/2})} \le C\left(\|u\|_{L^\infty(B_1)}+\|f\|_{C^{k, \alpha}(B_1)}\right).
\end{equation}
for some constant $C$ depending only on $n$, $\alpha$, and $k$; and then we are done by a covering argument (see Remark~\ref{rem.covering_argument}). We prove it by induction on $k$, and it follows applying the induction hypothesis to derivatives of $u$. Notice that \eqref{eq.toshowk}  deals with balls $B_{1/2}$ and $B_1$, but after a rescaling and covering argument (see Remark \ref{rem.covering_argument}), it could also be stated in balls $B_{1/2}$ and $B_{3/4}$ (we will use it in this setting). 

The base case, $k = 0$, already holds by Theorem~\ref{thm.Schauder_estimates_L}. Let us now assume that \eqref{eq.toshowk} holds for $k = m-1$, and we will show it for $k = m$. 

In this case, let us differentiate $\Delta u = f$ to get $\Delta \de_i u = \de_i f$, for $i\in \{1,\dots,n\}$. Applying \eqref{eq.toshowk} for $k = m-1$ to $\de_i u$ in balls $B_{1/2}$ and $B_{3/4}$, we get
\begin{align*}
\|\de_i u\|_{C^{m+1, \alpha}(B_{1/2})} & \le C\left(\|\de_i u\|_{L^\infty(B_{3/4})}+\|\de_i f\|_{C^{m-1, \alpha}(B_{3/4})}\right)\\
& \le C\left(\|u\|_{C^{2,\alpha}(B_{3/4})}+\| f\|_{C^{m, \alpha}(B_{3/4})}\right).
\end{align*}
Using now Theorem~\ref{thm.Schauder_estimates_L} in balls $B_{3/4}$ and $B_{1}$,
\[
\|\de_i u\|_{C^{m+1, \alpha}(B_{1/2})}\le  C\left(\|u\|_{L^\infty(B_1)} + \|f\|_{C^{\alpha}(B_{1})}+\| f\|_{C^{m, \alpha}(B_{3/4})}\right)
\]
This, together with the basic estimate from Theorem~\ref{thm.Schauder_estimates_L} for $\Delta u = f$, and used for each $i\in \{1,\dots,n\}$, directly yields that \eqref{eq.toshowk} holds for $k = m$.
\end{proof}

Similarly, if one wants to prove regularity estimates in other contexts, it is often enough to obtain the corresponding a priori estimate. For instance, using an estimate that we prove later in the chapter (in the more general context of non-divergence-form equations) we can immediately obtain also the proof of Proposition~\ref{prop.Schauder_estimates_L_bounded}.

\begin{proof}[Proof of Proposition~\ref{prop.Schauder_estimates_L_bounded}]
The proof is exactly the same as the proof of Corollary~\ref{cor.Schauder_estimates_L} but using Proposition~\ref{prop.Schauder_estimates_cont} from below instead of Theorem~\ref{thm.Schauder_estimates_L}. (Alternatively, see Remark~\ref{rem:zygmund}.)
\end{proof}

Let us now provide the first proof of Theorem~\ref{thm.Schauder_estimates_L}. The method used here was introduced by Wang in \cite{Wan}. 

\begin{proof}[First proof of Theorem~\ref{thm.Schauder_estimates_L}.]
We will prove that 
\[
 |D^2 u(z) - D^2 u(y) | \le C|z-y|^\alpha\left(\|u\|_{L^\infty(B_1)}+ [f]_{C^{0, \alpha}(B_1)}\right),
 \]
for all $y, z\in B_{1/32}$. After a translation, we assume that $y = 0$, so that the proof can be centered around 0. This will prove our theorem with estimates in a ball $B_{1/32}$, and the desired result in a ball of radius $\frac12$ follows by a covering argument. Moreover, after dividing the solution $u$ by $\|u\|_{L^\infty(B_1)}+ [f]_{C^{0, \alpha}(B_1)}$ if necessary, we may assume that $\|u\|_{L^\infty(B_1)}\le 1$ and $[f]_{C^{0, \alpha}(B_1)}\le 1$, and we just need to prove that for all $z\in B_{1/16}$,
\[
 |D^2 u(z) - D^2 u(0) | \le C|z|^\alpha.
 \]
  
Throughout the proof, we will use the following basic estimates for harmonic functions:
\begin{equation}
\label{eq.harmkap}
\Delta w = 0\quad\textrm{in} \quad B_r\quad\Rightarrow \quad \|D^\kappa w\|_{L^\infty(B_{r/2})} \le Cr^{-\kappa} \| w \|_{L^\infty(B_r)},
\end{equation}
where $C$ depends only on $n$ and $\kappa \in \N$. (In fact, we will only use $\kappa \in \{1, 2, 3\}$.)
Such estimate follows by rescaling the estimate \eqref{eq.estimatesuk} --- which corresponds to the case $r = 1$.

We will also use the estimate
\begin{equation}
\label{eq.harmkap2}
\Delta w = \lambda\quad\textrm{in} \quad B_r\quad\Rightarrow \quad \|D^2 w\|_{L^\infty(B_{r/2})} \le C\left(r^{-2} \| w \|_{L^\infty(B_r)} + |\lambda|\right),
\end{equation}
for some constant $C$ depending only on $n$. This estimate follows from \eqref{eq.harmkap} after subtracting $\frac{\lambda}{2n}|x|^2$. 

For $k = 0,1,2,\dots$, let $u_k$ be the solution to 
\[
\left\{
\begin{array}{rcll}
\Delta u_k & = & f(0) &\text{in } B_{2^{-k}}\\
u_k &=& u &\text{on } \partial B_{2^{-k}}.
\end{array}
\right.
\]
Then, $\Delta (u_k-u) = f(0)-f$, and by the rescaled version of Lemma~\ref{lem.maxPrinciple}
\begin{equation}
\label{eq.fromu0u}
\|u_k-u\|_{L^\infty(B_{2^{-k}})}\le C (2^{-k})^2\|f(0)-f\|_{L^\infty(B_{2^{-k}})}\le C2^{-k(2+\alpha)},
\end{equation}
where we are using that  $[f]_{C^{0, \alpha}(B_1)}\le 1$. 
Hence, the triangle inequality yields
\[
\|u_{k+1}-u_k\|_{L^\infty(B_{2^{-k-1}})}\le C2^{-k(2+\alpha)}.
\]
Since $u_{k+1}-u_k$ is harmonic, we have
\begin{equation}
\label{eq.d2uk}
\|D^2 (u_{k+1}-u_k)\|_{L^\infty(B_{2^{-k-2}})} \le C2^{2(k+1)}\|u_{k+1}-u_k\|_{L^\infty(B_{2^{-k-1}})} \le C2^{-k\alpha}. 
\end{equation}

Now, notice that
\begin{equation}
\label{eq.d2ulim}
D^2 u (0) = \lim_{k\to \infty} D^2 u_k (0).
\end{equation}
Indeed, let $\tilde u(x) := u(0) + x\cdot \nabla u(0) + \frac12 x\cdot D^2 u(0) x$ be the second order expansion of $u$ at 0. Then, since $u\in C^{2,\alpha}$, we have $\|\tilde u - u\|_{L^\infty(B_r)} \le Cr^{2+\alpha} = o(r^2)$. Using that $\tilde u - u_k$ is harmonic together with \eqref{eq.harmkap} we deduce
\begin{align*}
|D^2 u_k(0) - D^2 u(0)| & \le \|D^2 (u_k - \tilde u)\|_{L^\infty(B_{2^{-k-1}})}\\
& \le C 2^{2k} \|u_k - \tilde u\|_{L^\infty(B_{2^{-k}})}\\
& = C 2^{2k} \|u - \tilde u\|_{L^\infty(\de B_{2^{-k}})} \to 0\quad\textrm{as}\quad k\to \infty. 
\end{align*}

Now, for any point $z$ near the origin, we have 
\begin{align*}
|D^2 u(z) - D^2u(0)|\le |D^2& u_k(0) -D^2 u(0)|\\
&+|D^2 u_k(0)-D^2 u_k(z)|+|D^2 u_k(z)-D^2 u(z)|.
\end{align*}
For a given $z\in B_{1/16}$, we choose $k\in \N$ such that \[
2^{-k-4}\le |z|\le 2^{-k-3}.
\]
Thanks to \eqref{eq.d2uk}-\eqref{eq.d2ulim}, and by the triangle  inequality, we get
\[
|D^2 u_k(0)-D^2 u(0)|\le \sum_{j = k}^\infty |D^2 u_j(0)-D^2 u_{j+1}(0)|\le C\sum_{j = k }^\infty 2^{-j\alpha} = C2^{-k\alpha},
\]
where we use that $\alpha\in (0, 1)$. 

In order to estimate $|D^2 u(z)-D^2 u_k(z)|$, the same argument can be repeated around $z$ instead of $0$. That is, take solutions of $\Delta v_j = f(z)$ in $B_{2^{-j}(z)}$ and $v_j = u$ on $\partial B_{2^{-j}}(z)$. Then,
\[
|D^2 u_k(z) - D^2u(z) |\le |D^2 u_k(z) -D^2 v_k(z)|+|D^2 v_k(z) - D^2 u(z)|.
\]
The second term above can be bounded by $C2^{-k\alpha}$ arguing as before. For the first term, we use \eqref{eq.harmkap2} by noticing that $\Delta (u_k - v_k) = f(0)-f(z)$ in $B_{2^{-k}}\cap B_{2^{-k}}(z)\supset B_{2^{-k-1}}(z)$ (recall $|z|\le 2^{-k-3}$), so that, in $B_{2^{-2-k}(z)}$ we have 
\begin{align*}
|D^2 u_k(z) - D^2 v_k(z)|&  \le \|D^2(u_k-v_k)\|_{L^\infty(B_{2^{-2-k}}(z))}\\
& \le C2^{2k}\|u_k-v_k\|_{L^\infty(B_{2^{-k-1}(z)})} + C|f(z)-f(0)|\\
& \le C2^{2k}\|u_k-v_k\|_{L^\infty(B_{2^{-k-1}(z)})} + C2^{-k\alpha},
\end{align*}
where we use, again, that $|z|\le 2^{-k-3}$, and $[f]_{C^{0, \alpha}(B_1)}\le 1$ 

Finally, from \eqref{eq.fromu0u}, we know that 
\[
\|u_k-u\|_{L^\infty(B_{2^{-k-1}}(z))}\le \|u_k-u\|_{L^\infty(B_{2^{-k}})}\le C2^{-k(2+\alpha)},
\]
and 
\[
\|u-v_k\|_{L^\infty(B_{2^{-k-1}(z)})}\le C2^{-k(2+\alpha)},
\]
which gives 
\[
\|u_k-v_k\|_{L^\infty(B_{2^{-k-1}(z)})}\le C2^{-k(2+\alpha)}.
\]
Thus, we deduce that 
\[
|D^2 u_k (z) -D^2 u (z) |\le C2^{-k\alpha}. 
\]
Finally, to estimate $|D^2 u_k(z)-D^2 u_k(0)|$, we denote $h_j := u_j-u_{j-1}$ for $j = 1,2,\dots, k$. Since $h_j$ are harmonic, by \eqref{eq.harmkap} with $\kappa = 3$ and using that $B_{2^{-k-3}}\subset B_{2^{-j-1}}$, we see that
\begin{align*}
\left|\frac{D^2 h_j(z)-D^2 h_j (0)}{|z|}\right|& \le \|D^3 h_j\|_{L^\infty(B_{2^{-k-3}})}\\
& \le C2^{3j}\|h_j\|_{L^\infty(B_{2^{-j}})}\le C2^{j(1-\alpha)}.
\end{align*}

Hence, 
\begin{align*}
|D^2 u_k(0) - D^2 u_k(z) | & \le |D^2 u_0(z) -D^2 u_0(0)|+\sum_{j = 1}^k |D^2 h_j(z) - D^2 h_j(0)|\\
& \le C|z|\|u_0\|_{L^\infty(B_1)}+C  |z|\sum_{j = 1}^k 2^{j(1-\alpha)}.
\end{align*}
We have also used here that,  if we define $w := u_0 - \frac{f(0)}{2n} |x|^2 + \frac{f(0)}{2n}$ then $w$ is harmonic, $D^3 w = D^3 u_0$, and $w = u_0$ on $\partial B_1$, and
\begin{align*}
|z|^{-1}|D^2 u_0(z) -D^2 u_0(0)|& \le \|D^3 u_0\|_{L^\infty(B_{1/2})} = \|D^3 w\|_{L^\infty(B_{1/2})} \\
& \le C \|w\|_{L^\infty(B_{1})} = C \|u_0\|_{L^\infty(\partial B_{1})} \le C\|u_0\|_{L^\infty(B_1)},
\end{align*}
by higher order regularity estimates for harmonic functions (see~\eqref{eq.harmkap}) and the maximum principle (Lemma~\ref{lem.maxPrinciple}). Note that here the constant $C$ depends only on $n$. 

Combined with the fact that, from \eqref{eq.fromu0u},
\[
\|u_0\|_{L^\infty(B_1)}\le C, 
\]
(where we also use $\|u\|_{L^\infty(B_1)}\le 1$) and $|z|\le 2^{-k-3}$, we deduce that 
\[
|D^2 u_k(0) - D^2 u_k(z) |  \le C|z|+ C|z| 2^{k(1-\alpha)}\le C2^{-k\alpha}.
\]

We finish by noticing that $|z|\ge 2^{-k-4}$ and combining all the last inequalities we reach
 \[
 |D^2 u(z) - D^2 u(0) |\le C2^{-k\alpha}\le C|z|^\alpha
 \]
 for all $z\in B_{1/16}$. That is, 
 \[
[D^2 u]_{C^{0,\alpha}(B_{1/16})}\le C.
\]
Now, thanks to the interpolation inequalities (see \eqref{ch0-interp2} with $\eps = 1$),
\begin{align*}
\|u\|_{C^{2,\alpha}(B_{1/16})} & = \|u\|_{C^{2}(B_{1/16})} +[D^2 u]_{C^{0,\alpha}(B_{1/16})}\\
& \le C\|u\|_{L^\infty(B_{1/16})} + 2 [D^2 u]_{C^{0,\alpha}(B_{1/16})}\le C.
\end{align*}
We finish by recalling that we divided the solution $u$ by $\|u\|_{L^\infty(B_1)}+ [f]_{C^{0, \alpha}(B_1)}$, and we use a covering argument to get the desired result (see Remark~\ref{rem.covering_argument}).
\end{proof}

For the second proof of Theorem~\ref{thm.Schauder_estimates_L}, we use the methods from \cite{M}, originally from \cite{Caf89}.

\begin{proof}[Second proof of Theorem~\ref{thm.Schauder_estimates_L}.]
After subtracting $\frac{f(0)}{2n}|x|^2$ we may assume that $f(0) = 0$. After dividing $u$ by $\|u\|_{L^\infty(B_1)}+\eps^{-1}\|f\|_{C^{0, \alpha}(B_1)}$ if necessary, we may also assume that $\|u\|_{L^\infty(B_1)}\le 1$ and $\|f\|_{C^{0, \alpha}(B_1)}\le\eps$, where $\eps> 0$ is a constant to be chosen depending only on $n$ and $\alpha$. 
After these simplifications, it is enough to show that
\begin{equation}
\label{eq.udone}
\|u\|_{C^{2, \alpha}(B_{1/2})}\le C
\end{equation}
for some constant $C$ depending only on $n$ and $\alpha$. 

We will show that, for every $x\in B_{1/2}$, there exist  a sequence of quadratic polynomials, $(P_k)_{k\in \N}$, and a $\rho_\circ< 1$ such that
\begin{equation}
\label{eq.H_QP}
\|u-P_k\|_{L^\infty(B_{\rho_\circ^k}(x))}\leq C_\circ\rho_\circ^{k(2+\alpha)}\qquad\textrm{for all }\, k\in \N,
\end{equation}
for some constant $C_\circ$. By property \ref{it.H5} from Chapter~\ref{ch.0}, this yields that $[D^2u]_{C^{0, \alpha}(B_{1/2})} \le CC_\circ$. After using an interpolation inequality \eqref{ch0-interp2}, we get \eqref{eq.udone}.

We will prove \eqref{eq.H_QP} for $x = 0$ (after a translation, it follows for all $x\in B_{1/2}$). We are going to use that $\Delta u = f$, $\|u\|_{L^\infty(B_1)}\le 1$, $f(0 ) = 0$ and $[f]_{C^{0, \alpha}(B_1)}\le \eps$.

Notice that $\|\Delta u\|_{C^{0,\alpha}(B_1)} = \|f\|_{C^{0,\alpha}(B_1)}\le 2\eps$, i.e., $u$ is $2\eps$-close in H\"older norm to a harmonic function: let $w$ be such that $\Delta w = 0$ and $w = u $ on $\de B_1$.  Then, $\Delta(u- w) = f$ in $B_1$, and $u-v = 0$ on $\de B_1$, so that by Lemma~\ref{lem.maxPrinciple},
\begin{equation}
\label{eq.connor1}
\|u-w\|_{L^\infty(B_1)}\le C'\|f\|_{L^\infty(B_1)} \le C\eps,
\end{equation}
for some $C$ universal (we are only using $\|f\|_{L^\infty(B_1)}\le 2\eps$, and not using its $C^\alpha$ norm at this point). The function $w$ is harmonic and $|w|\le 1$ (since $|u|\le 1$). Therefore, it has a quadratic Taylor polynomial $P_1$ at the origin, which satisfies $\Delta P_1 \equiv 0$ and $|P_1|\le C$. Moreover, since $w$ is harmonic (and in particular $w\in C^3$), we have
\begin{equation}
\label{eq.connor2}
\|w-P_1 \|_{L^\infty(B_r)}\le Cr^3\qquad\textrm{for all }\, r \le 1,
\end{equation}
for some $C$ depending only on $n$. 

Combining \eqref{eq.connor1} and \eqref{eq.connor2} we obtain 
\[
\|u-P_1\|_{L^\infty(B_r)}\le C(r^3+\eps)\qquad\textrm{for all }\, r \le 1.
\]
Choose now $r_\circ$ small enough such that $Cr_\circ^3\le \frac12 r_\circ^{2+\alpha}$ (notice $\alpha < 1$), and $\eps$ small enough such that $C\eps< \frac12 r_\circ^{2+\alpha}$. (Notice that both $r_\circ$ and $\eps$ can be chosen depending only on $n$ and $\alpha$.) Then,
\[
\|u-P_1\|_{L^\infty(B_{r_\circ})}\le r_\circ^{2+\alpha}. 
\]
Let us now define 
\[
u_2(x) := \frac{(u-P_1)(r_\circ x)}{r_\circ^{2+\alpha} }.\]
Notice that $\|u_2\|_{L^\infty(B_1)}\le 1$ and $\Delta u_2(x) = r_\circ^{-\alpha}f(r_\circ x)=:f_2(x)$. Then, $f_2(0) = 0$ and  $[f_2]_{C^{0,\alpha}(B_1)}\le [f]_{C^{0,\alpha}(B_1)}\le  \eps$. That is, the same hypotheses as before are fulfilled. 
Repeating the same procedure,  there exists a polynomial $P_2$ such that
\[
\|u_2 - P_2\|_{L^\infty(B_{r_\circ})}\le r_\circ^{2+\alpha}. 
\]
That is, substituting back, 
\[
\big\|u - P_1 - r_\circ^{2+\alpha}P_2(x/r_\circ)\big\|_{L^\infty(B_{r^2_\circ})}\le r_\circ^{2(2+\alpha)}. 
\]
Continuing iteratively, for every $k\in \N$ we can define 
\[
u_{k+1}(x):= \frac{(u_k-P_k)(r_\circ x)}{r_\circ^{2+\alpha}},
\]
which satisfies
\[
\|u_{k+1}\|_{L^\infty(B_1)}\le 1,\quad\Delta u_{k+1}(x) = r_\circ^{-\alpha} f_{k}(r_\circ x) = r_\circ^{-k\alpha} f(r_\circ^k x) =: f_{k+1}(x),
\]
and there exists some $P_{k+1}$ such that
\[
\|u_{k+1}-P_{k+1}\|_{L^\infty(B_{r_\circ})}\le r_\circ^{2+\alpha}. 
\]
Substituting back,
\[
\big\|u - P_1 - r_\circ^{2+\alpha}P_2(x/r_\circ) - \dots - r_\circ^{k(2+\alpha)} P_{k+1}(x/r_\circ^k)\big\|_{L^\infty(B_{r^{k+1}_\circ})}\le r_\circ^{(k+1)(2+\alpha)}. 
\]
That is, we have constructed a sequence of quadratic polynomials approximating $u$ in a decreasing sequence of balls around 0; which shows that \eqref{eq.H_QP} holds around 0. After a translation, the same argument can be repeated around any point $x\in B_{1/2}$, so that, by \ref{it.H5} we are done. 
\end{proof}
\begin{rem}
\label{rem:zygmund}
When $\alpha = 0$, the previous proof implies that if $f\in L^\infty(B_1)$ then, by \eqref{eq.H_QP}, $\nabla u$ is in  the Zygmund space $\Lambda^1(B_1)$; see Remark~\ref{rem.zygmund} in the Appendix~\ref{app.A} for more details. In particular, we also get a proof of Proposition~\ref{prop.Schauder_estimates_L_bounded}.
\end{rem}

Notice that in the previous proof we have not directly used that $u$ is $C^2$. In fact, the only properties of $u$ (and the Laplacian) we have used are that the maximum principle holds  and that $\Delta (u(rx)) = r^2 (\Delta u)(rx)$.

In particular, the second proof of Theorem~\ref{thm.Schauder_estimates_L} is {\it not} an a priori estimate, and rather it says that any weak solution to the Laplace equation with $C^\alpha$ right-hand side is $C^{2,\alpha}$. That is, we have directly proved Corollary~\ref{cor.Schauder_estimates_L}.

\section{Schauder estimates for operators in non-divergence form}

After proving the Schauder estimates for the Laplacian, we will study now more general second order linear elliptic operators. We start with operators in non-divergence form. The type of equation we are interested in is 
\[
\boxed{
{\rm tr} \big(A(x)D^2 u(x)\big) = \sum_{i,j = 1}^n a_{ij}(x)\partial_{ij}u(x) = f(x) \quad\textrm{in}\quad B_1
}
\]
where the matrix $A(x) = (a_{ij}(x))_{ij}$ is uniformly elliptic --- in the sense that \eqref{eq.ellipt} below holds --- and  $a_{ij}(x) \in C^{0,\alpha}(B_1)$. We will prove the following a priori estimates. 

\begin{thm}[Schauder estimates in non-divergence form]\index{Schauder estimates!Non-divergence form}
\label{thm.Schauder_estimates}
Let $\alpha\in (0, 1)$, and let $u\in C^{2,\alpha}$ be any solution to 
\[
\sum_{i, j = 1}^n a_{ij}(x) \partial_{ij} u = f(x) \quad\textrm{in} \quad B_1,
\]
with $f\in C^{0, \alpha}(B_1)$ and $a_{ij}(x) \in C^{0,  \alpha}(B_1)$, and $(a_{ij}(x))_{ij}$  fulfilling the ellipticity condition
\begin{equation}
\label{eq.ellipt}
0< \lambda\,{\rm Id} \le (a_{ij}(x))_{ij} \le\Lambda\,{\rm Id}\quad\textrm{in} \quad B_1,
\end{equation}
for some $0 < \lambda \le\Lambda < \infty$. Then,
\[
\|u\|_{C^{2, \alpha}(B_{1/2})} \le C \left(\|u\|_{L^\infty(B_1)} + \|f\|_{C^{0, \alpha}(B_1)}\right)
\]
for some constant $C$ depending only on $\alpha$, $n$, $\lambda$, $\Lambda$, and $\|a_{ij}\|_{C^{0, \alpha}(B_1)}$.
\end{thm}

As for the Laplacian, we will provide two different proofs of the previous result. 
On the other hand, as a consequence of the previous result, we also obtain  higher order Schauder estimates in non-divergence form. 

\begin{cor}[Higher order Schauder estimates in non-divergence form]\index{Higher order Schauder estimates!Non-divergence form}
\label{cor.Schauder_estimates_HO}
Let $u\in C^{k+2,\alpha}$ be a solution to 
\[
\sum_{i, j = 1}^n a_{ij}(x) \partial_{ij} u = f(x) \quad\textrm{in} \quad B_1,
\]
with $f\in C^{k,  \alpha}(B_1)$ and $a_{ij}(x) \in C^{k,  \alpha}(B_1)$ for some $\alpha\in (0, 1)$, $k\in \N$, and $(a_{ij}(x))_{ij}$  fulfilling the ellipticity conditions \eqref{eq.ellipt} for some $0 < \lambda \le\Lambda < \infty$. Then,
\[
\|u\|_{C^{k+2, \alpha}(B_{1/2})} \le C \left(\|u\|_{L^\infty(B_1)} + \|f\|_{C^{k, \alpha}(B_1)}\right)
\]
for some constant $C$ depending only on $\alpha$, $k$, $n$, $\lambda$, $\Lambda$, and $\|a_{ij}\|_{C^{k, \alpha}(B_1)}$.
\end{cor}

\begin{rem}[Ellipticity condition]\index{Uniform ellipticity condition!Linear equations}
The uniform ellipticity condition in $B_1$, \eqref{eq.ellipt}, is a quantification of the fact that the matrix 
\[
A(x) := (a_{ij}(x))_{ij}
\]
is uniformly positive definite and uniformly bounded as well. Notice that we can always assume that $A(x)$ is symmetric (from $\partial_{ij} u = \partial_{ji} u$). We recall that the inequality $A_1 \le A_2$ for symmetric matrices $A_1, A_2\in \mathcal{M}_n$ has to be understood in the sense that $A_2-A_1$ is positive semi-definite. Alternatively, \eqref{eq.ellipt} will hold if 
\[
0< \lambda|\xi|^2 \le \sum_{i,j= 1}^n \xi_i \xi_j a_{ij}(x)\le \Lambda|\xi|^2\quad\textrm{for all} \quad x\in B_1
\]
for all $\xi \in  \R^n$. 
\end{rem}

\begin{rem}[Constant coefficients]
\label{rem.constcoef}
Let us start by understanding the case of constant coefficients, 
\[
\sum_{i,j = 1}^n a_{ij} \partial_{ij} u(x) = 0 \quad\textrm{in}\quad B_1,
\]
where $a_{ij}$ are constants and satisfy  the uniform ellipticity assumption, 
\[
0 < \lambda{\rm Id} \le (a_{ij})_{ij}  \le \Lambda{\rm Id},
\]
for $0< \lambda \le \Lambda < \infty$. 

Let us denote $A := (a_{ij})_{ij}\in \mathcal{M}_n$. Then, $A$ is a symmetric positive definite matrix, and therefore has a unique positive definite square root $A^{1/2}$. After an affine change of variables 
 \[
 z = A^{1/2} x ,
 \] 
 the equation 
 \[
 \sum_{i, j = 1}^n a_{ij}\partial_{x_i x_j} u = 0\quad\textrm{becomes}\quad \sum_{i = 1}^n \partial_{z_i z_i} u = 0
 \]
 or $\Delta_z u = 0$. Indeed,
 \[
\sum_{i, j = 1}^n a_{ij} \partial_{x_i x_j} u = {\rm tr}(AD^2_x u ) = {\rm tr}(A^{1/2}D^2_x uA^{1/2} )  = {\rm tr}(D_z^2 u) = \Delta_z u.
 \]
 Therefore (and since $0 < \lambda{\rm Id} \le A \le \Lambda{\rm Id}$), the case of constant coefficients (uniformly elliptic) can be reduced to the case of {\em harmonic functions}.
 
 Thanks to the uniform ellipticity, the change of variables is {\em not} degenerate, and thus the estimates on $\|u\|_{C^{2, \alpha}}$ that we get depend only on $\alpha$, $n$, $\lambda$, and $\Lambda$ (but not on $A$). Similarly, after changing variables, there could be a shrinking of the domain, say that the $C^{2, \alpha}$ norm of $u$ is bounded in $B_\rho$ instead of $B_{1/2}$, for some $\rho < 1/2$. Once again, since the change is non-degenerate, such $\rho$ depends only on $n$, $\lambda$, and $\Lambda$, and one can complete the proof by a covering argument in $B_{1/2}$ (see Remark~\ref{rem.covering_argument}). 
\end{rem}

\subsection*{The maximum principle}

We state the maximum principle for equations in non-divergence form, which will be used in this section. 

\begin{prop}[Maximum Principle in non-divergence form]\index{Maximum principle!Non-divergence form}
\label{prop.weakMaxNonDiv}
Let $\Omega\subset\R^n$ be any bounded open set. Suppose that $u\in C^0(\overline{\Omega})\cap C^2(\Omega)$ satisfies
\[
\sum_{i, j = 1}^n a_{ij}(x) \partial_{ij} u \ge 0 \quad\textrm{in} \quad \Omega,
\]
where $(a_{ij}(x))_{ij}$ satisfy
\[
0< \lambda\,{\rm Id} \le (a_{ij}(x))_{ij} \quad\textrm{in} \quad \Omega.
\]
Then,
\[
\sup_{\Omega} u  = \sup_{\partial \Omega} u. 
\]
\end{prop}
\begin{proof}
Let us begin by showing the maximum principle in the case
\begin{equation}
\label{eq.MaxPrStrict}
\sum_{i, j = 1}^n a_{ij}(x) \partial_{ij} u> 0 \quad\textrm{in} \quad B_1,
\end{equation}
that is, when we have a strict inequality. We show it by contradiction: suppose that there exists some $x_\circ\in \Omega$ such that $\sup_{\Omega} u = u(x_\circ)$. Since it is an interior maximum, we must have $\nabla u (x_\circ) = 0$ and $D^2 u(x_\circ) \le 0$, that is, $D^2 u(x_\circ)$ is a negative semi-definite symmetric matrix. In particular, all its eigenvalues are non-positive, and after a change of variables we have that
\[
P^T D^2 u(x_\circ) P = {\rm diag}(\lambda_1,\dots,\lambda_n) := D_{x_\circ}
\]
for some orthogonal $n\times n$ matrix $P$, and with $\lambda_i \le 0$ for all $1\le i \le n$. Let $A(x) = (a_{ij}(x))_{ij}$, and let $A^P(x_\circ) := P^T A(x_\circ) P$. Then, since $A(x_\circ)$ is positive definite, so is $A^P(x_\circ) = (a_{ij}^P(x_\circ))_{ij}$. In particular, $a_{ii}^P(x_\circ) \ge 0$ for all $1\le i \le n$. Then, 
\[
{\rm tr} (A(x_\circ) D^2 u(x_\circ)) = {\rm tr}(A(x_\circ)PD_{x_\circ} P^T)= {\rm tr}(P^T A(x_\circ)PD_{x_\circ})
\]
and, therefore
\[
0< {\rm tr} (A(x_\circ) D^2 u(x_\circ)) = {\rm tr}(A^P(x_\circ) D_{x_\circ}) = \sum_{i = 1}^n a_{ii}^P(x_\circ) \lambda_i \le 0,
\]
a contradiction. Here, we used that $a_{ii}^P(x_\circ) \ge 0$ and $\lambda_i \le 0$ for all $1\le i \le n$. This shows that the maximum principle holds when the strict inequality \eqref{eq.MaxPrStrict} is satisfied. 

Let us now remove this hypothesis. Let $R$ be large enough such that $B_R\supset \Omega$ --- after a translation, we can take $R =\frac12{\rm diam}(\Omega)$. Consider now the function
\[
u_\eps (x) := u(x) +\eps e^{x_1} \quad\textrm{for}\quad x\in \Omega,
\]
for $\eps > 0$. Notice that, 
\[
\sum_{i, j = 1}^n a_{ij}(x) \partial_{ij} u_\eps (x) \ge \lambda \eps e^{x_1}>0 \quad\textrm{in}\quad \Omega.
\]
In particular, we can apply the result for \eqref{eq.MaxPrStrict} to obtain that 
\[
\sup_\Omega u \le \sup_{\Omega} u_\eps = \sup_{\partial \Omega} u_\eps \le \sup_{\partial \Omega} u +  \eps e^{R}. 
\]
By letting $\eps\downarrow 0$, we obtain the desired result. 
\end{proof}

As a consequence, we find: 

\begin{lem}
\label{lem.maxPrinciple_2}
 Let $\Omega\subset\R^n$ be a bounded open set, and let $u\in C^0(\overline{\Omega})\cap C^2(\Omega)$ be a function satisfying 
\begin{equation*}
\left\{ 
\begin{array}{rcll}
\sum_{i, j = 1}^n a_{ij}(x) \partial_{ij} u & = & f &\text{in } \Omega\\
u &=& g &\text{on } \partial \Omega,
\end{array}\right.
\end{equation*}
where $(a_{ij})_{ij}$ fulfill the ellipticity conditions \eqref{eq.ellipt} for some $0 < \lambda \le\Lambda < \infty$.
Then, 
\[\|u\|_{L^\infty(\Omega)}\leq C\bigl(\|f\|_{L^\infty(\Omega)}+\|g\|_{L^\infty(\partial\Omega)}\bigr),\]
for a constant $C$ depending only on the diameter of $\Omega$, $\lambda$, and $\Lambda$.
\end{lem}
\begin{proof}
This follows exactly as in the proof of Lemma~\ref{lem.maxPrinciple} using Proposition~\ref{prop.weakMaxNonDiv}.
\end{proof}

\subsection*{Proof of Schauder estimates}

Let us now proceed with the proof of Schauder estimates for equations in non-divergence form, Theorem~\ref{thm.Schauder_estimates}. We will first prove (in two ways) the following proposition, which is a weaker version of the estimate we want to show. 

We will later prove that, in fact, such estimate is enough to prove Theorem~\ref{thm.Schauder_estimates}. 

\begin{prop}
\label{prop.Schauder_estimates}
Let $u\in C^{2,\alpha}$ be a solution to 
\[
\sum_{i, j = 1}^n a_{ij}(x) \partial_{ij} u = f(x) \quad\textrm{in} \quad B_1,
\]
with $f\in C^{0, \alpha}(B_1)$ and $a_{ij}(x) \in C^{0,  \alpha}(B_1)$ for some $\alpha\in (0, 1)$, and $(a_{ij}(x))_{ij}$  fulfilling the ellipticity condition
\[
0< \lambda {\rm Id} \le (a_{ij}(x))_{ij} \le\Lambda{\rm Id}\quad\textrm{in} \quad B_1,
\]
for some $0 < \lambda \le\Lambda < \infty$. Then, for any $\delta> 0$,
\[
[D^2 u]_{C^{0, \alpha}(B_{1/2})} \le \delta[D^2 u]_{C^{0,\alpha}(B_1)} +C_\delta \left(\|u\|_{L^\infty(B_1)} + \|f\|_{C^{0, \alpha}(B_1)}\right),
\]
for some constant $C_\delta$ depending only on $\delta$, $\alpha$, $n$, $\lambda$, $\Lambda$, and $\|a_{ij}\|_{C^{0, \alpha}(B_1)}$.
\end{prop}

Notice that, the previous statement is almost what we want: if we could let $\delta\downarrow 0$ and $C_\delta$ remained bounded, Theorem~\ref{thm.Schauder_estimates} would be proved (after using interpolation inequalities \eqref{ch0-interp2}). On the other hand, if the H\"older norm was in $B_{1/2}$ instead of $B_1$, choosing $\delta = \frac12$ would also complete the proof. As we will see, although it is not so straightforward, Proposition~\ref{prop.Schauder_estimates} is just one step away from the final result. 

Let us provide two different proofs of Proposition~\ref{prop.Schauder_estimates}. The first proof is a sketch that follows the same spirit as the first proof of Theorem~\ref{thm.Schauder_estimates}. The second proof is through a blow-up argument (by contradiction). 

\begin{proof}[First Proof of Proposition~\ref{prop.Schauder_estimates}]
The proof is very similar to the case of the Laplacian, the first proof of Theorem~\ref{thm.Schauder_estimates_L}. 

We define $u_k$ as the solution to 
\[
\left\{\begin{array}{rcll}
\sum_{i,j = 1}^n a_{ij}(0) \partial_{ij} u_k & = & f(0) & \textrm{in } B_{2^{-k}}\\
u_k & = & u & \textrm{on } \partial B_{2^{-k}}.
\end{array}
\right.
\]
(We freeze the coefficients at zero.) Then, 
\[
v_k := u- u_k
\]
satisfies
\[
\sum_{i, j = 1}^n a_{ij} (0) \partial_{ij} v_k = f(x)-f(0) +\sum_{i, j  =1}^n \big(a_{ij} (0)-a_{ij}(x) \big) \partial_{ij} u\quad\textrm{in} \quad B_{2^{-k}}.
\]
By the maximum principle (Lemma~\ref{lem.maxPrinciple_2}) we get
\begin{align*}
\|u-u_k\|_{L^\infty(B_{ 2^{-k}})} \le C 2^{-2k}\bigg(& 2^{-\alpha k}\|f\|_{C^{0, \alpha}(B_{ 2^{-k}})}  \\
& +2^{-\alpha k} \|D^2 u\|_{L^\infty(B_{ 2^{-k}})}\sum_{i,j = 1}^n \|a_{ij}\|_{C^{0,\alpha}(B_{ 2^{-k}})}   \bigg).
\end{align*}
Thus, 
\[
\|u_k -u_{k+1}\|_{L^\infty(2^{-k-1})} \le C 2^{-k(2+\alpha)}\left(\|f\|_{C^{0, \alpha}(B_{2^{-k}})}+\|D^2 u\|_{L^\infty(B_{2^{-k}})}\right),
\]
where the constant $C$ depends only on $\alpha$, $n$, $\lambda$, $\Lambda$, and $\|a_{ij}\|_{C^{0, \alpha}(B_1)}$. 

Following the exact same proof as in the case of the Laplacian, $\Delta u = f(x) $, we now get 
\[
[D^2 u]_{C^{0,\alpha}(B_{1/2})} \le C\left(\|u\|_{L^\infty( B_1)} + \|f\|_{C^{0, \alpha}(B_1)} + \|D^2 u \|_{L^\infty(B_1)}\right) .
\]

This is {\em almost} exactly what we wanted to prove. However, we have an extra term $\|D^2 u \|_{L^\infty(B_1)}$ on the right-hand side. This can be dealt with by means of interpolation inequalities. 

We use that, for any $\eps > 0$, there is $C_\eps$ such that 
\begin{equation}
\label{eq.interpolation}
\|D^2 u \|_{L^\infty(B_1)}\le \eps [D^2 u ]_{C^{0, \alpha}(B_1)} + C_\eps\|u\|_{L^\infty(B_1)}
\end{equation}
see \eqref{ch0-interp2} in Chapter~\ref{ch.0}. 

The idea is that, since the $\|D^2 u \|_{L^\infty}$ term is lower order, we can absorb it in the left-hand side by paying the price of adding more $\|u\|_{L^\infty}$ norm on the right-hand side. 

Namely, we have
\begin{align*}
[D^2 u]_{C^{0,\alpha}(B_{1/2})} & \le C\left(\|u\|_{L^\infty( B_1)} + \|f\|_{C^{0, \alpha}(B_1)} + \|D^2 u \|_{L^\infty(B_1)}\right) \\
\textrm{\tiny (by interpolation) } & \le C\left(C_\eps \|u\|_{L^\infty( B_1)} + \|f\|_{C^{0, \alpha}(B_1)} + \eps[D^2 u ]_{C^{0, \alpha}(B_1)}\right) \\
 & \le C_{\delta}\left( \|u\|_{L^\infty( B_1)} + \|f\|_{C^{0, \alpha}(B_1)}\right) + \delta[D^2 u ]_{C^{0, \alpha}(B_1)}, 
\end{align*}
where we have used the interpolation inequality, and in the last step we have chosen $\eps = \delta / C> 0$. The constant $C_{\delta}$ depends only on $\delta$, $\alpha$, $n$, $\lambda$, $\Lambda$, and $\|a_{ij}\|_{C^{0, \alpha}(B_1)}$.
This concludes the proof. 
\end{proof}

For the second proof of Proposition~\ref{prop.Schauder_estimates} we use a robust blow-up method due to L. Simon, \cite{Sim}. For simplicity, we will first prove it for the Laplacian case. After proving it for the Laplacian, we explain in detail how to adapt the method for the more general non-divergence operators. 

\begin{proof}[Second Proof of Proposition~\ref{prop.Schauder_estimates}]
Assume first that $(a_{ij}(x))_{ij} = {\rm Id}$, that is, $\Delta u = f$ in $B_1$. We then explain the modifications needed to show the result in the general case, $\sum_{i,j=1}^n a_{ij}(x)\partial_{ij}u (x)= f(x)$ in $B_1$.

Thanks to interpolation inequalities we only need to prove the following estimate for any $\delta > 0$ sufficiently small,
\begin{equation}
\label{eq.contr.SAL}
[D^2u]_{C^{0,\alpha}(B_{1/2})}\le \delta[D^2u]_{C^{0,\alpha}(B_{1})}+ C_\delta\left(\|D^2u\|_{L^\infty(B_1)}+[f]_{C^{0, \alpha}(B_1)}\right)
\end{equation}
for all $u \in C^{2,\alpha}(B_1)$ with $\Delta u = f$ in $B_1$. Indeed, if \eqref{eq.contr.SAL} holds then, by interpolation \eqref{eq.interpolation}, with $\eps = \delta/C_\delta$, 
\[
[D^2u]_{C^{0,\alpha}(B_{1/2})}\le 2\delta[D^2u]_{C^{0,\alpha}(B_{1})}+ C_\delta \left(\|u\|_{L^\infty(B_1)}+[f]_{C^{0, \alpha}(B_1)}\right)
\]
for some new $C_\delta$ depending only on $\delta$, $n$ and $\alpha$, which is the desired result. 

We will now show that \eqref{eq.contr.SAL} holds by contradiction, for some $C_\delta$ depending only on $\delta$, $n$, and $\alpha$. Indeed, suppose that it does not hold. Then, there exist sequences $u_k\in C^{2, \alpha}(B_1)$ and $f_k\in C^{0,\alpha}(B_1)$ for $k\in \N$ such that
\[
\Delta u_k = f_k\quad\textrm{in} \quad B_1,
\]
and for a fixed small constant $\delta_\circ  >0$ we have
\begin{equation}
\label{eq.contr.SAL_2}
[D^2u_k]_{C^{0,\alpha}(B_{1/2})}> \delta_\circ[D^2u_k]_{C^{0,\alpha}(B_{1})}+ k\left(\|D^2u_k\|_{L^\infty(B_1)}+[f_k]_{C^{0, \alpha}(B_1)}\right).
\end{equation}
We now have to reach a contradiction. 

Select $x_k, y_k\in B_{1/2}$ such that 
\begin{equation}
\label{eq.halfeq}
\frac{|D^2u_k(x_k)-D^2u_k(y_k)|}{|x_k-y_k|^\alpha}\ge \frac12 [D^2u_k]_{C^{0,\alpha}(B_{1/2})}
\end{equation}
and let 
\[
\rho_k := |x_k-y_k|.
\]
Observe that we must necessarily have $\rho_k \to 0 $ as $k \to \infty$, since 
\begin{align*}
\frac12 [D^2u_k]_{C^{0,\alpha}(B_{1/2})}& \le \frac{|D^2u_k(x_k)-D^2u_k(y_k)|}{\rho_k^\alpha}\\
 & \le \frac{2\|D^2u_k\|_{L^\infty(B_1)}}{\rho_k^\alpha}\le \frac{2[D^2 u_k]_{C^{0,\alpha}(B_{1/2})}}{k\rho_k^\alpha},
\end{align*}
where we have used \eqref{eq.contr.SAL_2} in the last inequality. Thus, 
\[
\rho_k\le Ck^{-\frac{1}{\alpha}}\to 0 \quad\textrm{as}\quad k\to \infty
\]
Now, we rescale and blow up. Define
\[
\tilde u_k(x) := \frac{u_k(x_k+\rho_k x)-p_k(x)}{\rho_k^{2+\alpha}[D^2 u_k]_{C^{0,\alpha}(B_1)}},\quad\tilde f_k(x) := \frac{f_k(x_k+\rho_k x)-f_k(x_k)}{\rho_k^\alpha[D^2u_k]_{C^{0,\alpha}(B_1)}},
\]
where the quadratic polynomial $p_k$ is chosen so that 
\begin{equation}
\label{eq.AA1}
\tilde u_k(0) = |\nabla \tilde u_k(0)| = |D^2 \tilde u_k(0)| = 0.
\end{equation}
Namely, 
\[
p_k(z) := u_k(x_k) + \rho_k \sum_{i = 1}^n \partial_i u_k(x_k) z_i +  \frac{1}{2} \rho_k^2 \sum_{i, j = 1}^n \partial_{ij}u_k(x_k) z_i z_j. 
\]
It is now a simple computation to check that
\begin{equation}
\label{eq.simplecomputation}
\Delta \tilde u_k = \tilde f_k\quad\textrm{in} \quad B_{1/(2\rho_k)}.
\end{equation}
Let us also denote 
\[
\xi_k := \frac{y_k-x_k}{\rho_k}\in \mathbb{S}^{n-1}.
\]

Notice that 
\begin{equation}
\label{eq.AA2}
[D^2  \tilde u_k]_{C^{0,\alpha}\big(B_{1/{(2\rho_k)}}\big)} \le 1,\qquad\textrm{and}\qquad \big|D^2 \tilde u_k (\xi_k)\big|> \frac{\delta_\circ}{2}, 
\end{equation}
where for the second inequality we use \eqref{eq.contr.SAL_2}-\eqref{eq.halfeq}. 

Since $\tilde u_k$ are uniformly bounded in compact subsets, and bounded in the $C^{2, \alpha}$ norm (see \eqref{eq.AA1}-\eqref{eq.AA2}), we have by Arzel\`a--Ascoli that the sequence $\tilde u_k$ converges (up to a subsequence and in the $C^2$ norm) to a $C^{2, \alpha}$ function $\tilde u$ on compact subsets of $\R^n$. Moreover, again up to a subsequence, we have that $\xi_k\to \xi\in \mathbb{S}^{n-1}$.

By the properties of $\tilde u_k$, we deduce that $\tilde u$ satisfies 
\begin{equation}
\label{eq.contradiction2}
\tilde u (0) = |\nabla \tilde u(0) | = |D^2 \tilde u (0)| = 0,\quad[D^2\tilde u ]_{C^{0,\alpha}(\R^n)}\le 1,\quad |D^2\tilde u(\xi)|>\frac{\delta_\circ}{2}.
\end{equation}
On the other hand, for any $R\ge 1$ we have 
\begin{align*}
\|\tilde f_k\|_{L^\infty(B_R)} & = \sup_{x\in B_R} \frac{|f_k(x_k+\rho_k x)-f_k(x_k)|}{\rho_k^\alpha[D^2 u_k]_{C^{0,\alpha}(B_1)}}
 \le \frac{(\rho_kR)^\alpha[f_k]_{C^{0,\alpha}(B_1)}}{\rho_k^\alpha[D^2 u_k]_{C^{0,\alpha}(B_1)}}\\
& \le \frac{R^\alpha[D^2u_k]_{C^{0,\alpha}(B_{1/2})}}{k[D^2u_k]_{C^{0,\alpha}(B_1)}}\le \frac{R^\alpha } {k}\to 0,\textrm{ as } k \to \infty. 
\end{align*}
Thus, $\tilde f_k \to 0$ uniformly on compact sets of $\R^n$. Together with the fact that $\tilde u_k \to \tilde u$ in the $C^2$ norm in compact sets, we deduce (recall \eqref{eq.simplecomputation})
\[
\Delta \tilde u = 0\quad\textrm{in}\quad \R^n. 
\]
That is, $\tilde u$ is harmonic and, in particular, so is $\partial_{ij} \tilde u$ for any $i, j  = 1,\dots,n$. 

Let us now use the three properties in \eqref{eq.contradiction2} to get a contradiction. 
First notice that we have $[D^2 \tilde u ]_{C^{0, \alpha}(\R^n)} \le 1$. Thus, $D^2 \tilde u$ has sub-linear growth at infinity, and by Liouville's theorem (Proposition~\ref{cor.Liouville}) we find that $D^2\tilde u$ is constant. That is, $\tilde u$ is a quadratic polynomial, which also fulfills $\tilde u (0) = |\nabla \tilde u(0) | = |D^2 \tilde u (0)| = 0$. The only possibility is that $\tilde u \equiv 0$ in $\R^n$, which is a contradiction with $|D^2\tilde u(\xi)|>\frac{\delta_\circ}{2}$.

Thus, the proposition is proved in the case of the Laplacian.

We now treat the case of variable coefficients,
\[
\sum_{i, j = 1}^n a_{ij}(x) \partial_{ij} u(x) = f(x) \quad\textrm{in} \quad B_1,
\]
with $a_{ij}(x) $ uniformly elliptic in $B_1$ (i.e., $0<\lambda {\rm Id}\le (a_{ij}(x))_{ij}\le \Lambda{\rm Id}$ for $x\in B_1$) and with $\|a_{ij}\|_{C^{0,\alpha}(B_1)}\le M<\infty$ for some $M$.
The proof is essentially the same. As before, we proceed by contradiction, by assuming that there exist sequences $u_k$, $f_k$, and $a_{ij}^{(k)}$ such that
\[
\sum_{i, j = 1}^n a^{(k)}_{ij}(x) \partial_{ij} u_k(x) = f_k(x) \quad\textrm{in} \quad B_1,
\]
and \eqref{eq.contr.SAL_2} holds.

The only difference with respect to the Laplacian case is the equation satisfied by $\tilde u_k$. Let us define,
\[
\tilde a_{ij}^{(k)} (x) := a_{ij}^{(k)}(x_k+\rho_k x).
\]

Notice that 
\[
[\tilde a^{(k)}_{ij}]_{C^{0,\alpha}(B_{{1}/{(2\rho_k)}})}\le \rho_k^\alpha[a^{(k)}_{ij}]_{C^{0,\alpha}(B_1)}\to 0,\quad\textrm{as}\quad k\to \infty.
\]
In particular, up to subsequences, $\tilde a_{ij}^{(k)}$ converges uniformly in compact sets to some $\tilde a_{ij}$ with $[\tilde a^{(k)}_{ij}]_{C^{0,\alpha}(\R^n)} = 0$, i.e., $\tilde a_{ij}$ is constant. Then $\tilde u_k$ satisfies
\[
\sum_{i, j = 1}^n \tilde a_{ij}^{(k)} \partial_{ij}\tilde u_k = \tilde f_k (x) -\sum_{i,j = 1}^n\frac{\left(a_{ij}^{(k)}(x_k+\rho_k x) -  a_{ij}^{(k)}(x_k)\right)\partial_{ij} u_k(x_k)}{\rho_k^\alpha [D^2 u_k]_{C^{0, \alpha}(B_1)}}.
\]
Thus, 
\begin{align*}
\left|\sum_{i, j = 1}^n \tilde a_{ij}^{(k)} \partial_{ij}\tilde u_k - \tilde f_k (x)\right| & \le \sum_{i,j = 1}^n\frac{|x|^\alpha\rho_k^\alpha [a_{ij}^{(k)}]_{C^{0,\alpha}(B_1)}\|\partial_{ij} u_k\|_{L^\infty(B_{1})}}{\rho_k^\alpha [D^2 u_k]_{C^{0, \alpha}(B_1)}}\\
& \le C|x|^\alpha \frac{\|D^2 u_k\|_{L^\infty(B_{1})}}{[D^2 u_k]_{C^{0, \alpha}(B_1)}} \le C|x|^\alpha \frac{\|D^2 u_k\|_{L^\infty(B_{1})}}{[D^2 u_k]_{C^{0, \alpha}(B_{1/2})}}.
\end{align*}
Using \eqref{eq.contr.SAL_2} we deduce that, for any $x\in B_{\sigma}$ for some fixed $\sigma\in (0, \infty)$, and for $k$ large enough,
\[
\left|\sum_{i, j = 1}^n \tilde a_{ij}^{(k)} \partial_{ij}\tilde u_k - \tilde f_k (x)\right|\le  C(\sigma) \frac{\|D^2 u_k\|_{L^\infty(B_{1})}}{[D^2 u_k]_{C^{0, \alpha}(B_{1/2})}}\le \frac{C(\sigma)}{k}.
\]
Taking the limit $k\to \infty$ (and recalling that $\tilde f_k\to 0$ uniformly in compact sets) we get 
\[
\sum_{i, j = 1}^n \tilde a_{ij} \partial_{ij}\tilde u = 0\quad\textrm{in}\quad \R^n,
\]
an equation with constant coefficients, which is equivalent to $\Delta  \tilde u = 0$ in $\R^n$ (see Remark~\ref{rem.constcoef}), and we reach a contradiction as well.
\end{proof}

We can now proceed with the proof of the Schauder estimates in non-divergence form. 
Namely, we will show how to go from Proposition~\ref{prop.Schauder_estimates} to Theorem~\ref{thm.Schauder_estimates}. As with the previous results, we will do it in two different ways. In this case, however, both ways reduce to the same idea.

\begin{proof}[First Proof of Theorem~\ref{thm.Schauder_estimates}]
Define the semi-norm
\[
[D^2u]_{\alpha; B_1}^* := \sup_{B_\rho(x_\circ)\subset B_1} \rho^{2+\alpha}[D^2u]_{C^{0, \alpha}(B_{\rho/2}(x_\circ))}.
\]
Notice that this norm measures in a precise way how the $C^{2,\alpha}$ norm of $U$ blows up as we approach $\de B_1$. 

From the fact that H\"older semi-norms are sub-additive with respect to unions of convex sets,
\begin{equation}
\label{eq.comparable}
[D^2 u]^*_{\alpha; B_1}\le C\sup_{B_\rho(x_\circ)\subset B_1} \rho^{2+\alpha}[D^2u]_{C^{0, \alpha}(B_{\rho/4}(x_\circ))}
\end{equation}
(and, in fact, they are comparable) for some constant $C$ depending only on $\alpha$ and $n$. Indeed, for any fixed ball $B_\rho(x_\circ)\subset B_1$, we cover $B_{\rho/2}(x_\circ)$ with $N$ smaller balls $(B_{\rho/8}(z_j))_{1\le j\le N}$, which, since $B_{\rho/2}(z_j)\subset B_1$, gives
\[
\left(\frac{\rho}{2}\right)^{2+\alpha} [D^2u]_{C^{0, \alpha}(B_{\rho/8}(z_j))} \le \sup_{B_\rho(x_\circ)\subset B_1} \rho^{2+\alpha}[D^2u]_{C^{0, \alpha}(B_{\rho/4}(x_\circ))}.
\]
Thus, 
\begin{align*}
\rho^{2+\alpha}[D^2u]_{C^{0, \alpha}(B_{\rho/2}(x_\circ))} & \le \rho^{2+\alpha}\sum_{j = 1}^N [D^2u]_{C^{0, \alpha}(B_{\rho/8}(z_j))}\\
& \le 2^{2+\alpha} N\sup_{B_\rho(x_\circ)\subset B_1} \rho^{2+\alpha}[D^2u]_{C^{0, \alpha}(B_{\rho/4}(x_\circ))}.
\end{align*}
Taking the supremum on the left-hand side gives \eqref{eq.comparable}.

Applying the inequality 
\[
 [D^2 u]_{C^{0, \alpha}(B_{1/2})}\le \delta [D^2 u]_{C^{0, \alpha}(B_{1})} + C_\delta\left( \|u\|_{L^\infty( B_1)} + \|f\|_{C^{0, \alpha}(B_1)}\right)
\]
from Proposition~\ref{prop.Schauder_estimates} to any ball $B_{\rho/2}(x_\circ)\subset B_\rho(x_\circ)\subset B_1$ we get 
\begin{align*}
 \rho^{2+\alpha}[D^2 u]_{C^{0, \alpha}(B_{\rho/4})}& \le \delta\rho^{2+\alpha} [D^2 u]_{C^{0, \alpha}(B_{\rho/2})} + C_\delta\left( \|u\|_{L^\infty( B_1)} + \|f\|_{C^{0, \alpha}(B_1)}\right)\\
 & \le \delta[D^2 u]^*_{\alpha; B_1} + C_\delta\left( \|u\|_{L^\infty( B_1)} + \|f\|_{C^{0, \alpha}(B_1)}\right).
\end{align*}
Taking the supremum and using \eqref{eq.comparable} we get 
\begin{align*}
\frac{1}{C}[D^2 u]^*_{\alpha; B_1}& \le \sup_{B_\rho(x_\circ)\subset B_1} \rho^{2+\alpha}[D^2u]_{C^{0, \alpha}(B_{\rho/4}(x_\circ))}
 \\ 
 & \le  \delta[D^2 u]^*_{\alpha; B_1} + C\left( \|u\|_{L^\infty( B_1)} + \|f\|_{C^{0, \alpha}(B_1)}\right).
\end{align*}
Now, if we fix a small enough $\delta > 0$, we can absorb the $[D^2 u]^*_{\alpha; B_1}$ term on the left-hand side to get 
\[
[D^2 u]_{C^{0, \alpha}(B_{1/2})}\le [D^2 u]^*_{\alpha; B_1} \le  C_\delta\left( \|u\|_{L^\infty( B_1)} + \|f\|_{C^{0, \alpha}(B_1)}\right),
\]
which, after interpolation (see \eqref{ch0-interp2}) gives the desired result.
\end{proof}

We also give an alternative proof of Theorem~\ref{thm.Schauder_estimates} by directly using the following abstract lemma. Such lemma constitutes a generalization of the previous proof. 


\begin{lem}
\label{lem.SAL}
Let $k\in \R$ and $\gamma > 0$. Let $S$ be a non-negative function on the class of open convex subsets of $B_1$, and suppose that $S$ is sub-additive. That is, if $A, A_1, \dots, A_N$ are open convex subsets of $B_1$ with $A\subset \bigcup_{j = 1}^N A_j$, then  $S(A) \le \sum_{j = 1}^N S(A_j)$.

Then, there is $\delta > 0$ small (depending only on $n$ and $k$) such that,
if 
\[
\rho^k S(B_{\rho/2}(x_\circ))\le \delta \rho^k S(B_\rho(x_\circ))+\gamma \quad\textrm{for all } B_\rho(x_\circ) \subset B_1,
\]
then 
\[
S(B_{1/2}) \le C\gamma,
\]
for some $C$ depending only on $n$ and $k$. 
\end{lem}
\begin{proof}
Let 
\[
Q := \sup_{B_\rho(x_\circ)\subset B_1} \rho^k S(B_{\rho/2}(x_\circ)).
\]
Thanks to the assumption in the Lemma, we get 
\[
\left(\frac{\rho}{2}\right)^k S(B_{\rho/4}(x_\circ))\le \delta \left(\frac{\rho}{2}\right)^k S(B_{\rho/2}(x_\circ))+\gamma\le \delta Q+\gamma,\quad\textrm{for all } B_\rho(x_\circ)\subset B_1. 
\]
Taking now the supremum for all $B_\rho(x_\circ)\subset B_1$ we get 
\[
\tilde Q := \sup_{B_\rho(x_\circ) \subset B_1} \left(\frac{\rho}{2}\right)^k S(B_{\rho/4}(x_\circ)) \le \delta Q +\gamma.
\]

We now claim that 
\begin{equation}
\label{eq.claimSAL}
Q \le C\tilde Q,
\end{equation}
for some $C$ depending only on $n$ and $k$. This will yield 
\[
\frac{1}{C} Q \le\tilde Q \le \delta Q +\gamma \Rightarrow Q\le \tilde C \gamma
\]
if $\delta > 0$ is small enough depending only on $n$ and $k$. Thus, we have to show \eqref{eq.claimSAL}. 

Take any $B_\rho(x_\circ)\subset B_1$, and cover $B_{\rho/2}(x_\circ)$ with a finite collection of smaller balls $B_{\rho/8}(z_j)$ ($j = 1,2,\dots,N$), with $z_j \in B_{\rho/2}(x_\circ)$ and $N\le C$ (universally bounded depending only on the dimension). Since $B_{\rho /2}(z_j)\subset B_1$ we then have 
\[
\left(\frac{\rho}{4}\right)^k S(B_{\rho/8}(z_j))\le \tilde Q. 
\]
Adding up over all indices $j$, and using the sub-additivity of $S$, we obtain
\[
\rho^k S(B_{\rho/2}(x_\circ))\le \sum_{j = 1}^N \rho^kS(B_{\rho/8}(z_j))\le N 4^k\tilde Q = C\tilde Q. 
\]
Taking the supremum, we reach \eqref{eq.claimSAL}. 
\end{proof}
\begin{proof}[Second Proof of Theorem~\ref{thm.Schauder_estimates}]
We use Lemma~\ref{lem.SAL}, with $k = \alpha$ and 
\[
S(A):=  [D^2 u]_{C^{0, \alpha}(A)},
\]
which is sub-additive on open convex subsets. From the estimate in Proposition~\ref{prop.Schauder_estimates}, fixing $\delta>0$ from Lemma~\ref{lem.SAL} (which depends only on $\alpha$ and~$n$) we know
\[
 [D^2 u]_{C^{0, \alpha}(B_{1/2})}\le C\left( \|u\|_{L^\infty( B_1)} + \|f\|_{C^{0, \alpha}(B_1)}\right) + \delta [D^2 u ]_{C^{0, \alpha}(B_1)}.
\]
Rescaling\footnote{The rescaling is done by considering the estimate on $u_\rho(x) = u(x_\circ+\rho x)$, which fulfills $\sum a^{(\rho)}_{ij}(x) \de_{ij} u_\rho(x) = \rho^{2}f(x_\circ+\rho x) =: f_\rho(x)$ in $B_1$, with $a_{ij}^{(\rho)}(x) = a_{ij}(x_\circ + \rho x)$ (notice that $\|a_{ij}^{(\rho)}\|_{C^{0,\alpha}(B_1)}\le \|a_{ij}\|_{C^{0,\alpha}(B_1)}$). Then, $[D^2 u_\rho]_{C^{0,\alpha}(B_{1/2})} = \rho^{2+\alpha}[D^2 u]_{C^{0,\alpha}(B_{\rho/2}(x_\circ))}$ and $[f_\rho]_{C^{0,\alpha}(B_{1})} = \rho^{2+\alpha}[f]_{C^{0,\alpha}(B_{\rho}(x_\circ))}$.} to $B_\rho(x_\circ)$  with $\rho \le 1$ we obtain
\begin{align*}
\rho^{2+\alpha}[&D^2u ]_{C^{0, \alpha}(B_{\rho/2}(x_\circ))} \le  \\
&\le \delta  \rho^{2+\alpha} [D^2 u ]_{C^{0, \alpha}(B_\rho(x_\circ))} \\
&\qquad +C\left( \|u\|_{L^\infty( B_\rho(x_\circ))} + \rho^2 \|f\|_{L^\infty(B_\rho(x_\circ))}+\rho^{2+\alpha}[f]_{C^{0, \alpha}(B_\rho(x_\circ))}\right) \\
& \le \delta  \rho^{2+\alpha} [D^2 u ]_{C^{0, \alpha}(B_\rho)} +C\left( \|u\|_{L^\infty( B_1)} + \|f\|_{C^{0,\alpha}(B_1)}\right).
\end{align*}
This is exactly 
\[
\rho^k S(B_{\rho/2}(x_\circ))\le \delta \rho^k S(B_\rho(x_\circ))+\gamma ,
\]
with 
\[
\gamma = C_\delta\left( \|u\|_{L^\infty( B_1)} + \|f\|_{C^{0,\alpha}(B_1)}\right).
\]

Thus, thanks to Lemma~\ref{lem.SAL}, we immediately deduce 
\[
S(B_{1/2})\le C\gamma,
\]
that is, 
\[
 [D^2 u]_{C^{0, \alpha}(B_{1/2})}\le C\left( \|u\|_{L^\infty( B_1)} + \|f\|_{C^{0, \alpha}(B_1)}\right).
\]
Therefore, after using interpolation inequalities (see \eqref{ch0-interp2}) we get
\[
 \|u\|_{C^{2, \alpha}(B_{1/2})}\le C\left( \|u\|_{L^\infty( B_1)} + \|f\|_{C^{0, \alpha}(B_1)}\right)
\]
as desired.
\end{proof}
%

We finish this section by proving Corollary~\ref{cor.Schauder_estimates_HO}.

\begin{proof}[Proof of Corollary~\ref{cor.Schauder_estimates_HO}]
We follow the proof of Corollary~\ref{cor.Schauder_estimates_L_HO}. We will show by induction on $k$ that
\begin{equation}
\label{eq.toshowknondiv}
\|u\|_{C^{k+2, \alpha}(B_{1/2})} \le C\left(\|u\|_{L^\infty(B_1)}+\|f\|_{C^{k, \alpha}(B_1)}\right)
\end{equation}
for some constant $C$ depending only on $n$, $\alpha$, $k$, $\lambda$, $\Lambda$, and $\|a_{ij}\|_{C^{k, \alpha}(B_1)}$. We apply the induction hypothesis to derivatives of the equation in non-divergence form. 

As in the proof of Corollary~\ref{cor.Schauder_estimates_L_HO}, \eqref{eq.toshowknondiv}  deals with balls $B_{1/2}$ and $B_1$, but after a rescaling and covering argument (see Remark~\ref{rem.covering_argument}), it could also be stated in balls $B_{1/2}$ and $B_{3/4}$. 

The base case, $k = 0$, already holds by Theorem~\ref{thm.Schauder_estimates}. Let us now assume that \eqref{eq.toshowknondiv} holds for $k = m-1$, and we will show it for $k = m$. 

We differentiate the non-divergence-form equation with respect to $\de_e$ to get
\[
\sum_{i, j = 1}^n \ a_{ij}(x) \partial_{ij} \de_e u (x) = \de_e f(x)-\sum_{i, j = 1}^n \de_e a_{ij}(x) \partial_{ij} u (x)  \quad\textrm{in} \quad B_1.
\]

Now, we apply the estimate \eqref{eq.toshowknondiv} with $k = m-1$ to $\de_e u$ in the previous expression,  in balls $B_{1/2}$ and $B_{3/4}$,  to get
\begin{align*}
\|\de_e u\|_{C^{m+1, \alpha}(B_{1/2})} & \le C\bigg(\|\de_e u\|_{L^\infty(B_{3/4})}+\|\de_e f\|_{C^{m-1, \alpha}(B_{3/4})} \\
& \qquad\qquad \qquad + \sum_{i, j = 1}^n \|\de_e a_{ij} \de_{ij} u\|_{C^{m-1, \alpha}(B_{3/4})} \bigg). 
\end{align*}
Notice that 
\begin{align*}
\|\de_e a_{ij} \de_{ij} u\|_{C^{m-1, \alpha}(B_{3/4})} & \le \|\de_e a_{ij}\|_{C^{m-1, \alpha}(B_{3/4})}\|\de_{ij} u\|_{C^{m-1, \alpha}(B_{3/4})}\\
& = C\|\de_{ij} u\|_{C^{m-1, \alpha}(B_{3/4})}\\
& \le C\left(\|u\|_{L^\infty(B_1)}+\|f\|_{C^{m-1, \alpha}(B_1)}\right),
\end{align*}
where in the last inequality we have used the induction hypothesis in balls $B_{3/4}$ and $B_1$ (see Remark~\ref{rem.covering_argument}). Using that $\|\de_e u\|_{L^\infty(B_{3/4})}\le \|u\|_{C^{2, \alpha}(B_{3/4})}$ we can use the base case (with balls $B_{3/4}$ and $B_1$) of \eqref{eq.toshowknondiv} to bound this term. In all, we obtain that
\[
\|\de_e u\|_{C^{m+1, \alpha}(B_{1/2})}  \le C\left(\|u\|_{L^\infty(B_1)}+\|f\|_{C^{m-1, \alpha}(B_1)}\right),
\]
which, combined with the base case, and for every $e\in \mathbb{S}^{n-1}$, yields the desired estimate. 
\end{proof}

\section{Schauder estimates for operators  in divergence form}

We will next prove Schauder estimates for operators in divergence form. In particular, we will study the equation
\begin{equation}
\label{eq.div}
\boxed{
\divv \big(A(x)\nabla u(x)\big)  = \sum_{i,j = 1}^n\partial_i\big(a_{ij}(x) \partial_j u(x)\big)= f(x) \quad\textrm{in}\quad B_1,
}
\end{equation}
where $A(x) := (a_{ij}(x))_{ij}$ is uniformly elliptic, and $a_{ij}(x )\in C^{0,\alpha}$. 
Notice that, a priori, the expression \eqref{eq.div} does not make sense even for $C^\infty$ functions~$u$: we are taking derivatives of $a_{ij}(x)$, which is only $C^{0,\alpha}$. That is why we need to define a weak notion of solution to \eqref{eq.div}. Thus, we will say that $u\in H^1(B_1)$ solves \eqref{eq.div} weakly if 
\[
\int_{B_1} \nabla \phi(y)\cdot  A(y) \nabla u(y)\, dy = -\int_{B_1} \phi(y) f(y)\, dy\qquad\textrm{for all}\quad\phi\in C^\infty_c(B_1).
\]

We will prove the following: 

\begin{thm}[Schauder estimates in  divergence form]\index{Schauder estimates!Divergence form}
\label{thm.Schauder_estimates_div}
Let $u\in C^{1,\alpha}$ be a weak solution to 
\[
\sum_{i, j = 1}^n \partial_i\big(a_{ij}(x) \partial_{j} u (x)\big) = f(x) \quad\textrm{in} \quad B_1,
\]
with $f\in L^q(B_1)$ for $q \ge \frac{n}{1-\alpha}$, and $a_{ij}(x) \in C^{0,  \alpha}(B_1)$ for some $\alpha\in (0, 1)$, such that $(a_{ij}(x))_{ij}$  fulfills the ellipticity condition
\begin{equation}
\label{eq.ellipt_2}
0< \lambda\,{\rm Id} \le (a_{ij}(x))_{ij} \le\Lambda\,{\rm Id}\quad\textrm{in} \quad B_1,
\end{equation}
for some $0 < \lambda \le\Lambda < \infty$. Then,
\[
\|u\|_{C^{1, \alpha}(B_{1/2})} \le C \left(\|u\|_{L^\infty(B_1)} + \|f\|_{L^q(B_1)}\right)
\]
for some constant $C$ depending only on $\alpha$, $n$, $\lambda$, $\Lambda$, and $\|a_{ij}\|_{C^{0, \alpha}(B_1)}$.
\end{thm}

And as a consequence, we also get higher order Schauder estimates for operators in divergence form.

\begin{cor}[Higher order Schauder estimates in  divergence form]\index{Higher order Schauder estimates!Divergence form}
\label{cor.Schauder_estimates_div_HO}
Let $u\in C^{k+1,\alpha}$ be a weak solution to 
\[
\sum_{i, j = 1}^n \partial_i\big(a_{ij}(x) \partial_{j} u (x)\big) = f(x) \quad\textrm{in} \quad B_1,
\]
with $f\in C^{k-1+\alpha}(B_1)$ and $a_{ij}(x) \in C^{k,  \alpha}(B_1)$ for some $\alpha\in (0, 1)$, $k\in\N$, such that $(a_{ij}(x))_{ij}$  fulfills the ellipticity condition \eqref{eq.ellipt_2} for some $0 < \lambda \le\Lambda < \infty$. Then,
\[
\|u\|_{C^{k+1, \alpha}(B_{1/2})} \le C \left(\|u\|_{L^\infty(B_1)} + \|f\|_{C^{k-1, \alpha}(B_1)}\right)
\]
for some constant $C$ depending only on $\alpha$, $k$, $n$, $\lambda$, $\Lambda$, and $\|a_{ij}\|_{C^{k, \alpha}(B_1)}$.
\end{cor}

\subsection*{The maximum principle} As in the case of operators in non-divergence form, we also have a  maximum principle for equations in divergence form. 

\begin{prop}[Maximum Principle in divergence form]\index{Maximum principle!Divergence form}
Let $\Omega\subset\R^n$ be a bounded open set. Suppose that $u\in H^{1}(\Omega)$ satisfies, in the weak sense, 
\[
\sum_{i, j = 1}^n \partial_i \big(a_{ij}(x) \partial_{j} u (x)\big) \ge 0 \quad\textrm{in} \quad \Omega,
\]
where $(a_{ij}(x))_{ij}\in L^\infty(\Omega)$  fulfill the  pointwise ellipticity condition,
\[
0< (a_{ij}(x))_{ij} \quad\textrm{in} \quad \Omega.
\]
Then,
\[
\sup_{\Omega} u  = \sup_{\partial \Omega} u. 
\]
\end{prop}

\begin{proof}
We know that, denoting $A(x) = (a_{ij}(x))_{ij}$, 
\[
\int_\Omega \nabla \phi \cdot A(x) \nabla u \, dx \le 0\quad\textrm{for all}\quad\phi\in C^\infty_c(\Omega),~\phi \ge 0.  
\]
In particular, by approximation (see \ref{it.S7} in Chapter~\ref{ch.0}), the previous expression holds for all $\phi \in H^1_0(\Omega)$  such that $\phi \ge 0$. We take, as test function, $\phi(x) := (u - \sup_{\de \Omega} u)^+\in H^1_0(\Omega)$, where $f^+ := \max\{f, 0\}$ denotes the positive part. Then, 
\[
\int_\Omega \nabla \phi \cdot A(x) \nabla \phi \, dx = \int_\Omega \nabla \phi \cdot A(x) \nabla u \, dx\le 0.
\]
Since $A(x) > 0$, this implies that $\nabla\phi \equiv 0$, and $\phi$ is constant. Since $\phi \in H^1_0(\Omega)$, this implies that $\phi\equiv 0$, that is, $u \le\sup_{\de\Omega} u$ in $\Omega$, as wanted. 
\end{proof}

\subsection*{Proof of Schauder estimates}

We proceed with the proof of Theorem~\ref{thm.Schauder_estimates_div}. We will do so via a blow-up argument, in the spirit of the second proof of Proposition~\ref{prop.Schauder_estimates}.

\begin{proof}[Proof of Theorem~\ref{thm.Schauder_estimates_div}]

As in the (second) proof of Proposition~\ref{prop.Schauder_estimates}, we will show that, for any $\delta > 0$,
\begin{equation}
\label{eq.contr.SAL_div}
[\nabla u]_{C^{0,\alpha}(B_{1/2})}\le \delta[\nabla u]_{C^{0,\alpha}(B_{1})}+ C_\delta\left(\|\nabla u\|_{L^\infty(B_1)}+\|f \|_{L^q(B_1)}\right)
\end{equation}
for all $u \in C^{1,\alpha}(B_1)$ such that 
\[
\divv(A(x)\nabla u(x)) = \sum_{i,j= 1}^n \partial_i\left(a_{ij}(x) \partial_j u(x)\right) = f(x),\quad\textrm{weakly in}\quad B_1. 
\]
This yields 
\[
\|u\|_{C^{1,\alpha}(B_{1/2})}\le \delta[\nabla  u]_{C^{0,\alpha}(B_{1})}+ C_\delta \left(\|u\|_{L^\infty(B_1)}+\|f\|_{L^q(B_1)}\right)
\]
and so, proceeding as in the proof of Theorem~\ref{thm.Schauder_estimates} by using Lemma~\ref{lem.SAL}, (or, alternatively, adapting the first proof of Theorem~\ref{thm.Schauder_estimates}), we get the desired result. Let us focus, therefore, on the proof of \eqref{eq.contr.SAL_div}: 

Suppose that it does not hold. Then, there exist sequences $u_k\in C^{1, \alpha}(B_1)$ and $f_k\in L^q(B_1)$ for $k\in \N$ such that
\[
\divv\big(A_k(x)u_k(x)\big) = f_k(x)\quad\textrm{weakly in} \quad B_1,
\]
and for a fixed small constant $\delta_\circ  >0$ we have
\begin{equation}
\label{eq.contr.SAL_2_div}
[\nabla u_k]_{C^{0,\alpha}(B_{1/2})}> \delta_\circ[\nabla u_k]_{C^{0,\alpha}(B_{1})}+ k\left(\|\nabla u_k\|_{L^\infty(B_1)}+\|f_k\|_{L^q(B_1)}\right).
\end{equation}
We now have to reach a contradiction. 

Select $x_k, y_k\in B_{1/2}$ such that 
\begin{equation}
\label{eq.halfeq_div}
\frac{|\nabla u_k(x_k)-\nabla u_k(y_k)|}{|x_k-y_k|^\alpha}\ge \frac12 [\nabla u_k]_{C^{0,\alpha}(B_{1/2})}
\end{equation}
and let 
\[
\rho_k := \frac{|x_k-y_k|}{2},\quad\textrm{and}\quad z_k := \frac{x_k+y_k}{2}.
\]
Then, as in Proposition~\ref{prop.Schauder_estimates}, $\rho_k \le Ck^{-\frac{1}{\alpha}} \to 0$ as $ k \to \infty$. Define
\[
\tilde u_k(x) := \frac{u_k(z_k+\rho_k x)+u_k(z_k-\rho_k x)-2u_k(z_k)}{\rho_k^{1+\alpha}[\nabla  u_k]_{C^{0,\alpha}(B_1)}}.
\]
Then,
\begin{equation}
\label{eq.AA1_d}
\tilde u_k(0) = |\nabla \tilde u_k(0)| = 0.
\end{equation}
We remark that here, instead of defining $\tilde u_k$ as in Proposition~\ref{prop.Schauder_estimates} (i.e., subtracting a quadratic polynomial), we have used second order incremental quotients. 

Let us also denote 
\[
\xi_k := \frac{y_k-x_k}{2\rho_k}\in \mathbb{S}^{n-1}.
\]

Notice that 
\begin{equation}
\label{eq.AA2_d}
[\nabla   \tilde u_k]_{C^{0,\alpha}\left(B_{1/(2\rho_k)}\right)} \le 2,\qquad\textrm{and}\qquad |\nabla  \tilde u_k (\xi_k)|> \frac{\delta_\circ}{2}, 
\end{equation}
where for the second inequality we use \eqref{eq.contr.SAL_2_div} and \eqref{eq.halfeq_div}. 

Since $\tilde u_k$ are uniformly bounded in compact subsets, and bounded in the $C^{1, \alpha}$ norm (due to \eqref{eq.AA1_d} and \eqref{eq.AA2_d}), it follows by Arzel\`a--Ascoli that the sequence $\tilde u_k$ converges (in the $C^1$ norm) to a $C^{1, \alpha}$ function $\tilde u$ on compact subsets of $\R^n$ (up to a subsequence). Moreover, again up to a subsequence, we have that $\xi_k\to \xi\in \mathbb{S}^{n-1}$.

By the properties of $\tilde u_k$, we deduce that $\tilde u$ satisfies 
\begin{equation}
\label{eq.contradiction2_div}
\tilde u (0) = |\nabla \tilde u(0) | =  0,\quad[\nabla \tilde u ]_{C^{0,\alpha}(\R^n)}\le 2,\quad |\nabla \tilde u(\xi)|>\frac{\delta_\circ}{2}.
\end{equation}

Let us   check which equation does $\tilde u_k$ satisfy. Let
\[
\tilde a_{ij}^{(k)} (x) := a_{ij}^{(k)}(z_k+\rho_k x),
\]
so that, as in Proposition~\ref{prop.Schauder_estimates}, $\tilde a_{ij}^{(k)}$ converges uniformly in compact sets to some $\tilde a_{ij}$ constant. For any $\phi\in C^\infty_c(B_1)$, we know that 
\begin{equation}
\label{eq.divweknow}
\int_{B_1} \nabla\phi \cdot A_k(x) \nabla u_k = -\int_{B_1} f_k \phi.
\end{equation}
Let $\tilde A_k(x) := A_k(z_k + \rho_k x) = (\tilde a_{ij}^{(k)}(x))_{ij}$. Let $\phi \in C_c^\infty(\R^n)$, and let $k$ be large enough so that ${\rm supp}\, \phi \subset \, B_{1/(2\rho_k)}$. Let
\[
\int \nabla\phi \cdot \tilde A_k(x) \nabla \tilde u_k = \textrm{I} - \textrm{II},
\]
where 
\begin{align*}
\textrm{I} & = \frac{1}{\rho_k^{\alpha}[\nabla u_k]_{C^{0,\alpha}(B_1)}} \int \nabla\phi(x) \cdot A_k(z_k+\rho_k x) \nabla u_k(z_k+\rho_k x)\, dx \\
& = \frac{1}{\rho_k^{\alpha}[\nabla u_k]_{C^{0,\alpha}(B_1)}} \int \nabla_y\left(\phi\left(\rho_k^{-1}(y-z_k)\right)\right) \cdot A_k(y) \nabla u_k(y) \rho_k^{-n+1}\, dy\\
 & = \frac{-\rho_k^{1-\alpha}}{[\nabla u_k]_{C^{0,\alpha}(B_1)}}  \int \phi (x) f_k(z_k+\rho_k x) \, dx,
\end{align*}
thanks to \eqref{eq.divweknow}, and 
\[
\textrm{II} = \frac{1}{\rho_k^{\alpha}[\nabla u_k]_{C^{0,\alpha}(B_1)}} \int \nabla\phi(x) \cdot A_k(z_k+\rho_k x) \nabla u_k(z_k-\rho_k x)\, dx = \textrm{II}_{i} + \textrm{II}_{ii}.
\]
Here, we have denoted by $\textrm{II}_i$ and $\textrm{II}_{ii}$ the following quantities: 
\[
\textrm{II}_i = \frac{1}{\rho_k^{\alpha}[\nabla u_k]_{C^{0,\alpha}(B_1)}} \int \nabla\phi(x) \cdot (A_k(z_k+\rho_k x)  - A_k(z_k-\rho_k x)) \nabla u_k(z_k-\rho_k x)\, dx
\]
and 
\begin{align*}
\textrm{II}_{ii}& = \frac{1}{\rho_k^{\alpha}[\nabla u_k]_{C^{0,\alpha}(B_1)}} \int \nabla\phi(x) \cdot A_k(z_k-\rho_k x) \nabla u_k(z_k-\rho_k x)\, dx \\
& = \frac{\rho_k^{1-\alpha}}{[\nabla u_k]_{C^{0,\alpha}(B_1)}}  \int \phi (x) f_k(z_k-\rho_k x) \, dx.
\end{align*}
Let us now show that 
\[
\left|\int \nabla\phi \cdot \tilde A_k \nabla \tilde u_k \right|\to 0,\quad\textrm{as}\quad k\to \infty
\]
for all $\phi \in C^\infty_c(\R^n)$, by bounding each term separately. 

Notice that, for $\frac{1}{q}+\frac{1}{q'} = 1$,
\begin{align*}
\left|\int \phi (x) f_k(z_k+\rho_k x) \, dx \right|& \le \left(\int |\phi|^{q'}\right)^{\frac{1}{q'}}\left(\int |f_k(z_k+\rho_k x)|^q\, dx\right)^{\frac{1}{q}}
\\ 
& \le C(\phi) \|f_k\|_{L^q(B_1)}\rho_k^{-\frac{n}{q}}.
\end{align*}
Then,
\[
|\textrm{I}|\le C(\phi)\rho_k^{1-\alpha-\frac{n}{q}}\frac{\|f_k\|_{L^q(B_1)}}{[\nabla u_k]_{C^{0,\alpha}(B_1)}} \le C(\phi)\rho_k^{1-\alpha-\frac{n}{q}}k^{-1}\to 0,\quad\textrm{as}\quad k \to \infty,
\]
as long as $1-\alpha-\frac{n}{q}\ge 0$, that is, $q \ge \frac{n}{1-\alpha}$. In the last step we have used \eqref{eq.contr.SAL_2_div}. Similarly, 
\[
|\textrm{II}_{ii}|\to 0,\quad\textrm{as}\quad k \to \infty,
\]
since $q \ge \frac{n}{1-\alpha}$. Finally, 
\begin{align*}
|\textrm{II}_i|& \le \frac{[A_k]_{C^{0,\alpha}(B_1)}}{[\nabla u_k]_{C^{0,\alpha}(B_1)}} \int |\nabla\phi| |x|^\alpha \|\nabla u\|_{L^\infty(B_1)}\, dx\\
& \le C(\phi) \frac{\|\nabla u\|_{L^\infty(B_1)}}{[\nabla u_k]_{C^{0,\alpha}(B_1)}} \le \frac{C(\phi)}{k}\to 0,\quad\textrm{as}\quad k \to \infty.
\end{align*}
Here, we used again \eqref{eq.contr.SAL_2_div}. That is, $|\textrm{II}_i|\to 0 $ uniformly in compact sets of~$\R^n$. 

Then we conclude that, for any $\phi\in C_c^\infty(\R^n)$, 
\[
\left|\int \nabla\phi \cdot \tilde A_k(x) \nabla \tilde u_k \right| \to 0,\quad\textrm{as}\quad k \to \infty. 
\]
By taking limits, up to a subsequence we will have that $\tilde A_k\to \tilde A$ uniformly in compact sets, where $\tilde A$ is a constant coefficient matrix.
Thus, we deduce that 
\[
\int \nabla\phi \cdot \tilde A \nabla \tilde u = 0\quad\textrm{for all}\quad \phi\in C_c^\infty(\R^n).
\]
This means that, after a change of variables, $\tilde u$ is harmonic (recall Remark~\ref{rem.constcoef}). By Liouville's theorem (Proposition~\ref{cor.Liouville}) we obtain that $\nabla \tilde u$ must be constant (since it is harmonic, and $[\nabla \tilde u]_{C^{0, \alpha}(\R^n)}\le 2$). However, $\nabla \tilde u (0) = 0$ and $\nabla \tilde u (\xi)\neq 0$ (see \eqref{eq.contradiction2_div}), a contradiction. 
\end{proof}

\begin{proof}[Proof of Corollary~\ref{cor.Schauder_estimates_div_HO}]
We proceed by induction on $k$. The case $k = 0$ is due to Theorem~\ref{thm.Schauder_estimates_div}. Then, let us assume that
\begin{equation}
\label{eq.inddiv}
\|u\|_{C^{k+1, \alpha}(B_{1/2})} \le C \left(\|u\|_{L^\infty(B_1)} + \|f\|_{C^{k-1,\alpha}(B_1)}\right)
\end{equation}
holds for all $k \le m-1$, and let us show it for $k = m$. 

To do so, notice that, since $a_{ij}(x) \in C^{m, \alpha}$, and $m \ge 1$, we can compute the derivatives in the divergence-form equation, to get
\[
\sum_{i, j = 1}^n a_{ij}(x) \partial_{ij} u = f(x)-\sum_{i, j = 1}^n \partial_ia_{ij}(x) \partial_{j} u \quad\textrm{in} \quad B_1,
\]
that is, a non-divergence-form equation, where the right-hand side is in $C^{m-1, \alpha}$. Applying the higher order Schauder estimates for equations in non-divergence form, Corollary~\ref{cor.Schauder_estimates_HO} (in balls $B_{1/2}$ and $B_{3/4}$), we get that 
\begin{align*}
\|u\|_{C^{m+1, \alpha}(B_{1/2})} & \le C \bigg(\|u\|_{L^\infty(B_{3/4})} + \|f\|_{C^{m-1,\alpha}(B_{3/4})}+ \\
& \qquad\qquad + \sum_{i, j = 1}^n \|\partial_i(a_{ij})\|_{C^{m-1, \alpha}(B_{3/4})} \|\partial_{j} u \|_{C^{m-1, \alpha}(B_{3/4})}\bigg),
\end{align*}
that is, 
\[
\|u\|_{C^{m+1, \alpha}(B_{1/2})}  \le C \bigg(\|u\|_{L^\infty(B_{3/4})} + \|f\|_{C^{m-1,\alpha}(B_{3/4})} + \| u \|_{C^{m, \alpha}(B_{3/4})}\bigg),
\]
where the constant $C$ depends only on $n$, $\alpha$, $\lambda$, $\Lambda$, and $\|a_{ij}\|_{C^{m, \alpha}(B_1)}$. Using now the hypothesis induction, \eqref{eq.inddiv} for $k = m-1$, in balls $B_{3/4}$ and $B_1$, completes the proof.
\end{proof}

\section{The case of continuous coefficients}

Let us finish this chapter by studying equations in divergence and non-divergence form with continuous coefficients. 

In this section we establish a priori Schauder estimates for \eqref{eq.mainEPDE} and \eqref{eq.mainEPDE_div} whenever $a_{ij}\in C^0(B_1)$ (and the right-hand side is bounded or in $L^n$ respectively). This kind of estimates will be useful in the next chapters. 

In this limiting case (when $\alpha \downarrow 0$), one could extrapolate from the previous results that the solution has respectively bounded $C^2$ and $C^1$ norm. However, this is not true. 

We will show, instead, that we gain \emph{almost} two derivatives. Namely, for any $\eps > 0$, the solution has bounded $C^{2-\eps}$ and $C^{1-\eps}$ norm. More precisely, we prove below the following results:

\begin{prop}
\label{prop.Schauder_estimates_cont}\index{Schauder estimates for continuous coefficients!Non-divergence form}
Let $u\in C^{2}$ be any solution to 
\[
\sum_{i, j = 1}^n a_{ij}(x) \partial_{ij} u = f(x) \quad\textrm{in} \quad B_1,
\]
with $f\in L^\infty(B_1)$ and $a_{ij} \in C^{0}(B_1)$ for some  $(a_{ij}(x))_{ij}$  satisfying \eqref{eq.ellipt} for some $0 < \lambda \le\Lambda$. Then, for any $\eps > 0$,
\[
\|u\|_{C^{1, 1-\eps}(B_{1/2})} \le C_\eps \left(\|u\|_{L^\infty(B_1)} + \|f\|_{L^\infty(B_1)}\right)
\]
for some constant $C_\eps$ depending only on $\eps$, $n$, $\lambda$, $\Lambda$, and $(a_{ij})_{ij}$.
\end{prop}

That is,  we are not gaining two {\it full} derivatives, but instead we are losing an arbitrarily small factor. This loss is paired with the fact that the constant $C_\eps$ diverges when $\eps\downarrow 0$; see \cite{JMV09, EM17} for counterexamples in the case $\eps = 0$. This is also consistent with what occurs with the Laplacian (see the counterexample at the beginning of Section~\ref{sec.Sch_Lap}).

We remark that the dependence of $C$ on $(a_{ij})_{i,j}$ in the previous proposition is a dependence on the modulus of continuity of $(a_{ij})_{i,j}$. That is, if $\omega: [0, \infty)\to [0, \infty)$ is a continuous monotone function with $\omega(0) = 0$ and such that
\[
|a_{ij}(x)-a_{ij}(y)|\le\omega(|x-y|), \quad\textrm{for all}\quad x, y\in B_1,
\]
then the constant in the previous proposition depends on $\omega$ rather than on~$(a_{ij})_{i,j}$.

For divergence-form equations we have the following:

\begin{prop}
\label{prop.Schauder_estimates_div}\index{Schauder estimates for continuous coefficients!Divergence form}
Let $u\in C^{1}$ be a weak solution to 
\[
\sum_{i, j = 1}^n \partial_i\big(a_{ij}(x) \partial_{j} u\big) = f(x) \quad\textrm{in} \quad B_1,
\]
with $f\in L^n(B_1)$ and  $a_{ij}(x) \in C^{0}(B_1)$ satisfying the ellipticity conditions \eqref{eq.ellipt_2} for some $0 < \lambda \le\Lambda$. Then,  for any $\eps > 0$,
\[
\|u\|_{C^{1-\eps}(B_{1/2})} \le C_\eps \left(\|u\|_{L^\infty(B_1)} + \|f\|_{L^\infty(B_1)}\right)
\]
for some constant $C_\eps$ depending only on $\eps$, $n$, $\lambda$, $\Lambda$, and $(a_{ij})_{ij}$.
\end{prop}

The proofs of the previous two propositions are analogous to those of the Schauder estimates for operators in non-divergence and divergence form respectively. 

We give short sketches of the proofs of Propositions~\ref{prop.Schauder_estimates_cont} and \ref{prop.Schauder_estimates_div} that contain all the essential information regarding the steps to take.

\begin{proof}[Sketch of the proof of Proposition~\ref{prop.Schauder_estimates_cont}]
We give a short sketch of the proof in the case $(a_{ij}(x))_{ij} = {\rm Id}$, and leave the details to the reader. The proof sketched follows the same steps and arguments as the second proof of Proposition~\ref{prop.Schauder_estimates}.

Proceeding analogously, and after using Lemma~\ref{lem.SAL}, (cf. first or second proof of Theorem~\ref{thm.Schauder_estimates}), we just need to show that for any $\delta > 0$
\[
[\nabla u]_{C^{1-\eps}(B_{1/2})} \le \delta [\nabla u]_{C^{1-\eps}(B_{1})} + C_\delta(\|\nabla u\|_{L^\infty(B_1)} + \|f\|_{L^\infty(B_1)}),
\]
for some $C_\delta$.

By contradiction, suppose that we have a sequence $f_k\in L^\infty(B_1)$, $u_k\in C^2(B_1)$, and coefficients $(a_{ij}^{(k)})_{ij}$ with a common modulus of continuity, such that $\sum_{i, j= 1}^n a_{ij}^{(k)}\partial_{ij} u_k = f_k$ and
\begin{equation}
\label{eq.contr.SAL_2_CC}
[\nabla u_k]_{C^{1-\eps}(B_{1/2})} > \delta_\circ [\nabla u_k]_{C^{1-\eps}(B_{1})} + k(\|\nabla u_k\|_{L^\infty(B_1)} + \|f_k\|_{L^\infty(B_1)}),
\end{equation}
for some $\delta_\circ > 0$.

Select $x_k, y_k\in B_{1/2}$ such that 
\begin{equation}
\label{eq.halfeq_CC}
\frac{|\nabla u_k(x_k)-\nabla u_k(y_k)|}{|x_k-y_k|^{1-\eps}}\ge \frac12 [\nabla u_k]_{C^{1-\eps}(B_{1/2})}
\end{equation}
and let $\rho_k := |x_k-y_k|$, so that as in the second proof of Proposition~\ref{prop.Schauder_estimates} $\rho_k\downarrow 0$.  Define
\[
\tilde u_k(x) := \frac{u_k(x_k+\rho_k x)-u_k(x_k) - \rho_k \nabla u_k(x_k)\cdot x}{\rho_k^{2-\eps}[\nabla u_k]_{C^{1-\eps}(B_1)}}\]
and
\[
\quad\tilde f_k(x) := \rho_k^{\eps}\frac{f_k(x_k+\rho_k x)-f_k(x_k)}{[\nabla u_k]_{C^{1-\eps}(B_1)}},
\]
so that 
\begin{equation}
\label{eq.equation_visc}
\tilde u_k(0) = |\nabla \tilde u_k(0)|  = 0,\qquad \sum_{i, j = 1}^n \tilde a_{ij}^{(k)}\de_{ij} \tilde u_k = \tilde f_k\quad\textrm{in}\quad B_{1/(2\rho_k)}.
\end{equation}
where 
\[
\tilde a_{ij}^{(k)}(x) := a^{(k)}_{ij}(z_k +\rho_k x).
\]
Denoting $\xi_k := \frac{y_k-x_k}{\rho_k}\in \mathbb{S}^{n-1}$, we have 
\[
[\nabla  \tilde u_k]_{C^{1-\eps}\big(B_{\frac{1}{2\rho_k}}\big)} \le 1,\qquad\textrm{and}\qquad \big|\nabla \tilde u_k (\xi_k)\big|> \frac{\delta_\circ}{2}, 
\]
by means of  \eqref{eq.contr.SAL_2_CC} and \eqref{eq.halfeq_CC}. 

As in Proposition~\ref{prop.Schauder_estimates}, $\tilde u_k$ converges (up to a subsequence and in the $C^1$ norm) to a $C^{1,1-\eps}$ function $\tilde u$ on compact subsets of $\R^n$, and $\xi_k\to \xi\in \mathbb{S}^{n-1}$. Furthermore, \[
\tilde u (0) = |\nabla \tilde u(0) | = 0,\quad[\nabla \tilde u ]_{C^{1-\eps}(\R^n)}\le 1,\quad |\nabla \tilde u(\xi)|>\frac{\delta_\circ}{2}.
\]

On the other hand, for any $R\ge 1$ we have 
\begin{align*}
\|\tilde f_k\|_{L^\infty(B_R)} & \le \frac{\rho_k^\eps} {k}\to 0,\textrm{ as } k \to \infty,
\end{align*}
and that, from the uniform modulus of continuity of $a_{ij}^{(k)}$, $\tilde a_{ij}^{(k)}(x) \to \tilde a_{ij}$
locally uniformly in $\R^n$, where the limiting coefficients $\tilde a_{ij}$ are constant. At this point, in the  equation \eqref{eq.equation_visc} the coefficients converge locally uniformly to constant coefficients, and the solutions $\tilde u_k$ converge simply in $C^1$. The passage to the limit is now more involved than before: in order to do it, we need the notion of viscosity solutions (see Definition~\ref{ch0-viscosity} and Section~\ref{sec.43}) and the fact that they are stable under uniform limits (see Proposition~\ref{prop.stability_viscosity}). In all, we can show that the limiting $\tilde u$ satisfies
\[
\sum_{i, j= 1}^n \tilde a_{ij}\partial_{ij}\tilde u = 0 \qquad\text{in}\quad \R^n
\]
(in the viscosity sense).  Hence, the limiting solution $\tilde u$ is harmonic (after changing variables) and we reach a contradiction as in the second proof of Proposition~\ref{prop.Schauder_estimates}. 
\end{proof}

In order to prove the convergence of the sequence in the proof of Proposition~\ref{prop.Schauder_estimates_div} we will need the following lemma:
\begin{lem}
\label{lem.takelim}
Let $u\in H^1(B_1)$ satisfy 
\begin{equation}
\label{eq.theequation}
{\rm div}(A (x)\nabla u(x)) = f(x)\quad\text{in}\quad B_1,
\end{equation}
in the weak sense, for some $f\in L^2(B_1)$ and $A(x) = (a_{ij}(x))_{ij}$ uniformly elliptic with ellipticity constants $\lambda$ and $\Lambda$ (see \eqref{eq.ellipt}). Then
\[
\|\nabla u\|_{L^{2}(B_{1/2})}\le C(\|u\|_{L^2(B_1)}+\|f\|_{L^2(B_1)})
\]
for some $C$ depending only on $\lambda$, $\Lambda$, and $n$. 
\end{lem}
\begin{proof}
Let us prove the lemma in the case $A(x)$ is symmetric for all $x\in B_1$. 

Let $\eta\in C^\infty_c(B_1)$ be arbitrary with $\eta \equiv 1$ in $B_{1/2}$, and observe that 
\[
\int_{B_{1/2}} |\nabla u|^2 \le C \int_{B_1} \nabla (u\eta) \cdot A(x) \nabla (u\eta) \, dx
\]
by ellipticity. In particular, since $A(x)$ is symmetric for all $x\in B_1$ we can use that $\nabla(u\eta)\cdot A \nabla(u\eta) = u^2 \nabla \eta\cdot A\nabla\eta +   \nabla(u\eta^2) \cdot A\nabla u$ and the equation \eqref{eq.theequation} to get 
\[
\int_{B_{1/2}} |\nabla u|^2 \le C \int_{B_1} u^2 |\nabla \eta|^2 + C\int_{B_1} |fu|\eta^2.
\]
By H\"older's inequality, we get the desired estimate. We refer to the proof of Lemma~\ref{lem.energyinequality} for more details on the proof and on the non-symmetric case in a very similar situation. 
\end{proof}

Let us now give the proof of Proposition~\ref{prop.Schauder_estimates_div}. 

\begin{proof}[Proof of Proposition~\ref{prop.Schauder_estimates_div}]
The proof is by contradiction and proceeds as the proof of Theorem~\ref{thm.Schauder_estimates_div}, with the analogous modifications introduced in the Sketch of the proof of Proposition~\ref{prop.Schauder_estimates_cont} with respect to the proof of Proposition~\ref{prop.Schauder_estimates}. 

Observe that, in this case, we should define $\tilde u_k(x)$ as first order incremental quotients: 
\[
\tilde u_k(x) = \frac{u_k(x_k + \rho_k x)  -u_k(x_k)}{\rho_k^{1-\eps} [u_k]_{C^{1-\eps}(B_1)}}
\]
so that we directly have (differently from the proof of Theorem~\ref{thm.Schauder_estimates_div}) that $\tilde u_k$ satisfies:
\begin{equation}
\label{eq.takelim}
{\rm div}(\tilde A_k(x)\nabla  \tilde u_k(x)) = \tilde f_k(x)\quad\text{in}\quad B_{{1}/{(2\rho_k)}},\qquad \tilde f_k(x) = \frac{\rho_k^{1+\eps} f_k(x_k +\rho_k x)}{[u_k]_{C^{1-\eps}(B_1)}},
\end{equation}
in the weak sense, where $\tilde A_k(x) := A_k(x_k + \rho_k x)$ and $\|\tilde f_k\|_{L^n(B_{1/(2\rho_k)})}\downarrow 0$ as $k\to \infty$. In particular, $\tilde A_k(x)$ converges to some constant matrix $\tilde A_\infty$ locally uniformly by uniform continuity of $A_k$.

On the other hand, observe that each $\tilde u_k$ is in $H^1$ (since they are $C^1$ by assumption), and they are locally uniformly in $L^2$ (since they are uniformly locally bounded). Hence, we can apply Lemma~\ref{lem.takelim} to get that $\tilde u_k$ are locally uniformly bounded in $H^1$. In particular, by \ref{it.S4} from Chapter~\ref{ch.0} (see \eqref{ch0-weak-conv}) $\nabla \tilde u_k$ converges weakly to $\nabla \tilde u_\infty$. Thus:
\[
\int_{\R^n} \nabla \phi \cdot \tilde A_k\nabla \tilde u_k \to \int_{\R^n}\nabla \phi \cdot \tilde A_\infty \nabla \tilde u_\infty\quad\text{as}\quad k\to \infty, \quad \text{for all }\phi\in C^\infty_c(\R^n),
\]
and from \eqref{eq.takelim} we have that $u_\infty$ is harmonic (after changing variables) in $\R^n$.  The contradiction is now reached, again, by the Liouville theorem, Proposition~\ref{cor.Liouville}.
\end{proof}

\begin{rem}
The blow-up technique is a common tool in analysis that has great versatility. In particular, the technique presented in this section is due to L. Simon, \cite{Sim}, and can be applied in a similar fashion to many different situations. 
We have seen the technique applied in interior a priori estimates for linear second-order equations, both in divergence and non-divergence form, and blow-up arguments like the one presented above can be adapted also to boundary estimates, parabolic equations, nonlinear equations, and even integro-differential equations. 
\end{rem}

\section{Boundary regularity}

We finish the chapter by stating the corresponding results to Corollaries~\ref{cor.Schauder_estimates_L} and \ref{cor.Schauder_estimates_L_HO} for the global (up to the boundary) estimates, for a sufficiently smooth domain.

For the sake of readability we state the result for the Laplacian, but there exists an analogous result for uniformly elliptic equations in non-divergence form (with the corresponding regularity on the coefficients).

\begin{thm}[Boundary regularity]\index{Boundary regularity Laplace equation}
\label{thm.Global_Schauder_estimates}
Let $\alpha\in (0, 1)$ and $k\in \N$ with $k\ge 2$, and let $\Omega$ be a bounded $C^{k, \alpha}$ domain of $\R^n$. Let $u\in H^1({\Omega})$ be a weak solution to
\begin{equation}
\label{eq.glob_eq}
\left\{
\begin{array}{rcll}
\Delta u &=& f& \textrm{in }\Omega\\
u & = & g& \textrm{on }\partial \Omega,
\end{array}
\right.
\end{equation}
for some $f\in C^{k-2,\alpha}(\overline{\Omega})$, $g\in C^{k,\alpha}(\de\Omega)$. 

Then, $u\in C^{k,\alpha}(\overline{\Omega})$ and
\[
\|u\|_{C^{k, \alpha}(\overline\Omega)}\le C\left(\|f\|_{C^{k-2,\alpha}(\overline\Omega)}+\|g\|_{C^{k,\alpha}(\partial \Omega)}\right),
\]
for some constant $C$ depending only on $\alpha$, $n$, $k$, and $\Omega$. 
\end{thm}
\begin{rem}
Notice that in this case we do not need a term $\|u\|_{L^\infty(\Omega)}$ on the right-hand side because, thanks to the maximum principle (Lemma~\ref{lem.maxPrinciple_2}), 
\[
\max_{\overline{\Omega}} u \le C\left(\max_{\partial \Omega} g+ \|f\|_{L^\infty(\Omega)}\right)
\]
for some $C$ depending only on $\Omega$, $\lambda$, $\Lambda$, and $M$. 
\end{rem}

Theorem~\ref{thm.Global_Schauder_estimates} can be proved using similar techniques (correspondingly adapted) to the ones in the previous sections: after a blow-up, points near the boundary behave like in a local problem in the half-space (that is, the blow-up flattens $\de\Omega$), and we can reach a contradiction with Liouville's theorem in the half-space.

One might wonder what happens under lower regularity assumptions on the domain (we refer to \cite{Ken94, K} for further reading in this direction). In such case, similar regularity results hold in $C^{1,\alpha}$ (and even $C^1$) domains, but when $\Omega$ is merely Lipschitz, almost all regularity is lost. Namely, assume that $u$ solves \eqref{eq.glob_eq}, with $f$ and $g$ smooth enough. Then,
\begin{itemize}
\item If $\Omega$ is a $C^{1,\alpha}$ domain, then solutions are $C^{1, \alpha}(\overline{\Omega})$.
\item If $\Omega$ is a $C^1$ domain, then solutions are $C^{1-\eps}(\overline{\Omega})$ for all $\eps > 0$, but not $C^{0,1}(\overline{\Omega})$ in general. 
\item If $\Omega$ is a Lipschitz domain, then solutions are $C^\gamma(\overline{\Omega})$ for some small $\gamma > 0$ that depends on the Lipschitz norm of the domain, and this is optimal.
\end{itemize}

We see that, if $\Omega$ is a Lipschitz domain, then essentially all regularity is lost. If one thinks on the blow-up and compactness method, it is clear that Lipschitz domains are quite different from $C^1$. Indeed, Lipschitz domains do \emph{not} get flatter by doing a blow-up (they remain Lipschitz, with the same Lipschitz norm). Thus, one cannot improve regularity by blowing up. Solutions turn out to be $C^\gamma$ for some small $\gamma > 0$ and, in general, not better. 


%
%
%

\chapter{Nonlinear variational PDE \& Hilbert's XIXth problem}
\label{ch.2}
\vspace{-2cm}

{\it \small Eine der begrifflich merkw\"urdigsten Thatsachen in den Elementen der Theorie der analytischen Functionen erblicke ich darin, da{\ss} es partielle Differentialgleichungen giebt, deren Integrale s\"amtlich notwendig analytische Funktionen der unabh\"angigen Variabeln sind, die also, kurz gesagt, nur analytischer L\"osungen f\"ahig sind.}
\begin{flushright}
--- David Hilbert (1900).
\end{flushright}

\vspace{0.4cm} 

Up until this point, we have studied linear elliptic PDEs. In this chapter we start the study of {\em nonlinear} elliptic PDEs. 

More precisely, we study {\em variational} nonlinear PDEs, that is, those that appear in the Calculus of Variations (minimizing an energy functional). In particular, our main goal is to introduce and solve Hilbert's XIXth problem\footnote{The original statement by Hilbert says that ``{\it there exist partial differential equations whose integrals are all of necessity analytic functions of the independent variables, that is, in short, equations susceptible of none but analytic solutions}'', and refers to solutions to what he calls ``regular variational problems'', involving convex (in $\nabla w$) and analytic operators of the form $L(\nabla w, w, x)$. We deal here with $L(\nabla w)$ for simplicity.}. 

\vspace{2mm}

\begin{center} 
 \fbox{
\begin{minipage}{0.85\textwidth}
\vspace{1.5mm}

\noindent \ $\bullet$\ \textbf{\underline{\smash{Hilbert's XIXth problem (1900)}}}\index{Hilbert XIXth problem}: Consider any local minimizer of energy functionals of the form 
\[
\mathcal{E}(w) := \int_\Omega L(\nabla w) \, dx,
\]
where $L:\R^n\to \R$ is smooth and uniformly convex, and $\Omega\subset\R^n$. 

Is is true that all local minimizers to this type of problems are smooth?

\vspace{2mm}

\end{minipage}
}

\end{center}

\vspace{2mm}

Notice that, given a boundary condition
\[
u = g \quad\textrm{on}\quad\partial \Omega,
\]
one can show that there is a unique minimizer to this problem, $u\in H^1(\Omega)$, with $u|_{\de\Omega} = g$. That is, there exists a unique $u\in H^1(\Omega)$ such that $u$ minimizes the functional $\mathcal{E}(w) := \int_\Omega L(\nabla w)\, dx$, among all functions $w\in H^1(\Omega)$ such that $w|_{\partial \Omega} = g$. We will be more precise about this in the first two sections of this chapter.

The question in Hilbert's XIXth problem is that of {\em regularity}: Is such minimizer $u$ smooth? 

\begin{rem}[On the convexity assumption]
\label{rem.conv}
The uniform convexity of the function is what gives us existence and uniqueness of a minimizer (see Theorem~\ref{thm.existuniq} below). Moreover, from the point of view of regularity, if $L$ is not convex and reaches its minimum at two different points, then even in dimension $n = 1$ there exist counterexamples to regularity. 

If $n = 1$ and $L$ has a minimum at two points $p_1< p_2$, then we can construct Lipschitz only minimizers zigzagging with slopes $p_1$ and $p_2$ (e.g., if $p_1 = -1$ and $p_2 = 1$, then $u(x) =|x|$ would be a minimizer). 

Thus, the convexity assumption is needed. 
\end{rem}

\section{Overview}
\label{sec.overview}
Hilbert's XIXth problem as posed above is a generalization of the minimization of the Dirichlet integral,
\[
\int_\Omega|\nabla w |^2\, dx.
\] 
Local minimizers of the Dirichlet integral verify the corresponding Euler--Lagrange equation, which in this case is the Laplace equation
\[
\Delta w = 0\quad\textrm{in}\quad\Omega.
\]
Solutions to this PDE, as seen in Chapter~\ref{ch.1}, are known to be $C^\infty$ in the interior of~$\Omega$.

Thus, the Dirichlet integral case $L(p) = |p|^2$ is extremely simple. Surprisingly, the general case is far more difficult, and its resolution took more than 50 years.

First, let us be more precise about the problem: by a {\em local minimizer} of $\mathcal{E}(w) =\int_\Omega L(\nabla w) \, dx$, we mean a function $u\in H^1(\Omega)$ such that 
\[
\mathcal{E}(u)\le \mathcal{E}(u+\phi)\quad\textrm{for all}\quad\phi\in C^\infty_c(\Omega).
\]
The {\em uniform convexity } of the functional is equivalent to 
\begin{equation}
\label{eq.unifconv}
0<\lambda {\rm Id}\le D^2L(p)\le\Lambda{\rm Id}\quad\textrm{for all}\quad p \in \R^n,
\end{equation}
(i.e., uniform convexity of $L$).  Notice the analogy with the uniform ellipticity from the previous chapter.

Now, what is the PDE satisfied by minimizers of $\mathcal{E}(u)$? (Namely, the Euler--Lagrange equation of the problem.) If $u\in H^1(\Omega)$ is a local minimizer, then 
\[
\mathcal{E}(u)\le \mathcal{E}(u+\eps\phi)\quad\textrm{for all}~~\phi\in C^\infty_c(\Omega), \quad\textrm{and all}~~\eps \in \R. 
\]
Hence, 
\[
\int_\Omega L(\nabla u)\, dx\le \int_\Omega L(\nabla u + \eps \nabla \phi)\, dx\quad\textrm{for all}~~\phi\in C^\infty_c(\Omega), \quad\textrm{and all}~~\eps \in \R,
\]
and thus, as a function of $\eps$, it has a minimum at $\eps = 0$. Taking derivatives in $\eps$ we reach 
\[
0 = \frac{d}{d\eps}\bigg|_{\eps = 0}\int_\Omega L(\nabla u + \eps\nabla \phi)\, dx = \int_\Omega DL(\nabla u)\nabla \phi\, dx.
\]

The {\em weak formulation} of the Euler--Lagrange equation is then
\begin{equation}
\label{eq.var_nonlinear2}
\int_\Omega DL(\nabla u)\nabla \phi\, dx = 0\quad\textrm{for all}~~\phi\in C^\infty_c(\Omega).
\end{equation}
That is, $u$ solves in the weak sense the PDE
\begin{equation}
\label{eq.var_nonlinear}
\boxed{
\divv \left(DL(\nabla u )\right) = 0\quad\textrm{in}\quad\Omega. 
}
\end{equation}
(This derivation will be properly justified in Theorem~\ref{thm.existuniq} below.)

If $u$ is $C^2$, \eqref{eq.var_nonlinear} is equivalent to 
\begin{equation}
\label{eq.uisC2}
\sum_{i,j = 1}^n(\partial_{ij}L)(\nabla u)\partial_{ij} u = 0\quad\textrm{in}\quad\Omega. 
\end{equation}

By uniform convexity of $L$, this is a (nonlinear) {\em uniformly elliptic} PDE. What can we say about the regularity of $u$? 

\subsection*{Regularity of local minimizers: First approach}

Let us assume that $u$ is smooth enough so that it solves \eqref{eq.uisC2}. We can regard \eqref{eq.uisC2} as a linear equation with variable coefficients, by denoting
\[
a_{ij}(x):= (\partial_{ij} L)(\nabla u(x) ),
\]
and we notice that, by uniform convexity of $L$, we have 
%
%
%
%
%
\[
0< \lambda\,{\rm Id}\le (a_{ij}(x))_{ij}\le \Lambda\,{\rm Id}.
\]
Moreover, if $\nabla u\in C^{0,\alpha}$, then $a_{ij}\in C^{0,\alpha}$. In particular, using Schauder estimates (see Theorem~\ref{thm.Schauder_estimates}), we have
\begin{equation}
\label{eq.sch_int}
u\in C^{1,\alpha}\Rightarrow a_{ij}\in C^{0,\alpha}\Rightarrow u \in C^{2, \alpha}.
\end{equation}
We can then bootstrap the regularity and get $C^\infty$: 
\[
 u \in C^{2, \alpha}\Rightarrow \nabla u \in C^{1,\alpha}\Rightarrow a_{ij}\in C^{1,\alpha}\Rightarrow u \in C^{3,\alpha}\Rightarrow\dots\Rightarrow u\in C^\infty.
\]

In fact, using the linear estimates for continuous coefficients, one can actually get $u\in C^1\Rightarrow a_{ij}\in C^0\Rightarrow u\in C^{1,\alpha}$. We remark that while the previous implications are true at a formal level, we did not properly argue the use of Schauder estimates. Indeed, our results for Schauder estimates in both non-divergence form (Theorem~\ref{thm.Schauder_estimates}) and divergence form (Theorem~\ref{thm.Schauder_estimates_div}) are {\it a priori}, i.e., they already assume regularity on $u$. We show how to use them in Theorem~\ref{thm.C1aimpCinf} below to prove the results we want and expect.

\subsection*{Equations with bounded measurable coefficients}
\index{Equation in divergence form with bounded measurable coefficients}
We have argued that using perturbative results for linear equations (Schauder estimates), one expects to prove that 
\[
u\in C^1\quad \Longrightarrow \quad u\in C^\infty.
\]

However, this approach does not allow us to prove any regularity if we do not know a priori that $u \in C^1$. The main open question in Hilbert's XIXth problem was then 
\[
\textrm{is it true that } u \in H^1\Rightarrow u \in C^1~\textrm{?}
\]

This problem was open for many years, and it was finally solved (independently and almost at the same time) by De Giorgi \cite{deGiorgi} and Nash \cite{Nash0, Nash}. 

\begin{thm}[De Giorgi--Nash]\index{De Giorgi--Nash}
\label{thm.DN_short}
Let $u$ be a local minimizer of 
\[
\mathcal{E}(w) = \int_\Omega L(\nabla w)\, dx,
\]
with $L$ uniformly convex and smooth. Then, $u\in C^{1,\alpha}$ for some $\alpha > 0$. 
\end{thm}

This theorem solved Hilbert's XIXth problem.

In order to show regularity of local minimizers $u$ of $\mathcal{E}(w) = \int_\Omega L(\nabla w)\, dx$, with $w\in H^1(\Omega)$, we first notice that they solve (in the weak sense) the nonlinear elliptic equation
\[
\divv\left(DL(\nabla u)\right) = 0\quad\textrm{in}\quad \Omega. 
\]

The first idea in the proof is to consider derivatives of $u$, $v = \partial_e u$, and to show that they solve an elliptic PDE as well. 

If we differentiate the equation $\divv\left(DL(\nabla u)\right) = 0$ with respect to $e\in \mathbb{S}^{n-1}$, we get 
\[
\divv\left(D^2L(\nabla u )\nabla\partial_e u\right) = 0\quad\textrm{in}\quad\Omega.
\]
Denoting (as before) $v :=\partial_e u$, $a_{ij}(x) := \partial_{ij}L (\nabla u (x) )$ and $A(x) := \left(a_{ij}(x)\right)_{ij}$, we can write this equation as 
\[
\divv\left(A(x)\nabla v\right)  = 0\quad\textrm{in}\quad\Omega.
\]

This is a linear, uniformly elliptic equation in divergence form, but we do not have \emph{any} regularity of $A(x)$ in the $x$-variable. We only know that the equation is uniformly elliptic. 

This is called a (uniformly elliptic) equation in divergence form  with {\em bounded measurable coefficients}. (Recall that the uniform convexity of $L$ yields $0<\lambda{\rm Id}\le A(x) \le \Lambda{\rm Id}$.)

De Giorgi and Nash established a new regularity result for such type of equations, see Theorem~\ref{thm.DGN_Om}.

The aim of this Chapter is to provide a complete and detailed proof of the solution to Hilbert's XIXth problem. We will follow De Giorgi's approach.

\section{Existence and basic estimates}

We start by showing the existence and uniqueness of minimizers of $\mathcal{E}$ among the class of $H^1(\Omega)$ functions with prescribed boundary data. That is, we  want a statement analogous to Theorem~\ref{ch0-existence}, but with the functional involving $L$ instead. We recall that we denote by $u|_{\de\Omega}$ the trace of $u$ on $\de\Omega$; see \ref{it.S5} in Chapter~\ref{ch.0}.

\begin{thm}[Existence and uniqueness of minimizers]\index{Existence and uniqueness!Minimizers convex functional}
\label{thm.existuniq}
Assume that $\Omega\subset \R^n$ is any bounded Lipschitz domain, and that
\begin{equation}
\label{eq.nonempt}
\left\{w\in H^1(\Omega) : w|_{\de\Omega} = g\right\} \neq \varnothing.
\end{equation}
Let $L:\R^n\to \R$ be smooth and uniformly convex, see \eqref{eq.unifconv}. Let
\[
\mathcal{E}(w) := \int_\Omega L(\nabla w) \, dx.
\]
Then, there exists a unique minimizer $u\in H^1(\Omega)$ with $u|_{\de\Omega} = g$. Moreover, $u$ solves \eqref{eq.var_nonlinear} in the weak sense. 
\end{thm}

In order to prove the existence and uniqueness theorem for minimizers, we need first to show the following result on the lower semi-continuity of the energy in this context. We provide two different proofs.

\begin{lem}[Lower semi-continuity of the functional]\label{lem.lscE} Let $\Omega\subset \R^n$ be a bounded domain. Let $L:\R^n\to \R$ be smooth and uniformly convex, see \eqref{eq.unifconv}; and let
\[
\mathcal{E}(w) := \int_\Omega L(\nabla w) \, dx.
\]

Then, $\mathcal{E}$ is weakly lower semi-continuous in $H^1(\Omega)$. That is, if $H^1(\Omega)\ni w_k \rightharpoonup w\in H^1(\Omega)$ weakly in $H^1(\Omega)$, then 
\[
\mathcal{E}(w) \le \liminf_{k\to \infty}\mathcal{E}(w_k). 
\]
\end{lem}
\begin{proof}[First proof]
Let us define the set
\[
\mathcal{A}(t) := \left\{v\in H^1(\Omega) : \mathcal{E}(v) \le t\right\}. 
\]
Notice that, by convexity of $\mathcal{E}$, $\mathcal{A}(t)$ is convex as well. Let us show that it is closed, i.e., if $\mathcal{A}(t) \ni w_k \to w$ strongly in $H^1(\Omega)$, then $w\in \mathcal{A}(t)$. This simply follows by noticing that, up to a subsequence, $\nabla w_k \to \nabla w$ almost everywhere, so that, by Fatou's lemma, 
\[
\mathcal{E}(w) = \int_\Omega L(\nabla w) \le \liminf_{k\to \infty} \int_\Omega L(\nabla w_k) \le t,
\]
that is, $w\in \mathcal{A}(t)$. Therefore, $\mathcal{A}(t)$ is closed (with respect to the $H^1(\Omega)$ convergence), and it is convex. By a standard result in functional analysis (closed and convex sets are weakly closed; see, for example, \cite[Theorem 3.7]{Brezis}), $\mathcal{A}(t)$ is also closed under weak convergence; namely, if $\mathcal{A}(t) \ni w_k\rightharpoonup w$ weakly in $H^1(\Omega)$ then $w\in \mathcal{A}(t)$. 

Let us now consider a sequence weakly converging in $H^1(\Omega)$, $w_k \rightharpoonup w$, and let us denote $t^* := \liminf_{k\to \infty}\mathcal{E}(w_k)$. For any $\eps > 0$, there exists some subsequence $k_{j,\eps}$ such that $w_{k_{j,\eps}} \rightharpoonup w$ weakly in $H^1(\Omega)$ and $\mathcal{E}(w_{k_{j,\eps}}) \le t^* + \eps$. That is, $w_{k_{j, \eps}}\in \mathcal{A}(t^*+\eps)$, and therefore, since $\mathcal{A}(t)$ is weakly closed (in $H^1(\Omega)$) for all $t$,  we have $w\in \mathcal{A}(t^*+\eps)$ and $\mathcal{E}(w) \le t^*+\eps$. Since this can be done for any $\eps > 0$, we reach that $\mathcal{E}(w) \le t^*$, and therefore, we have shown the weak lower semi-continuity of $\mathcal{E}$ in $H^1(\Omega)$. 
\end{proof}

\begin{proof}[Second proof]
Let us prove the lower semi-continuity of the functional by means of a different proof, from \cite{Mag11}. We will actually show that if $u_k, u \in W^{1,1}(\Omega)$ and $u_k \to u$ in $L^1_{\rm loc}(\Omega)$, then 
\[
\int_\Omega L(\nabla u) \le \liminf_{k\to \infty} \int_\Omega L(\nabla u_k). 
\]
In particular, since $\Omega$ is bounded, we can apply this result to the sequences in $H^1(\Omega)$ converging weakly in $H^1(\Omega)$ (by \ref{it.S2} from Chapter~\ref{ch.0}). Let $\eta\in C^\infty_c(B_1)$ be a smooth function with $\eta \ge 0$ and $\int_{B_1} \eta = 1$, and let $\eta_\eps(x) = \eps^{-n} \eta(x/\eps)$, so that we can consider the mollifications
\[
(u_k)_\eps (x) := (u_k * \eta_\eps)(x)  = \int_{B_\eps} u(x-y) \eta_\eps(y) \, dy,\qquad u_\eps(x) := (u*\eta_\eps)(x). 
\]
Let $\Omega'\subset \Omega$ be such that for all $x\in \Omega'$, $B_\eps(x)\subset \Omega$. In particular, since $u_k\to u$ in $L^1_{\rm loc}(\Omega)$, we have $\nabla (u_k)_\eps(x) \to \nabla u_\eps(x)$ for every $x\in \Omega'$. From the smoothness of $L$ we also have that $L(\nabla (u_k)_\eps(x) ) \to L(\nabla u_\eps(x) ) $ and by Fatou's lemma (recall that we may assume $L\ge 0$)
\begin{equation}
\label{eq.comb_with}
\int_{\Omega'} L(\nabla u_\eps)\le \liminf_{k\to \infty} \int_{\Omega'} L(\nabla (u_k)_\eps).
\end{equation}
Noticing now that $\nabla (u_k)_\eps = (\nabla u_k)_\eps$ and using Jensen's inequality (since $L$ is convex and $\int \eta_\eps = 1$) we have
\[
L(\nabla (u_k)_\eps) = L \left(\int_{B_\eps(x)} \eta_\eps(x-y) \nabla u_k (y)\, dy\right) \le \int_{B_\eps(x)} \eta_\eps(x-y) L(\nabla u_k(y))\, dy
\]
which leads to 
\begin{align*}
\int_{\Omega'} L(\nabla (u_k)_\eps) & \le \int_{\Omega'}\left\{\int_{B_\eps(x)}\eta_\eps(x-y) L(\nabla u_k(y))\, dy\right\}\, dx\\
& \le \int_{I_\eps(\Omega')}L(\nabla u_k(y)) \int_{B_\eps(y)\cap \Omega'} \eta_\eps(x-y)\, dx\, dy\le \int_\Omega L(\nabla u_k),
\end{align*}
where $I_\eps(\Omega')\subset \Omega$ denotes an $\eps$-neighborhood of $\Omega'$. Combined with \eqref{eq.comb_with}, this yields
\[
\int_{\Omega'}L(\nabla u_\eps) \le \liminf_{k\to \infty} \int_{\Omega} L (\nabla u_k).
\]

Now, since $u\in W^{1,1}(\Omega)$, we have $\nabla u_\eps \to \nabla u$ as $\eps\downarrow 0$ almost everywhere in $\Omega'$ and so, again by Fatou's Lemma, we can let $\eps\downarrow 0$ to deduce 
\[
\int_{\Omega'}L(\nabla u) \le \liminf_{k\to \infty}\int_{\Omega} L(\nabla u_k).
\]
By taking an increasing sequence of sets $\Omega'$ whose union is $\Omega$ we reach the desired result. 
\end{proof}

We can now prove Theorem~\ref{thm.existuniq}.

\begin{proof}[Proof of Theorem~\ref{thm.existuniq}]
We divide the proof into three different parts. 
\\[0.1cm]
{\it \underline{\smash{Step 1}}.} If $u$ is a local minimizer, then it solves \eqref{eq.var_nonlinear} in the weak sense. This follows from the fact that
\[
\int_\Omega  L(\nabla u)\, dx\le \int_\Omega L(\nabla u + \eps \nabla \phi)\, dx\quad\textrm{for all}~~\eps,\quad\textrm{for all}~~\phi\in C^\infty_c(\Omega). 
\] 
Indeed, notice that the integrals are bounded ($L$ being uniformly convex, i.e., at most quadratic at infinity, and $\nabla u \in L^2$). Since $L$ is smooth, we can take a Taylor expansion
\[
L(\nabla u+\eps \nabla \phi) \le L(\nabla u) + \eps DL(\nabla u) \nabla \phi + \frac{\eps^2}{2}|\nabla \phi|^2 \sup_{p \in \R^n}  \left|D^2 L(p)\right|.
\]
Recalling from \eqref{eq.unifconv} that $\left|D^2 L\right|$ is bounded by $\Lambda$, and plugging it back into the integral we obtain 
\[
-  \Lambda \frac{\eps}{2}|\nabla \phi|^2 \le \int_\Omega DL(\nabla u) \nabla \phi \, dx\quad \textrm{for all}~~\eps>0,\quad\textrm{for all}~~\phi\in C^\infty_c(\Omega).
\] 
Letting $\eps$ go to zero, we reach that 
\[
 \int_\Omega DL(\nabla u) \nabla \phi \, dx\ge 0\quad\textrm{for all}~~\phi\in C^\infty_c(\Omega).
\] 
On the other hand, taking $-\phi$ instead of $\phi$, we reach the equality \eqref{eq.var_nonlinear2}, as we wanted to see. 
\\[0.1cm]
{\it \underline{\smash{Step 2}}.} Let us now show the existence of a solution. 

Since $L$ is uniformly convex (see \eqref{eq.unifconv}) it has a unique minimum. That is, there exists $p_L\in \R^n$ such that $L(p) \ge L(p_L)$ for all $p\in\R^n$. In particular, since $L$ is smooth, $\nabla L(p_L) = 0$ and thus, from the uniform convexity \eqref{eq.unifconv} we have that
\[
0 <\lambda |p|^2 \le L(p - p_L) - L(p_L) \le \Lambda|p|^2,\quad\textrm{for all}\quad p\in\R^n.
\]

Without loss of generality, by taking $\tilde L(p) = L(p-p_L) - L(p_L)$ if necessary, we may assume that $L(0) = 0$ and $\nabla L(0) = 0$, so that we have 
\begin{equation}
\label{eq.unifconv2}
0 <\lambda |p|^2 \le L(p) \le \Lambda|p|^2,\quad\textrm{for all}\quad p\in\R^n.
\end{equation}
(Notice that we may assume that because if $u$ is a minimizer for $L$, then $u+\langle p_L, x\rangle$ is a minimizer for $\tilde L$, since the domain is bounded and therefore the integral of $L(p_L)$ is finite.)

Let 
\[\mathcal{E}_\circ=\inf\left\{\int_{ \Omega}  L(\nabla w)\,dx\,:\, w\in H^1( \Omega),\ w|_{\de \Omega}=g\right\},\]
that is, the infimum value of $\mathcal E(w)$ among all admissible functions $w$. Notice that, by assumption \eqref{eq.nonempt}, such infimum exists. Indeed, if $w\in H^1(\Omega)$, by \eqref{eq.unifconv2} we have that
\[
\mathcal{E}(w) = \int_{ \Omega}  L(\nabla w) \le \Lambda \int_{ \Omega} |\nabla w|^2 = \Lambda  \|\nabla w \|^2_{L^2( \Omega)}< \infty
\]
that is, the energy functional is bounded for functions in $H^1(\Omega)$. 

Let us take a minimizing sequence of functions. That is, we take $\{u_k\}$ such that $u_k\in H^1(\Omega)$, $u_k|_{\de\Omega}=g$, and  $\mathcal E(u_k)\to \mathcal{E}_\circ$ as $k\to\infty$. We begin by showing that $\mathcal{E}(u_k)$ are bounded, and that $u_k$ is a sequence bounded in $H^1(\Omega)$.  By \eqref{eq.unifconv2}, 
\[
 \lambda\|\nabla u_k\|^2_{L^2(\Omega)} \le \lambda\int_\Omega |\nabla u_k |^2  \le \int_{\Omega} L(\nabla u_k) \le \mathcal{E}(u_k) < \infty,
\]
That is, since $\mathcal{E}(u_k)$ is uniformly bounded (being a convergent sequence with non-infinite elements), we reach that $ \|\nabla u_k\|^2_{L^2(\Omega)}$ is uniformly bounded. 
Thus, by the Poincar\'e inequality (see Theorem~\ref{ch0-Poinc}) the sequence $u_k$ is uniformly bounded in $H^1(\Omega)$. 

In particular, there exists a subsequence $u_{k_j}$ converging strongly in $L^2(\Omega)$ and weakly in $H^1(\Omega)$ to some $u\in H^1(\Omega)$, $u_k \rightharpoonup u$ weakly in $H^1(\Omega)$. By the weak lower semi-continuity (Lemma~\ref{lem.lscE}) we reach that
\[
\mathcal{E}(u) \le \liminf_{k\to \infty}\mathcal{E}(u_k) = \mathcal{E}_\circ, 
\]
so that $\mathcal{E}(u) = \mathcal{E}_\circ$ (by minimality) and therefore $u$ is a minimizer. 
\\[0.1cm]
{\it \underline{\smash{Step 3}}.} We finish the proof by showing the uniqueness of such minimizer.

This follows from the uniform convexity. Indeed, since $L$ is uniformly convex, if $p\neq q$, then 
\[
\frac{L(p) + L(q)}{2} >  L\left(\frac{p+q}{2}\right).
\]
Let $u, v \in H^1(\Omega)$ be two distinct minimizers with the same boundary data ($\mathcal{E}(u) = \mathcal{E}(v) = \mathcal{E}_\circ$). In particular, $\nabla u \not\equiv \nabla v$ in $\Omega$, so that  $U := \{x\in \Omega : \nabla u \neq \nabla v\} \subset\Omega$ has positive measure. Thus,  
\[
\frac{L(\nabla u) + L(\nabla v)}{2} >  L\left(\frac{\nabla u +\nabla v}{2}\right)\quad\textrm{in}\quad U,
\]
so that, since $|U| > 0$, 
\[
\frac12 \int_U \big\{L(\nabla u) + L(\nabla v)\big\} > \int_U L\left(\frac{\nabla u +\nabla v}{2}\right).
\]
Since the integrals are equal in $\Omega\setminus U$, we reach 
\[
\mathcal{E}_\circ = \frac12 \int_\Omega \big\{L(\nabla u) + L(\nabla v)\big\} > \int_\Omega L\left(\frac{\nabla u +\nabla v}{2}\right) \ge \mathcal{E}_\circ,
\]
where the last inequality comes from the minimality of $\mathcal{E}_\circ$. We have reached a contradiction, and thus, the minimizer is unique.
\end{proof}

We next give a complete and rigorous proof of the formal argumentation from the previous section, where we explained that $C^1$ solutions are $C^\infty$. 

\begin{thm}
\label{thm.C1aimpCinf}
Let $u\in H^1(\Omega)$ be a local minimizer of 
\[
\mathcal{E}(w) =\int_\Omega L(\nabla w)\, dx,
\]
with $L$ uniformly convex and smooth. Assume that $u\in C^{1}$. Then $u\in C^\infty$. 
\end{thm}
\begin{proof}
We know that if $u\in C^1$ and $u$ is a minimizer of $\mathcal{E}(w)$, then 
\[
\int_\Omega DL(\nabla u(x))\nabla \phi(x)\, dx = 0\quad\textrm{for all }\phi\in C^\infty_c(\Omega).
\]
Let $h\in \R^n$, and assume that $|h|$ is small. We have, in particular, that
\[
\int_\Omega \bigg(DL(\nabla u(x+h))-DL(\nabla u(x))\bigg)\nabla \phi(x)\, dx = 0\quad\textrm{for all }\phi\in C^\infty_c(\Omega_h),
\]
where $\Omega_h := \{x\in \Omega: {\rm dist}(x, \de\Omega) > |h|\}$. Notice that, by the fundamental theorem of calculus for line integrals, we can write 
\begin{align*}
&DL(\nabla u(x+h))-DL(\nabla u(x)) = \\
&\qquad= \int_0^1 D^2 L\bigg(t\nabla u(x+h)+(1-t)\nabla u(x)\bigg) \bigg(\nabla u(x+h)-\nabla u(x)\bigg) \, dt.
\end{align*}
If we define 
\[
\tilde A(x) := \int_0^1 D^2 L\bigg(t\nabla u(x+h)+(1-t)\nabla u(x)\bigg)\, dt,
\]
then $\tilde A(x)$ is uniformly elliptic (since $L$ is uniformly convex), and continuous (since $L$ is smooth and $\nabla u$ is continuous). Then, by the previous argumentation,
\[
\int_\Omega \nabla \left(\frac{u(x+h)-u(x)}{|h|}\right)\cdot \tilde A(x) \nabla \phi(x)\, dx = 0\quad\textrm{for all }\phi\in C^\infty_c(\Omega_h),
\]
that is, $\frac{u(\cdot+h)-u}{|h|}$ solves weakly 
\[
\divv\left(\tilde A(x)\nabla \left[\frac{u(x+h)-u(x)}{|h|}\right]\right) = 0\quad\textrm{for }x\in \Omega_h.
\]
Moreover, notice that $\frac{u(\cdot+h)-u}{|h|}$ is $C^1$ for all $h\neq 0$, since $u$ is $C^1$. Thus, by the Schauder-type estimates for operators in divergence form and continuous coefficients(Proposition~\ref{prop.Schauder_estimates_div}),
\[
\left\|\frac{u(\cdot+h)-u}{|h|}\right\|_{C^{\beta}(B_{\rho/2}(x_\circ))} \le C(\rho) \left\|\frac{u(\cdot+h)-u}{|h|}\right\|_{L^\infty(B_\rho(x_\circ))} \le C,
\]
for all $B_\rho(x_\circ) \subset \Omega_h$ and $\beta\in (0,1)$. In the last inequality we used that $\nabla u$ is continuous (and thus, bounded). Notice that the constant $C(\rho)$ is independent of $h$ (but might depend on $\beta$). In particular, from \ref{it.H7} in Chapter~\ref{ch.0}, namely \eqref{eq.H7} with $\alpha = 1$, we obtain that $u \in C^{1, \beta}(\overline{\Omega_h})$ for all $h\in \R^n$. Letting $|h|\downarrow 0$ we get that $u \in C^{1, \beta}$ inside $\Omega$.

We want to repeat the previous reasoning, noticing now that $\tilde A(x)$ is $C^{0, \beta}$ (since $\nabla u\in C^{0, \beta}$ and $L$ is smooth). That is, $\frac{u(\cdot+h)-u}{|h|}\in C^{1,\beta}(\Omega)$ and fulfills
\[
\divv\left(\tilde A(x)\nabla\left[\frac{u(x+h)-u(x)}{|h|}\right]\right) = 0\quad\textrm{for }x\in \Omega_h, 
\]
in the weak sense, with $\tilde A\in C^{\beta}$ and uniformly elliptic. By Theorem~\ref{thm.Schauder_estimates_div}, 
\[
\left\|\frac{u(\cdot+h)-u}{|h|}\right\|_{C^{1, \beta}(B_{\rho/2})} \le C(\rho) \left\|\frac{u(\cdot+h)-u}{|h|}\right\|_{L^\infty(B_\rho)} \le C,
\]
for all $B_\rho\subset\Omega_h$, and again, thanks to \ref{it.H7}, \eqref{eq.H7}, we obtain that $u \in C^{2, \beta}(\Omega)$. We can now proceed iteratively using the higher order interior Schauder estimates in divergence form (Corollary~\ref{cor.Schauder_estimates_div_HO}) to obtain that $u\in C^k(\Omega)$ for all $k\in \N$, i.e, $u\in C^\infty$ inside $\Omega$. 
\end{proof}
\begin{rem}
Notice that in the formal proof \eqref{eq.sch_int} we were using Schauder estimates in non-divergence form, since we were already assuming that the solution $u$ was $C^2$. Here, in the proof of Theorem~\ref{thm.C1aimpCinf}, we need to \emph{differentiate} the equation (in incremental quotients) and then we obtain an equation in divergence form whose coefficients have the right regularity. Thus, in the actual proof we are using Schauder estimates for equations in divergence form instead.  
\end{rem}

\section{De Giorgi's proof}

The result of De Giorgi and Nash regarding the regularity of solutions to equations with bounded measurable coefficients is the following (see the discussion in Section~\ref{sec.overview}). 

\begin{thm}[De Giorgi--Nash]\index{De Giorgi--Nash}\index{Equation in divergence form with bounded measurable coefficients}
\label{thm.DGN_Om}
Let $v\in H^1(\Omega)$ be any weak solution to 
\begin{equation}
\label{eq.divformA}
\divv\left(A(x)\nabla v\right)  = 0\quad\textrm{in}\quad\Omega,
\end{equation}
with $0<\lambda{\rm Id}\le A(x) \le \Lambda{\rm Id}$. Then, there exists some $\alpha > 0$ such that $v\in C^{0,\alpha}(\tilde\Omega)$ for any $\tilde \Omega\subset\subset\Omega$, with 
\[
\|v\|_{C^{0,\alpha}(\tilde\Omega)}\le C\|v\|_{L^2(\Omega)}.
\]
The constant $C$ depends only on $n$, $\lambda$, $\Lambda$, $\Omega$, and $\tilde\Omega$. The constant $\alpha > 0$ depends only on $n$, $\lambda$, and $\Lambda$. 
\end{thm}

This theorem yields Theorem~\ref{thm.DN_short}, and combined with previous discussions, solved Hilbert's XIXth problem. Indeed, if $u\in H^1(\Omega)$ is any local minimizer of $\mathcal{E}(w) =\int_\Omega L(\nabla w)\, dx$, then any derivative of $u$, $v = \partial_e u$, solves \eqref{eq.divformA}. 

Thanks to Theorem~\ref{thm.DGN_Om} we will have 
\[
u\in H^1(\Omega)\Rightarrow v\in L^2(\Omega)\xRightarrow[\begin{subarray}{c}\rm De Giorgi\\ \rm -Nash\end{subarray}]{} v\in C^{0,\alpha}(\tilde \Omega)\Rightarrow u \in C^{1,\alpha}\xRightarrow[\rm Schauder]{} u \in C^\infty. 
\]
This will be proved in detail in Section~\ref{sec.solution}.

Theorem~\ref{thm.DGN_Om} is significantly different in spirit than all the results on elliptic regularity which existed before. 

Most of the previous results can be seen as perturbation of the Laplace equation (they are perturbative results). In Schauder-type estimates, we always use that, when zooming in a solution at a point, the operator gets closer and closer to the Laplacian. 

In De Giorgi's theorem, this is not true anymore. The uniform ellipticity is preserved by scaling, but the equation is not better, nor closer to the Laplace equation. 

\subsection*{General ideas of the proof}
We will follow the approach of De Giorgi. 

From now on, we denote $\mathcal{L}$ any operator of the form 
\begin{equation}
\label{eq.mL}
\mathcal{L} v := -\divv(A(x)\nabla v),\quad
\begin{array}{l}
\textrm{where $A(x)$ is uniformly elliptic}\\
\textrm{with ellipticity constants $0 < \lambda\le\Lambda$}.
\end{array}
\end{equation}

By a standard covering argument (cf. Remark~\ref{rem.covering_argument}), we only need to prove the estimate for $\Omega = B_1$ and $\tilde \Omega = B_{1/2}$.

Throughout the proof, we will use that, if $v$ solves $\mathcal{L}v = 0$, then $\tilde v (x) := C v(x_\circ+rx)$ solves an equation of the same kind, $\tilde{\mathcal{L}} \tilde v = 0$, for some operator $\tilde{\mathcal{L}}$ with the same ellipticity constants as $\mathcal{L}$ --- given by $\tilde{\mathcal{L}} \tilde v = \divv\big(A(x_\circ + rx) \nabla \tilde v\big)$.

De Giorgi's proof is  split into two steps: 
\\[1mm]
\underline{\smash{First step:}} Show that $\|v\|_{L^\infty} \le C\|v\|_{L^2}$
\\[1mm]
\underline{\smash{Second step:}} Show that $\|v\|_{C^{0,\alpha}}\le C\|v\|_{L^\infty}$.
\vspace{1mm}

In the first step, we work on the family of balls (see Figure~\ref{fig.7})
\[
\tilde B_k := \left\{x : |x|\le \frac12+2^{-k-1} \right\}.
\]
\begin{figure}
\includegraphics{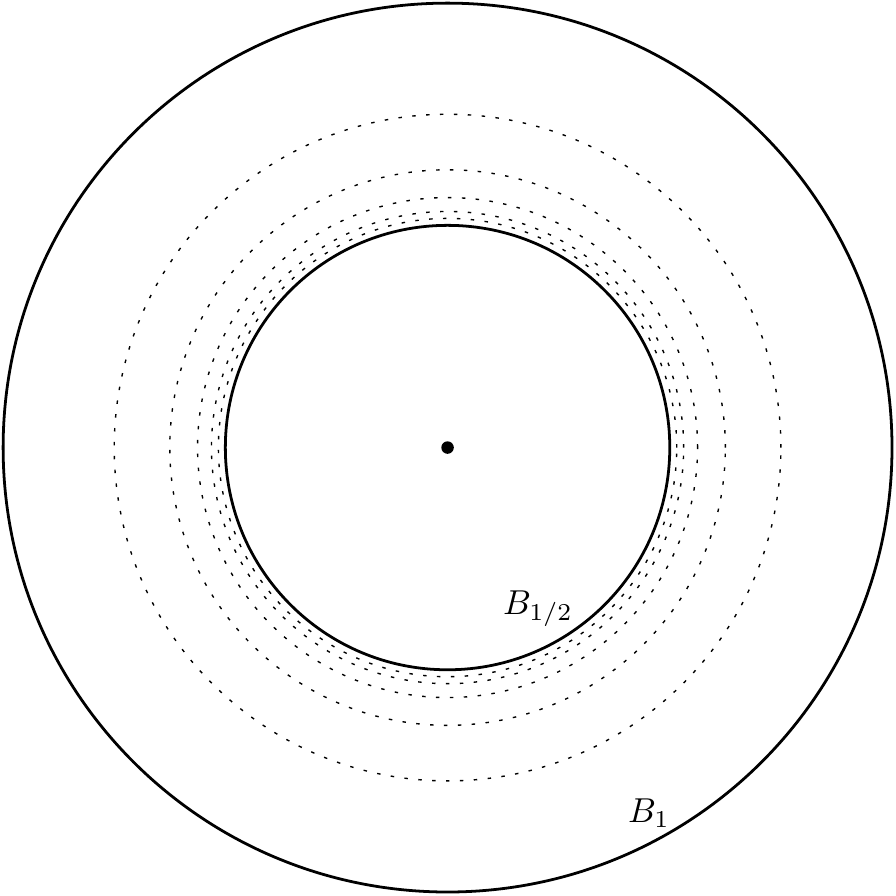}
\caption{Representation of the family of balls $\tilde B_k$.}
\label{fig.7}
\end{figure}
Note that $\tilde B_0 = B_1$, and $\tilde B_k$ converges to $B_{1/2}$ as $k\to \infty$. 

We assume $\|v\|_{L^2(B_1)}\leq \delta\ll 1$ and then consider the truncated functions
\[
v_k:= (v-C_k)_+\quad\textrm{with}\quad C_k := 1-2^{-k},
\]
and the numbers 
\[
V_k \approx \int_{\tilde B_k}|v_k|^2\, dx.
\]

Then, the main point is to derive an estimate of the form 
\begin{equation}
\label{eq.boundUk}
V_k\le C^k  V_{k-1}^\beta\quad\textrm{for some}\quad\beta > 1,
\end{equation}
 for some constant $C$ depending only on $n$, $\lambda$, and $\Lambda$. This previous inequality implies that $V_k\to 0$ as $k\to \infty$ if $V_0$ is small enough. 
 In particular, $v_\infty = (v-1)_+$ is equal to zero in $B_{1/2}$, and so $v \le 1$ in $B_{1/2}$.

Notice that our equation $\mathcal{L}v = 0$ in $B_1$ is linear, while the bound \eqref{eq.boundUk} is nonlinear.  The ``game'' consists in using the Sobolev inequality (which gives control of $L^p$ norms of $v_k$ in terms of $L^2$ norms of $\nabla v_k$), combined with an energy inequality, which gives a ``reversed'' Poincar\'e inequality, i.e., a control of $\|\nabla v_k\|_{L^2}$ in terms of $\|v_k\|_{L^2}$. 

Once we have the first step $v\in L^2\Rightarrow v \in L^\infty$, the second step consists of showing an oscillation-decay lemma
\[
\mathcal{L}v = 0\quad\textrm{in}\quad B_1\quad\Longrightarrow \quad \osc_{B_{1/2}} v \le (1-\theta)\osc_{B_1} v.
\]
This implies the $C^{0,\alpha}$ regularity of $v$ (as we saw in Corollary~\ref{cor.Holder_regularity_1}). 

In the next proofs we follow \cite{CV10, Vas16}.

\subsection*{De Giorgi's first step: from $L^2$ to $L^\infty$}

The two main ingredients are the Sobolev inequality
\[
\|v\|_{L^p(\R^n)}\le C\|\nabla v\|_{L^2(\R^n)},\qquad p = \frac{2n}{n-2},
\]
(see Theorem~\ref{ch0-Sob}) and the following energy inequality (the Caccioppoli inequality):
\begin{lem}[Energy inequality] \index{Energy inequality}
\label{lem.energyinequality}
Let $v\in H^1(B_1)$ with $v \ge 0$ such that $\mathcal{L} v \le 0$ in $B_1$, for some $\mathcal{L}$ of the form \eqref{eq.mL}. Then, for any $\varphi\in C^\infty_c(B_1)$ we have
\[
\int_{B_1}|\nabla (\varphi v)|^2\, dx \le C\|\nabla \varphi\|^2_{L^\infty(B_1)}\int_{B_1\cap {\rm supp }\, \varphi} v^2\, dx,
\]
where $C$ depends only on $n$, $\lambda$, and $\Lambda$. 
\end{lem}

\begin{proof}
Notice that the weak formulation of $-\divv(A(x)\nabla v)\le 0$ in $B_1$ is 
\[
\int_{B_1} \nabla \eta\cdot A (x) \nabla v\, dx\le 0\quad\textrm{for all}\quad\eta\in H^1_0(B_1), \eta\ge 0. 
\]
Take $\eta = \varphi^2v$, to get 
\[
\int_{B_1} \nabla(\varphi^2v)\cdot A (x)\nabla v \, dx \le 0.
\]
Now, we want to ``bring one of the $\varphi$ from the first gradient to the second gradient''. Indeed, using 
\begin{align*}
\nabla (\varphi^2 v)& = \varphi\nabla (\varphi v) +(\varphi v)\nabla \varphi,\\
\nabla (\varphi v)& = \varphi\nabla v +v\nabla \varphi,
\end{align*}
we get 
\begin{align*}
0& \ge \int_{B_1} \nabla(\varphi^2 v)\cdot A(x) \nabla v\, dx\\
& = \int_{B_1} \varphi  \nabla(\varphi v)\cdot A (x)\nabla v\, dx+\int_{B_1}\varphi v \,\nabla\varphi \cdot A (x)\nabla v\, dx\\
& = \int_{B_1} \nabla(\varphi v)\cdot A(x) \nabla(\varphi v)\, dx-\int_{B_1}v \nabla(\varphi v)\cdot A (x) \nabla \varphi\, dx\\
& \hspace{6cm}+\int_{B_1}\varphi v \, \nabla \varphi\cdot A (x)\nabla v\, dx\\
& = \int_{B_1} \nabla(\varphi v)\cdot A (x)\nabla(\varphi v)\, dx - \int_{B_1}v \nabla(\varphi v)\cdot( A(x)-A^T(x)) \nabla\varphi \, dx
\\
& \hspace{6cm}-\int_{B_1}v^2\nabla \varphi \cdot A (x) \nabla \varphi\, dx.
\end{align*}
Let us first bound the term involving $(A - A^T)$. By H\"older's inequality, using the uniform ellipticity of $A$ and that $(A - A^T)^2 \le 4\Lambda^2{\rm Id}$, we get
\begin{align*}
\int_{B_1} v \nabla(& \varphi v)\cdot( A(x)-A^T(x)) \nabla\varphi \, dx \\
& \le \left(\int_{B_1} |v\,( A(x)-A^T(x)) \nabla\varphi|^2 \, dx \right)^{\frac12}\left(\int_{B_1} |\nabla(\varphi v)|^2 \, dx \right)^{\frac12}\\
& \le 2\frac{\Lambda}{\lambda^{\frac12}}\left(\int_{B_1} |v\nabla\varphi|^2 \, dx \right)^{\frac12}\left(\int_{B_1} \nabla(\varphi v)A(x)\nabla(\varphi v) \, dx \right)^{\frac12}\\
& \le \frac12 \int_{B_1} \nabla(\varphi v)A(x)\nabla(\varphi v) \, dx + 2\frac{\Lambda^2}{\lambda}\int_{B_1} |v\nabla\varphi|^2\, dx,
\end{align*}
where in the last inequality we are using that $2ab \le a^2 + b^2$. Combining the previous inequalities, we obtain that 
\[
 2\frac{\Lambda^2}{\lambda}\int_{B_1} |v\nabla\varphi|^2\, dx \ge  \frac12 \int_{B_1} \nabla(\varphi v)\cdot A(x) \nabla(\varphi v)\, dx  -\int_{B_1}v^2\nabla \varphi \cdot A(x)  \nabla \varphi\, dx.
\]

Therefore, we deduce 
\begin{align*}
\lambda\int_{B_1}|\nabla (\varphi v)|^2\, dx& \le \int_{B_1}\nabla (\varphi v)\cdot A(x) \nabla (\varphi v)\, dx \\
& \le 2 \int_{B_1} v^2\,\nabla \varphi\cdot A(x) \nabla \varphi\, dx+4\frac{\Lambda^2}{\lambda}\int_{B_1} |v\nabla\varphi|^2\, dx\\
& \le  \left(2\Lambda + 4 \frac{\Lambda^2}{\lambda}\right) \|\nabla \varphi\|_{L^\infty(B_1)}^2\int_{B_1\cap {\rm supp}\, \varphi} v^2\, dx,
\end{align*}
and the lemma is proved. 
\end{proof}

We will use the energy inequality (from the previous lemma) applied to the function 
\[
v_+:= \max\{v, 0\}. 
\]
Before doing so, let us show that if $\mathcal{L}v \le 0$ (i.e., $v$ is a subsolution), then $\mathcal{L}v_+ \le 0$ (i.e., $v_+$ is a subsolution as well). (More generally, the maximum of two subsolutions is always a subsolution.)


\begin{lem} 
\label{lem.pospartSH}
Let $\mathcal{L}$ be of the form \eqref{eq.mL}, let $v\in H^1(B_1)$ be such that $\mathcal{L} v \le 0$ in $B_1$. Then, $\mathcal{L} v_+ \le 0$. 
\end{lem} 
\begin{proof}

We proceed by approximation. Let $F\in C^\infty(\R)$ be a smooth, non-decreasing, convex function, with globally bounded first derivatives. We start by showing that $\mathcal{L}(F(v))\le 0$ in $B_1$.

Notice that if $v\in W^{1, 2}(B_1)$, then $F(v) \in W^{1, 2}(B_1)$ as well. 

We know that $\mathcal{L}(v) \le 0$, i.e., 
\[
\int_{B_1} \nabla \eta\cdot A \nabla v\, dx\le 0\quad\textrm{for all}\quad\eta\in H^1_0(B_1), \eta\ge 0. 
\]
Let us now compute, for any fixed $\eta\in H^1_0(\Omega)$ satisfying $\eta\ge 0$, $\mathcal{L}(F(v))$. Notice that the weak formulation still makes sense. 
\begin{align*}
\int_{B_1} \nabla \eta\cdot A \nabla F(v)\, dx & = \int_{B_1} F'(v) \nabla \eta\cdot A \nabla v\, dx\\
& = \int_{B_1} \nabla (F'(v)\eta)\cdot A \nabla v\, dx-\int_{B_1} \eta F''(v) \nabla v\cdot A \nabla v\, dx.
\end{align*}
The first term is non-positive, since $F'(v) \eta\in H^1_0(B_1)$ and $F'(v) \ge 0$ ($F$ is non-decreasing), so that $F'(v) \eta$ is an admissible test function. The second term is also non-positive, since $\eta F''(v) \ge 0$ and $\nabla v \cdot A \nabla v \ge 0$ by ellipticity (and the integral is well defined, since $\eta F''(v)$ can be assumed to be bounded by approximation, and $\int_{B_1} \nabla v \cdot A \nabla v \le \Lambda \|\nabla v\|_{L^2(B_1)}^2$). Therefore, 
\[
\int_{B_1} \nabla \eta\cdot A \nabla F(v)\, dx \le 0, 
\] 
and the proof is complete. We finish by taking smooth approximations of the positive part function, $F_\eps$, converging uniformly in compact sets to $F(x) = \max\{x, 0\}$. Notice that this can be done in such a way that $\|F_\eps(v)\|_{W^{1, 2}(B_1)}\le C$, for some $C$ independent of $\eps > 0$, which gives the desired result. 
\end{proof}

We want to prove the following.
\begin{prop}[from $L^2$ to $L^\infty$] 
\label{prop.L2Linfty} 
Let $\mathcal{L}$ be of the form \eqref{eq.mL}, and let $v\in H^1(B_1)$ be a solution to 
\[
\mathcal{L}v \le 0\quad\textrm{in}\quad B_1.
\]
Then 
\[
\|v_+\|_{L^\infty(B_{1/2})}\le C\|v_+\|_{L^2(B_1)},
\]
for some constant $C$ depending only on $n$, $\lambda$, and $\Lambda$. 
\end{prop}

We will prove, in fact, the following (which is actually equivalent):

\begin{prop}[from $L^2$ to $L^\infty$]
\label{prop.L2Linfty_2} 
Let $\mathcal{L}$ be of the form \eqref{eq.mL}. There exists a constant $\delta > 0$ depending only on $n$, $\lambda$, and $\Lambda$, such that if $v\in H^1(B_1)$ solves 
\[
\mathcal{L}v \le 0\quad\textrm{in}\quad B_1\quad\text{and}\quad 
\int_{B_1}v_+^2 \le \delta,
\]
then 
\[
v\le 1\quad\textrm{in}\quad B_{1/2}. 
\]
\end{prop}

\begin{proof}
Define, as introduced in the general ideas of the proof, for $k\ge 0$,
\[
\tilde B_k := \left\{|x|\le \frac12+2^{-k-1}\right\},
\]
\[
v_k := (v-C_k)_+\quad\textrm{with}\quad C_k = 1-2^{-k},
\]
and let $\varphi_k$ be a family of shrinking cut-off functions $0\le \varphi_k\le 1$ that fulfill
\[
\varphi_k\in C^\infty_c(B_1),\quad\varphi_k = \left\{
\begin{array}{ll}
1 & \textrm{ in } \tilde B_k\\
0 & \textrm{ in } \tilde B_{k-1}^c
\end{array}
\right.,
\quad\textrm{and}\quad |\nabla \varphi_k|\le C2^k~\textrm{ in }~\tilde B_{k-1}\setminus \tilde B_k,
\]
where $C$ here depends only on $n$.

Let
\[
V_k := \int_{B_1}\varphi^2_kv_k^2\, dx.
\]
Now, the Sobolev inequality, and the energy inequality (Lemma~\ref{lem.energyinequality}) give
\begin{align*}
\left(\int_{B_1} |\varphi_{k+1}v_{k+1}|^p\, dx\right)^{\frac2p}& \le C\left(\int_{B_1} |\nabla(\varphi_{k+1}v_{k+1})|^2\, dx\right)\\
& \le C 2^{2k}\int_{\tilde B_k}|v_{k+1}|^2\, dx\\
& \le C2^{2k}\int_{B_1} (\varphi_kv_k)^2\, dx = C2^{2k} V_k, 
\end{align*}
for $p = \frac{2n}{n-2}$ if $n \ge 3$. If $n = 1$ or $n = 2$, we can take $p = 4$.

On the other hand, by H\"older's inequality,
\[
V_{k+1} = \int_{B_1}\varphi_{k+1}^2v_{k+1}^2\, dx\le \left(\int_{B_1}(\varphi_{k+1}v_{k+1})^p\, dx\right)^{\frac{2}{p}} 	\big|\{\varphi_{k+1}v_{k+1} > 0\}\big|^\gamma, 
\]
where $\gamma := \frac{2}{n}$ (if $n = 1$ or $n = 2$, $\gamma = \frac12$). Here, we are using that $\int_A  f\le (\int_A |f|^{p/2})^{2/p}|A|^\gamma$.

Now, from Chebyshev's inequality and the definition of $v_k$ and $\varphi_k$,
\begin{align*}
\big|\{\varphi_{k+1}v_{k+1} > 0\}\big| &\le \big|\{\varphi_k v_k > 2^{-k-1}\}\big|\\
& = \big|\{\varphi_k^2v_k^2 > 2^{-2k-2}\}\big|\\
& \le 2^{2(k+1)}\int_{B_1} \varphi_k^2 v_k^2 \, dx = 2^{2(k+1)}V_k.
\end{align*}

Apart from Chebyshev's inequality, we are using here that if $v_{k+1} > 0$ and $\varphi_{k+1} > 0$, then $v_k > 2^{-k-1}$ and $\varphi_k = 1$. Thus, combining the previous inequalities, we get 
\begin{align*}
V_{k+1}& \le \left(\int_{B_1}(\varphi_{k+1}v_{k+1})^p\, dx\right)^\frac{2}{p} \big|\{\varphi_{k+1}v_{k+1} > 0\}\big|^\gamma \\
& \le C2^{2k}V_k\big(2^{2(k+1)}V_k\big)^\gamma \le C^{k+1}V_k^{1+\gamma},
\end{align*}
where we recall $\gamma = \frac2n$ if $n\ge 3$, and $\gamma = \frac12$ otherwise; and $C$ depends only on $n$, $\lambda$, and $\Lambda$. 

Now, we claim that, if $\delta > 0$ is small enough, then 
\[
\left\{
\begin{array}{rcccl}
0& \le & V_{k+1} & \le &C^{k+1}V_k^{1+\gamma}\\
0& \le & V_{0} & \le &\delta
\end{array}
\right.
\quad \Longrightarrow \quad V_k \to 0\quad\textrm{as}\quad k \to \infty. 
\]

Indeed, in order to see this it is enough to check by induction that if $V_0 \le C^{-1/\gamma-1/\gamma^2}$ then 
\[
V_k^\gamma\le \frac{C^{-k-1}}{(2C)^{\frac1\gamma}},
\]
which is a simple computation.
Alternatively, one could check by induction that $V_k \le C^{(1+\gamma)^k\left(\sum_{i = 1}^k \frac{i}{(1+\gamma)^i}\right)} V_0^{(1+\gamma)^k}$. 

Hence, we have proved that 
\[
V_k = \int_{B_1}(\varphi_k v_k)^2\, dx\to 0 \quad\textrm{as}\quad k \to \infty. 
\]
Passing to the limit, we get 
\[
\int_{B_{1/2}}(v-1)^2_+\, dx = 0,
\]
and thus, $v \le 1$ in $B_{1/2}$, as wanted. 
\end{proof}

\begin{proof}[Proof of Proposition~\ref{prop.L2Linfty}]
To deduce the Proposition~\ref{prop.L2Linfty} from Proposition~\ref{prop.L2Linfty_2}, just use $\tilde v := {\sqrt{\delta}}\, v/{\|v_+\|_{L^2(B_1)}}$ (which solves the same equation). 
\end{proof}

This proves the first part of the estimate
\begin{equation}
\label{eq.L2Linfty}
\mathcal{L} v \le 0\quad\textrm{in}\quad B_1\qquad \Longrightarrow \qquad \|v_+\|_{L^\infty(B_{1/2})}\le C\|v_+\|_{L^2(B_1)}. 
\end{equation}

Notice that, as a direct consequence, we have the $L^2$ to $L^\infty$ estimate. Indeed, if $\mathcal{L}v = 0$ then $\mathcal{L} v_+ \le 0$ (see Lemma~\ref{lem.pospartSH}) but also $\mathcal{L} v_- \le 0$, where $v_- := \max\{0, -v\}$. Thus, $\|v_-\|_{L^\infty(B_{1/2})}\le C\|v_-\|_{L^2(B_1)}$, and since $\|v\|_{L^2(B_1)} = \|v_+\|_{L^2(B_1)}+ \|v_-\|_{L^2(B_1)}$, combining the estimate for $v_+$ and $v_-$ we get 
\begin{equation}
\label{eq.L2Linfty0}
\mathcal{L} v = 0\quad\textrm{in}\quad B_1\qquad \Longrightarrow \qquad \|v\|_{L^\infty(B_{1/2})}\le C\|v\|_{L^2(B_1)},
\end{equation}
as we wanted to see.

\begin{rem}[Moser's proof]
The proof of \eqref{eq.L2Linfty} here presented is the original proof of De Giorgi. The first ingredient in the proof was to use $\varphi^2v_+$ as a test function in the weak formulation of our PDE to get the energy inequality from Lemma~\ref{lem.energyinequality},
\[
\int_{B_1}|\nabla(\varphi v)^2|\, dx\le C\int_{B_1} v^2|\nabla \varphi|^2\, dx\quad\textrm{for all}\quad \varphi\in C^\infty_c(B_1). 
\]

Roughly speaking, this inequality said that $v$ cannot jump too quickly (the gradient is controlled by $v$ itself). 

Moser did something similar, but taking $\eta = \varphi^2(v_+)^\beta$ instead, for some $\beta \ge 1$, to get the inequality 
\[
\int_{B_1}\left|\nabla\left(v^{\frac{\beta+1}{2}}\varphi\right)^2\right|\, dx\le C\int_{B_1} v^{\beta+1}|\nabla \varphi|^2\, dx\quad\textrm{for all}\quad \varphi\in C^\infty_c(B_1). 
\]
Combining this with Sobolev's inequality, one gets 
\[
\left(\int_{B_{r_1}} v^{q\gamma}\, dx\right)^{\frac{1}{q\gamma}}\le \left(\frac{C}{|r_2-r_1|^2}\int_{B_{r_2}} v^{q}\, dx\right)^{\frac{1}{q}},
\]
where $\gamma =\frac{2^*}{2}> 1$ and $q = \beta+1$. Taking a sequence of $r_k \downarrow \frac12$ as in  De Giorgi's proof, one obtains 
\[
\|v\|_{L^{2\gamma^k}(B_{r_k})}\le C\|v\|_{L^2(B_1)},
\]
and taking $k\to \infty$ we obtain the $L^\infty$ bound in $B_{1/2}$. 

We refer the interested reader to \cite[Chapter 4]{HL} for a full proof. 
\end{rem}

\subsection*{De Giorgi's second step: $L^\infty$ to $C^{0,\alpha}$}

We next prove the second step of  De Giorgi's estimate. We want to prove:

\begin{prop}[Oscillation decay] \index{Oscillation decay}
\label{prop.osc_decay_L}
Let $\mathcal{L}$ be of the form \eqref{eq.mL}. Let $v\in H^1(B_2)$ be a solution to 
\[
\mathcal{L}v = 0\quad\textrm{in}\quad B_2. 
\]
Then, 
\[
\osc_{B_{1/2}} v \le (1-\theta)\osc_{B_2} v
\]
for some $\theta > 0$ small depending only on $n$, $\lambda$, and $\Lambda$. 
\end{prop}

As we saw in Chapter~\ref{ch.1} (see Corollary~\ref{cor.Holder_regularity_1}), this proposition immediately implies $C^{0,\alpha}$ regularity of solutions. 

As shown next, Proposition~\ref{prop.osc_decay_L} follows from the following lemma. 

\begin{lem}
\label{lem.osc_decay}
Let $\mathcal{L}$ be of the form \eqref{eq.mL}, and let $v\in H^1(B_2)$ be such that
\[
v\le 1\quad\textrm{in}\quad B_2,\qquad\textrm{and}\qquad
\mathcal{L}v\le 0\quad\textrm{in}\quad B_2.
\] 
Assume that 
\[
\big|\{v \le 0\}\cap B_1\big|\ge\mu > 0. 
\]
Then, 
\[
\sup_{B_{1/2}} v \le 1-\gamma,
\]
for some small $\gamma > 0$ depending only on $n$, $\lambda$, $\Lambda$, and $\mu$. 
\end{lem}

\begin{figure}
\includegraphics[scale = 1.25]{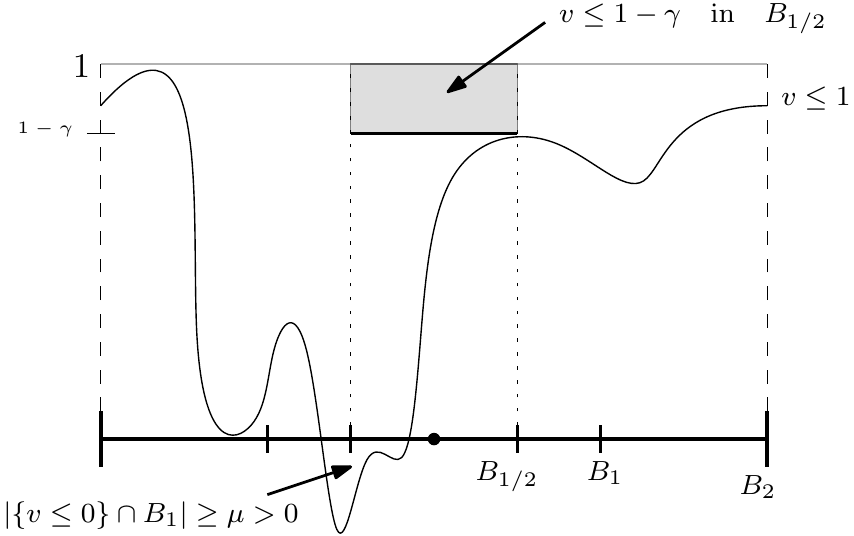}
\caption{Graphical representation of $v$ with $\mathcal{L} v \le 0$ from Lemma~\ref{lem.osc_decay}.}
\label{fig.8}
\end{figure}

In other words, if $v\le 1$, and it is ``far from 1'' in a set of non-zero measure, then $v$ cannot be close to 1 in $B_{1/2}$. (See Figure~\ref{fig.8}.)

Let us show how this lemma yields the oscillation decay: 

\begin{proof}[Proof of Proposition~\ref{prop.osc_decay_L}]
Consider the function 
\[
w(x) := \frac{2}{\osc_{B_2} v}\left(v(x) - \frac{\sup_{B_2} v +\inf_{B_2} v}{2}\right)
\]
and notice that 
\[
-1\le w \le 1\quad\textrm{in}\quad B_2,
\]
(in fact, $\osc_{B_2} w = 2$). 
Let us assume that $\big|\{w \le 0\}\cap B_1\big|\ge \frac12|B_1|$ (otherwise, we can take $-w$ instead). Then, by Lemma~\ref{lem.osc_decay}, we get
\[
w\le 1-\gamma\quad\textrm{in}\quad  B_{1/2},
\]
and thus 
\[
\osc_{B_{1/2}} w\le 2-\gamma.
\]
This yields
\[
\osc_{B_{1/2}} v \le \left(1-\frac\gamma2\right)\osc_{B_2} v,
\]
and thus the proposition is proved. 
\end{proof}

To prove Lemma~\ref{lem.osc_decay}, we will need the following De Giorgi isoperimetric inequality. It is a kind of a quantitative version of the fact that an $H^1$ function cannot have a jump discontinuity. 

\begin{lem}
\label{lem.osc_dec_2.2}
Let $w\in H^1(B_1)$ be such that 
\[
\int_{B_1}|\nabla w|^2\, dx\le C_\circ. 
\]
Let 
\[
A:= \{w\le 0\}\cap B_1,\quad D:= \left\{w\ge \frac12\right\}\cap B_1,\quad E:= \left\{0<w<\frac12\right\}\cap B_1.
\]
Then, we have 
\[
C_\circ |E|\ge c |A|^2 \cdot |D|^2
\]
for some constant $c$ depending only on $n$. 
\end{lem}

\begin{proof}
We define $\bar w$ in $B_1$ as $\bar w = w$ in $E$, $\bar w \equiv 0$ in $A$ and $\bar w = \frac12$ in $D$. In this way, $\nabla \bar w \equiv 0$ in $B_1\setminus E$ and $\int_{B_1}|\nabla \bar w|^2 \le C_\circ$. 

Let us denote the average of $\bar w$ in $B_{1}$ by $\bar w_{B_{1}} := \ave_{B_{1}} \bar w(x) \, dx$. Then, 
\begin{align*}
|A|\cdot |D|& \le 2\int_A\int_D |\bar w(x)-\bar w(y)|\, dx\, dy \\
& \le 2\int_{B_{1}}\int_{B_{1}} \left(\big|\bar w(x) - \bar w_{B_{1}} \big| + \big|\bar w(y) - \bar w_{B_{1}}\big| \right)\, dx\, dy \\
& = 4|B_{1}|\int_{B_{1}} \big|\bar w(x) - \bar w_{B_{1}} \big|\, dx \le C \int_{E} |\nabla \bar w(x)|\, dx, 
\end{align*}
where in the last step we have used the Poincar\'e inequality (Theorem~\ref{ch0-Poinc} with $p= 1$) and the fact that $\nabla \bar w \equiv 0$ in $B_1\setminus E$. Thus, by H\"older's inequality we reach 
\[
|A|\cdot |D|\le C \int_{E} |\nabla \bar w| \le C \left(\int_{E} |\nabla \bar w|^2\right)^{1/2}|E|^{1/2}\le CC_\circ^{1/2}|E|^{1/2},
\]
as we wanted to see. 
\end{proof}

Finally, we prove Lemma~\ref{lem.osc_decay}:

\begin{proof}[Proof of Lemma~\ref{lem.osc_decay}]
Consider the sequence 
\[
w_k := 2^k\left[v-(1-2^{-k})\right]_+.
\]
Notice that $w_k \le 1$ in $B_2$ since $v\le 1$ in $B_2$. Moreover, $\mathcal{L}w_k \le 0$ in $B_2$. 

Using the energy inequality (Lemma~\ref{lem.energyinequality}), we easily get that 
\[
\int_{B_1}|\nabla w_k|^2 \le C\int_{B_2}w_k^2 \le C_\circ \qquad (\textrm{notice } 0\le w_k\le 1\textrm{ in } B_2).
\]
We also have
\[
\big|\{w_k\le 0 \}\cap B_1\big|\ge \mu > 0
\]
(by the assumption on $v$). We now apply Lemma~\ref{lem.osc_dec_2.2} recursively to $w_k$, as long as 
\[
\int_{B_1}w_{k+1}^2 \ge \delta^2.
\]

We get
\[
\left|\left\{w_k\ge\frac12\right\}\cap B_1\right| \ge \big|\{w_{k+1}>0\}\cap B_1\big| \ge \int_{B_1}w_{k+1}^2 \ge \delta^2. 
\]
Thus, from Lemma~\ref{lem.osc_dec_2.2}, 
\[
\left|\left\{0<w_k<\frac12\right\}\cap B_1\right|\ge \frac{c}{C_\circ}\delta^4\mu^2 = \beta > 0,
\]
where $\beta > 0$ is independent of $k$, and depends only on $n$, $\delta$, and $\mu$. 

But notice that the sets $\left\{0<w_k<\frac12\right\}$ are disjoint for all $k\in \N$, therefore we cannot have the previous inequality for every $k$. This means that, for some $k_\circ\in \N$ (depending only on $n$ and $\beta$) we have 
\[
\int_{B_1} w_{k_\circ}^2 < \delta^2
\]
and, hence, by the $L^2$ to $L^\infty$ estimate from Proposition~\ref{prop.L2Linfty} 
\[
\|w_+\|_{L^\infty(B_{1/2})}\le C\|w_+\|_{L^2(B_1)}.
\]
We get
\[
\|w_{k_\circ}\|_{L^\infty(B_{1/2})}\le C\delta\le \frac12,
\]
provided that $\delta > 0$ is small enough, depending only on $n$, $\lambda$, and $\Lambda$. This means that $w_{k_\circ}\le \frac12$ in $B_{1/2}$, and thus 
\[
v\le \frac12 2^{-k_\circ} + \left(1-2^{-k_\circ}\right)\le 1-2^{-k_\circ-1} = 1-\gamma
\]
as desired, where $k_\circ$ (and therefore, $\gamma$) depends only on $n$, $\lambda$, $\Lambda$, and $\mu$. 
\end{proof}

Summarizing, we have now proved Lemma~\ref{lem.osc_decay} (by using the $L^2$ to $L^\infty$ estimate and Lemma~\ref{lem.osc_dec_2.2}). 
Then, Lemma~\ref{lem.osc_decay} implies the oscillation decay, and the oscillation decay implies the H\"older regularity. 

\begin{thm}
\label{thm.DGCaLi}
Let $\mathcal{L}$ be of the form \eqref{eq.mL}, and let $v\in H^1(B_1)$ solve 
\[
\mathcal{L}v = 0\quad\textrm{in}\quad B_1. 
\]
Then, 
\[
\|v\|_{C^{0,\alpha}(B_{1/2})}\le C\|v\|_{L^\infty(B_1)}
\]
for some $\alpha > 0$ and $C$ depending only on $n$, $\lambda$, and $\Lambda$. 
\end{thm}
\begin{proof}
The theorem follows from the oscillation decay, in much the same way as Corollary~\ref{cor.Holder_regularity_1}). 
\end{proof}

Combining this last result with the $L^2$ to $L^\infty$ estimate, Proposition~\ref{prop.L2Linfty}, we finally obtain the theorem of De Giorgi--Nash.
\begin{thm}
\label{thm.DGN_B}
Let $v\in H^1(B_1)$ be a weak solution to $\divv\left(A(x)\nabla v\right)  = 0$ in~$B_1$, with $0<\lambda\,{\rm Id}\le A(x) \le \Lambda\,{\rm Id}$. Then, there exists some $\alpha > 0$ such that $v\in C^{0,\alpha}(B_{1/2})$ and
\[
\|v\|_{C^{0,\alpha}(B_{1/2})}\le C\|v\|_{L^2(B_1)}.
\]
The constants $C$ and $\alpha > 0$ depend only on $n$, $\lambda$, and $\Lambda$.
\end{thm}

\begin{proof}
The result follows from Theorem~\ref{thm.DGCaLi} combined with Proposition~\ref{prop.L2Linfty} (by \eqref{eq.L2Linfty0}). 
\end{proof}

As a consequence of the previous result, we have:
\begin{proof}[Proof of Theorem \ref{thm.DGN_Om}]
It follows by Theorem \ref{thm.DGN_B} by a covering argument.
\end{proof}

In particular, as shown below,  Theorem~\ref{thm.DGN_B} solved Hilbert's XIXth problem.

This is one of the main results for which De Giorgi got the Wolf Prize in 1990, and Nash got the Abel Prize in 2015. 
It has been speculated that if only one of them had solved Hilbert's XIXth problem, he would also have received the Fields Medal for the proof.

\begin{rem}[Harnack's inequality]
Even though it is not needed to prove Theorem \ref{thm.DGN_B}, it is interesting to notice that with some more work one can also prove Harnack's inequality for operators of the form $\divv\left(A(x)\nabla v\right)$; see \cite{LZ17,Mo61}.
\end{rem}

\section{Solution to Hilbert's XIXth problem}
\label{sec.solution}
In this chapter, we have proved the {\em interior regularity} result for $v\in H^1(B_1)$
\[
\divv\big(A(x) \nabla v\big)  = 0\quad\textrm{in}\quad B_1\qquad\Longrightarrow \qquad\|v\|_{C^{0,\alpha}(B_{1/2})}\le C\|v\|_{L^2(B_1)},
\]
for some small $\alpha > 0$ depending only on $n$, $\lambda$, and $\Lambda$. 

For a general domain $\Omega\subset \R^n$, this gives the estimate for $v\in H^1(\Omega)$
\[
\divv\big(A(x) \nabla v\big) = 0\quad\textrm{in}\quad \Omega\quad\Longrightarrow \quad\|v\|_{C^{0,\alpha}(\tilde\Omega)}\le C\|v\|_{L^2(\Omega)},
\]
for any $\tilde\Omega\subset\subset\Omega$  (with a constant $C$ that depends only on $n$, $\lambda$, $\Lambda$, $\Omega$, and~$\tilde \Omega$).

Thanks to this, one can in fact solve Hilbert's XIXth problem:

\vspace{2mm}

\begin{center} 
 \fbox{
\begin{minipage}{0.85\textwidth}
\vspace{2.5mm}

\begin{thm*}
\index{Hilbert XIXth problem} Let $u\in H^1(\Omega)$ be any local minimizer of 
\[
\mathcal{E}(w) := \int_\Omega L(\nabla w) \, dx,
\]
where $L$ is  smooth and uniformly convex, and $\Omega\subset\R^n$ is bounded. Then, $u$ is $C^\infty$ in $\Omega$. 
\end{thm*}
\vspace{0mm}
\end{minipage}
}

\end{center}

\vspace{2mm}

\begin{proof}
Any local minimizer $u$ satisfies
\[
\int_\Omega DL(\nabla u) \nabla\phi\, dx = 0,\qquad\textrm{for all}\quad \phi\in C^\infty_c(\Omega). 
\]
(This is the weak formulation of $\divv(DL(\nabla u)) = 0$ in $\Omega$.)

As in the proof of Theorem~\ref{thm.C1aimpCinf}, if we define for any $h\in \R^n$,
\[
\tilde A(x) := \int_0^1 D^2 L\big(t\nabla u(x+h)+(1-t)\nabla u(x)\big)\, dt,
\]
then $\tilde A(x)$ is uniformly elliptic (since $L$ is uniformly convex,  $0< \lambda{\rm Id}\le D^2L(p)\le \Lambda{\rm Id}$). We have that $\frac{u(\cdot+h)-u}{|h|}\in H^1(\Omega_h)$ fulfills (again, see Theorem~\ref{thm.C1aimpCinf})
\[
\int_\Omega \nabla \left(\frac{u(x+h)-u(x)}{|h|}\right)\cdot \tilde A(x) \nabla \phi(x)\, dx = 0\quad\textrm{for all }\phi\in C^\infty_c(\Omega_h),
\]
that is, $\frac{u(\cdot+h)-u}{|h|}$ solves weakly 
\[
\divv\left(\tilde A\nabla\left[\frac{u(\cdot+h)-u}{|h|}\right]\right) = 0,\quad\textrm{in}\quad\Omega_h. 
\]
(We recall $\Omega_h := \{x\in \Omega : {\rm dist}(x, \de\Omega)>  |h|\}$.)

%
%
%

By the estimate of De Giorgi and Nash (Theorem~\ref{thm.DGN_Om}), we find that 
\[
\left\|\frac{u(\cdot+h)-u}{|h|}\right\|_{C^{0,\alpha}(\tilde \Omega)}\le C\left\|\frac{u(\cdot+h)-u}{|h|}\right\|_{L^2(\Omega_h)}\le C\|\nabla u\|_{L^2(\Omega)} 
\]
for any $\tilde \Omega\subset\subset\Omega_h$ (see \ref{it.S9} in Chapter~\ref{ch.0}). By \ref{it.H7}, since the constant $C$ is independent of $h$, this yields
\[
\|u\|_{C^{1,\alpha}(\tilde\Omega)}\le C\|u\|_{H^1(\Omega)}.
\]
Now, once $u$ is $C^{1,\alpha}$, we are done by Theorem~\ref{thm.C1aimpCinf}. 
\end{proof}
%

\section{Further results and open problems}

Let us finish this chapter by mentioning some state-of-the art results and open problems regarding the minimization of convex energy functionals. 

As we have explained, the minimization of a convex functional is a classical problem in the Calculus of Variations. Namely, 
\begin{equation}
\label{eq.minLW}
\min_{w\in \mathcal{W}} \int_\Omega L(\nabla w)\, dx,
\end{equation}
with $L:\R^n\to \R$ convex, $\Omega\subset \R^n$, and some appropriate class of functions $\mathcal{W}$, say, with prescribed trace on $\de\Omega$. Hilbert's XIXth problem deals with the case in which $L$ is \emph{uniformly} convex and smooth, to obtain nice regularity results. In Remark~\ref{rem.conv} we discuss that lack of convexity can yield non-uniqueness of minimizers, but it is not that clear what occurs if we simply remove the condition on the uniform convexity, but maintain the \emph{strict} convexity. In fact, functionals involving functions $L$ that only involve strict convexity (that is, $D^2 L$ could have 0 and $\infty$ as eigenvalues in some sets) appear naturally in some applications: anisotropic surface tensions, traffic flow, and statistical mechanics (see \cite{M19B} and the references therein). 

Minimizers of \eqref{eq.minLW} are known to be Lipschitz (under enough smoothness of the domain and boundary data) by the comparison principle. 
Thus, the following natural question is to whether first derivatives of minimizers are continuous: 
\[
\textrm{If $L$ is strictly convex, are minimizers to \eqref{eq.minLW} $C^1$?}
\]
The answer to that question has been investigated in the last years. The problem was first addressed by De Silva and Savin in \cite{DS10}, where they studied the case of dimension 2:

 \begin{thm}[\cite{DS10}]
 Let $u$ be a Lipschitz minimizer to  \eqref{eq.minLW} in $\R^2$, and suppose that $L$ is strictly convex. Assume that the set of points where $D^2L$ has some eigenvalue equal to 0 or $\infty$ is finite. Then, $u\in C^1$. 
 \end{thm}
 
  Later, Mooney in \cite{M19B} studied the problem in higher dimensions and showed that the question has a negative answer, in general, in dimensions $n\ge 4$:
  \begin{thm}[\cite{M19B}]
  In $\R^4$ there exists a Lipschitz minimizer to \eqref{eq.minLW}, with $L$ strictly convex, that is not $C^1$. 
  \end{thm}
  
  In the example by Mooney, the minimizer is analytic outside the origin (having a singularity there), and the corresponding functional has a Hessian with an eigenvalue going to $\infty$ in $\{x_1^2+x_2^2 = x_3^2+x_4^2\}\cap \sqrt{2} \mathbb{S}^3$, but otherwise, the eigenvalues are uniformly bounded from below away from zero.

It is currently an open question what happens in dimension $n =3$, as well as what happens for general strictly convex functionals in $\R^2$.

\chapter{Fully nonlinear elliptic PDE}
\label{ch.3}

Second order nonlinear elliptic PDEs in their most general form can be written as
\begin{equation}
\label{eq.fully}
F(D^2 u , \nabla u , u, x ) = 0\quad\textrm{in}\quad \Omega\subset\R^n.
\end{equation}
Understanding the regularity of solutions to these equations has been a major research direction since the mid-20th century. 

These are called \emph{fully nonlinear elliptic equations}. 
Besides their own interest in PDE theory, they arise in Probability Theory (stochastic control, differential games; see Appendix~\ref{app.B} for a probabilistic interpretation), and in Geometry. 

Thanks to Schauder-type estimates, under natural assumptions on the dependence on $\nabla u$, $u$, and $x$, the regularity for \eqref{eq.fully} can be reduced to understanding solutions to \index{Fully nonlinear equation}
\begin{equation}
\label{eq.fully2}
\boxed{F(D^2 u) = 0\quad\textrm{in}\quad \Omega\subset\R^n.}
\end{equation}
Indeed, some of the ``perturbative'' methods that we used in Chapter~\ref{ch.1} to prove Schauder estimates for linear equations $\sum a_{ij}(x) \partial_{ij} u = f(x)$ in $\Omega\subset\R^n$ work in such fully nonlinear setting, too. 
For simplicity, we will focus here on the study of \eqref{eq.fully2}.

In the next sections we will discuss the following:
\begin{itemize}
\item[--] What is ellipticity for solutions to \eqref{eq.fully2}?
\item[--] Existence and uniqueness of solutions.
\item[--] Regularity of solutions to \eqref{eq.fully2}. 
\end{itemize}

We will \emph{not} prove all the main known results of this Chapter, but only give an overview of what is known. 
We refer to the books \cite{CC} and \cite{NTV} for more details about this topic.

\section{What is ellipticity?}
\label{sec.ellipt}
There are (at least) two possible ways to define ellipticity: 
\begin{itemize}
\item[--] Linearizing the equation. 
\item[--] ``Imposing'' that the comparison principle holds.
\end{itemize}
We will see that they are essentially the same. 

\begin{defi}\index{Ellipticity condition!Fully nonlinear equations}
Let $F: \R^{n\times n}\to \R$. We say that $F$ is \emph{elliptic} if for any two symmetric matrices $A, B\in \R^{n\times n}$ such that $A \ge B$ (i.e., $A-B$ is positive semi-definite) we have 
\[
F(A) \ge F(B),
\]
with strict inequality if $A > B$ (i.e., $A-B$ positive definite). 
\end{defi}
The Laplace equation $\Delta u = 0$ corresponds to the case $F(M) ={\rm tr}\,M $. For a linear equation (with constant coefficients)
\[
\sum_{i,j = 1}^n a_{ij} \partial_{ij} u = 0,
\]
$F$ is given by $F(M) = {\rm tr}\, (AM)$, where $A = (a_{ij})_{i,j}$. This equation is elliptic if and only if the coefficient matrix $A$ is positive definite. Therefore, it coincides with our notion of ellipticity for linear equations.

\begin{rem}[Comparison Principle]
If a $C^2$ function $v$ touches $u\in C^2$ from below at a point $x_\circ$ (i.e. $u\ge v$ everywhere, and $u(x_\circ) = v(x_\circ)$; see Figure~\ref{fig.9}), then it follows that 
\[
\nabla u (x_\circ ) = \nabla v (x_\circ),\qquad D^2 u(x_\circ) \ge D^2 v(x_\circ).
\]
Therefore, for these functions we would have $F(D^2 u(x_\circ)) \ge F(D^2 v(x_\circ))$ if $F$ is elliptic. This is essential when proving the comparison principle.
\begin{figure}
\includegraphics[scale = 1.3]{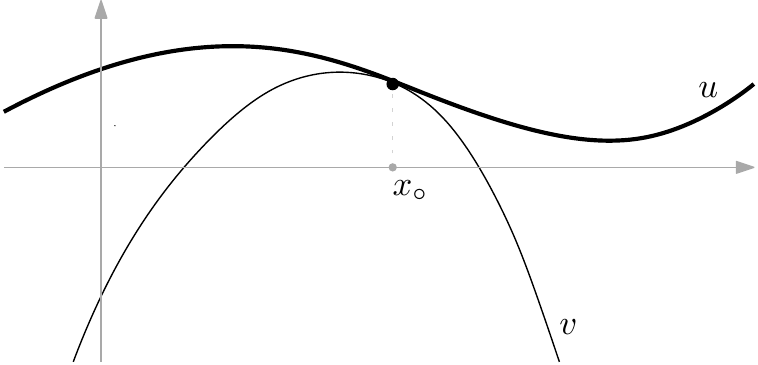}
\caption{The function $v$ touches $u$ from below at $x_\circ$.}
\label{fig.9}
\end{figure}
\end{rem}

\begin{prop}[Comparison Principle]\index{Comparison principle!Fully nonlinear equations (classical)}
\label{prop.comp_princ_C2}
Assume that $F$ is elliptic, and $\Omega\subset\R^n$ is bounded. Let $u, v\in C^2(\Omega)\cap C^0(\overline{\Omega})$. Then,
\[
\left\{
\begin{array}{rcll}
u & \ge & v& \textrm{on   } ~\partial \Omega\\
F(D^2 u)& \le & F(D^2 v)& \textrm{in   }~\Omega.
\end{array}
\right.
\quad
\Longrightarrow 
\quad u \ge v\quad\textrm{in}\quad \overline{\Omega}. 
\]
\end{prop}

\begin{proof} We separate into two cases. 
\\[0.1cm]
{\it \underline{\smash{Case 1}}.} Assume first that $F(D^2 u) < F(D^2 v)$ in $\Omega$ (with strict inequality). If the conclusion is false, then the function $u-v$ would have an interior minimum inside $\Omega$, say at $x_\circ \in \Omega$. Then, we would have $D^2(u-v)(x_\circ) \ge 0$. Therefore, $D^2 u(x_\circ)\ge D^2 v(x_\circ)$ and by ellipticity of $F$, this yields $F(D^2 u(x_\circ))\ge F(D^2 v(x_\circ))$. This is a contradiction with $F(D^2 u) < F(D^2 v)$ in $\Omega$, and hence $u \ge v$ in $\overline{\Omega}$.
\\[0.1cm]
{\it \underline{\smash{Case 2}}.} Assume now $F(D^2 u)\le F(D^2 v)$ in $\Omega$. 
Then, we can define 
\[
\bar u(x) := u(x) +\eps\left(c_\Omega - |x|^2\right),
\]
where $c_\Omega > 0$ is a constant such that $c_\Omega-|x|^2 > 0$ in $\overline{\Omega}$ (recall that $\Omega$ is bounded). 

Then, we have $\bar u \ge u $ on $\partial \Omega$, and $D^2\bar u = D^2 u - 2\eps{\rm Id}$. Thus, by ellipticity,
\[
F(D^2\bar u ) = F(D^2 u - 2\eps{\rm Id}) < F(D^2 u) \le F(D^2 v) \quad\textrm{in}\quad \Omega. 
\]
By Case 1, 
\[
\left\{
\begin{array}{rcll}
\bar u & \ge & v& \textrm{on   } ~\partial \Omega,\\
F(D^2 \bar u)& < & F(D^2 v)& \textrm{in   }~\Omega.
\end{array}
\right.
\quad
\Longrightarrow 
\quad \bar u \ge v\quad\textrm{in}\quad \Omega. 
\]

This gives 
\[
u(x) +\eps\left(c_\Omega - |x|^2\right) \ge v(x)\quad\textrm{in}\quad \Omega.
\]
Letting $\eps\downarrow 0$ we deduce that $u \ge v$ in $\Omega$.
\end{proof}

Thus, we see that ellipticity is exactly what we need in order to prove the comparison principle. We will see that \emph{uniform ellipticity} (analogously to the case of linear equations) implies, in fact, the regularity of solutions. 

\begin{defi}
\label{defi.unifellipt}\index{Uniform ellipticity condition!Fully nonlinear equations}
Let $F: \R^{n\times n}\to \R$. Then $F$ is \emph{uniformly elliptic} if there are $0 < \lambda\le \Lambda$ (the \emph{ellipticity constants}\index{Ellipticity constants!Fully nonlinear equations}), such that for every symmetric matrices $M$, $N$ with $N \ge 0$ (that is, positive semi-definite), we have 
\[
\lambda\|N\|\le F(M+N) - F(M) \le \Lambda\|N\|,
\]
where $\|N\|:= {\rm tr}\left( (N^T N)^{1/2}\right) = {\rm tr}(N)$ is the sum of the (absolute value of the) eigenvalues.
\end{defi}

We remark that our choice of matrix norm in the previous definition is not standard. In $\R^n$, all norms are equivalent and thus we could have chosen any other norm. This definition of norm, however, avoids dealing with constants in future computations. 

Of course, uniform ellipticity implies ellipticity, in a quantitative way. 

For linear equations, i.e. $F(M) = {\rm tr}\, (AM)$, uniform ellipticity is equivalent to 
\[
0<\lambda  {\rm Id}\le A\le \Lambda{\rm Id},
\]
as usual. 

The alternative way to see ellipticity is by linearizing the equation:

Assume $F\in C^1$ (which is not always the case!).
We consider the functions
\[
F_{ij}(M) := \frac{\de F}{\de M_{ij}} (M),
\] 
i.e., the first derivative of $F(M)$ with respect to the component $M_{ij}$ of the matrix $M$. 

Then, it is immediate to see that \index{Uniform ellipticity condition!Fully nonlinear equations}
\[
\begin{split}
\textrm{$F$ is uniformly elliptic} & ~~~\Longleftrightarrow ~~~ 0< \lambda\,{\rm Id}\le (F_{ij}(M))_{i,j}\le \Lambda\,{\rm Id},\qquad \forall M\\
& ~~~\Longleftrightarrow ~~~\textrm{the linearized equation is uniformly elliptic.}
\end{split}
\]

Therefore, at least when $F$ is $C^1$, uniform ellipticity can be seen as uniform ellipticity of the linearized equation. 

In general, though, the uniform ellipticity condition implies that $F$ is Lipschitz, but not always $C^1$.  There are, in fact, important examples of equations $F(D^2 u) = 0$ in which the corresponding $F$ is Lipschitz but not~$C^1$. In this case, the previous characterization of ellipticity through the derivatives of $F$ still holds, understanding now that they are defined almost every\-where.

\begin{rem}[Convex (or concave) equations]\index{Convex fully nonlinear equations}
An important subclass of equations $F(D^2 u) = 0$ are those for which $F$ is convex (or concave). Namely, $F(M)$ as a function $F:\R^{n\times n}\to \R$ is convex (or concave). In this case, the equation can be written as a \emph{Bellman equation} \index{Bellman equation}(see \eqref{eq.maxPDE}), as 
\[
F(D^2 u) = \max_{\alpha\in \mathcal{A}} \{L_\alpha u\} = 0, 
\]
where $\{L_\alpha\}_{\alpha\in \mathcal{A}}$ is a family of linear operators of the form 
\[
L_\alpha u := \sum_{i,j=1}^n a_{ij}^\alpha \de_{ij} u + c_\alpha,
\]
for a family of coefficients $\{a_{ij}^\alpha\}_{\alpha\in \mathcal{A}}$ uniformly elliptic, with ellipticity constants $\lambda$ and $\Lambda$.

Notice that if $u$ solves $F(D^2 u) = 0$, with $F$ convex, then $v = -u$ solves $G(D^2 v) = 0$, with $G(M) = -F(-M)$, and therefore, $G$ is concave. 
\end{rem}

\subsection*{Pucci operators} 
Within the class of fully nonlinear uniformly elliptic operators with ellipticity constants $\lambda$ and $\Lambda$, the \emph{extremal} or \emph{Pucci} operators, denoted by $\mathcal{M}^+$ and $\mathcal{M}^-$, are those that attain the \emph{extreme} values (from above and below, respectively). Alternatively, every other elliptic operator with the same ellipticity constants is ordered with respect to them in the sense of \eqref{eq.pucciext} below. 

We define $\mathcal{M}^\pm$ as follows.
\begin{defi}\label{def.pucci}\index{Pucci operators}\index{Extremal operators}
Given $0<\lambda\le \Lambda$, the \emph{extremal} or \emph{Pucci} operators with ellipticity constants $\lambda$ and $\Lambda$, $\mathcal{M}^\pm:\R^{n\times n}\to \R$, are defined as 
\begin{equation}
\label{eq.Pucci}
\begin{split}
\mathcal{M}^-(M) & := \inf_{\lambda{\rm Id}\le (a_{ij})_{i,j}\le \Lambda{\rm Id}}\bigg\{\sum_{i, j = 1}^n a_{ij}M_{ij} \bigg\} = \inf_{\lambda{\rm Id}\le A\le \Lambda{\rm Id}}\left\{{\rm tr}\,(AM)\right\}\\
\mathcal{M}^+(M) & := \sup_{\lambda{\rm Id}\le (a_{ij})_{i,j}\le \Lambda{\rm Id}}\bigg\{\sum_{i, j = 1}^n a_{ij}M_{ij}\bigg\}= \sup_{\lambda{\rm Id}\le A\le \Lambda{\rm Id}}\left\{{\rm tr}\,(AM)\right\},
\end{split}
\end{equation}
for any symmetric matrix $M$. They are uniformly elliptic operators, with ellipticity constants $\lambda$ and $\Lambda$. 
\end{defi}

In particular, from the definition we have 
\[
\mathcal{M}^\pm(\alpha M) = \alpha\mathcal{M}^\pm(M), \quad\textrm{for all}\quad \alpha \ge 0. 
\]

Notice that $\mathcal{M}^\pm= \mathcal{M}^\pm_{n, \lambda, \Lambda}$. In general, however, the dependence on the ellipticity constants and the dimension will be clear in the corresponding context, and thus we will drop it in the notation. 

Sometimes, it is easier to define the Pucci operators through the eigenvalues of the corresponding matrix, appropriately weighted with the ellipticity constants, in the following way. 

\begin{lem} 
\label{lem.Pucci} 
The Pucci operators as defined in \eqref{eq.Pucci} can be equivalently defined as
\begin{equation}
\label{eq.Pucci2}
\begin{split}
\mathcal{M}^-(M) & = \lambda \sum_{\mu_i > 0} \mu_i + \Lambda \sum_{\mu_i < 0 }\mu_i = \lambda\|M_+\| - \Lambda \|M_-\|,\\
\mathcal{M}^+(M) & = \Lambda \sum_{\mu_i > 0} \mu_i + \lambda \sum_{\mu_i < 0 }\mu_i = \Lambda\|M_+\| - \lambda\|M_-\|,
\end{split}
\end{equation}
where $\mu_i = \mu_i(M)$ denote the eigenvalues of the symmetric matrix $M$, the matrices $M_+$ and $M_-$ are such that $M_\pm \ge 0$, $M = M_+-M_-$, and $\|A\| = {\rm tr}\left((A^T A)^{1/2}\right)$.  
\end{lem}

\begin{proof}
The proof follows directly using the following rearrangement-type inequalities involving the eigenvalues and the product of two symmetric matrices $A$ and $B$: 
\[
\sum_{i = 1}^n \lambda_i(A)\lambda_{n-i}(B) \le {\rm tr}(AB) \le\sum_{i = 1}^n \lambda_i(A) \lambda_i(B),
\]
where $\lambda_1(A)\le \dots\le \lambda_n(A)$ denote the ordered eigenvalues of $A$, and $\lambda_1(B)\le \dots\le \lambda_n(B)$ denote the ordered eigenvalues of $B$.
\end{proof}

From the definition of uniform ellipticity of $F$ (Definition~ \ref{defi.unifellipt}) it follows that, given two symmetric matrices $M$, $N$, 
\[
\lambda \|N_+\|-\Lambda\|N_-\| \le F(M+N)-F(M) \le \Lambda\|N_+\|-\lambda\|N_-\|,
\]
where $N = N_+ -N_-$, and $N_\pm \ge 0$. Thus, by Lemma~\ref{lem.Pucci},  
\begin{equation}
\label{eq.pucciext0}
\mathcal{M}^-(N)\le F(M+N) -F(M) \le \mathcal{M}^+(N). 
\end{equation}
If we take $M = 0$, we see that
\begin{equation}
\label{eq.pucciext}
\mathcal{M}^-(N) \le F(N) -F(0) \le \mathcal{M}^+(N),
\end{equation}
so these operators are like the ``worse case'' from above and below --- up to a constant, $F(0)$.
(Recall that $\mathcal{M}^\pm$ are fully nonlinear uniformly elliptic operators with ellipticity constants $\lambda$, $\Lambda$.)

If we further assume that $F(0) = 0$, we see that if $u$ solves any equation of the form $F(D^2 u) = 0$ then in particular
\begin{equation}
\label{eq.nondiv_bmc}
\mathcal{M}^-(D^2 u) \le 0 \le \mathcal{M}^+(D^2 u). 
\end{equation}

\begin{rem}\label{rem.eqnondivbmc}\index{Equation in non-divergence form with bounded measurable coefficients}Equation \eqref{eq.nondiv_bmc} is called \emph{equation in non-divergence form with bounded measurable coefficients}. Indeed, notice that given some uniformly elliptic coefficients $(a_{ij}(x))_{i,j}$ with no regularity assumption on $x$, if $u\in C^2$ fulfills $\sum_{i,j} a_{ij}(x) \de_{ij} u$ then in particular \eqref{eq.nondiv_bmc} holds. On the other hand, if \eqref{eq.nondiv_bmc} holds for some $u\in C^2$, one can recover some uniformly elliptic coefficients $(a_{ij}(x))_{i,j}$ such that $\sum_{i,j} a_{ij}(x) \de_{ij} u$.
\end{rem}

\section{Equations in two variables}

Before going into the general theory of existence and regularity for fully nonlinear equations in $\R^n$, let us study a simpler case: fully nonlinear equations in two variables.

The main regularity estimate in this context is due to Nirenberg \cite{Nir}, and reads as follows.

\begin{thm}
\label{thm.2D}\index{Equations in two variables (fully nonlinear)}
Let $F:\R^{2\times 2}\to \R$ be uniformly elliptic with ellipticity constants $\lambda$ and $\Lambda$. Let $u\in C^2(B_1)$ solve 
\[
F(D^2u ) = 0\quad\textrm{in}\quad B_1\subset\R^2.
\]
Then,
\[
\|u\|_{C^{2,\alpha}(B_{1/2})}\le C\|u \|_{L^\infty(B_1)},
\]
for some constants $\alpha > 0$ and $C$ depending only on $\lambda$ and $\Lambda$. 
\end{thm}

The idea of the proof is the following: define $v := \de_e u$, and differentiate the equation $F(D^2u)=0$ in the $e$ direction, to get
\begin{equation}
\label{eq:newtag1}
\sum_{i, j = 1}^2 a_{ij}(x)\de_{ij} v(x) = 0\quad\textrm{in}\quad B_1\subset\R^2,
\end{equation}
where $a_{ij}(x) := F_{ij}(D^2 u(x))$ for $i,j\in \{1,2\}$. 
Since $F$ is uniformly elliptic, we have $a_{22}(x) \ge \lambda > 0$. Thus, we can divide \eqref{eq:newtag1} by $a_{22}(x)$ to obtain 
\begin{equation}
\label{eq:newtag2}
a(x) \de_{11}v(x) + b(x) \de_{12}v(x) +\de_{22}v(x) = 0,
\end{equation}
for some coefficients 
\[
a(x) = \frac{a_{11}(x)}{a_{22}(x)}\quad\text{and} \quad b(x) = \frac{a_{12}(x)+a_{21}(x)}{a_{22}(x)} = \frac{2a_{12}(x)}{a_{22}(x)}.
\] 
If we write $w := \partial_{1} v$ and differentiate \eqref{eq:newtag2} with respect to $x_1$, we get
\[
\de_{1}\big(a(x) \de_{1}w(x) + b(x)\de_{2} w(x) \big) + \de_{22}w(x) = \divv(A(x)\nabla w) = 0,
\]
where 
$$
A(x) := \left(\begin{matrix}
a(x) & b(x)\\
0 & 1
\end{matrix}\right).
$$
That is, $w$ solves an equation in divergence form, and $A$ is uniformly elliptic, with ellipticity constants depending on  $\lambda$  and $\Lambda$. 
Thus, by the De~Giorgi--Nash result (Theorem~\ref{thm.DGN_Om}) one has $\partial_1 v=w\in C^{0,\alpha}(B_{1/2})$.
Since the roles of $x_1$ and $x_2$ can be changed, and since $v=\partial_e u$ (with $e\in \mathbb S^{n-1}$ arbitrary), we deduce that $u\in C^{2,\alpha}(B_{1/2})$. 

Let us now formally prove it. The idea is the one presented in the lines above, where we used that $u\in C^4$. In reality we can only use that $u\in C^2$, so we proceed by means of incremental quotients. 

\begin{proof}[Proof of Theorem~\ref{thm.2D}]
Let us define 
\[
v(x) = \frac{u(x+h)-u(x)}{|h|}\in C^2(B_{1-|h|}),
\]
with $|h|< \frac14$, and   proceed similarly to Theorem~\ref{thm.C1aimpCinf}. Since $F$ is translation invariant,  we have
\[
F(D^2 u(x) ) = 0,\quad F(D^2 u(x+h) ) = 0 \quad\textrm{in}\quad B_{1-|h|}. 
\]
Then, by the fundamental theorem of calculus for line integrals,
\[
0 = F(D^2 u(x+h) ) - F(D^2 u(x)) = \sum_{i,j = 1}^2 a_{ij}(x) \de_{ij} \big(u(x+h)-u(x)\big),
\]
where
\[
a_{ij}(x) = \int_0^1 F_{ij}\big(t D^2 u(x+h)  + (1-t) D^2 u(x) \big)\, dt.
\]
Since $F$ is uniformly elliptic,  $(a_{ij})_{i,j}$ is uniformly elliptic (with the same ellipticity constants). That is, $v\in C^2(B_{1-|h|})$ solves an equation in non-divergence form
\[
a_{11}(x) \de_{11}v(x)+ 2a_{12}(x)\de_{12}v(x)  + a_{22}(x) \de_{22}v(x) = 0 \quad\textrm{in}\quad B_{1-|h|},
\]
where $a_{12} = a_{21}$ and $\de_{12}v = \de_{21}v$ because $v\in C^2$. From the ellipticity conditions, we have $\lambda \le a_{22}(x)\le \Lambda$, and we can divide by $a_{22}(x)$ to get 
\[
a(x) \de_{11}v(x)+ b(x)\de_{12} v(x)+\de_{22}v(x) = 0 \quad\textrm{in}\quad B_{1-|h|}.
\]
Let 
$$
A(x) := \left(\begin{matrix}
a(x) & b(x)\\
0 & 1
\end{matrix}\right).
$$
It is straightforward to check that $A$ is uniformly elliptic, with ellipticity constants $\lambda/\Lambda$ and $\Lambda/\lambda$. 
Let $\eta\in C^2_c(B_{1-|h|})$ and notice that, by integration by parts, 
\[
\int_{B_{1-|h|}} \partial_2\eta \,\de_{12}v = \int_{B_{1-|h|}} \partial_1\eta \,\de_{22}v.
\]
Thus,
\begin{align*}
\int_{B_{1-|h|}} \nabla \eta \cdot A(x) \nabla \de_{1}v & = \int_{B_{1-|h|}} \nabla \eta(x) \cdot \left(\begin{matrix}
a(x)\de_{11}v(x) + b(x) \de_{12}v(x) \\
\de_{12}v (x)
\end{matrix}\right)\, dx\\
& = \int_{B_{1-|h|}} \big\{\partial_1\eta \,\big(a(x) \de_{11}v + b(x) \de_{12}v \big)+ \partial_2\eta \, \de_{12}v\big\}\, dx\\
& = \int_{B_{1-|h|}} \partial_1\eta\, \big(a(x) \de_{11}v + b(x) \de_{12}v +  \de_{22}v\big)\, dx\\
& = 0. 
\end{align*}
That is, $\de_{1}v$ solves an equation with bounded measurable coefficients $A(x)$ in divergence form. Thus, by the De~Giorgi--Nash theorem (see Theorem~\ref{thm.DGCaLi}), we know that $\de_{1}v\in C^\alpha$ and 
\[
\|\de_{1}v\|_{C^{0, \alpha}(B_{1/2})} \le C\|\de_{1}v\|_{L^\infty(B_{1-|h|})}\le C\|\de_{1}u\|_{C^{0,1}(B_{1})},
\]
(notice that we can go from $B_1$ to $B_{1-|h|}$ in Theorem~\ref{thm.DGCaLi} by a covering argument for $|h|$ small), for some constant $C$ depending only on $\lambda$ and $\Lambda$. By letting $|h|\to 0$, thanks to \ref{it.H7}, we obtain that 
\[
\|\nabla \de_{1}u\|_{C^{0, \alpha}(B_{1/2})} \le C\|\de_{1}v\|_{L^\infty(B_{1-|h|})}\le C\|\de_{1}u\|_{C^{0,1}(B_{1})},
\]
for some constant $C$ depending only on $\lambda$  and $\Lambda$.
By symmetry, the same inequality is true for $\de_{2}v$ (and $\de_{2}u$), so that
\[
\|u\|_{C^{2, \alpha}(B_{1/2})} \le C\|u\|_{C^{1,1}(B_{1})}.
\]

Notice that, by interpolation inequalities (see \eqref{ch0-interp2}), for each $\eps > 0$, there exists some $C_\eps > 0$ such that 
\[
\|u\|_{C^{1, 1}(B_{1/2})} \le \eps \|u\|_{C^{2, \alpha}(B_{1})}+ C_\eps\|u\|_{L^\infty(B_{1})}.
\]
Now, the proof can be concluded by means of Lemma~\ref{lem.SAL} analogously to what has been done in the proof of Theorem~\ref{thm.Schauder_estimates}. 
\end{proof}

Thus, as we can see, in the two-dimensional case it is rather easy to show a priori $C^{2,\alpha}$ estimates for solutions to the fully nonlinear equation. Thanks to these estimates, by means of the continuity method (see \cite{GT} or \cite{HL}) one can actually show the existence of $C^{2,\alpha}$ solutions for the Dirichlet problem. 

Nonetheless, as we will see, it turns out that in higher dimensions such an a priori estimate is no longer available, and one needs to prove existence of solutions in a different way, by introducing a new notion of weak solution (viscosity solutions). 

This is what we do in the next section.

\section{Existence of solutions}
\label{sec.43}

We now turn our attention to fully nonlinear elliptic equations in $\R^n$.

The first question to understand is the existence of solutions: given  a nice domain $\Omega\subset\R^n$, and a nice boundary data $g: \de\Omega\to \R$, can we always solve the following Dirichlet problem? 
\[
\left\{
\begin{array}{rcll}
F(D^2 u) & = & 0 & \textrm{in }~\Omega\\
u & = & g & \textrm{on }~\de\Omega.
\end{array}
\right.
\]

Notice that here we cannot construct the solution by minimizing a functional, since these fully nonlinear equations do not come, in general, from any energy functional. 

To construct the solution, we only have two options: 
\begin{itemize}
\item[--] Prove ``a priori estimates'' and then use the continuity method. 
\item[--] Use the comparison principle and Perron's method. 
\end{itemize}

The continuity method is reasonably easy to use, but we need $C^{2, \alpha}$ estimates for solutions up to the boundary. This is a very difficult problem, and in fact, in general we do \emph{not} have $C^{2, \alpha}$ estimates for these equations in~$\R^n$. 

Therefore, we need to construct some kind of generalized notion of solution: \emph{viscosity solutions}.

The right concept of solution must be so that we have 
\begin{itemize}
\item \underline{\smash{Existence of solutions}}. 
\item \underline{\smash{Comparison principle}} (and in particular, uniqueness of solutions).
\item \underline{\smash{Stability}} (so that limits of solutions are solutions). 
\end{itemize}

Notice that if we consider only $C^2$ solutions, then we have the comparison principle (and it is easy to prove), but we may not be able to prove existence. 

On the other hand, if we relax the notion of solution, then we may be able to easily prove the existence of a solution, but then it will be more difficult to prove the uniqueness/comparison principle.

The right notion of generalized solution is the one given in Definition~\ref{defi.visco} below, known as viscosity solutions. For subsolutions in the viscosity sense, this notion only requires that the function is upper semi-continuous (USC), while for supersolutions in the viscosity sense, this notion can be checked on lower semi-continuous (LSC) functions. This is important in the proof of existence of solutions. 

We recall that a function $f$ is said to be upper semi-continuous at $x_\circ$ if 
\[
\limsup_{x\to x_\circ} f(x) \le f(x_\circ).
\]
 Similarly, it is lower semi-continuous at $x_\circ$ if \[\liminf_{x\to x_\circ} f(x) \ge f(x_\circ).\]

We refer to \cite{S-viscosity} for a nice introduction to viscosity solutions to elliptic equations. 

\begin{defi}[Viscosity solutions]\index{Viscosity solution}
\label{defi.visco}
Let $F:\R^{n \times n}\to \R$ be uniformly elliptic, and consider the PDE
\[
F(D^2 u) = 0\quad\textrm{in}\quad\Omega. 
\]
$\bullet$ \  \index{Viscosity subsolution} We say that $u\in {\rm USC}(\overline{\Omega})$ is a \emph{subsolution} (in the viscosity sense), and we write $F(D^2u) \ge 0$, if for any $\phi\in C^2(\Omega)$ such that $\phi \ge u$ in $\Omega$ and $\phi(x_\circ) = u(x_\circ)$, $x_\circ \in \Omega$, we have $F(D^2\phi(x_\circ)) \ge 0$.

\vspace{1mm}

\noindent $\bullet$ \   \index{Viscosity supersolution} We say that $u\in {\rm LSC}(\overline{\Omega})$ is a \emph{supersolution} (in the viscosity sense), and we write $F(D^2u) \le 0$, if for any $\phi\in C^2(\Omega)$ such that $\phi \le u$ in $\Omega$ and $\phi(x_\circ) = u(x_\circ)$, $x_\circ \in \Omega$, we have $F(D^2\phi(x_\circ)) \le 0$.

\vspace{1mm}

\noindent $\bullet$ \   \index{Viscosity solution} We say that $u\in C(\overline{\Omega})$ solves $F(D^2 u) = 0$ in $\Omega$ in the viscosity sense if it is both a subsolution and a supersolution.
\end{defi}

Notice that there may be points $x_\circ\in \Omega$ at which no function $\phi\in C^2$ touches $u$ at $x_\circ$ (from above and/or from below). This is allowed by the previous definition. 

\begin{rem}[Some history]
The concept of viscosity solution was introduced in 1983 by Crandall and P.-L. Lions in the study of first-order equations. 
During a few years, the work on viscosity solutions focused on first-order equations, because it was not known whether second-order uniformly elliptic PDEs would have a unique viscosity solution (or if the comparison principle would hold for these solutions). 
In 1988 the comparison principle for viscosity solutions was finally proved by Jensen \cite{Jensen}, and in subsequent years the concept has become prevalent in the analysis of elliptic PDEs. 

In 1994, P.-L. Lions received the Fields Medal for his contributions to nonlinear PDEs, one of his major contributions being his work on viscosity solutions \cite{ICM94}.
\end{rem}

A key result in the theory of viscosity solutions is the following (see \cite{Jensen, CC}).

\begin{thm}[Comparison principle for viscosity solutions]\index{Comparison principle!Fully nonlinear equations (viscosity)}
\label{thm.comppr}
Let $\Omega\subset\R^n$ be any bounded domain, and $F:\R^{n\times n}\to \R$ be uniformly elliptic. 
Assume that $u \in {\rm LSC}(\overline{\Omega})$ and $v \in {\rm USC}(\overline{\Omega})$ satisfy 
\begin{equation}
\label{eq:tag01}
u \ge v\quad\textrm{on}\quad\de\Omega,
\end{equation}
and
\begin{equation}
\label{eq:tag02}
F(D^2 u) \le 0 \le F(D^2 v)\quad\textrm{in}\quad\Omega\quad\textrm{in the viscosity sense}. 
\end{equation}
Then,
\[
u \ge v\quad\textrm{in}\quad\overline{\Omega}. 
\]
\end{thm}

We already proved this for $C^2$ functions $u$ in Proposition~\ref{prop.comp_princ_C2}, and the proof was very simple. 
For viscosity solutions the proof is more involved.

The main step in the proof of the comparison principle is the following.

\begin{prop} \label{prop.star}
Let $\Omega\subset\R^n$ be any bounded domain, and $F:\R^{n\times n}\to \R$ be uniformly elliptic. 
Assume that $u \in {\rm LSC}(\overline{\Omega})$ and $v \in {\rm USC}(\overline{\Omega})$ are bounded functions that satisfy \eqref{eq:tag01} and \eqref{eq:tag02}. Then,
\[
\mathcal M^-(D^2(u-v)) \leq 0\quad\textrm{in}\quad {\Omega}. 
\]
\end{prop}

We refer the reader to \cite[Theorem 5.3]{CC} for a proof of such result, where it is proved assuming that $u, v\in C(\overline{\Omega})$. 
The same proof works under the hypotheses here presented. 

The comparison principle follows using Proposition \ref{prop.star} and the next lemma. 

\begin{lem} \label{lem.star}
Let $\Omega\subset\R^n$ be any bounded domain, and assume that $w \in {\rm LSC}(\overline{\Omega})$  satisfies 
\[
w \ge 0\quad\textrm{on}\quad\de\Omega,
\]
and
\[
\mathcal M^-(D^2w) \leq 0\quad\textrm{in}\quad {\Omega}. 
\]
Then, $w\geq 0$ in $\Omega$.
\end{lem}

\begin{proof}
The proof is similar to that of Proposition \ref{max-princ-viscosity}.
Indeed, first notice that after a rescaling we may assume $\Omega\subset B_1$, and assume by contradiction that $w$ has a negative minimum in $\Omega$.
Then, since $w\geq0$ on $\partial \Omega$, we have $\min_{\overline \Omega} w=-\delta$, with  $\delta>0$, and the minimum is achieved in $\Omega$. 

Let us now consider $0<\varepsilon<\delta$, and $v(x):= -\kappa+\varepsilon(|x|^2-1)$, with $\kappa>0$ (that is, a sufficiently flat paraboloid).

\begin{figure}
\includegraphics[scale = 1.3]{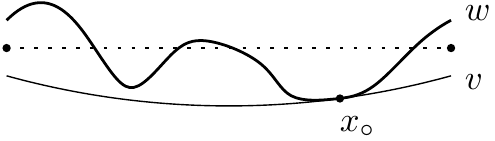}
\caption{We slide $v$ from below until it touches $w$ at a point $x_\circ$.}
\label{fig.9_413}
\end{figure}

Now, notice that $v<0$ on $\partial\Omega$, and we can choose $\kappa>0$ so that $v$ touches $w$ from below at a point inside $\Omega$.
In other words, there is $\kappa>0$ such that $w\geq v$ in $\Omega$, and $w(x_\circ)= v(x_\circ)$ for some $x_\circ\in \Omega$. (See Figure~\ref{fig.9_413}.)
Then, by definition of viscosity supersolution, we have
\[\mathcal M^-(D^2v)(x_\circ)\leq0.\]
However, a direct computation gives $\mathcal M^-(D^2v)=\mathcal M^-(2\varepsilon\textrm{Id})\equiv 2\lambda n\varepsilon>0$ in~$\Omega$, a contradiction.
\end{proof}


Once we have the comparison principle for viscosity solutions, we can use \emph{Perron's method} \index{Perron's method for viscosity solutions} to prove existence of solutions. 
We next do this, following \cite{S-viscosity}.

First let us notice that, for any bounded function $u$ in $\overline \Omega\subset \R^n$, we may define its \emph{upper semi-continuous envelope} as
\[u^*(x):= \sup\{\limsup_k u(x_k) : x_k\to x\},\]
where the supremum is taken among all sequences $\overline\Omega\ni x_k\to x$.
Notice that $u^*$ is the smallest function satisfying $u^*\in {\rm USC}(\overline\Omega)$ and $u^*\geq u$.
Similarly, we define the lower semi-continuous envelope of $u$ as 
\begin{equation}
\label{eq.lscenv}
u_*(x):= \inf\{\liminf_k u(x_k) : x_k\to x\}.
\end{equation}

We will need the following lemma, which is a generalization of the fact that the maximum of subsolutions is also a subsolution.

\begin{lem}
\label{lem.limsuplim}
Let $F:\R^{n\times n}\to \R$ be uniformly elliptic, and let $\Omega\subset \R^n$ be any bounded domain.

Let $(u_a)_{a\in \mathcal A}$ be a family of subsolutions: $u_a \in {\rm USC}(\overline{\Omega})$, and $F(D^2 u_a)\ge 0$ in $\Omega$, for all $a\in \mathcal A$. 
Let
\[u(x):= \sup_{a\in \mathcal A} u_a,\]
and let
\[u^*(x) := \sup\big\{\limsup_{k\to \infty} u(x_k) : x_k \to x\big\}.\]
Then, $u^*\in {\rm USC}(\overline \Omega)$ is a subsolution: $F(D^2u^*) \ge 0$ in $\Omega$.
\end{lem}

\begin{proof}
We divide the proof into two steps. 
\\[0.1cm]
{\it \underline{\smash{Step 1}}.} In the first part, we show that if $u^*$ has a strict local maximum at $x_\circ$, then one can extract sequences of indices $a_k\in \mathcal A$ for $k\in \N$, and of points $x_{k}\in \overline{\Omega}$, such that $x_k\to x_\circ$, $u_{a_k}$ has a local maximum at $x_{k}$, and $u_{a_k}(x_{k})\to u^*(x_\circ)$.  

By definition of $u^*(x_\circ)$, we can extract a sequence of indices $(a_j)_{j\in \N}$, $a_j\in \mathcal A$, and of points $y_j\to x_\circ$, such that $u_{a_j}(y_j)\to u^*(x_\circ)$. Now let us prove that we can extract a further subsequence $a_k:= a_{j_k}$ such that our desired conclusion holds. 

Indeed, let $r> 0$ be such that $u^*(y) < u^*(x_\circ)$ for $y\in B_r(x_\circ)\setminus\{x_\circ\}$, and let $\rho > 0$ be so small  that, if $K_\rho := B_r(x_\circ)\setminus B_\rho(x_\circ)$, then 
\[
\max_{K_\rho} u^* \le u^*(x_\circ) - \delta,
\]
for some $\delta >0$.

Now notice that, for $j$ large enough, $u_{a_j} \le  u^*(x_\circ) - \delta/2$ in $K_\rho$. Otherwise, there would be $j_m \to \infty$ and $z_m$ such that $u_{a_{j_m}}(z_m) > u^*(x_\circ) - \delta/2 \ge \max_{K_\rho} u^*+\delta/2$. Since $K_\rho$ is compact, up to a subsequence, $z_m \to z_\infty$ for some $z_\infty$ in $K_\rho$ such that
\[
u^*(z_\infty) \ge  \limsup_{m\to \infty} u_{a_{j_m}} (z_m) > \max_{K_\rho} u^*+\delta/2.
\]
A contradiction. Thus,  $u_{a_j} \le  u^*(x_\circ) - \delta/2$ in $K_\rho$ for $j$ large enough.

Let now $x_j\in \overline{B_r(x_\circ)}$ be the point where the maximum of $u_{a_j}$ in $\overline{B_r(x_\circ)}$ is attained. In particular, $u_{a_j}(x_j) \ge u_{a_j}(y_j) \to u^*(x_\circ)$, that is, $u_{a_j}(x_j) \ge u^*(x_\circ) - \delta/4$ for $j$ large enough. Since $u_{a_j} \le u^*(x_\circ)-\delta/2$ in $K_\rho$ (again, for $j$ large enough), this implies that $x_j \in B_\rho(x_\circ)$. That is, $u_{k_j}$ attains its maximum in $B_r(x_\circ)$, inside $B_\rho(x_\circ)$. By repeating this argument choosing smaller $\rho> 0$, we can extract a subsequence $a_k:= a_{j_k}$ to get the desired result. Notice that $x_j \to x_\circ$, and that by construction, $u_{a_j}(x_j) \ge u_{a_j} (y_j) \to u^*(x_\circ)$, so that $u_{a_j}(x_j) \to u^*(x_\circ)$. This completes the first part of the proof. 

Notice that so far we have not used that $u_a$ are subsolutions.
\\[0.1cm]
{\it \underline{\smash{Step 2}}.}
Let us now proceed with the second part of the proof, which proves the lemma. Let $\phi\in C^2$ be such that $\phi(x_\circ) = u^*(x_\circ)$ and $u\le \phi$ around $x_\circ$  (that is, $u-\phi$ attains its local maximum at $x_\circ$), with $x_\circ\in \Omega$. By considering $\bar \phi (x)= \phi (x)+ |x- x_\circ|^4$, we have that $u-\bar \phi$ attains a strict local maximum at $x_\circ$. We apply now the first part of the proof with $v_a := u_{a} - \bar \phi$. That is, there exist sequences of indices $(a_k)_{k\in \N}$, and points $x_k \to x_\circ$ such that $u_{a_k} - \bar \phi$ attains its local maximum at $x_k$ and $u_{a_k}(x_k) \to u^*(x_\circ)$ (since $\bar \phi$ is continuous). In particular, since $u_{a_k}$ are subsolutions in the viscosity sense,  we have
\[
F\big(D^2\bar \phi(x_k) \big) \ge 0 \quad \Longrightarrow \quad F\big(D^2\bar \phi(x_\circ) \big) = F\big(D^2\phi(x_\circ) \big) \ge 0, 
\]
by continuity of $F$ and $D^2\phi$. Thus, $u$ is a viscosity subsolution. 
\end{proof}

We can now prove the existence of viscosity solutions. 
To do so, we assume that we are given a bounded domain $\Omega\subset \R^n$ such that
\begin{equation}
\label{eq.Pprop}
\begin{split}
&\textrm{for every $x_\circ \in \de\Omega$, there exists some $\psi_+\in C^2(\overline{\Omega})$ such that}\\
&\psi_+(x_\circ) = 0,\quad \psi_+|_{\de\Omega\setminus \{x_\circ\}} > 0,\quad \textrm{and}~\mathcal{M}^+(D^2\psi_+) \le 0 \textrm{ in }\Omega,
\end{split}
\end{equation}
where we recall that $\mathcal{M}^+$ is the Pucci operator defined in \eqref{eq.Pucci} with ellipticity constants $\lambda$ and $\Lambda$. 
Notice that, if \eqref{eq.Pprop} holds, then we also have that for every $x_\circ \in \de\Omega$, there exists some $\psi_-\in C^2(\overline{\Omega})$ such that $\psi_-(x_\circ) = 0$, $\psi_-|_{\de\Omega\setminus \{x_\circ\}} < 0$, and
\[
\mathcal{M}^-(D^2\psi_-) \ge 0\textrm{ in }\Omega,
\]
where $\psi_-$ is simply given by $\psi_- = -\psi_+$. 

We will later show that any bounded $C^2$ domain satisfies \eqref{eq.Pprop}, for any constants $0<\lambda\leq \Lambda$.

\begin{rem}
\label{rem.f00}
In the following results, we will often assume that $F(0) = 0$. Otherwise, if $F(0) \neq 0$, we can consider the uniformly elliptic operator $\tilde F_t(D^2 u) := F\left(D^2(u+t|x|^2/2)\right) = F(D ^2u + t{\rm Id})$ instead. Then, $\tilde F_t(0) = F(t{\rm Id})$, and we can choose $t\in \R$ such that $F(t{\rm Id}) = 0$. Indeed, if $F(0) > 0$, by \eqref{eq.pucciext} $\tilde F_t(0) = F(t{\rm Id}) \le \mathcal{M}^+(t{\rm Id}) + F(0) = tn\lambda + F(0) < 0$ for $t <0 $ negative enough. Since $\tilde F_0(0) = F(0)  > 0$, by continuity of $\tilde F_t$ in $t$, we are done for some $t\in \big[-\frac{F(0)}{n\lambda}, 0\big)$. The case $F(0) < 0$ follows analogously. 
\end{rem}

\begin{thm}[Existence and uniqueness of viscosity solutions]\index{Existence and uniqueness!Viscosity solutions}
\label{thm.exist_sol}
Let $F:\R^{n\times n}\to \R$ be uniformly elliptic with ellipticity constants $\lambda$ and $\Lambda$, let $\Omega\subset\R^n$ be any bounded domain such that \eqref{eq.Pprop} holds, and let $g\in C(\de\Omega)$. 

Then, there exists a (unique) viscosity solution to the Dirichlet problem 
\[
\left\{
\begin{array}{rcll}
F(D^2 u) & = & 0 & \textrm{in }\Omega\\
u & = & g & \textrm{on }\de\Omega.
\end{array}
\right.
\]
\end{thm}

\begin{proof} 
The uniqueness follows directly from the comparison principle, Theorem~\ref{thm.comppr}.
Thanks to Remark~\ref{rem.f00}, we will assume $F(0) = 0$. 
The proof of existence follows by means of Perron's method, as shown next. 

Let us define the set of all subsolutions as 
\[
\mathcal{A} := \big\{v\in {\rm USC}(\overline{\Omega}) : F(D^2 v)\ge 0~~\textrm{ in }~~\Omega, ~~ v \le g ~~\textrm{ on }~~\de\Omega\big\}.
\]
Then, we can define the pointwise supremum of all subsolutions in $\mathcal{A}$,
\[
u(x) := \sup_{v\in \mathcal{A}} v(x). 
\]
Notice that since the constant function $-\|g\|_{L^\infty(\de\Omega)}$ belongs to $\mathcal{A}$, such set is non-empty.
Notice also that all elements of $\mathcal{A}$ must be below the constant $\|g\|_{L^\infty(\de\Omega)}$ by the comparison principle, and thus $u$ is bounded.

We define the upper semi-continuous envelope 
\[u^*(x) = \sup\big\{\limsup_{k\to \infty} u(x_k) : x_k \to x\big\}.\]
Notice that, by Lemma~\ref{lem.limsuplim}, we have $F(D^2 u^*)\geq0$ in $\Omega$.

The strategy of the proof is as follows.
We first prove that $u^*=g$ on~$\partial\Omega$.
This implies that $u^*\in \mathcal A$, and therefore $u^*=u$.
Then, once this is done, we will define $u_*$ as the lower semi-continuous envelope of $u$, and show that $u_*$ is a supersolution.
By the comparison principle, this will imply that $u_*\geq u$, and thus $u_*=u$.
This means that $u$ is continuous, and that it is both a subsolution and a supersolution, as wanted.

\vspace{1mm}

\noindent {\it \underline{\smash{Step 1}}.} Let us start by showing that $u^* = g$ on $\de\Omega$, and that $u^*$ is continuous on $\de\Omega$. 
Namely, we show that for every $x_\circ\in \de\Omega$, and every $x_k \to x_\circ$ with $x_k \in \Omega$, then $\liminf_{k\to \infty} u^*(x_k) = \limsup_{k\to \infty} u^*(x_k) = g(x_\circ)$.

Let $\eps > 0$, and let us define
\[
w_\eps^- := g(x_\circ) - \eps + k_\eps\psi_- = g(x_\circ) - \eps - k_\eps\psi_+,
\]
where $k_\eps > 0$ is chosen large enough (depending on $\eps$ but also on $g$ and $\Omega$) such that $w_\eps^- \le g$ on $\de\Omega$, and $\psi_- = -\psi_+$ is the function given by property \eqref{eq.Pprop} at $x_\circ$. 
Let us also define
\[
w_\eps^+ := g(x_\circ) + \eps + k_\eps\psi_+,
\]
where $k_\eps > 0$ is such that $w_\eps^+ \ge g$ on $\de\Omega$ (without loss of generality, by taking it larger if necessary, we can assume it is the same as before).

By the properties of the extremal operators \eqref{eq.Pucci}, we have $\mathcal{M}^-(D^2 w_\eps^-) = k_\eps \mathcal{M}^-(D^2 \psi_-) \ge 0$ and $\mathcal{M}^+(D^2 w_\eps^+) = k_\eps \mathcal{M}^+(D^2 \psi_+) \le 0$ in $\Omega$. 
In particular, by \eqref{eq.pucciext} (recall $F(0) = 0$), 
\[
F(D^2 w_\eps^-)\ge 0\qquad \textrm{and}  \qquad F(D^2 w_\eps^+)\le 0\qquad\textrm{in}\quad\Omega,
\]
and $w_\eps^-\in \mathcal{A}$. 
Notice that, by continuity of $\psi_-$, for each $\eps > 0$ there exists some $\delta > 0$ such that $w_\eps^- \ge g(x_\circ) - 2\eps$ in $B_\delta(x_\circ)\cap \Omega$. 
This yields, $u^*\ge w_\eps^- \ge g(x_\circ) - 2\eps$ in $B_\delta(x_\circ)\cap \Omega$, so that if $x_k \to x_\circ$, then \[\liminf_{k\to \infty} u(x_k) \ge g(x_\circ)-2\eps.\]

On the other hand, by the comparison principle, all elements in $\mathcal{A}$ are below $w_\eps^+$ for any $\eps > 0$. 
Again, by continuity of $\psi_+$, for each $\eps > 0$ there exists some $\delta > 0$ such that $w_\eps^+ \le g(x_\circ) + 2 \eps$ in $B_\delta(x_\circ)\cap \Omega$. 
This yields, $u^*\le w_\eps^+ \le g(x_\circ) + 2\eps$ in $B_\delta(x_\circ)\cap \Omega$, so that if $x_k \to x_\circ$, then \[\limsup_{k\to \infty} u(x_k) \le g(x_\circ)+2\eps.\]

Since $\eps>0$ is arbitrary, we have that if $x_k \to x_\circ$, then 
\[
\lim_{k\to \infty} u^*(x_k) = g(x_\circ).
\]
Therefore, $u^* = g$ on $\de\Omega$ and $u$ is continuous on $\de\Omega$. In particular, we have $u^*\in\mathcal A$ and (since $u^*\geq u$) $u^*\equiv u$.
This means that $u\in {\rm USC}(\overline\Omega)$ and $F(D^2 u)\geq0$ in $\Omega$.

\noindent {\it \underline{\smash{Step 2}}.}
Now, we show that $u$ is a supersolution as well. 
To do so, we consider its lower semi-continuous envelope $u_*$, \eqref{eq.lscenv}, and  prove that $F(D^2 u_*)\leq 0$ in $\Omega$.

We start by noticing that, since $u$ is continuous on the boundary (by Step 1), then $u_* = g$ on $\de\Omega$. 
Assume by contradiction that $u_*$ is not a supersolution, that is, there exists some $x_\circ \in \Omega$ such that for some $\phi\in C^2$ we have $\phi(x_\circ) = u_*(x_\circ)$, $\phi \le u_*$, but $F(D^2\phi(x_\circ)) > 0$.

By taking $\bar \phi = \phi - |x-x_\circ|^4$ if necessary, we may assume that $\phi < u_*$ if $x\neq x_\circ$, and we still have $F(D^2 \phi(x_\circ)) > 0$. 
Notice that, by continuity of $F$ and $D^2\phi$, we have $F(D^2 \phi) > 0$ in $B_\rho(x_\circ)$ for some small $\rho > 0$. 

On the other hand, consider $\phi + \delta$ for $\delta > 0$, and  define $u_\delta := \max\{u, \phi+\delta\}$.
Since $\phi(x) < u_*(x) \le u(x) $ for $x\neq x_\circ$, we have for $\delta > 0$ small enough that $\phi_\delta < u$ outside $B_\rho(x_\circ)$. 

Now, notice that  $u_\delta$ is a subsolution, since it coincides with $u$ outside $B_\rho(x_\circ)$ and it is the maximum of two subsolutions in $B_\rho(x_\circ)$. 
This means that $u_\delta\in \mathcal A$, and thus $u_\delta \leq u$.
However, this means that $\phi+\delta\leq u$ everywhere in $\Omega$, and thus $\phi+\delta\leq u_*$, a contradiction. 
Thus, $u_*$ had to be a supersolution.

But then, again by the comparison principle, since $u$ is a subsolution and $u = u_* = g$ on $\de \Omega$, we get that $u_*\ge u$ in $\Omega$, which means that $u = u_*$.

Therefore, $u$ is continuous, both a subsolution and a supersolution, and $u = g$ on~$\partial\Omega$. 
This concludes the proof.
\end{proof}

As a consequence, we find the following.

\begin{cor} 
\label{cor.fig10}\index{Dirichlet problem!Fully nonlinear equations}
Let $\Omega$ be any bounded $C^2$ domain, and $F:\R^{n\times n}\to \R$ be uniformly elliptic. Then, for any continuous $g\in C(\de\Omega)$, the Dirichlet problem 
\[
\left\{
\begin{array}{rcll}
F(D^2 u) & = & 0 & \textrm{in }\Omega\\
u & = & g & \textrm{on }\de\Omega,
\end{array}
\right.
\]
has a unique viscosity solution.
\end{cor}

\begin{proof}
The result follows from the previous theorem, we just need to check that any $C^2$ domain fulfils \eqref{eq.Pprop}.
To do so, we need to construct an appropriate barrier at every boundary point $z\in \de\Omega$. 

Notice, that in the very simple case that $\Omega$ is strictly convex, such barrier $\psi_+$ can simply be a  hyperplane with zero level set tangent to $\Omega$ at a given boundary point, such that it is positive in $\Omega$. 

In general, since $\Omega$ is a bounded $C^2$ domain, it satisfies the exterior ball condition for some uniform radius $\rho > 0$: that is, for each point $x_\circ\in \de\Omega$ there exist  some point $z_{x_\circ} = z(x_\circ)\in \Omega^c$ and a ball $B_\rho(z_{x_\circ})$ such that $B_\rho(z_{x_\circ}) \subset \Omega^c$ and $B_\rho(z_{x_\circ}) \cap \de\Omega = \{x_\circ\}$. See Figure~\ref{fig.10}.

\begin{figure}
\includegraphics[scale = 1.3]{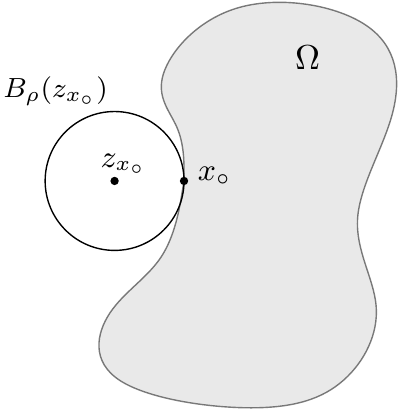}
\caption{Representation of the construction from the proof of Corollary~\ref{cor.fig10}.}
\label{fig.10}
\end{figure}

Let us construct the barrier $\psi_+$ from \eqref{eq.Pprop} for $C^2$ domains. We consider the function $\psi$ in $\R^n\setminus B_\rho$, for $\rho > 0$ given by the exterior ball condition, 
\[
\psi(x)= e^{-\alpha\rho^2}-e^{-\alpha|x|^2}, 
\]
for some $\alpha > 0$ also to be chosen.

Notice that 
\begin{align*}
e^{\alpha|x|^2} D^2\psi(x) & = -4\alpha^2 \left(\begin{matrix}
    x_1^2       & x_1 x_2 &  \dots & x_1 x_n \\
    x_2x_1       & x_2^2 &  \dots & x_2x_n \\
    \vdots & \vdots & \ddots & \vdots\\
    x_{n}x_1  &\dots     & \dots &  x_{n}^2
\end{matrix}\right) + 2\alpha{\rm Id}\\
&  = 2\alpha{\rm Id}-4\alpha^2 xx^T.
\end{align*}

Then, for $|x|\geq\rho$ we have
\[\begin{split}
e^{\alpha|x|^2}\mathcal{M}^+(D^2\psi) & \leq 2\alpha\mathcal{M}^+({\rm Id}) - 4\alpha^2\mathcal{M}^-(xx^T)
= 2\alpha n\Lambda - 4\alpha^2\lambda|x|^2 \\
&\leq 2\alpha (n\Lambda - 2\alpha\lambda \rho^2).
\end{split}\]
In particular, if we choose $\alpha \ge \frac{n\Lambda}{2\lambda\rho^2}$, we have
\[
\mathcal{M}^+(D^2\psi)  \le 0\quad \textrm{in}\quad B_\rho^c. 
\]
Therefore, translations of $\psi$ are good candidates for the function $\psi_+$ from \eqref{eq.Pprop}. 

Let now $x_\circ\in \de\Omega$ be any point on the boundary, and take $\psi_+(x) := \psi(x-z_{x_\circ})$. It is clear that $\psi_+(x_\circ) = 0$, and that $\psi_+(x) > 0$ for any $x\in \overline{\Omega}\setminus \{x_\circ\}$. On the other hand, from the discussion above we know that $\mathcal{M}^+(D^2\psi_+) \le 0$. Thus, $\Omega$ fulfills \eqref{eq.Pprop}. 
\end{proof}

\begin{rem}[Lipschitz domains]
It is actually possible to show that \eqref{eq.Pprop} holds for any bounded Lipschitz domain, too. 
In particular, this yields the existence of viscosity solutions in such class of domains.
\end{rem}

Finally, we also have the following:
\begin{prop}[Stability of viscosity solutions]\index{Stability of viscosity solutions}
\label{prop.stability_viscosity}
Let $F_k$ be a sequence of uniformly elliptic operators (with ellipticity constants $\lambda$ and $\Lambda$), and let $u_k\in C(\Omega)$ be such that $F_k(D^2 u_k) = 0$ in $\Omega$ in the viscosity sense. 

Assume that $F_k$ converges to $F$ uniformly in compact sets, and $u_k \to u$ uniformly in compact sets of $\Omega$. Then, $F(D^2 u) = 0$ in $\Omega$ in the viscosity sense.
\end{prop}
\begin{proof}
The proof uses the same ideas as the proof of Lemma~\ref{lem.limsuplim}. 

Let $x_\circ\in \Omega$ and $\phi\in C^2$ be such that $\phi(x_\circ) = u(x_\circ)$ and $\phi\le u$ in $\Omega$. By taking $\bar\phi(x) = \phi(x) + |x-x_\circ|^4$ we have that $u - \bar\phi$ attains a strict local maximum at $x_\circ$. 

Let now $v_k := u_k - \bar\phi$. Up to a subsequence, by Step 1 in the proof of Lemma~\ref{lem.limsuplim}, there exists a sequence $x_k\to x_\circ$ such that $u_k - \bar\phi$ attains a local maximum at $x_k$, and from the uniform convergence of $u_k$ to $u$, we also have $u_k(x_k) \to u(x_\circ)$. Since $u_k$ are, in particular, subsolutions in the viscosity sense for the operator $F_k$, we have that $F_k(D^2\bar\phi(x_k)) \ge 0$. Now, since $x_k\to x_\circ$, and $F_k$ converges uniformly to $F$, we get that, letting $k\to \infty$, $F(D^2\bar\phi(x_\circ)) = F(D^2\phi(x_\circ)) \ge 0$. 

In particular, $u$ is a viscosity subsolution for $F$. Doing the same for $-u$, we reach that $u$ is a viscosity solution. 
\end{proof}

\begin{rem}
We have seen that for fully nonlinear equations $F(D^2 u)=0$ we have existence, uniqueness, and stability  of viscosity solutions. The same can be done for more general equations like $F(D^2 u, x)=f(x)$, with continuous coefficients in $x$, see \cite{CC}. 
However, when we want to study linear equations in non-divergence form
\begin{equation}
\label{eq.linearbmc}
    \sum a_{ij}(x)\partial_{ij} u(x)=0 
\end{equation}
with bounded measurable coefficients, it turns out that viscosity solutions do not behave so well; see the counterexample in \cite{Nad97} (see also \cite{CCKS96}). This is the reason why, instead of defining viscosity solutions for a specific equation of the type \eqref{eq.linearbmc}, what we do is to say that $u$ solves an equation with bounded measurable coefficients (in non-divergence form) whenever it satisfies 
\[
\mathcal{M}^-(D^2 u) \le 0 \le \mathcal{M}^+(D^2 u)
\]
in viscosity sense, where $\mathcal{M}^\pm$ are the Pucci extremal operators (recall Definition~\ref{def.pucci}). As explained in Remark~\ref{rem.eqnondivbmc}, for $C^2$ functions $u$ these two inequalities are equivalent to saying that $u$ solves \eqref{eq.linearbmc} for \emph{some} coefficients $a_{ij}(x)$.
\end{rem}

\underline{\smash{Summarizing}}: For viscosity solutions we now have all we need in order to study regularity issues: 
\begin{itemize}
\item[--] Existence of solutions.
\item[--] Comparison principle. 
\item[--] Stability under uniform limits.
\end{itemize}

\section{Regularity of solutions: an overview}

In the last section we saw that for any (smooth) domain $\Omega\subset\R^n$ and any (continuous) boundary data $g$, one can find a unique viscosity solution $u\in C(\overline{\Omega})$ to the Dirichlet problem
\[
\left\{
\begin{array}{rcll}
F(D^2 u) & = & 0 & \textrm{in }~\Omega\\
u & = & g & \textrm{on }~\de\Omega.
\end{array}
\right.
\]

Now, the main question is that of regularity: 
\[
\begin{array}{l}
\textrm{If $u\in C({B_1})$ solves $F(D^2 u) = 0$ in $B_1$,}\\
\textrm{what can we say about the regularity of $u$?}
\end{array}
\]
Is the following implication true? 
\begin{equation}\label{question-fully}
\left\{
\begin{array}{c}
\textrm{$F\in C^\infty$ and}\\
\textrm{uniformly elliptic}\\
\& \\
\textrm{$F(D^2u) = 0$ in $B_1$}
\end{array}
\right.
\quad
\xRightarrow[]{\quad\textrm{\Large ?}\quad}
 \quad u \in C^\infty(B_{1/2}).
\end{equation}
This is in some sense a question analogous to Hilbert's XIXth problem.

\subsection*{Regularity for fully nonlinear equations: first results}

Assume that $u$ has some initial regularity, and that $F$ is $C^\infty$ and uniformly elliptic. 
Then, 
\[
F(D^2 u ) = 0 \quad \xrightarrow[\textrm{$\de_e$}]{} \quad \sum_{i,j = 1}^n F_{ij}(D^2 u) \de_{ij}(\de_e u) = 0,
\]
where 
\[
F_{ij} := \frac{\de F}{\de M_{ij}}
\]
is the derivative of $F(M)$ with respect to $M_{ij}$. Therefore, if we denote 
\[
a_{ij} (x) := F_{ij}(D^2 u(x)),
\]
we will then have 
\[
\textrm{$a_{ij}(x)$ is uniformly elliptic,}\quad 0<\lambda{\rm Id}\le (a_{ij}(x))_{i,j} \le\Lambda{\rm Id},
\]
thanks to the uniform ellipticity of $F$. 

Denoting 
\[
v = \de_e u,
\]
we have
\[
F(D^2 u ) = 0 \quad\Longrightarrow \quad v = \de_e u \quad\textrm{solves} \quad\sum_{i,j = 1}^n a_{ij}(x) \de_{ij}(\de_e u) = 0,
\]
where $a_{ij}(x) = F_{ij}(D^2 u(x))$. 

Now, if $u \in C^2$ (or $C^{2, \alpha}$), then the coefficients $a_{ij}(x)$ are continuous (or $C^{0,\alpha}$), and therefore we get, by Schauder-type estimates,
\[
u\in C^{2}\Rightarrow a_{ij}\in  C^0 \Rightarrow v \in C^{1,\alpha}\Rightarrow u \in C^{2,\alpha}\Rightarrow \dots\Rightarrow u\in C^\infty,
\]
where we use the bootstrap argument
\[
u\in C^{2,\alpha}\Rightarrow a_{ij} \in C^{0,\alpha}\Rightarrow  u \in C^{3,\alpha}\Rightarrow a_{ij}\in C^{1, \alpha}\Rightarrow \dots\Rightarrow u\in C^\infty.
\]

In other words, this suggests that the following result.

\begin{prop}
Let $F$ be uniformly elliptic and $C^\infty$. 
Let $u$ be any solution of $F(D^2 u) = 0$ in $B_1$, and assume that $u\in C^2$. 
Then, $u\in C^\infty$.
\end{prop}

\begin{proof}
The idea is the one presented in the lines above, but we can only use that $u\in C^2$ (in the previous argumentation, we used that $u$ is $C^3$). 
To do so, we make use of incremental quotients, as in Theorem~\ref{thm.C1aimpCinf}.

Let $u\in C^2(B_1)$, and let $h\in \R^n$ with $|h|$ small. 
Notice that $F$ is translation invariant, so 
\[
F(D^2 u(x) ) = 0,\quad F(D^2 u(x+h) ) = 0 \quad\textrm{in}\quad B_{1-|h|}. 
\]
Then, 
\[
0 = F(D^2 u(x+h) ) - F(D^2 u(x)) = \sum_{i,j = 1}^n a_{ij}(x) \de_{ij} \big(u(x+h)-u(x)\big) ,
\]
where
\[
a_{ij}(x) = \int_0^1 F_{ij}\big(t D^2 u(x+h)  + (1-t) D^2 u(x)\big)\, dt
\]
(cf. the proof of Theorem~\ref{thm.C1aimpCinf} or Theorem~\ref{thm.2D}). This is just the fundamental theorem of calculus. 
In particular, since $F$ is uniformly elliptic, $(a_{ij})_{i,j}$ is uniformly elliptic (with the same ellipticity constants). Since $u\in C^2$  and $F$ is smooth,  $a_{ij}$ are continuous. That is, $\frac{u(\cdot+h)-u}{|h|}$ solves the equation in non-divergence form
\[
\sum_{i,j= 1}^n a_{ij}(x)\de_{ij}\left(\frac{u(x+h)-u(x)}{|h|}\right) = 0\quad\textrm{in}\quad B_{1-|h|},
\]
for some continuous and uniformly elliptic coefficients $a_{ij}$. 
By the a priori estimates for equations with continuous coefficients (Proposition~\ref{prop.Schauder_estimates_cont}), we know that for any $\alpha\in(0,1)$ we have
\[
\left\|\frac{u(\cdot+h)-u}{|h|}\right\|_{C^{1, \alpha}(B_{1/2})} \le C \left\|\frac{u(\cdot+h)-u}{|h|}\right\|_{L^\infty(B_{3/4})}\le C\|u\|_{C^{0,1}(B_{3/4})},
\]
for some constant $C$ that is independent of $h$. 
By \ref{it.H7}, \eqref{eq.H7}, from Chapter~\ref{ch.0}, we reach that $u\in C^{2, \alpha}(B_{1/2})$, and by a covering argument $u\in C^{2, \alpha}$ inside $B_1$. 

Now, we proceed iteratively. 

Since  $u\in C^{2, \alpha}$ inside $B_1$, we have that $\frac{u(\cdot+h)-u}{|h|}\in C^{2, \alpha}$ inside $B_{1-|h|}$ for all $h$. Together with $F$ being smooth, this implies that $a_{ij}\in C^{0,\alpha}$ inside $B_{1-|h|}$. 
That is, now $\frac{u(\cdot+h)-u}{|h|}$ solves a non-divergence-form equation with H\"older continuous coefficients, and from Theorem~\ref{thm.Schauder_estimates} we get uniform bounds in the $C^{2,\alpha}$ norm for $\frac{u(\cdot+h)-u}{|h|}$, thus yielding that $u\in C^{3, \alpha}$ inside $B_1$. 
We can repeat this argument iteratively, using the higher order estimates from Corollary~\ref{cor.Schauder_estimates_HO}, to reach the desired result. 
\end{proof}

This is similar to what happened in Hilbert's XIXth problem: in that case we proved $C^1\Rightarrow C^\infty$.

Notice, however,  that for fully nonlinear equations, the ``gap to be filled'' (from $C^0$ to $C^2$) is ``bigger'' than in Hilbert's XIXth problem (from $H^1$ to $C^1$). 

Now, the central question to be answered is:
\[
\textrm{Is it true that solutions are always $C^2$?}
\]
 In particular, we wonder whether viscosity solutions are always classical solutions or not, and thus, whether the Dirichlet problem always admits a classical solution.

\subsection*{Regularity for fully nonlinear equations}

An important observation in the previous argument was the following:
\[
\left\{
\begin{array}{l}
\textrm{$u$ solves}\\
F(D^2 u) = 0
\end{array}
\right.
\quad
\Longrightarrow
\quad
\left\{
\begin{array}{l}
v = \de_e u \textrm{  solves  }\sum_{i,j= 1}^n a_{ij}(x) \de_{ij} v = 0,\\
\textrm{with  } a_{ij}(x) = F_{ij}(D^2 u(x) ).
\end{array}
\right. 
\]

This means that, at least formally, the derivatives of any solution to any fully nonlinear equation solve an equation with bounded measurable coefficients. 

This can be argued properly by looking at incremental quotients: 

Recall from \eqref{eq.pucciext0} the equivalence
\[
\begin{array}{c}
F \textrm{ is uniformly elliptic} \\[0.2cm]
\big\Updownarrow\\[0.2cm]
\mathcal{M}^-(D^2(u-v)) \le F(D^2 u) - F(D^2 v) \le \mathcal{M}^+(D^2(u-v)),
\end{array}
\]
where $\mathcal{M}^\pm$ are the Pucci operators. Thus,
\begin{align*}
 \mathcal{M}^-\big(D^2(u(x+h)-u(x))\big) & \le F(D^2 u(x+h)) - F(D^2 u(x)) \leq  \\&  \qquad\qquad \le \mathcal{M}^+\big(D^2(u(x+h)-u(x))\big)	.
\end{align*}
Using $F(D^2 u) = 0$ and denoting 
\[
v_h(x) = \frac{u(x+h)-u(x)}{|h|},
\]
we then reach
\[
\left\{
\begin{array}{rcl}
\mathcal{M}^+(D^2 v_h) & \ge & 0\\
\mathcal{M}^-(D^2 v_h) & \le & 0
\end{array}
\right.
\quad
\quad
\left(
\begin{array}{c}
\textrm{equation with bounded}\\
\textrm{measurable coefficients}
\end{array}
\right).
\]
The question is now: in case of divergence-form equations we proved
\[
\left\{
\begin{array}{l}
\textrm{equation with bounded}\\
\textrm{meas. coeff. }\divv(A(x) \nabla v) = 0
\end{array}
\right.
~\Longrightarrow
~
v \in C^{0,\alpha}\quad\textrm{(De Giorgi-Nash)}.
\]
Is there a similar result for equations in non-divergence form? The answer is \textsc{Yes}.

\begin{thm}[Krylov--Safonov, 1979]\index{Krylov--Safonov Theorem} \label{thm-kry-saf-bdd}
Let $0<\lambda\le \Lambda$ be the ellipticity constants, and $v\in C(B_1)$ be any solution to
\begin{equation}
\label{eq.forc2funct}
\left\{
\begin{array}{rcl}
\mathcal{M}^+(D^2 v) & \ge & 0\quad \textrm{in }~~ B_1\\
\mathcal{M}^-(D^2 v) & \le & 0\quad \textrm{in }~~ B_1,
\end{array}
\right.
\end{equation}
in the viscosity sense. Then,
\[
\|v\|_{C^{0,\alpha}(B_{1/2})}\le C\|v\|_{L^\infty(B_1)}
\]
for some small $\alpha > 0$ and $C$ depending only on $n$, $\lambda$, and $\Lambda$. 
\end{thm}

This result was proved in \cite{KS-KS} (for classical solutions); see also \cite{M19} for a more recent and simplified proof, and \cite{DS-quasi} for an extension of the result.

Recall that (see the end of Section~\ref{sec.ellipt}), for $C^2$ functions, \eqref{eq.forc2funct} is actually equivalent to $v$ solving an equation of the type $\sum_{i,j} a_{ij}(x) \de_{ij}v$ for some uniformly elliptic coefficients. This is why \eqref{eq.forc2funct} is called an {\em equation in non-divergence form with bounded measurable coefficients}. 

As a consequence of this result, we find the following. 
We assume for simplicity $F(0) = 0$, otherwise see Remark~\ref{rem.f00}. 

\begin{thm}[Krylov--Safonov, 1979]\index{Krylov--Safonov Theorem} \label{thm-kry-saf-fully}
Let $F$ be uniformly elliptic, $F(0) = 0$, and $u\in C(B_1)$ be any viscosity solution to 
\[
F(D^2 u) = 0\quad\textrm{in}\quad B_1.
\]
Then,
\[
\|u\|_{C^{1, \alpha}(B_{1/2})} \le C\|u\|_{L^\infty(B_1)}
\]
for some small $\alpha > 0$ and $C$ depending only on $n$, $\lambda$, and $\Lambda$. 
\end{thm}

\begin{proof}
By Proposition \ref{prop.star} (with $v\equiv0$), the function $u\in C(B_1)$ solves itself an equation with bounded measurable coefficients
\[
\left\{
\begin{array}{rcl}
\mathcal{M}^+(D^2 u) & \ge & 0\quad  \textrm{in }~~ B_1\\
\mathcal{M}^-(D^2 u) & \le & 0 \quad \textrm{in }~~ B_1.
\end{array}
\right.
\]
Therefore, by Theorem \ref{thm-kry-saf-bdd},  $u\in C^{0,\alpha}$ inside $B_1$. 
Now, for $\beta\in(0,1]$ take
\[
v_h(x) := \frac{u(x+h)-u(x)}{|h|^\beta},
\]
which (again by Proposition \ref{prop.star}) also solves an equation with bounded measurable coefficients,
\[
\left\{
\begin{array}{rcl}
\mathcal{M}^+(D^2 v_h) & \ge & 0\quad  \textrm{in }~~ B_{1-|h|}\\
\mathcal{M}^-(D^2 v_h) & \le & 0 \quad \textrm{in }~~ B_{1-|h|}.
\end{array}
\right.
\] 
Then, again by Theorem \ref{thm-kry-saf-bdd}, we have 
\[
\|v_h\|_{C^{0,\alpha}(B_{1/2})}\le C\|v_h\|_{L^\infty(B_{1-|h|})}\leq C\|u\|_{C^\beta(B_1)}.
\]
By \ref{it.H7}, we deduce that 
\[
\|u\|_{C^{\alpha+\beta}(B_{1/2})}\le C\|u\|_{C^\beta(B_1)},
\]
provided that $\alpha+\beta$ is not an integer and $\beta\leq 1$.

Using this estimate with $\beta=\alpha,2\alpha, ...,k\alpha$, one gets $C^{1,\alpha}$ regularity in a finite number of steps.
\end{proof}

\begin{rem} Observe that:

$\bullet~~$The $C^{0,\alpha}$ estimate for bounded measurable coefficients, Theorem \ref{thm-kry-saf-bdd},  is the best one can get in dimensions $n \ge 3$; see \cite{Saf}.

$\bullet~~$In a sense, Theorem \ref{thm-kry-saf-bdd} is the analogue of the result of De Giorgi--Nash for divergence-form equations. 
However, it is not enough to get $C^2$ regularity for solutions to fully nonlinear equations. 
\end{rem}

\noindent \underline{\smash{Summary}}: We have  $F(D^2 u) =0\Rightarrow u \in C^{1,\alpha}$ (for some small $\alpha >0$). Moreover, $u\in C^2\Rightarrow u \in C^\infty$. 
However, we have no idea (yet) if 
\[u\in C^{1,\alpha}\quad \xRightarrow[]{?} \quad u \in C^2.\]

In the two-dimensional case, as we have seen in Theorem~\ref{thm.2D} (as an a priori estimate), it turns out that one can do something better, and \emph{all} solutions are $C^{2, \alpha}$. 
This is because, in $\R^2$, solutions to equations with bounded measurable coefficients are not only $C^{0, \alpha}$, but $C^{1,\alpha}$.

As a consequence, we have the following. 


\begin{thm}\index{Equations in two variables (fully nonlinear)}
Let $F:\R^{2\times 2}\to \R$ be uniformly elliptic and smooth. 
Let $u\in C(B_1)$ be any viscosity solution to 
\[
F(D^2 u) = 0\quad\textrm{in}\quad  B_1\subset \R^2.
\]
Then $u \in C^\infty$. 
\end{thm}

This completely answers question \eqref{question-fully} in two dimensions.

In higher dimensions, a famous result established (independently) by Evans \cite{Evans-EK} and Krylov \cite{Krylov} gives the following. 

\begin{thm}[Evans--Krylov, 1982]\index{Evans--Krylov Theorem}
\label{thm.EK83}
Let $F$ be any {\bf convex} (or concave) uniformly elliptic operator, with $F(0) = 0$. Let $u\in C(B_1)$ be any viscosity solution to 
\[
F(D^2 u ) = 0\quad\textrm{in}\quad B_1.
\]

Then,
\[
\|u\|_{C^{2,\alpha}(B_{1/2})}\le C\|u\|_{L^\infty(B_1)},
\]
for some $\alpha > 0$ and $C$ depending only on $n$, $\lambda$, and $\Lambda$. In particular, if $F$ is smooth then $u\in C^\infty$.  
\end{thm}

We refer to \cite{CS-EK} for a shorter proof of such result.

\vspace{1mm}

Thus, for any solution to \eqref{eq.fully2}, with $F$ uniformly elliptic and smooth, we have: 

\begin{itemize}
\item If $u\in C^2$, then $u\in C^\infty$. 
\item $u\in C^{1,\alpha}$ always (Krylov--Safonov, 1979). 
\item In two dimensions, $u \in C^\infty$ (Nirenberg, 1952). 
\item If $F$ is convex, then $u\in C^\infty$ (Evans--Krylov, 1982)
\end{itemize}

\noindent \underline{\smash{Question}}: What happens in general? 

For decades it was an open problem to decide whether all solutions are $C^2$ or not. 
The question was finally answered by Nadirashvili and Vladuts in the 2000s \cite{NV1,NV2,NV3}:

\begin{thm}[Nadirashvili--Vladuts, 2007-2013]\index{Nadirashvili--Vladuts counterexamples}
There are solutions to \eqref{eq.fully2} that are not $C^2$. 
These counterexamples exist in dimensions $n \ge 5$. 

Moreover, for every $\tau > 0$, there exists a dimension $n$ and ellipticity constants $\lambda$ and $\Lambda$, such that there are solutions $u$ to $F(D^2 u ) = 0$ with $u\notin C^{1,\tau}$. 
\end{thm}

We refer to the monograph \cite{NTV} for more references and details. 

It is {\em not} known what happens in $\R^3$ and $\R^4$. 
This is one of the most remarkable open problems in elliptic PDEs.

\section{Further results and open problems}

As explained above, one of the main open questions regarding the problem
\begin{equation}
\label{eq.fur_reg_4}
F(D^2 u ) =  0\quad\textrm{in}\quad B_1\subset \R^n
\end{equation}
is the following:
\[
\textrm{Let $u$ be any solution to \eqref{eq.fur_reg_4} in $\R^3$ or $\R^4$. Is it true that $u\in C^2$?}
\]

We have seen that it is in general not true that solutions to fully nonlinear equations (in dimension $n\ge 5$) are $C^2$ under the assumption that $F$ is simply uniformly elliptic. 
Convexity, on the other hand, is a strong condition under which $C^2$ regularity is achieved, which, unfortunately, does not hold in some important applications. 
Even with this, it is still unclear what the optimal regularity of solutions is when $F$ is convex  and uniformly elliptic (not necessarily smooth). Theorem~\ref{thm.EK83} only gives, a priori, $C^{2, \alpha}$ regularity for some small $\alpha > 0$. 

These observations motivate, on the one hand, a more refined study for the regularity (and size of singularity) of solutions to general fully nonlinear elliptic equations, and on the other hand, the study of the optimal regularity under the convexity assumption. 

\subsection*{Partial regularity}

Recall that the ellipticity requirement for $F$ implies that $F$ is Lipschitz. 
Under the slightly more restrictive requirement that $F$ is also $C^1$, the following \emph{partial regularity} result was proved by Armstrong, Silvestre, and Smart in \cite{ASS}:

\begin{thm}[\cite{ASS}]
Let $F$ be uniformly elliptic, and assume in addition that $F\in C^1$. 
Let $u\in C^0(B_1)$ be any viscosity solution to \eqref{eq.fur_reg_4}.

Then, there exist  some $\eps > 0$ depending only on $n$, $\lambda$, $\Lambda$, and a closed subset $\Sigma\subset \overline{B_1}$ with $\dim_{\mathcal{H}} \Sigma \le n-\eps$, such that $u\in C^{2}(B_1\setminus \Sigma)$. 
\end{thm}

Here, $\dim_{\mathcal{H}}$ denotes the Hausdorff dimension of a set; see \cite{Mat}. 
Notice that if $\dim_{\mathcal{H}} \Sigma \le n-\eps$ then in particular $\Sigma$ has zero measure.

This result is the best known partial regularity result for solutions of (non-convex) fully nonlinear equations in dimensions $n\ge 3$. 
Notice that the size of the singular set is not known to be optimal (it could be much smaller!).
Moreover, it is an important open problem to decide whether the same statement holds without the regularity assumption $F\in C^1$.

\subsection*{Optimal regularity when $F$ is convex}
When $F$ is convex and uniformly elliptic, solutions to \eqref{eq.fur_reg_4}
are known to be $C^{2,\alpha}$ for some small $\alpha > 0$. If $F\in C^\infty$, a bootstrap argument then yields higher regularity for $u$, but the higher regularity of $F$ is needed. What happens if we just require $F$ to be convex and uniformly elliptic?

Since $F$ is convex, the expression \eqref{eq.fur_reg_4} can be reformulated as a supremum of linear uniformly elliptic operators as 
\[
\sup_{a\in \mathcal{A}} L_a u = 0 \quad\textrm{in}\quad B_1\subset \R^n,
\]
also known as Bellman equation (see \eqref{eq.maxPDE} in the Appendix~\ref{app.B}), where each of the operators $L_a$ is a linear uniformly elliptic operator.  

The question that remains open here is: 
\[
\textrm{What is the optimal regularity of solutions to Bellman equations?}
\]

In the simpler model of just two different operators, the previous equation is
\begin{equation}
\label{eq.maxtwo}
\max\{L_1 u, L_2 u\} = 0\quad\textrm{in}\quad B_1\subset \R^n.
\end{equation}

The best known result in this direction was proved by Caffarelli, De Silva, and Savin in 2018, and establishes the optimal regularity of solutions to \eqref{eq.maxtwo} in two dimensions:

\begin{thm}[\cite{CDS}]
Let $u$ be any viscosity solution to \eqref{eq.maxtwo} in $B_1\subset \R^2$. Then
\[
\|u\|_{C^{2,1}(B_{1/2})}\le C \|u\|_{L^\infty(B_1)},
\]
for some constant $C$ depending only on $\lambda$ and $\Lambda$. 
\end{thm}

The approach used in \cite{CDS} to show this result does not work in higher dimensions $n \ge 3$, and thus the following question remains open:
\[
\textrm{Let $u$ be any solution to \eqref{eq.maxtwo}, with $n \ge 3$. Is is true that $u\in C^{2,1}$?}
\]

%
%
%

\chapter{The obstacle problem}
\label{ch.4}


In this last chapter we focus our attention on a third type of nonlinear elliptic PDE: a \emph{free boundary problem}.
In this kind of problems we are no longer only interested in the regularity of a solution $u$, but also in the study of an a priori unknown interphase $\Gamma$ (the free boundary). 

As explained later, there is a wide variety of problems in physics, industry, biology, finance, and other areas which can be described by PDEs that exhibit {free boundaries}. 
Many of such problems can be written as variational inequalities, for which the solution is obtained by minimizing a constrained energy functional. 
And one of the most important and canonical examples is the \emph{obstacle problem}.\footnote{Other examples of important free boundary problems include the one-phase or Bernoulli problem, the thin or fractional obstacle problem, etc. We refer the interested reader to \cite{CS, PSU, Vel19, Fer21} and the references therein.}

Given a smooth function $\varphi$, the obstacle problem is the following:\index{Obstacle problem}
\begin{equation}\label{ch4-min}
\textrm{minimize}\qquad \frac12\int_\Omega|\nabla v|^2dx\qquad\textrm{among all functions}\ v\geq\varphi.
\end{equation}
Here, the minimization is subject to boundary conditions $v|_{\partial\Omega}=g$.

The interpretation of such problem is clear: One looks for the least ener\-gy function $v$, but the set of admissible functions consists only of functions that are above a certain ``obstacle'' $\varphi$.

In the two-dimensional case, one can think of the solution $v$ as a ``membrane'' which is elastic and is constrained to be above $\varphi$ (see Figure~\ref{fig.13}).

The Euler--Lagrange equation of the minimization problem is the following:\index{Euler--Lagrange for obstacle problem}
\begin{equation}\label{ch4-obst-pb}
\left\{\begin{array}{rcll}
v&\geq&\varphi & \textrm{in}\ \Omega \\
\Delta v&\leq&0 & \textrm{in}\ \Omega\\
\Delta v&=&0 & \textrm{in the set}\ \{v>\varphi\},
\end{array}\right.\end{equation}
together with the boundary conditions $v|_{\partial\Omega}=g$.

Indeed, notice that if we denote $\mathcal E(v)=\frac12\int_\Omega|\nabla v|^2dx$, then we will have
\[
\mathcal E(v+\varepsilon \eta)\geq \mathcal E(v)\quad\textrm{for every}\ {\varepsilon\geq0}\ \textrm{and}\ {\eta\geq0},\ \eta\in C^\infty_c(\Omega),\]
which yields $\Delta v\leq 0$ in $\Omega$.
That is, we can perturb $v$ with \emph{nonnegative} functions $(\varepsilon \eta)$ and we always get admissible functions $(v+\varepsilon \eta)$.
However, due to the constraint $v\geq\varphi$, we cannot perturb $v$ with negative functions in all of $\Omega$, but only in the set $\{v>\varphi\}$.
This is why we get $\Delta v\leq0$ \emph{everywhere} in $\Omega$, 
but $\Delta v=0$ \emph{only} in $\{v>\varphi\}$.
(We will show later that any minimizer $v$ of \eqref{ch4-min} is continuous, so that $\{v>\varphi\}$ is open.)

Alternatively, we may consider $u:=v-\varphi$, and the problem is equivalent to 
\begin{equation}\label{ch4-obst-pb2}
\left\{\begin{array}{rcll}
u&\geq&0 & \textrm{in}\ \Omega \\
\Delta u&\leq&f & \textrm{in}\ \Omega\\
\Delta u&=&f & \textrm{in the set}\ \{u>0\},
\end{array}\right.
\end{equation}
where $f:=-\Delta \varphi$.

\begin{figure}
\includegraphics[scale = 1]{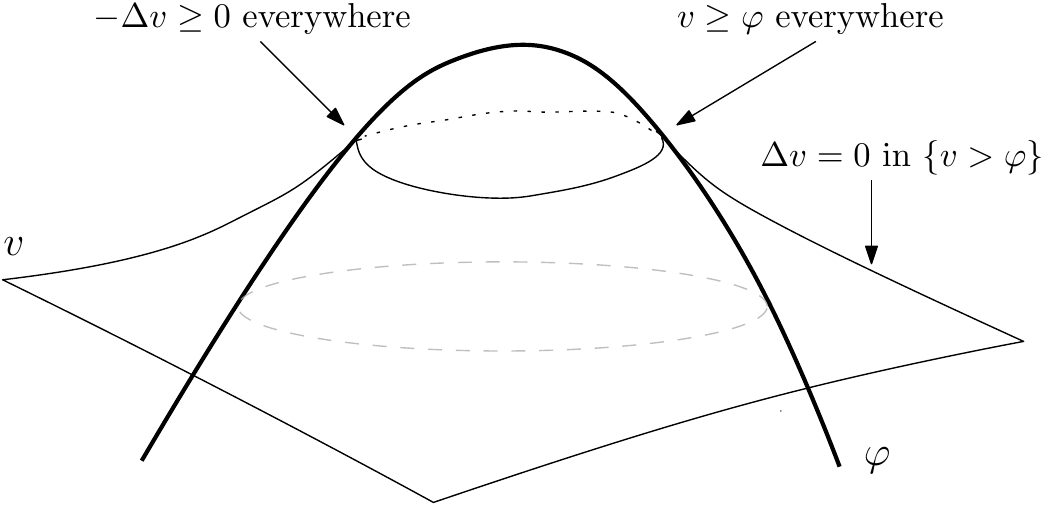}
\caption{The function $v$ minimizes the Dirichlet energy among all functions with the same boundary values situated above the obstacle.}
\label{fig.13}
\end{figure}

Such solution $u$ can be obtained as follows:
\begin{equation}\label{ch4-min2}
\qquad\textrm{minimize}
\quad \int_\Omega\left\{\frac12|\nabla u|^2+fu\right\}dx
\quad \textrm{among all functions}\ u\geq0.
\end{equation}
In other words, we can make the \emph{obstacle} just \emph{zero}, by adding a \emph{right-hand side} $f$.
Here, the minimization is subject to the boundary conditions $u|_{\partial\Omega}=\tilde g$, with $\tilde g:=g-\varphi$.

\subsection*{On the Euler--Lagrange equations}

As said above, the Euler--Lagrange equations of the minimization problem \eqref{ch4-min} are:
\begin{itemize}
\item[(i)] $v\geq\varphi$\, in $\Omega$ \ ($v$ is \emph{above} the \emph{obstacle}).

\item[(ii)] $\Delta v\leq 0$\, in $\Omega$ \ ($v$ is a \emph{supersolution}).

\item[(iii)] $\Delta v=0$\, in $\{v>\varphi\}$ \ ($v$ is \emph{harmonic} where it \emph{does not touch} the obstacle).
\end{itemize}
These are inequalities, rather than a single PDE.
Alternatively, one can write also the Euler--Lagrange equations in the following way:
\[
\min\{-\Delta v,\,v-\varphi\}=0\quad\textrm{in}\quad \Omega.
\]
(Notice that this resembles a fully nonlinear equation $\min\{L_1u,\,L_2u\}=0$, but in the present situation one of the two operators is of order zero.)

Of course, the same can be done for the equivalent problem \eqref{ch4-obst-pb2}.
In that case, moreover, the minimization problem \eqref{ch4-min2} is equivalent to
\begin{equation}\label{ch4-min2B}
\textrm{minimize}\qquad \int_\Omega\left\{\frac12|\nabla u|^2+fu^+\right\}dx,
\end{equation}
where $u^+=\max\{u,0\}$.
In this way, we can see the problem not as a constrained minimization but as a minimization problem with a \emph{non-smooth} term $u^+$ in the functional.
The Euler--Lagrange equation for this functional is then
\begin{equation}
\label{eq.ELOP}
\Delta u=f\chi_{\{u>0\}}\quad \textrm{in}\quad \Omega.
\end{equation}
(Here, $\chi_A$ denotes the characteristic function of a set $A\subset\R^n$.) We will show this in detail later.

\subsection*{The free boundary}

Let us take a closer look at the obstacle problem~\eqref{ch4-obst-pb2}.

One of the most important features of such problem is that it has \emph{two} unknowns: the \emph{solution} $u$, and the \emph{contact set} $\{u=0\}$. In other words, there are two regions in $\Omega$: one in which $u=0$; and one in which $\Delta u=f$.

These regions are characterized by the minimization problem \eqref{ch4-min2}.
Moreover, if we denote \index{Free boundary}
\[\Gamma:=\partial\{u>0\}\cap \Omega,\]
then this is called the \emph{free boundary}, see Figure~\ref{fig.14}. 

\begin{figure}
\includegraphics[scale = 0.9]{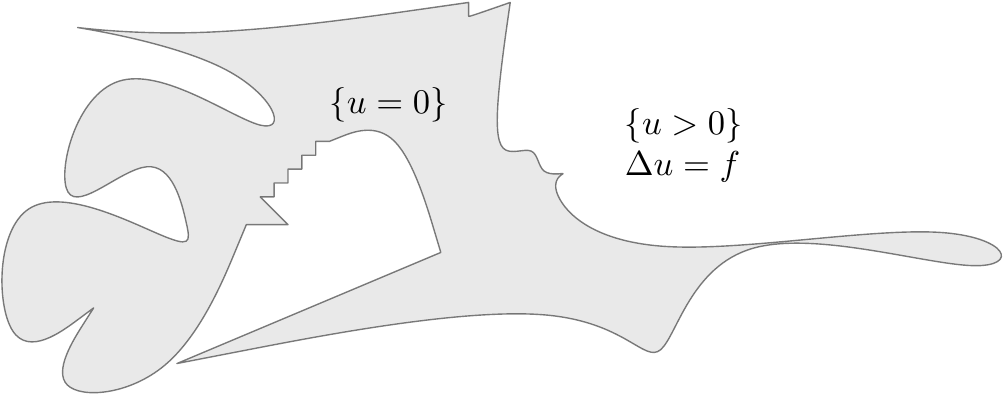}
\caption{The free boundary could, a priori, be very irregular.}
\label{fig.14}
\end{figure}

The obstacle problem is a \emph{free boundary problem}, as it involves an \emph{unknown interface} $\Gamma$ as part of the problem.

Moreover, it is not difficult to see that the fact that $u$ is a nonnegative supersolution must imply $\nabla u=0$ on $\Gamma$, that is, we will have that $u\geq0$ solves
\[
\left\{\begin{array}{rcll}
\Delta u&=&f & \textrm{in}\ \{u>0\}\\
u&=&0 & \textrm{on}\ \Gamma\\
\nabla u&=&0 & \textrm{on}\ \Gamma.
\end{array}\right.
\]
This is an alternative way to write the Euler--Lagrange equation of the problem.
In this way, the interface $\Gamma$ appears clearly, and we see that we have \emph{both} Dirichlet \emph{and} Neumann conditions on $\Gamma$.

This would usually be an over-determined problem (too many boundary conditions on $\Gamma$), but since $\Gamma$ is also free, it turns out that the problem has a unique solution (where $\Gamma$ is part of the solution, of course).

\section{Some motivations and applications}
\label{sec.appl}

Let us briefly comment on some of the main motivations and applications in the study of the obstacle problem, which are further developed in Appendix~\ref{app.C} (see also Appendix~\ref{app.B}). 
We refer to the books \cite{DL,KS, Rod87,Fri,PSU}, for more details and further applications of obstacle-type problems.

\subsection*{Fluid filtration}

The so-called Dam problem aims to describe the filtration of water inside a porous dam.
One considers a dam separating two reservoirs of water at different heights, made of a porous medium (permeable to water).
Then there is some transfer of water across the dam, and the interior of the dam has a wet part, where water flows, and a dry part. 
In this setting, an integral of the pressure (with respect to the height of the column of water at each point) solves the obstacle problem, and the free boundary corresponds precisely to the interphase separating the wet and dry parts of the dam.

\subsection*{Phase transitions}

The Stefan problem, dating back to the 19th century, is one of the most classical and important free boundary problems.
It describes the temperature of a homogeneous medium undergoing a phase change, typically a body of ice at zero degrees  submerged in water. 

In this context, it turns out that the integral of the temperature $\theta(x,t)$, namely $u(x,t):=\int_0^t\theta$, solves the parabolic version of the obstacle problem,
\[\left\{
\begin{array}{rcll}
u_t-\Delta u &=& \chi_{\{u>0\}}&\quad \textrm{in}\quad \Omega\times (0,T)\subset \R^3\times \R,\\
\partial_t u &\geq& 0,&\\
u &\geq& 0.&
\end{array}
\right.\]
The moving interphase separating the solid and liquid is exactly the free boundary $\partial\{u>0\}$.

\subsection*{Hele-Shaw flow}

This two-dimensional model, dating back to 1898, describes a fluid flow between two flat parallel plates separated by a very thin gap.
Various problems in fluid mechanics can be approximated to Hele-Shaw flows, and that is why understanding these flows is important.

A Hele-Shaw cell is an experimental device in which a viscous fluid is sandwiched in a narrow gap between two parallel plates. 
In certain regions, the gap is filled with fluid while in others the gap is filled with air. 
When liquid is injected inside the device through some  sinks (e.g. through a small hole on the top plate) the region filled with liquid grows.  
In this context, an integral of the pressure solves, for each fixed time~$t$, the obstacle problem.
In a similar way to the Dam problem, the free boundary corresponds to the interface between the fluid and the air regions.

\subsection*{Optimal stopping, finance}

In probability and finance, the obstacle problem appears when considering optimal stopping problems for stochastic processes.

Indeed, consider a random walk (Brownian motion) inside a domain $\Omega\subset\R^n$, and a payoff function $\varphi$ defined on the same domain. 
We can stop the random walk at any moment, and we get the payoff at that position. 
We want to maximize the expected payoff (by choosing appropriately the stopping strategy). 
Then, it turns out that the highest expected payoff $v(x)$ starting at a given position $x$ satisfies the obstacle problem \eqref{ch4-obst-pb}, where the contact set $\{v=\varphi\}$ is the region where we should immediately stop the random walk and get the payoff, while $\{v>\varphi\}$ is the region where we should wait (see Appendix~\ref{app.B} for more details).

\subsection*{Interacting particle systems}

Large systems of interacting particles arise in physical, biological, or material sciences.

In some   models, the particles attract each other when they are far, and experience a repulsive force when they are close. 
In other related models in statistical mechanics, the particles (e.g. electrons) repel with a Coulomb force and one wants to understand their behavior in presence of some external field that confines them. 

In this kind of models, a natural and interesting question is to determine the ``equilibrium configurations''.
For instance, in Coulomb systems the charges accumulate in some region with a well defined boundary.
Interestingly, these problems are equivalent to the obstacle problem --- namely, the electric potential $u=u(x)$ generated by the charges solves such problem --- and the contact set $\{u=0\}$ corresponds to the region in which the particles concentrate.

\subsection*{Quasi-Steady Electrochemical Shaping}

Consider a metal inside an electrolyte under the action of an electric potential, in such a way that the metal shrinks with time due to a chemical reaction. 
Then, the integral (in time) of the potential satisfies, for each fixed time, the obstacle problem, whose free boundary corresponds to the shape of the metal at that moment.

\subsection*{Heat control}

Trying to automatically control the temperature of a room using only heating devices, under suitable conditions, also yields the obstacle problem (in this case, for the temperature). Here, the free boundary separates the region where the heating devices are active and where they are not.

\subsection*{Elasticity}

Finally, in elasticity theory we probably find the most visual representation of the obstacle problem. 
Given a thin membrane that is affected only by tension forces (thus tries to minimize area), it \emph{approximately} satisfies the obstacle problem, where the contact region is the area where the membrane touches the obstacle.

\section{Basic properties of solutions I}

We proceed now to study the basic properties of solutions to the obstacle problem: existence of solutions, optimal regularity, and nondegeneracy.

We will first study all these properties for minimizers $v\geq\varphi$ of \eqref{ch4-min}, and then in the next section we will study independently minimizers $u\geq0$ of \eqref{ch4-min2} or \eqref{ch4-min2B}.

This is not only for completeness and clarity of presentation, but also to have both points of view.
For instance, the proof of the optimal regularity of solutions can be done in two completely different ways, one for each of the settings.

\subsection*{Existence of solutions}

Existence and uniqueness of solutions follows easily from the fact that the functional $\int_\Omega |\nabla v|^2dx$ is convex, and that we want to minimize it in the closed convex set $\{v\in H^1(\Omega): v\geq\varphi\}$.

Recall that $w|_{\de\Omega}$ denotes the trace of $w$ on $\de\Omega$ whenever it is defined. 

\begin{prop}[Existence and uniqueness]\label{obstacle-existence}\index{Existence and uniqueness!Obstacle problem}
Let $\Omega\subset\R^n$ be any bounded Lipschitz domain, and let $g:\partial\Omega\to\R$ and $\varphi\in H^1(\Omega)$ be such that
\[\mathcal C=\bigl\{w\in H^1(\Omega): w\geq\varphi\ \textrm{in}\ \Omega,\ w|_{\partial\Omega}=g\bigr\}\neq \varnothing.\]
Then, there exists a unique minimizer of $\int_\Omega|\nabla v|^2dx$ among all functions $v\in H^1(\Omega)$ satisfying $v\geq\varphi$ in $\Omega$ and $v|_{\partial\Omega}=g$.
\end{prop}

\begin{proof}
The proof is quite similar to that of Theorem \ref{ch0-existence}.
Indeed, let
\[\theta_\circ :=\inf\left\{\frac12\int_\Omega |\nabla w|^2dx\,:\, w\in H^1(\Omega),\ w|_{\partial\Omega}=g,\ w\geq \varphi\ \textrm{in}\ \Omega\right\},\]
that is, the infimum value of $\mathcal E(w)=\frac12\int_\Omega |\nabla w|^2dx$ among all admissible functions~$w$.

Let us take a sequence of functions $\{v_k\}$ such that
\begin{itemize}
\item $v_k\in H^1(\Omega)$

\item $v_k|_{\partial\Omega}=g$ and $v_k\geq \varphi$ in $\Omega$.

\item $\mathcal E(v_k)\to \theta_\circ $ as $k\to\infty$.
\end{itemize} 
By the Poincar\'e inequality (Theorem \ref{ch0-Poinc}), the sequence $\{v_k\}$ is uniformly bounded in $H^1(\Omega)$, and therefore a subsequence $\{v_{k_j}\}$ will converge to a certain function $v$ strongly in $L^2(\Omega)$ and weakly in $H^1(\Omega)$.
Moreover, by compactness of the trace operator (see \ref{it.S5} in Chapter \ref{ch.0}), we will have $v_{k_j}|_{\partial\Omega}\to v|_{\partial\Omega}$ in $L^2(\partial\Omega)$, so that $v|_{\partial\Omega}=g$.
Furthermore, such function $v$ will satisfy $\mathcal E(v)\leq \liminf_{j\to\infty}\mathcal E(v_{k_j})$ (by \eqref{ch0-weak-conv2}-\eqref{ch0-weak-conv3} from \ref{it.S4} in Chapter~\ref{ch.0}), and therefore it will be a minimizer of the energy functional.
Since $v_{k_j}\geq \varphi$ in $\Omega$ and $v_{k_j}\to v$ in $L^2(\Omega)$, we have $v\geq\varphi$ in $\Omega$.
Thus, we have proved the existence of a minimizer~$v$.

The uniqueness of the minimizer follows from the strict convexity of the functional $\mathcal E(v)$, exactly as in Theorem \ref{ch0-existence}.
\end{proof}

As in the case of harmonic functions, it is easy to show that if a function $v$ satisfies 
\[\left\{\begin{array}{rcll}
v&\geq&\varphi & \textrm{in}\ \Omega \\
\Delta v&\leq&0 & \textrm{in}\ \Omega\\
\Delta v&=&0 & \textrm{in the set}\ \{v>\varphi\},
\end{array}\right.\]
then it must actually be the minimizer of the functional.

There are two alternative ways to construct the solution to the obstacle problem: as the ``least supersolution above the obstacle'', or with a ``penalized problem''.
Let us briefly describe them.

\vspace{2mm}

\noindent$\bullet$ \underline{\smash{Least supersolution}}: This is related to the existence of viscosity solutions described in Chapter \ref{ch.3}. 
Indeed, we consider 
\[v(x):=\inf\biggl\{w(x):w\in C(\overline\Omega),\ -\Delta w\geq0\ \textrm{in}\ \Omega,\ w\geq\varphi\ \textrm{in}\ \Omega,\ w|_{\partial\Omega}\geq g\biggr\}.\]
Here, the inequality $-\Delta w\geq0$ in $\Omega$ has to be understood in the viscosity sense.

Then, as in Perron's method (recall Chapters \ref{ch.0} and \ref{ch.3}), it turns out that $v$ is itself a continuous supersolution, it satisfies $\Delta v=0$ in $\{v>\varphi\}$, and thus it solves the obstacle problem.
Therefore, 
\[\left\{\begin{array}{c} \textrm{least} \\ \textrm{supersolution}\end{array}\right\}\longleftrightarrow \left\{\begin{array}{c} \textrm{minimizer of} \\ \textrm{the functional}\end{array}\right\}.\]

\vspace{2mm}

\noindent$\bullet$ \underline{\smash{Penalized problem}}: We consider $\beta_\varepsilon:\R\to\R$ smooth and convex, converging to 
\[\beta_0(t):=\left\{\begin{array}{lll} 0 & \textrm{if} & t\geq0 \\ \infty & \textrm{if} & t<0.\end{array}\right.\]
We may take for example $\beta_\varepsilon(t):=e^{-t/\varepsilon}$, see Figure~\ref{fig.15}. 

\begin{figure}
\includegraphics{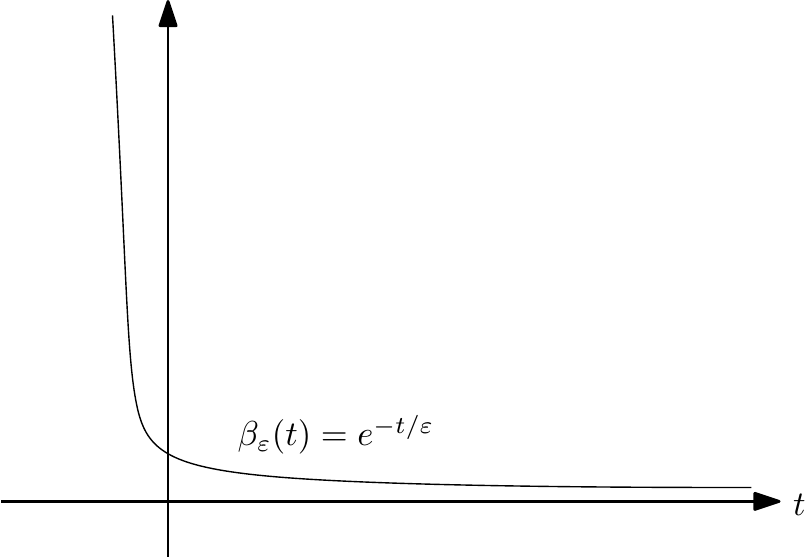}
\caption{The function $\beta_\eps \to \beta_0$ as $\eps \downarrow 0$.}
\label{fig.15}
\end{figure}

Then, we minimize the functional 
\[J_\varepsilon(v):=\frac12\int_\Omega |\nabla v|^2dx+\int_\Omega \beta_\varepsilon(v-\varphi)dx,\]
subject to the appropriate boundary conditions on $\partial\Omega$, and get a solution $v_\varepsilon\in C^\infty$ of $\Delta v_\varepsilon=\beta_\varepsilon'(v_\varepsilon-\varphi)$ in $\Omega$.

Since $\beta_\varepsilon'\leq 0$ everywhere, and $\beta_\varepsilon'(t)=0$ for $t\geq0$, we have
\[\left\{\begin{array}{rcll}-\Delta v_\varepsilon &\geq&0 & \quad \textrm{everywhere in}\ \Omega \\ \Delta v_\varepsilon &=&0 &\quad \textrm{in}\ \{v_\varepsilon>\varphi\}.\end{array}\right.\]
As $\varepsilon\to0$, we have $v_\varepsilon\to v$, where $v$ is the solution to the obstacle problem.
We refer to \cite{PSU} for more details.

\subsection*{Basic properties of solutions}

Let us next prove that any minimizer $v$ of \eqref{ch4-min} is actually continuous and solves \eqref{ch4-obst-pb}.

From now on we will ``forget'' about the regularity of the obstacle, and assume that it is as smooth as needed. This is why we will always be dealing with obstacles $\varphi\in C^\infty(\Omega)$. One gets analogous results under much weaker regularity assumptions on $\varphi$, which depend on the type of result to be proved. The role of the regularity of the obstacle is beyond the scope of this book, and thus we will always assume $\varphi$ to be smooth.

We start with the following lemma.

\begin{lem}\label{ch4-supersol}
Let $\Omega\subset\R^n$ be any bounded Lipschitz domain, $\varphi\in C^\infty(\Omega)$, and $v\in H^1(\Omega)$ be any minimizer of \eqref{ch4-min} subject to the boundary conditions $v|_{\partial\Omega}=g$.

Then, $-\Delta v\geq0$ in $\Omega$.
\end{lem}

\begin{proof}
Let 
\[\mathcal E(v)=\frac12\int_\Omega |\nabla v|^2dx.\]
Then, since $v$ minimizes $\mathcal{E}$ among all functions above the obstacle $\varphi$ (and with fixed boundary conditions on $\partial\Omega$), we have that
\[\mathcal E(v+\varepsilon \eta)\geq \mathcal E(v)\quad\textrm{for every}\ {\varepsilon\geq0}\ \textrm{and}\ {\eta\geq0},\ \eta\in C^\infty_c(\Omega).\]
This yields
\[\varepsilon\int_\Omega \nabla v\cdot\nabla \eta+\frac{\varepsilon^2}{2}\int_\Omega |\nabla \eta|^2dx\geq 0\quad\textrm{for every}\ {\varepsilon\geq0}\ \textrm{and}\ {\eta\geq0},\ \eta\in C^\infty_c(\Omega),\]
and thus
\[\int_\Omega \nabla v\cdot\nabla \eta\geq 0\quad\textrm{for every}\ {\eta\geq0},\ \eta\in C^\infty_c(\Omega).\]
This means that $-\Delta v\geq0$ in $\Omega$ in the weak sense, as desired.
\end{proof}

From here, by showing first that $\{v>\varphi\}$ is open, we obtain the Euler--Lagrange equations for the functional:

\begin{prop}\label{ch4-Euler-Lagrange}
Let $\Omega\subset\R^n$ be any bounded Lipschitz domain, $\varphi\in C^\infty(\Omega)$, and $v\in H^1(\Omega)$ be any minimizer of \eqref{ch4-min} subject to the boundary conditions $v|_{\partial\Omega}=g$.

Then, $v\in C(\Omega)$ and it satisfies 
\begin{equation}\label{Euler-Lagrange}
\left\{\begin{array}{rcll}
v&\geq&\varphi & \textrm{in}\ \Omega \\
\Delta v&\leq&0 & \textrm{in}\ \Omega\\
\Delta v&=&0 & \textrm{in}\ \{v>\varphi\} \cap\Omega.
\end{array}\right.
\end{equation}
\end{prop}

\begin{proof}
By construction, we already know that $v\ge \varphi$ in $\Omega$ and, thanks to Lemma~\ref{ch4-supersol}, $-\Delta v\ge 0$ in $\Omega$, i.e, $v$ is (weakly) superharmonic. Up to replacing $v$ in a set of measure zero, we may also assume that $v$ is lower semi-continuous (by Lemma~\ref{lem.lower_semi}). Thus, we only need to prove that $\Delta v=0$ in $\{v>\varphi\}\cap\Omega$ and that $v$ is, in fact, continuous.

In order to do that, let us show first that $\{v > \varphi\}\cap\Omega$ is open. Let $x_\circ\in \{v > \varphi\}\cap\Omega$ be such that $v(x_\circ) - \varphi(x_\circ) > \eps_\circ > 0$. Since $v$ is lower semi-continuous and $\varphi$ is continuous, there exists some $\delta > 0$ such that $v(x) - \varphi(x) > \eps_\circ/2$ for all $x\in B_\delta(x_\circ)$, and hence $B_\delta(x_\circ)\subset \{v > \varphi\}$. Since $x_\circ$ was arbitrary, this means that $\{v > \varphi\}$ is open. This implies, also, that $\Delta v = 0$ weakly in $\{v > \varphi\}\cap\Omega$. Indeed, for any $x_\circ\in \{v>\varphi\}$ and $\eta\in C^\infty_c(B_{\delta}(x_\circ))$ with $|\eta|\le 1$, we have $v\pm \eps\eta\geq \varphi$ in $\Omega$ for all $|\eps| < \eps_\circ/2$, and therefore it is an admissible competitor to the minimization problem.
Thus, we have $\mathcal E(v+\varepsilon \eta)\geq \mathcal E(v)$ for all $|\varepsilon|<\varepsilon_\circ $, and differentiating in $\varepsilon$ we deduce that $v$ is harmonic in $\{v>\varphi\}\cap\Omega$.

Finally, let us show that $v$ is continuous. We already know, by the regularity of harmonic functions (e.g. Corollary~\ref{ch0-smooth}), that $v$ is continuous in $\{v > \varphi\}\cap\Omega$. Let us now show that $v$ is continuous in $\{v = \varphi\}\cap \Omega$ as well.

Let $y_\circ\in \{v = \varphi\}\cap \Omega$, and let us argue by contradiction. That is, since $v$ is lower semi-continuous, let us assume that there is a sequence $y_k \to y_\circ$ such that $v(y_k) \to v(y_\circ) +\eps_\circ = \varphi(y_\circ) + \eps_\circ$ for some $\eps_\circ>0$. Since $\varphi$ is continuous, we may assume also that $y_k\in\{v > \varphi\}$. Let us denote by $z_k$ the projection of $y_k$ towards $\{v = \varphi\}$, so that $\delta_k := |z_k - y_\circ| \le 2|y_k - y_\circ|\downarrow 0$ and $v(z_k) \to v(y_\circ) = \varphi(y_\circ)$. Now, since $v$ is superharmonic by \eqref{eq.superharmonic_integral}, 
\[
v(z_k) \ge \ave_{B_{2\delta_k}(y_k)} v =  (1-2^{-n})\ave_{B_{2\delta_k}(y_k)\setminus B_{\delta_k}(y_k)} v + 2^{-n}\ave_{B_{\delta_k}(y_k)} v =I_1 + I_2. 
\]
Observe that, for the first term, since $v$ is lower semi-continuous and $\delta_k \downarrow 0$, we can assume that, for $k$ large enough, $v \ge \varphi(y_\circ) - 2^{-n} \eps_\circ$ in $B_{2\delta_k}$, so that $I_1 \ge (1-2^{-n}) [\varphi(y_\circ) - 2^{-n}\eps_\circ]$. On the other hand, since $v$ is harmonic in $B_{\delta_k}(y_k)$, we have by the mean-value property that $I_2 = 2^{-n} v(y_k)$. Combining everything, we get 
\[
v(z_k) \ge (1-2^{-n})[\varphi(y_\circ)-2^{-n}\eps_\circ] + 2^{-n} v(y_k) \to \varphi(y_\circ) +2^{-2n}\eps_\circ 
\] 
which contradicts the fact that we had $v(z_k) \to v(y_\circ) = \varphi(y_\circ)$. Hence, $v$ is continuous in $\Omega$.
\end{proof}

We next prove the following result, which says that $v$ can be characterized as the {least supersolution} above the obstacle.

\begin{prop}[Least supersolution]\label{ch4-least-supersol}\index{Least supersolution}
Let $\Omega\subset\R^n$ be any bounded Lipschitz domain, $\varphi\in H^1(\Omega)$, and $v\in H^1(\Omega)$ be any minimizer of \eqref{ch4-min} subject to the boundary conditions $v|_{\partial\Omega}=g$.

Then, for any function $w$ satisfying $-\Delta w\geq0$ in $\Omega$, $w\geq\varphi$ in $\Omega$, and $w|_{\partial\Omega}\geq v|_{\partial\Omega}$, we have $w\geq v$ in $\Omega$.
In other words, if $w$ is any supersolution above the obstacle $\varphi$, then $w\geq v$.
\end{prop}

\begin{proof}
If $w$ is any function satisfying $-\Delta w\geq0$ in $\Omega$, $w\geq\varphi$ in $\Omega$, and $w|_{\partial\Omega}\geq v|_{\partial\Omega}$, it simply follows from the maximum principle (Proposition~\ref{max-princ-weak}) that $w\geq v$.
Indeed, we have $-\Delta w\geq -\Delta v$ in $\Omega\cap \{v>\varphi\}$, and on the boundary of such set we have $w|_{\partial\Omega}\geq v|_{\partial\Omega}$ and $w\geq\varphi=v$ on $\{v=\varphi\}$.
\end{proof}

\subsection*{Optimal regularity of solutions}

Thanks to Proposition \ref{ch4-Euler-Lagrange}, we know that any minimizer of \eqref{ch4-min} is continuous and solves \eqref{Euler-Lagrange}.
From now on, we will actually localize the problem and study it in a ball:
\begin{equation}\label{ch4-obst-B1}
\left\{\begin{array}{rcll}
v&\geq&\varphi & \textrm{in}\ B_1 \\
\Delta v&\leq&0 & \textrm{in}\ B_1\\
\Delta v&=&0 & \textrm{in}\ \{v>\varphi\}\cap B_1.
\end{array}
\right.
\end{equation}

Our next goal is to answer the following question:
\[\textrm{\underline{Question}: }\textit{ What is the optimal regularity of solutions?}\]

\begin{figure}
\includegraphics{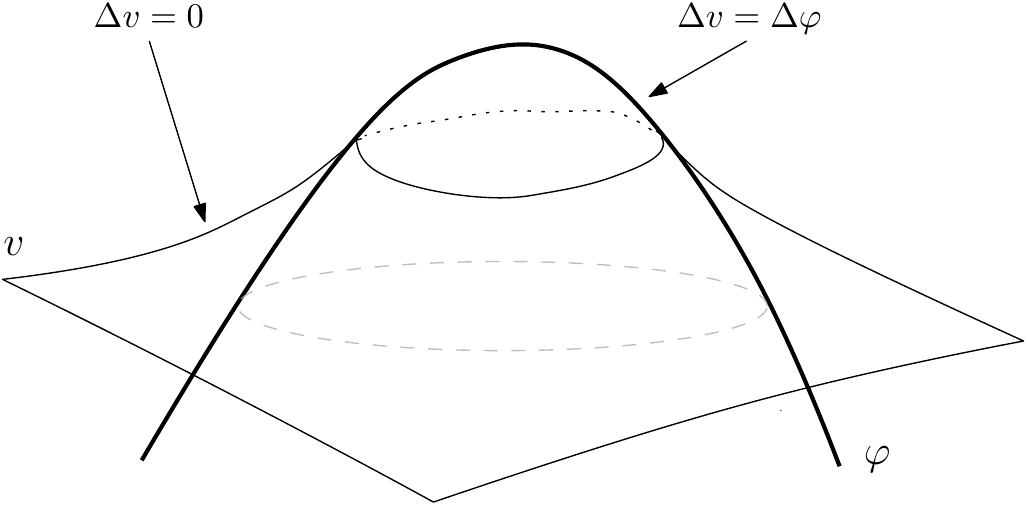}
\caption{Second derivatives are in general discontinuous across the free boundary.}
\label{fig.16}
\end{figure}

First, a few important considerations.
Notice that in the set $\{v>\varphi\}$ we have $\Delta v=0$, while in the interior of $\{v=\varphi\}$ we have $\Delta v=\Delta \varphi$ (since $v=\varphi$ there); see Figure~\ref{fig.16}.

Thus, since $\Delta \varphi$ is in general not zero, $\Delta v$ is \emph{discontinuous} across the free boundary $\partial\{v>\varphi\}$ in general.
In particular, $v\notin C^2$.

We will now prove that any minimizer of \eqref{ch4-min} is actually $C^{1,1}$, which gives the:
\[\textrm{\underline{Answer}: }\ v\in C^{1,1}\ \textit{ (second derivatives are bounded but not continuous)}\]

The precise statement and proof are given next.

\begin{thm}[Optimal regularity]\label{ch4-optimal-reg}\index{Optimal regularity, Obstacle problem}
Let $\varphi\in C^\infty(B_1)$, and $v$ be any solution to \eqref{ch4-obst-B1}.
Then, $v$ is $C^{1,1}$ in $B_{1/2}$, with the estimate
\[\|v\|_{C^{1,1}(B_{1/2})}\leq C\bigl(\|v\|_{L^\infty(B_{1})}+\|\varphi\|_{C^{1,1}(B_{1})}\bigr).\]
The constant $C$ depends only on $n$.
\end{thm}

To prove this, the main step is the following.

\begin{lem}\label{ch4-optimal-reg-lem}
Let $\varphi\in C^\infty(B_1)$, and $v$ be any solution to \eqref{ch4-obst-B1}.
Let $x_\circ \in \overline{B_{1/2}}$ be any point on $\{v=\varphi\}$.

Then, for any $r\in (0,\frac14)$ we have
\[0\leq \sup_{B_r(x_\circ )}(v-\varphi)\leq Cr^2,\]
with $C$ depending only on $n$ and $\|\varphi\|_{C^{1,1}(B_{1})}$.
\end{lem}

\begin{proof}
After dividing $v$ by a constant if necessary, we may assume that $\|\varphi\|_{C^{1,1}(B_{1})}\leq 1$.

Let $\ell(x):=\varphi(x_\circ )+\nabla \varphi(x_\circ )\cdot(x-x_\circ )$ be the linear part of $\varphi$ at $x_\circ $.
Let $r\in (0,\frac14)$.
Then, by $C^{1,1}$ regularity of $\varphi$, in $B_r(x_\circ )$ we have 
\[\ell(x)-r^2\leq \varphi(x)\leq v(x).\]
We want to show that, in the ball $B_r(x_\circ)$ (see Figure~\ref{fig.17}), we have
\[v(x)\leq \ell(x)+Cr^2.\]

For this, consider 
\[w(x):=v(x)-\bigl[\ell(x)-r^2\bigr].\]
This function $w$ satisfies $w\geq0$ in $B_r(x_\circ )$, and $-\Delta w=-\Delta v\geq0$ in $B_r(x_\circ )$.

Let us split $w$ into
\[w=w_1+w_2,\]
with 
\[
\left\{\begin{array}{rcll}
\Delta w_1 &=&0 &\ \textrm{in}\ B_r(x_\circ )\\
 w_1&=& w &\ \textrm{on}\ \partial B_r(x_\circ )
\end{array}\right.
\quad \textrm{and}\quad 
\left\{\begin{array}{rcll}
-\Delta w_2 &\geq&0 &\ \textrm{in}\ B_r(x_\circ )\\
 w_2&=& 0 &\ \textrm{on}\ \partial B_r(x_\circ ).
\end{array}\right.
\]

\begin{figure}
\includegraphics{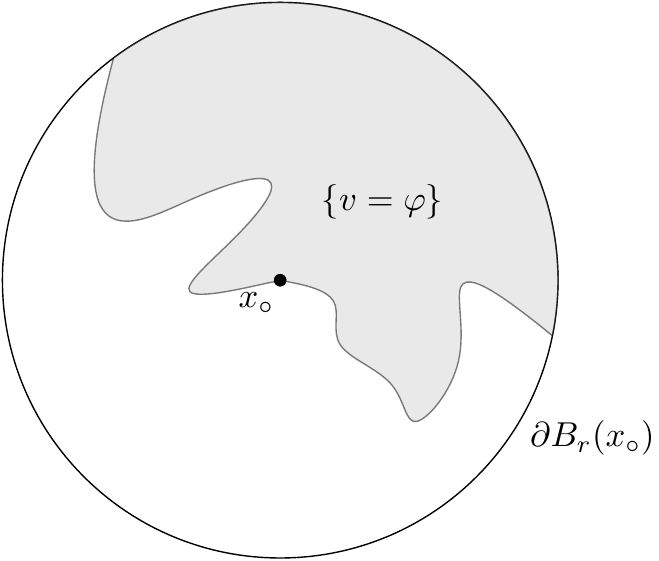}
\caption{The solution $v$ and a free boundary point $x_\circ $}
\label{fig.17}
\end{figure}

Notice that 
\[0\leq w_1\leq w \quad \textrm{and}\quad  0\leq w_2\leq w.\]
We have that 
\[w_1(x_\circ )\leq w(x_\circ )= v(x_\circ )-\bigl[\ell(x_\circ )-r^2\bigr]=r^2,\]
and thus by the Harnack inequality
\[\|w_1\|_{L^{\infty}(B_{r/2}(x_\circ ))}\leq Cr^2.\]

For $w_2$, notice that $\Delta w_2=\Delta v$, and in particular $\Delta w_2=0$ in $\{v>\varphi\}$.
This means that $w_2$ attains its maximum on $\{v=\varphi\}$.
But in the set $\{v=\varphi\}$ we have
\[w_2\leq w=\varphi-\bigl[\ell-r^2\bigr]\leq Cr^2,\]
and therefore we deduce that 
\[\|w_2\|_{L^\infty(B_r(x_\circ ))}\leq Cr^2.\]

Combining the bounds for $w_1$ and $w_2$, we get $\|w\|_{L^\infty(B_r(x_\circ ))}\leq Cr^2$.
Translating this into $v$, and using that $\|\varphi\|_{C^{1,1}(B_1)}\leq1$, we find $v-\varphi\leq Cr^2$ in $B_{r/2}(x_\circ )$.
\end{proof}

Therefore, we have proved that:
\[\textit{At every free boundary point }x_\circ ,\ v\textit{ separates from }\varphi\textit{ at most quadratically.}\]
As shown next, this easily implies the $C^{1,1}$ regularity.

\begin{proof}[Proof of Theorem \ref{ch4-optimal-reg}]
Dividing $v$ by a constant if necessary, we may assume that $\|v\|_{L^\infty(B_{1})}+\|\varphi\|_{C^{1,1}(B_{1})}\leq 1$.
 
We already know that $v\in C^\infty$ in the set $\{v>\varphi\}$ (since $v$ is harmonic), and also in the interior of the set $\{v=\varphi\}$ (since $\varphi\in C^\infty$). 
Moreover, on the interface $\Gamma=\partial\{v>\varphi\}$ we have proved the quadratic growth $\sup_{B_r(x_\circ )}(v-\varphi)\leq Cr^2$.
Let us prove that this yields the $C^{1,1}$ bound we want.

Let $x_1\in \{v>\varphi\}\cap B_{1/2}$, and let $x_\circ \in \Gamma$ be the closest free boundary point.
Denote $\rho=|x_1-x_\circ |$.
Then, we have $\Delta v=0$ in $B_\rho(x_1)$ (see the setting in Figure~\ref{fig.18}), and thus we have also $\Delta(v-\ell)=0$ in $B_\rho(x_1)$, where $\ell$ is the linear part of $\varphi$ at $x_\circ $. 
\begin{figure}
\includegraphics[scale = 1.2]{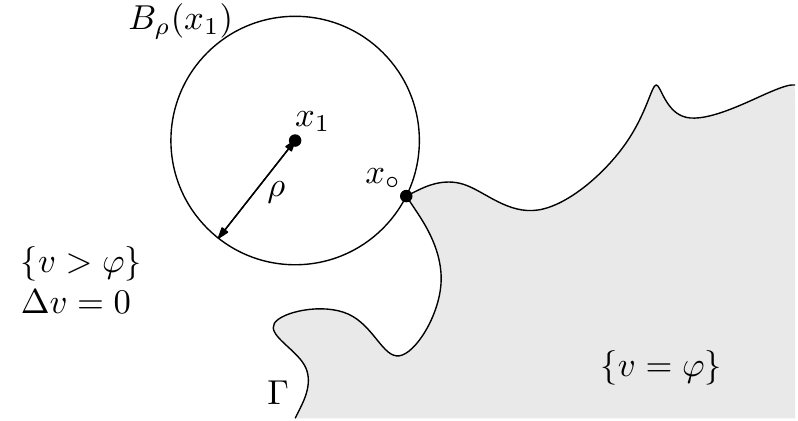}
\caption{A solution $v$ satisfying $\Delta v = 0$ in $B_\rho(x_1)\subset \{v > \varphi\}$.}
\label{fig.18}
\end{figure}

By estimates for harmonic functions, we find
\[\|D^2v\|_{L^\infty(B_{\rho/2}(x_1))}= \|D^2(v-\ell)\|_{L^\infty(B_{\rho/2}(x_1))} \leq \frac{C}{\rho^2}\|v-\ell\|_{L^\infty(B_{\rho}(x_1))}.\]
But by the growth proved in the previous Lemma, we have $\|v-\ell\|_{L^\infty(B_{\rho}(x_1))}\leq C\rho^2$, which yields 
\[\|D^2v\|_{L^\infty(B_{\rho/2}(x_1))}\leq \frac{C}{\rho^2}\,\rho^2=C.\]
In particular, $|D^2v(x_1)|\leq C$. We can do this for all $x_1\in \-\{v>\varphi\}\cap B_{1/2}$, and on $\partial\{v>\varphi\}$ we have quadratic growth by Lemma~\ref{ch4-optimal-reg-lem}, hence it follows that $\|v\|_{C^{1,1}(B_{1/2})}\leq C$, as wanted.
\end{proof}

The overall strategy of the proof of optimal regularity is summarized in Figure \ref{fig.19}.

\subsection*{Nondegeneracy}

We now want to prove that, at all free boundary points, $v$ separates from $\varphi$ \emph{at least} quadratically (we already know \emph{at most} quadratically).

That is, we want
\begin{equation} \label{ch4-nondeg}
0<cr^2\leq \sup_{B_r(x_\circ )}(v-\varphi)\leq Cr^2
\end{equation}
for all free boundary points $x_\circ \in \partial\{v>\varphi\}$.

This property is essential in order to study the free boundary later.

\begin{rem}
\label{rem.caff.vanish}
Since $-\Delta v\geq0$ \emph{everywhere},  it is clear that if $x_\circ \in \partial\{v>\varphi\}$ is a free boundary point, then necessarily $-\Delta \varphi(x_\circ )\geq0$ (otherwise we would have $-\Delta \varphi(x_\circ )<0$, and since $u$ touches $\varphi$ from above at $x_\circ $,  also $-\Delta v(x_\circ )<0$, a contradiction).

Moreover it can be proved that, in fact, if $\Delta \varphi$ and $\nabla\Delta\varphi$ do \emph{not} vanish simultaneously, then $-\Delta\varphi>0$ near \emph{all} free boundary points \cite{Caf98}.
\end{rem}

\begin{figure}
\includegraphics[width = \textwidth]{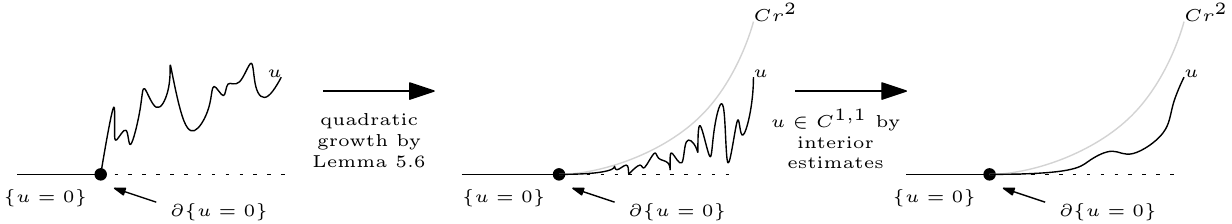}
\caption{Strategy of the proof of Theorem~\ref{ch4-optimal-reg}.}
\label{fig.19}
\end{figure}

This motivates the following:

\vspace{2mm}

\noindent\textbf{Assumption}: \emph{The obstacle} $\varphi$ \emph{satisfies} 
\[-\Delta \varphi\geq c_\circ >0\]
\emph{in the ball} $B_1$.
\vspace{2mm}

In particular, by Remark~\ref{rem.caff.vanish}, if $\Delta \varphi$ and $\nabla\Delta\varphi$ do {not} vanish simultaneously, then we have $-\Delta \varphi>0$ near any free boundary point, and thus by zooming in if necessary, we will always have that the assumption is satisfied in $B_1$, for some small $c_\circ >0$.

Thus, the only real assumption here is that $\Delta \varphi$ and $\nabla\Delta\varphi$ do {not} vanish simultaneously, which is a very mild assumption.
Moreover, this is in a sense a \emph{necessary} assumption: without this, the nondegeneracy \eqref{ch4-nondeg} does \emph{not} hold, and no regularity result can be proved for the free boundary. (Without the assumption, one can actually construct counterexamples in which the free boundary is a fractal set with infinite perimeter.)

\begin{prop}[Nondegeneracy]\label{ch4-prop-nondeg}\index{Nondegeneracy, Obstacle problem}
Let $\varphi\in C^\infty(B_1)$, and $v$ be any solution to \eqref{ch4-obst-B1}.
Assume that $\varphi$ satisfies $-\Delta \varphi\geq c_\circ >0$ in $B_1$.
Then, for every free boundary point $x_\circ \in \partial\{v>\varphi\}\cap B_{1/2}$, we have
\[\qquad\qquad  0<cr^2\leq \sup_{B_r(x_\circ )}(v-\varphi)\leq Cr^2\qquad \textrm{for all}\ r\in (0,{\textstyle\frac14}),\]
with a constant $c>0$ depending only on $n$ and $c_\circ $.
\end{prop}

\begin{proof}
Let $x_1\in \{v>\varphi\}$ be any point close to $x_\circ $ (we will then let $x_1\to x_\circ $ at the end of the proof).

Consider the function 
\[w(x):=v(x)-\varphi(x)-\frac{c_\circ }{2n}|x-x_1|^2.\]
Then, in $\{v>\varphi\}$ we have
\[\Delta w=\Delta v-\Delta \varphi-c_\circ =-\Delta\varphi-c_\circ \geq 0\]
and hence $-\Delta w\leq 0$ in $\{v>\varphi\}\cap B_r(x_1)$.
Moreover, $w(x_1)>0$.

By the maximum principle, $w$ attains a positive maximum on $\partial\bigl(\{v>\varphi\}\cap B_r(x_1)\bigr)$.
But on the free boundary $\partial\{v>\varphi\}$ we clearly have $w<0$.
Therefore, there is a point on $\partial B_r(x_1)$ at which $w>0$.
In other words, 
\[0<\sup_{\partial B_r(x_1)}w=\sup_{\partial B_r(x_1)}(v-\varphi)-\frac{c_\circ }{2n}\,r^2.\]
Letting now $x_1\to x_\circ $, we find $\sup_{\partial B_r(x_\circ )}(v-\varphi)\geq cr^2>0$, as desired.
\end{proof}

\subsection*{Summary of basic properties}

Let $v$ be any solution to the obstacle problem
\[
\left\{\begin{array}{rcll}
v&\geq&\varphi & \textrm{in}\ B_1 \\
\Delta v&\leq&0 & \textrm{in}\ B_1\\
\Delta v&=&0 & \textrm{in}\ \{v>\varphi\}\cap B_1.
\end{array}
\right.
\]
Then, we have:

\vspace{3mm}

\noindent$\bullet$ \underline{\smash{Optimal regularity}}:\quad $\displaystyle \|v\|_{C^{1,1}(B_{1/2})}\leq C\bigl(\|v\|_{L^\infty(B_{1})}+\|\varphi\|_{C^{1,1}(B_{1})}\bigr)$

\vspace{3mm}

\noindent$\bullet$ \underline{\smash{Nondegeneracy}}: If $-\Delta \varphi\geq c_\circ >0$, then 
\[\qquad\qquad 0<cr^2\leq \sup_{B_r(x_\circ )}(v-\varphi)\leq Cr^2\qquad \textrm{for all}\ r\in (0,{\textstyle\frac12})\]
at all free boundary points $x_\circ \in \partial\{v>\varphi\}\cap B_{1/2}$.

\vspace{3mm}

\noindent$\bullet$ \underline{\smash{Equivalence with zero obstacle}}: The problem is equivalent to 
\[
\left\{\begin{array}{rcll}
u&\geq&0 & \textrm{in}\ B_1 \\
\Delta u&\leq&f & \textrm{in}\ B_1\\
\Delta u&=&f & \textrm{in}\ \{u>0\}\cap B_1,
\end{array}
\right.
\]
where $f=-\Delta \varphi\geq c_\circ >0$.

\vspace{5mm}

We will next provide an alternative approach to the optimal regularity.

\section{Basic properties of solutions II}

We proceed now to study the basic properties of solutions $u\geq0$ to the obstacle problem \eqref{ch4-min2} or \eqref{ch4-min2B}.
As explained before, the main point here is that we prove optimal regularity independently from the previous Section.

Throughout this section we will always assume 
\[
f\ge 0\quad\text{in}\quad \Omega. 
\]

\subsection*{Existence of solutions}

Since problem \eqref{ch4-min2} is equivalent to \eqref{ch4-min}, existence and uniqueness of solutions follow  easily from Proposition \ref{obstacle-existence}, as shown next.

\begin{prop}[Existence and uniqueness]\index{Existence and uniqueness!Obstacle problem}
Let $\Omega\subset\R^n$ be any bounded Lipschitz domain, and let $g:\partial\Omega\to\R$  be such that
\[\mathcal C=\bigl\{u\in H^1(\Omega): u\geq0\ \textrm{in}\ \Omega,\ u|_{\partial\Omega}=g\bigr\}\neq \varnothing.\]
Then, for any $f\in L^2(\Omega)$ there exists a unique minimizer of 
\[\frac12\int_\Omega|\nabla u|^2dx+\int_\Omega fu\]
among all functions $u\in H^1(\Omega)$ satisfying $u\geq0$ in $\Omega$ and $u|_{\partial\Omega}=g$.
\end{prop}


\begin{proof}
We follow the proof of Proposition \ref{obstacle-existence}. Let
\[\theta_\circ :=\inf\left\{\frac12\int_\Omega |\nabla w|^2dx + \int_\Omega fw \,:\, w\in H^1(\Omega),\ w|_{\partial\Omega}=g,\ w\geq 0\ \textrm{in}\ \Omega\right\},\]
that is, the infimum value of $\mathcal E(w)=\frac12\int_\Omega |\nabla w|^2dx+\int_\Omega fw$ among all admis\-sible functions~$w$. Notice that, by H\"older's inequality, $\mathcal E(w) < +\infty$ if $w\in H^1(\Omega)$. 

We take again a sequence of functions $\{v_k\}$ such that $v_k\in H^1(\Omega)$, $v_k|_{\partial\Omega}=g$, $v_k\geq 0$ in $\Omega$, and $\mathcal E(v_k)\to \theta_\circ $ as $k\to\infty$. By the Poincar\'e inequality (Theorem \ref{ch0-Poinc}), H\"older's inequality, and the fact that $\mathcal E(v_k)\le  \theta_\circ +1$, for $k$ large enough
\begin{align*}
\|v_k\|_{H^1 (\Omega)}^2 & \le C \left(\int_\Omega|\nabla v_k|^2 + \int_{\partial \Omega} g^2 \right)
\le C \left(\theta_\circ + 1 +  \int_\Omega |fv_k| + \frac12 \int_{\partial \Omega} g^2 \right)\\
 & \le C \left(\theta_\circ + 1+ \|f\|_{L^2(\Omega)}\|v_k\|_{H^1(\Omega)} + \frac12 \int_{\partial \Omega} g^2 \right).
\end{align*}
In particular, $\|v_k\|_{H^1(\Omega)}\le C$ for some constant $C$ depending only on $n$, $\Omega$, $g$, $f$, and $\theta_\circ$ (recall that $g\in L^2(\partial\Omega)$ by the trace theorem, \ref{it.S5} in Chapter \ref{ch.0}). Hence, a subsequence $\{v_{k_j}\}$ converges to a certain function $v$ strongly in $L^2(\Omega)$ and weakly in $H^1(\Omega)$.
By compactness of the trace operator $v_{k_j}|_{\partial\Omega}\to v|_{\partial\Omega} = g$ in $L^2(\partial\Omega)$.
Furthermore, $v$ satisfies $\mathcal E(v)\leq \liminf_{j\to\infty}\mathcal E(v_{k_j})$ (by \eqref{ch0-weak-conv2}-\eqref{ch0-weak-conv3} from \ref{it.S4} and weak convergence), and therefore it will be a minimizer of the energy functional.
Since $v_{k_j}\geq 0$ in $\Omega$ and $v_{k_j}\to v$ in $L^2(\Omega)$, we have $v\geq 0$ in $\Omega$. Thus, there is a minimizer~$v$.

The uniqueness of the minimizer follows from the strict convexity of the functional $\mathcal{E}(v)$, exactly as in Theorem \ref{ch0-existence}.
\end{proof}

\begin{rem}
Alternatively, we could have denoted $v := u +\varphi$ with $\varphi$ such that $-\Delta \varphi = f$ in $\Omega$, and use Proposition~\ref{obstacle-existence}.
\end{rem}

Furthermore, we have the following equivalence. (Recall that we denote $u^+ = \max\{u, 0\}$, and $u^- = \max\{-u, 0\}$, so that $u = u^+-u^-$.)

\begin{prop}\label{ch4-equivalence-3pb}
Let $\Omega\subset\R^n$ be any bounded Lipschitz domain, and let $g:\partial\Omega\to\R$ be such that
\[\mathcal C=\bigl\{u\in H^1(\Omega): u\geq0\ \textrm{in}\ \Omega,\ u|_{\partial\Omega}=g\bigr\}\neq \varnothing.\]

Then, the following are equivalent.
\begin{itemize}
\item[(i)] $u$ minimizes $\frac12\int_\Omega |\nabla u|^2+\int_\Omega fu$ among all functions satisfying $u\geq0$ in $\Omega$ and $u|_{\partial\Omega}=g$.

\item[(ii)] $u$ minimizes $\frac12\int_\Omega |\nabla u|^2+\int_\Omega fu^+$ among all functions satisfying $u|_{\partial\Omega}=g$.
\end{itemize}
\end{prop}

\begin{proof}
The two functionals coincide whenever $u\geq0$.
Thus, the only key point is to prove that the minimizer in (ii) must be nonnegative, i.e., $u=u^+$. (Notice that $\mathcal{C}\neq \varnothing$ implies that $g\ge 0$ on $\de\Omega$.)
To show this, recall that the positive part of any $H^1$ function is still in $H^1$, and moreover $|\nabla u|^2=|\nabla u^+|^2+|\nabla u^-|^2$ (see \ref{it.S10} in Chapter~\ref{ch.0}).
Thus, we have that (recall that $f \ge 0$ in $\Omega$)
\[\frac12\int_\Omega |\nabla u^+|^2+\int_\Omega fu^+ \leq \frac12\int_\Omega |\nabla u|^2+\int_\Omega fu^+,\]
with strict inequality unless $u=u^+$.
This means that any minimizer $u$ of the functional in (ii) must be nonnegative, and thus we are done.
\end{proof}

\subsection*{Basic properties of solutions}

Let us next prove that any minimizer of \eqref{ch4-min2} is actually a solution to \eqref{equation-PSU} below.

We recall that we are always assuming that obstacles are as smooth as necessary, $\varphi\in C^\infty(\Omega)$, and therefore we assume here that $f\in C^\infty(\Omega)$ as well. 

\begin{prop}\label{ch4-equation-PB-f}
Let $\Omega\subset\R^n$ be any bounded Lipschitz domain, $f\in C^\infty(\Omega)$, and $u\in H^1(\Omega)$ be any minimizer of \eqref{ch4-min2} subject to the boundary conditions $u|_{\partial\Omega}=g$.

Then, $u$ solves 
\begin{equation}\label{equation-PSU}
\Delta u=f\chi_{\{u>0\}}\quad \textrm{in}\quad \Omega
\end{equation}
in the weak sense. 

In particular, $u$ is $C^{1,\alpha}$ inside $\Omega$, for every $\alpha\in (0,1)$.
\end{prop}

\begin{proof}
Notice that, by Proposition \ref{ch4-equivalence-3pb}, $u$ is actually a minimizer of 
\[\mathcal E(u)= \frac12\int_\Omega |\nabla u|^2+\int_\Omega fu^+\]
subject to the boundary conditions $u|_{\partial\Omega}=g$.

Thus, for any $\eta\in H^1_0(\Omega)$ and $\eps > 0$ we have 
\[\mathcal E(u+\varepsilon\eta)\geq \mathcal E(u).\]
In particular, we obtain 
\[
0\le \lim_{\eps\downarrow 0} \frac{\mathcal{E}(u+\eps \eta) - \mathcal{E}(u)}{\eps} = \int_\Omega \nabla u \cdot \nabla \eta + \lim_{\eps\downarrow0}\int_{\Omega} f \frac{(u+\eps\eta)^+-u^+}{\eps}.
\]

Notice that 
\[\lim_{\varepsilon\downarrow 0}\frac{(u+\varepsilon\eta)^+-u^+}{\varepsilon}= \left\{ \begin{array}{ll}
\eta & \qquad \textrm{in}\quad \{u>0\} \\
\eta^+& \qquad \textrm{in}\quad \{u=0\},
\end{array}\right.\]
so that we have 
\[
\int_\Omega \nabla u \cdot \nabla \eta + \int_{\Omega} f \eta\chi_{\{u > 0\}} +\int_{\Omega} f \eta^+\chi_{\{u = 0\}}\ge 0\quad\text{for all}\quad \eta\in H^1_0(\Omega). 
\]
Assume first that $\eta \ge 0$, so that 
\[
\int_\Omega \nabla u \cdot \nabla \eta + \int_{\Omega} f \eta \ge 0\quad\text{for all}\quad \eta\in H^1_0(\Omega), \quad \eta\ge 0,
\]
which implies that $\Delta u \le f$ in the weak sense. On the other hand, if $\eta \le 0$, then
\[
\int_\Omega \nabla u \cdot \nabla \eta + \int_{\Omega} f \eta\chi_{\{u > 0\}} \ge 0\quad\text{for all}\quad \eta\in H^1_0(\Omega), \quad \eta\le 0,
\]
which implies that $\Delta u \ge f \chi_{\{u  > 0\}}$ in the weak sense. In all (recall that $f \ge 0$),
\[
f \chi_{\{u > 0\}} \le \Delta u \le f\quad\text{in}\quad\Omega.
\]
(In particular, notice that $\Delta u = f $ in $\{u > 0\}$.) Now, since $f$ is smooth, this implies that $\Delta u \in L^\infty_{\rm loc}(\Omega)$. By Proposition~\ref{prop.Schauder_estimates_L_bounded} we deduce that $u\in C^{1,1-\varepsilon}$ for every $\varepsilon>0$. Moreover, since $\Delta u \in L^\infty_{\rm loc}(\Omega)$ we have $\Delta u \in L^2_{\rm loc}(\Omega)$ and by Calder\'on-Zygmund estimates (see, for example, Remark~\ref{rem.CZ2}) we have $u\in W^{2, 2}_{\rm loc}(\Omega)$. Thus, $\Delta u = 0$ almost everywhere in the level set $\{u = 0\}$ (see \ref{it.S10} in Chapter~\ref{ch.0}) and we have 
\[
\Delta u=f\chi_{\{u>0\}}\quad \textrm{a.e. in}\quad \Omega.
\]
From here we deduce that $\Delta u=f\chi_{\{u>0\}}$ in $\Omega$ in the weak sense.
\end{proof}

Notice that in the previous Section, when dealing with minimizers $v$ of \eqref{ch4-min}, it was not easy to prove that $v$ is continuous (see Proposition \ref{ch4-Euler-Lagrange}).
Here, instead, thanks to Proposition \ref{ch4-equation-PB-f} we simply used Schauder-type estimates for the Laplacian to directly deduce that $u$ is $C^{1,1-\varepsilon}$, which is the almost-optimal regularity of solutions.

\subsection*{Optimal regularity of solutions}

Thanks to the previous results, we know that any minimizer of \eqref{ch4-min2} is continuous and solves \eqref{equation-PSU}.
From now on, we will localize the problem and study it in a ball:
\begin{equation}\label{ch4-obst-B1-u}
\left\{
\begin{array}{rcll}
u&\geq&0 & \textrm{in}\ B_1 \\
\Delta u&=&f\chi_{\{u>0\}} & \textrm{in}\ B_1.
\end{array}
\right.
\end{equation}

Our next goal is to answer the following question:
\[\textrm{\underline{Question}: }\textit{ What is the optimal regularity of solutions?}\]

First, a few important considerations.
Notice that in the set $\{u>0\}$ we have $\Delta u=f$, while in the interior of $\{u=0\}$ we have $\Delta u=0$ (since $u\equiv 0$ there).

Thus, since $f$ is in general not zero, $\Delta u$ is \emph{discontinuous} across the free boundary $\partial\{u>0\}$ in general.
In particular, $u\notin C^2$.

We will now prove that any minimizer of \eqref{ch4-min2} is actually $C^{1,1}$, which gives the:
\[\textrm{\underline{Answer}: }\ u\in C^{1,1}\ \textit{ (second derivatives are bounded but not continuous)}\]

The precise statement and proof are given next.

\begin{thm}[Optimal regularity]\label{ch4-optimal-reg-u}\index{Optimal regularity, Obstacle problem}
Let $f\in C^\infty(B_1)$, and let $u$ be any solution to \eqref{ch4-obst-B1-u}.
Then, $u$ is $C^{1,1}$ inside $B_{1/2}$, with the estimate
\[\|u\|_{C^{1,1}(B_{1/2})}\leq C\bigl(\|u\|_{L^\infty(B_{1})}+\|f\|_{{\rm Lip}(B_{1})}\bigr).\]
The constant $C$ depends only on $n$.
\end{thm}

To prove this, the main step is the following.

\begin{lem}\label{ch4-optimal-reg-lem-u}
Let $u$ be any solution to \eqref{ch4-obst-B1-u}.
Let $x_\circ \in \overline{B_{1/2}}$ be any point on $\{u=0\}$.
Then, for any $r\in (0,\frac14)$ we have
\[0\leq \sup_{B_r(x_\circ )}u\leq Cr^2,\]
with $C$ depending only on $n$ and $\|f\|_{L^\infty(B_{1})}$.
\end{lem}

\begin{proof}
We have that $\Delta u=f\chi_{\{u>0\}}$ in $B_1$, with $f\chi_{\{u>0\}}\in L^\infty(B_1)$.
Thus, since $u\geq0$, we can use the Harnack inequality (Theorem \ref{thm.Harnack.f}) for the equation $\Delta u=f\chi_{\{u>0\}}$ in $B_{2r}(x_\circ )$, to find
\[\sup_{B_r(x_\circ )} u \leq C\left(\inf_{B_r(x_\circ )}u + r^2\|f\chi_{\{u>0\}}\|_{L^\infty(B_{2r}(x_\circ ))} \right).\]
Since $u\geq0$ and $u(x_\circ )=0$, this yields $\sup_{B_r(x_\circ )} u \leq C\|f\|_{L^\infty(B_1)}r^2$, as wanted.
\end{proof}

Notice that this proof is significantly shorter than the one given in the previous Section (Lemma \ref{ch4-optimal-reg-lem}).
This is an advantage of using the formulation~\eqref{equation-PSU}.

We have proved the following: 
\[\textit{At every free boundary point }x_\circ ,\ u\textit{ grows (at most) quadratically.}\]
As shown next, this easily implies the $C^{1,1}$ regularity.

\begin{proof}[Proof of Theorem \ref{ch4-optimal-reg-u}]
Dividing $u$ by a constant if necessary, we may assume that $\|u\|_{L^\infty(B_{1})}+\|f\|_{{\rm Lip}(B_{1})}\leq 1$.
 
We already know that $u\in C^\infty$ in the set $\{u>0\}$ (since $\Delta u=f\in C^\infty$ there), and also inside the set $\{u=0\}$ (since $u=0$ there). 
Moreover, on the interface $\Gamma=\partial\{u>0\}$ we have proved the quadratic growth $\sup_{B_r(x_\circ )}u\leq Cr^2$.
Let us prove that this yields the $C^{1,1}$ bound we want.

Let $x_1\in \{u>0\}\cap B_{1/2}$, and let $x_\circ \in \Gamma$ be the closest free boundary point.
Denote $\rho=|x_1-x_\circ |$.
Then, we have $\Delta u=f$ in $B_\rho(x_1)$.

By Schauder estimates, we find
\[\|D^2u\|_{L^\infty(B_{\rho/2}(x_1))} \leq C\left(\frac{1}{\rho^2}\|u\|_{L^\infty(B_{\rho}(x_1))}+\|f\|_{{\rm Lip}(B_{1})}\right).\]
But by the growth proved in the previous Lemma, we have $\|u\|_{L^\infty(B_{\rho}(x_1))}\leq C\rho^2$, which yields 
\[\|D^2u\|_{L^\infty(B_{\rho/2}(x_1))}\leq C.\]
In particular, 
\[|D^2u(x_1)|\leq C.\]
We can do this for each $x_1\in \{u>0\}\cap B_{1/2}$, and therefore $\|u\|_{C^{1,1}(B_{1/2})}\leq C$, as wanted.
\end{proof}

Also, notice that as a consequence of the previous results, we have that as soon as the solution to \eqref{ch4-obst-B1-u} has non-empty contact set, then its $C^{1,1}$ norm is universally bounded. 

\begin{cor}
Let $u$ be any solution to \eqref{ch4-obst-B1-u}, and let us assume that $u(0) = 0$ and $\|f\|_{{\rm Lip}(B_1)}\le 1$. Then, 
\[
\|u\|_{C^{1, 1}(B_{1/2})}\le C
\]
for some $C$ depending only on $n$. 
\end{cor}
\begin{proof}
It is an immediate consequence of Theorem~\ref{ch4-optimal-reg-u} combined with Lemma~\ref{ch4-optimal-reg-lem-u}. 
\end{proof}

\subsection*{Nondegeneracy}

For completeness, we now state the nondegeneracy in this setting (analogously to Proposition~\ref{ch4-prop-nondeg}). That is, at all free boundary points, $u$ grows \emph{at least} quadratically (we already know \emph{at most} quadratically). We want:
\[
0<cr^2\leq \sup_{B_r(x_\circ )}u\leq Cr^2
\]
for all free boundary points $x_\circ \in \partial\{u>0\}$.

This property is essential in order to study the free boundary later.
As before, for this we need the following natural assumption:

\vspace{2mm}

\noindent\textbf{Assumption}: \emph{The right-hand side} $f$ \emph{satisfies} 
\[f\geq c_\circ >0\]
\emph{in the ball} $B_1$.
\vspace{2mm}

\begin{prop}[Nondegeneracy]\label{ch4-prop-nondeg2}\index{Nondegeneracy, Obstacle problem}
Let $u$ be any solution to \eqref{ch4-obst-B1-u}.
Assume that $f\geq c_\circ >0$ in $B_1$.
Then, for every free boundary point $x_\circ \in \partial\{u>0\}\cap B_{1/2}$, we have
\[\qquad\qquad  0<cr^2\leq \sup_{B_r(x_\circ )}u\leq Cr^2\qquad \textrm{for all}\ r\in (0,{\textstyle\frac12}),\]
with a constant $c>0$ depending only on $n$ and $c_\circ $.
\end{prop}

\begin{proof}
The proof is the one from Proposition \ref{ch4-prop-nondeg}.
\end{proof}

\subsection*{Summary of basic properties}

Let $u$ be any solution to the obstacle problem
\[
\left\{
\begin{array}{rcll}
u&\geq&0 & \textrm{in}\ B_1, \\
\Delta u&=&f\chi_{\{u>0\}} & \textrm{in}\ B_1.
\end{array}
\right.
\]
Then, we have:

\vspace{3mm}

\noindent$\bullet$ \underline{\smash{Optimal regularity}}:\qquad $\displaystyle \|u\|_{C^{1,1}(B_{1/2})}\leq C\left(\|u\|_{L^\infty(B_{1})}+\|f\|_{{\rm Lip}(B_{1})}\right)$

\vspace{3mm}

\noindent$\bullet$ \underline{\smash{Nondegeneracy}}: If $f\geq c_\circ >0$, then 
\[\qquad\qquad 0<cr^2\leq \sup_{B_r(x_\circ )}u\leq Cr^2\qquad \textrm{for all}\ r\in (0,{\textstyle\frac12})\]
at all free boundary points $x_\circ \in \partial\{u>0\}\cap B_{1/2}$.

\vspace{5mm}

Using these properties, we can now start the study of the free boundary.

\section{Regularity of free boundaries: an overview}

From now on, we consider any solution to
\begin{equation}\label{ch4-obst-f-1}
\left\{
\begin{array}{l}
u\in C^{1,1}(B_1),
\vspace{1mm}
\\
u\geq0\quad \textrm{in}\ B_1,
\vspace{1mm}
\\
\Delta u=f\quad \textrm{in}\ \{u>0\},
\end{array}
\right.
\end{equation}
(see Figure~\ref{fig.20}) with 
\begin{equation}\label{ch4-obst-f-2}
f\geq c_\circ >0\qquad\textrm{and}\qquad f\in C^\infty.
\end{equation}

\begin{figure}
\includegraphics[scale = 0.9]{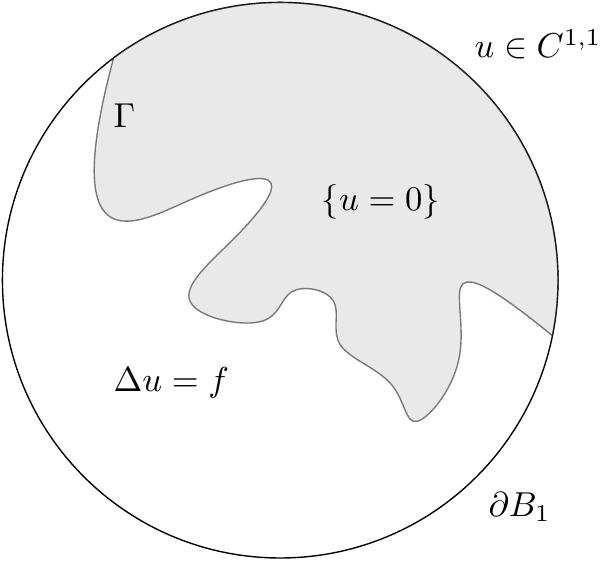}
\caption{A solution to the obstacle problem in $B_1$.}
\label{fig.20}
\end{figure}

Notice that on the interface 
\[\Gamma=\partial\{u>0\}\cap B_1\]
we have that 
\[\begin{split} u&=0\quad \textrm{on}\ \Gamma,\\ 
\nabla u&=0\quad \textrm{on}\ \Gamma.\end{split}\]

The central mathematical challenge in the obstacle problem is to
\[\textit{understand the geometry/regularity of the free boundary }\Gamma.\]

Notice that, even if we already know the optimal regularity of $u$ (it is $C^{1,1}$), we know nothing about the free boundary $\Gamma$.
A priori $\Gamma$ could be a very irregular object, even a fractal set with infinite perimeter.

As we will see, under the natural assumption $f\geq c_\circ >0$, it turns out that free boundaries are always smooth, possibly outside a certain set of singular points. In fact, in our proofs we will assume for simplicity that $f\equiv 1$ (or constant). We do that in order to avoid $x$-dependence and the technicalities associated to it, which gives cleaner proofs. In this way, the main ideas behind the regularity of free boundaries are exposed.

\subsection*{Regularity of free boundaries: main results}

Assume from now on that $u$ solves \eqref{ch4-obst-f-1}-\eqref{ch4-obst-f-2}.
Then, the main known results on the free bound\-a\-ry $\Gamma=\partial\{u>0\}$ can be summarized as follows:

\vspace{3mm}

\noindent$\bullet$ At every free boundary point $x_\circ \in \Gamma$, we have
\[\qquad\qquad 0<cr^2\leq \sup_{B_r(x_\circ )}u\leq Cr^2\qquad \qquad\forall r\in \left(0,r_\circ\right).\]

\vspace{0mm}

\noindent$\bullet$ The free boundary $\Gamma$ splits into \emph{regular points} and \emph{singular points}.

\vspace{3mm}

\noindent$\bullet$ \index{Regular points} The set of \emph{regular points} is an open subset of the free boundary, and $\Gamma$ is $C^\infty$ near these points.

\vspace{3mm}

\noindent$\bullet$ \index{Singular points} \emph{Singular points} are those at which the contact set $\{u=0\}$ has \emph{zero density}, and these points (if any) are contained in an $(n-1)$-dimensional $C^1$ manifold.

\vspace{3mm}

Summarizing, \emph{the free boundary is smooth, possibly outside a certain set of singular points}. See Figure~\ref{fig.21}.

\begin{figure}
\includegraphics{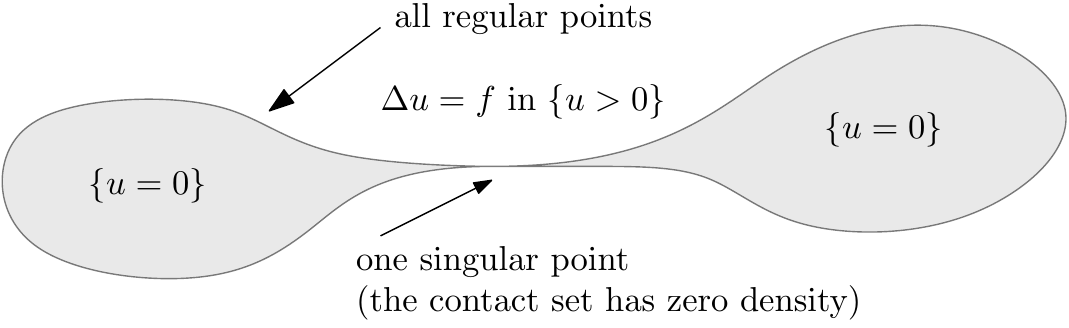}
\caption{Singular points are those where the contact set has zero density.} 
\label{fig.21}
\end{figure}

So far, we have not even proved that $\Gamma$ has finite perimeter, or anything at all about $\Gamma$.
Our goal will be to prove that $\Gamma$ \emph{is} $C^\infty$ \emph{near regular points}.
This is the main and most important result in the obstacle problem.
It was proved by Caffarelli in 1977, and it is one of the major results for which he received the Wolf Prize in 2012 and the Shaw Prize in 2018.

\subsection*{Overview of the strategy}

To prove these regularity results for the free boundary, one considers \emph{blow-ups}.
Namely, given any free boundary point~$x_\circ$ for a solution $u$ of \eqref{ch4-obst-f-1}-\eqref{ch4-obst-f-2}, one takes the rescalings
\[u_r(x):= \frac{u(x_\circ +rx)}{r^2},\]
with $r>0$ small.
This is like ``zooming in'' at a free boundary point.

The factor $r^{-2}$ is chosen so that
\[\|u_r\|_{L^{\infty}(B_1)}\approx 1\]
as $r\to0$; recall that $0<cr^2\leq \sup_{B_r(x_\circ )}u\leq Cr^2$.

Then, by $C^{1,1}$ estimates, we will prove that a subsequence of $u_r$ converges to a function $u_0$ locally uniformly in $\R^n$ as $r\to0$.
Such function $u_0$ is called a \emph{blow-up of} $u$ \emph{at}~$x_\circ $.

Any blow-up $u_0$ is a \emph{global} solution to the obstacle problem, with $f\equiv 1$ (or with $f\equiv \textrm{constant}>0$).

Then, the main issue is to \emph{classify blow-ups}: that is, to show that 
\[\begin{split} 
\textrm{either}\qquad & u_0(x)={\textstyle \frac12}(x\cdot e)_+^2 \qquad\qquad\,\, \textrm{(this happens at regular points)} \\
\textrm{or} \qquad & u_0(x)={\textstyle \frac12}x^TAx \qquad\qquad\quad \textrm{(this happens at singular points)}. \end{split}\]
Here, $e\in \mathbb{S}^{n-1}$ is a unit vector, and $A\geq0$ is a positive semi-definite matrix satisfying $\textrm{tr}A=1$.
Notice that the contact set $\{u_0=0\}$ becomes a half-space in case of regular points, while it has zero measure in case of singular points; see Figure~\ref{fig.22}.

\begin{figure}
\includegraphics[width =\textwidth]{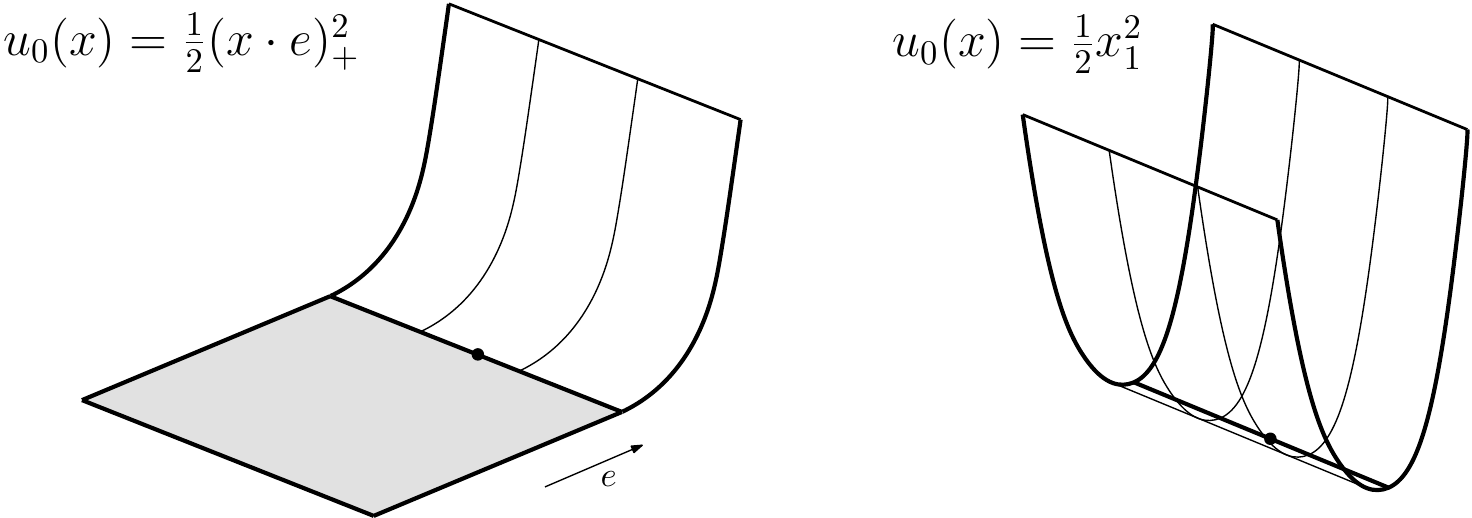}
\caption{Possible blow-ups of the solution to the obstacle problem at free boundary points.}
\label{fig.22}
\end{figure}

Once this is done, one has to ``transfer'' the information from the blow-up $u_0$ to the original solution $u$.
Namely, one shows that, in fact, the free boundary is $C^{1,\alpha}$ near regular points (for some small $\alpha>0$).

Finally, once we know that the free boundary is $C^{1,\alpha}$,  we will ``bootstrap'' the regularity to $C^\infty$.
This is in a somewhat similar spirit as in \emph{Hilbert's XIXth problem} (Chapter \ref{ch.2}), where the really difficult point was to prove that minimizers are always $C^{1,\alpha}$. 
Once this was done, by Schauder estimates (Chapter \ref{ch.1}) and a bootstrap argument we saw that solutions are actually~$C^\infty$.

Classifying blow-ups is not easy.
Generally speaking, classifying blow-ups is of similar difficulty to proving regularity estimates --- recall the blow-up arguments in Chapter~\ref{ch.1}. 

Thus, how can we classify blow-ups?
Do we get any extra information on $u_0$ that we did not have for $u$?
(Otherwise it seems hopeless!)

The answer is \emph{yes}: \textsc{Convexity}.
We will prove that all blow-ups are always \emph{convex}.
This is a huge improvement, since this yields that the contact set $\{u_0=0\}$ is also convex.
Prior to that, we will also show that blow-ups are also \emph{homogeneous}.

So, before the blow-up we had no information on the set $\{u=0\}$, but after the blow-up we get that $\{u_0=0\}$ \emph{is a convex cone}.
Thanks to this we will be able to classify blow-ups, and thus to prove the regularity of the free boundary.

\vspace{3mm}

The main steps in the proof of the regularity of the free boundary will be the following:
\begin{enumerate}
\item $0<cr^2\leq \sup_{B_r(x_\circ )}u\leq Cr^2$

\vspace{1mm}

\item Blow-ups $u_0$ are \emph{homogeneous} and \emph{convex}.

\vspace{1mm}

\item If the contact set has \emph{positive density} at $x_\circ $, then $u_0(x)={\textstyle \frac12}(x\cdot e)_+^2$.

\vspace{1mm}

\item Deduce that the free boundary is $C^{1,\alpha}$ near $x_\circ $.

\vspace{1mm}

\item Deduce that the free boundary is $C^\infty$ near $x_\circ $.
\end{enumerate}

\vspace{2mm}

The proof we will present here for the convexity of blow-ups is new, based on the fact that they are homogeneous.
We refer to \cite{Caf98}, \cite{PSU}, \cite{W}, and \cite{KN}, for different proofs of the classification of blow-ups and/or of the regularity of free boundaries.

\section{Classification of blow-ups}
\index{Classification of blow-ups}

The aim of this Section is to classify all possible blow-ups $u_0$.
For this, we will first prove that blow-ups are homogeneous, then we will prove that they are convex, and finally we will establish their complete classification.


\subsection*{Homogeneity of blow-ups}

We start by proving that blow-ups are homogeneous.
This is not essential in the proof of the regularity of the free boundary (see \cite{Caf98}), but it actually simplifies it.

Recall that, for simplicity, from now on we will assume that $f\equiv1$ in~$B_1$.
This is only to avoid $x$-dependence in the equation, it simplifies some proofs.

Therefore, from now on we consider a solution $u$ satisfying (see Figure~\ref{fig.23}):
\begin{equation}\label{ch4-obst-f=1}
\begin{array}{l}
u\in C^{1,1}(B_1) 
\vspace{1mm}
\\
u\geq0\quad \textrm{in}\ B_1 
\vspace{1mm}
\\
\Delta u=1\quad \textrm{in}\ \{u>0\} 
\vspace{1mm}
\\
0\ \textrm{is a free boundary point.}
\end{array}
\end{equation}
We will prove all the results around the origin (without loss of generality).

\begin{figure}
\includegraphics{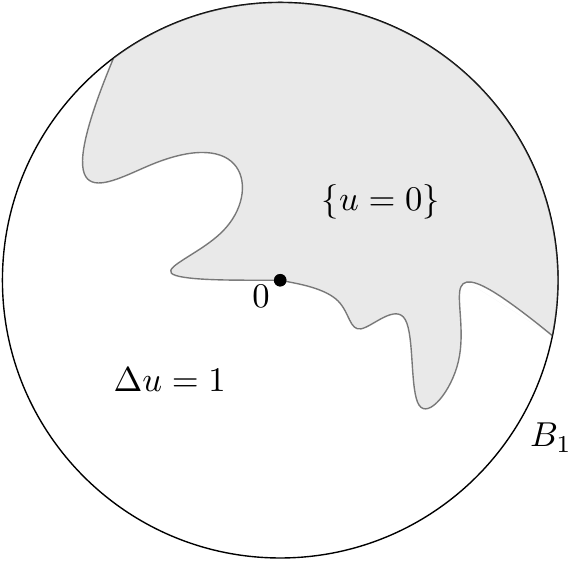}
\caption{A solution $u$ to the obstacle problem with $f\equiv1$.}
\label{fig.23}
\end{figure}

We will show that, for the original solution $u$ in $B_1$, the closer we look at a free boundary point $x_\circ $, the closer is the solution to being homogeneous.

\begin{prop}[Homogeneity of blow-ups]\label{cor-Weiss}\index{Homogeneity of blow-ups}
Let $u$ be any solution to \eqref{ch4-obst-f=1}.
Then, any blow-up of $u$ at $0$ is homogeneous of degree $2$.
\end{prop}

It is important to remark that not all global solutions to the obstacle problem in $\R^n$ are homogeneous.
There exist global solutions $u_0$ that are convex, $C^{1,1}$, and whose contact set $\{u_0=0\}$ is an ellipsoid, for example.
However, thanks to the previous result, we find that such non-homogeneous solutions cannot appear as blow-ups, i.e., that all blow-ups must be homogeneous.

We provide two different proofs of Proposition~\ref{cor-Weiss}. The first one uses a monotonicity formula as introduced by Weiss; while the second one does not require any monotonicity formula and is due to Spruck. 

\subsubsection*{Homogeneity of blow-ups \textit{\`a la} Weiss}
For the first proof of Proposition~\ref{cor-Weiss}, we need the following monotonicity formula due to Weiss \cite{W}.

\begin{thm}[Weiss' monotonicity formula]\label{thm-Weiss}\index{Weiss monotonicity formula}
Let $u$ be any solution to \eqref{ch4-obst-f=1}.
Then, the quantity 
\begin{equation}
\label{weissenergy}
W_u(r):=\frac{1}{r^{n+2}}\int_{B_r}\left\{{\textstyle \frac12}|\nabla u|^2+u\right\} -\frac{1}{r^{n+3}}\int_{\partial B_r}u^2
\end{equation}
is monotone in $r$, that is,
\[\frac{d}{dr}W_u(r)=\frac{1}{r^{n+4}}\int_{\partial B_r}(x\cdot \nabla u-2u)^2dx\geq 0\]
for $r\in(0,1)$.
\end{thm}

\begin{proof}
Let $u_r(x)=r^{-2}u(rx)$, and observe that
\[W_u(r)=\int_{B_1}\left\{{\textstyle \frac12}|\nabla u_r|^2+u_r\right\} -\int_{\partial B_1}u_r^2.\]
Using this, together with 
\[\frac{d}{dr}(\nabla u_r)=\nabla \frac{d}{dr}u_r,\]
we find
\[\frac{d}{dr}W_u(r)=\int_{B_1}\left\{\nabla u_r\cdot \nabla \frac{d}{dr}u_r+\frac{d}{dr}u_r\right\}-2\int_{\partial B_1}u_r\frac{d}{dr}u_r.\]
Now, integrating by parts we get
\[\int_{B_1}\nabla u_r\cdot \nabla \frac{d}{dr}u_r=-\int_{B_1}\Delta u_r\frac{d}{dr}u_r+\int_{\partial B_1}\partial_\nu(u_r)\frac{d}{dr}u_r.\]
Since $\Delta u_r=1$ in $\{u_r>0\}$ and $\frac{d}{dr}u_r=0$ in $\{u_r=0\}$, we have
\[\int_{B_1}\nabla u_r\cdot \nabla \frac{d}{dr}u_r=-\int_{B_1}\frac{d}{dr}u_r+\int_{\partial B_1}\partial_\nu(u_r)\frac{d}{dr}u_r.\]
Thus, we deduce
\[\frac{d}{dr}W_u(r)=\int_{\partial B_1}\partial_\nu(u_r)\frac{d}{dr}u_r-2\int_{\partial B_1}u_r\frac{d}{dr}u_r.\]
Using that on $\partial B_1$ we have $\partial_\nu=x\cdot \nabla$, combined with
\[
\frac{d}{dr}u_r=\frac1r\left\{x\cdot\nabla u_r-2u_r\right\}
\]
yields
\[\frac{d}{dr}W_u(r)=\frac1r\int_{\partial B_1}\left(x\cdot\nabla u_r-2u_r\right)^2,\]
which gives the desired result.
\end{proof}

%
%
%

We now give the:

\begin{proof}[First proof of Proposition \ref{cor-Weiss}]
Let $u_r(x)=r^{-2}u(rx)$, and notice that we have the scaling property 
\[W_{u_r}(\rho)=W_u(\rho r),\]
for any $r,\rho>0$.

If $u_0$ is any blow-up of $u$ at $0$ then there is a sequence $r_j\to0$ satisfying $u_{r_j}\to u_0$ in $C^1_{\rm loc}(\R^n)$.
Thus, for any $\rho>0$ we have
\[W_{u_0}(\rho)=\lim_{r_j\to0}W_{u_{r_j}}(\rho)=\lim_{r_j\to0}W_{u}(\rho r_j)=W_u(0^+).\]
Notice that the limit $W_u(0^+):=\lim_{r\to0}W_u(r)$ exists by monotonicity of $W$  and since $u\in C^{1,1}$ implies $W_u(r)\ge - C$ for all $r \ge 0$.

Hence, the function $W_{u_0}(\rho)$ is \emph{constant} in $\rho$.
However, by Theorem \ref{thm-Weiss} this yields that $x\cdot \nabla u_0-2u_0\equiv0$ in $\R^n$, and therefore $u_0$ is homogeneous of degree~$2$.
\end{proof}

\begin{rem}
Here, we used that a $C^1$ function $u_0$ is $2$-homogeneous (i.e. $u_0(\lambda x)=\lambda^2 u_0(x)$ for all $\lambda\in \R_+$) if and only if $x\cdot \nabla u_0\equiv 2u_0$.
This is because $\partial_\lambda|_{\lambda=1}\left\{\lambda^{-2}u_0(\lambda x)\right\}=x\cdot \nabla u_0-2u_0$.
\end{rem}

\subsubsection*{Homogeneity of blow-ups \textit{\`a la} Spruck}
We present an alternative (and quite different) proof of the homogeneity of blow-ups. Such proof is due to Spruck \cite{Spr83} and is not based on any monotonicity formula.

\begin{proof}[Second proof of Proposition~\ref{cor-Weiss}]
Let $u_0$ be a blow-up given by the limit along a sequence $r_k\downarrow 0$, 
\[
u_0(x) := \lim_{k\to \infty} r_k^{-2} u(r_k x). 
\]

By taking polar coordinates $(\varrho, \theta)\in [0, +\infty)\times \mathbb{S}^{n-1}$ with $x = \varrho \theta$, and by denoting $\tilde u_0 (\varrho, \theta) = u_0 (\varrho \theta) = u_0(x)$, we will prove that $u_0(x) = \varrho^2\tilde u_0(1, \theta) = |x|^2u_0(x/|x|)$.

Let us define $\tau := -\log\varrho$, $\tilde u(\varrho, \theta) = u(x)$, and $\psi = \psi(\tau, \theta)$ as 
\[
\psi(\tau, \theta) := \varrho^{-2} \tilde u(\varrho, \theta)= e^{2\tau}u(e^{-\tau}\theta)
\]
for $\tau \ge 0$. We observe that, since $\|u\|_{L^\infty(B_r)}\le C r^2$, $\psi$ is bounded. Moreover,  $\psi \in C^1((0, \infty)\times \mathbb{S}^{n-1})\cap C^2(\{\psi > 0\})$ from the regularity of $u$; and $\partial_\tau \psi$ and $\nabla_\theta \psi$ are not only continuous, but also uniformly bounded in $[0, \infty)\times \mathbb{S}^{n-1}$. Indeed,
\[
\big|\nabla_\theta \psi(\tau, \theta)\big| \le e^{\tau}\big|\nabla u(e^{-\tau}\theta) \big|\le C,
\]
since $\|\nabla u\|_{L^\infty(B_r)}\le C r$ by $C^{1,1}$ regularity and the fact that $\nabla u(0)  =0$. For the same reason we also obtain
\[
\big|\partial_\tau \psi(\tau, \theta)\big| \le 2\psi(\tau, \theta)  +e^{\tau}\big|\nabla u(e^{-\tau}\theta)\big|\le C.
\]
Observe that, by assumption, if we denote $\tau_k := -\log r_k$, 
\begin{equation}
\label{eq.goestozerohi}
\psi(\tau_k, \theta) \to \tilde u_0(1, \theta)\quad  \text{uniformly on $\mathbb{S}^{n-1}$, as $k\to \infty$}.
\end{equation}

Let us now write an equation for $\psi$. In order to do that, since we know that $\Delta u = \chi_{\{u > 0\}}$ and $\chi_{\{u > 0\}} = \chi_{\{\psi > 0\}}$, we have
\[
\Delta \big(\varrho^2 \psi(-\log \varrho, \theta)\big) = \chi_{\{\psi > 0\}}.
\]
By expanding the Laplacian in polar coordinates, $\Delta = \partial_{\varrho \varrho} +\frac{n-1}{\varrho}\partial_\varrho + \varrho^{-2}\Delta_{\mathbb{S}^{n-1}}$ (where $\Delta_{\mathbb{S}^{n-1}}$ denotes the spherical Laplacian, i.e. the Laplace--Beltrami operator on $\mathbb{S}^{n-1}$) we obtain 
\begin{equation}
\label{eq.goingbackto}
2n \psi- (n+2)\partial_\tau \psi+\partial_{\tau\tau}\psi + \Delta_{\mathbb{S}^{n-1}}\psi = \chi_{\{\psi > 0\}}.
\end{equation}

We multiply the previous equality by $\partial_\tau \psi$, and integrate in $[0, \tau]\times\mathbb{S}^{n-1}$. We can consider the terms separately, integrating in $\tau$ first,
\[
2n \int_{\mathbb{S}^{n-1}}\int_{0}^\tau \psi\partial_\tau \psi = n \int_{\mathbb{S}^{n-1}} \big(\psi^2(\tau,\theta) - \psi^2(0, \theta)\big)\, d\theta
\]
and
\[
 \int_{\mathbb{S}^{n-1}}\int_{0}^\tau \partial_{\tau\tau} \psi\partial_\tau \psi = \frac12 \int_{\mathbb{S}^{n-1}} \big((\partial_\tau\psi)^2(\tau,\theta) - (\partial_\tau\psi)^2(0, \theta)\big)\, d\theta,
\]
and then integrating by parts in $\theta$ first, to integrate in $\tau$ afterwards: 
\[
\begin{split}
\int_{0}^\tau \int_{\mathbb{S}^{n-1}}\Delta_{\mathbb{S}^{n-1}} \psi\partial_\tau \psi & = -  \frac12\int_{0}^\tau \int_{\mathbb{S}^{n-1}} \partial_\tau |\nabla_\theta \psi|^2\\
& = \frac12 \int_{\mathbb{S}^{n-1}} \big(|\nabla_\theta \psi|^2(0,\theta) - |\nabla_\theta \psi|^2(\tau, \theta)\big)\, d\theta.
\end{split}
\]
Finally, since $\partial_\tau \psi = 0$ whenever $\psi = 0$, we have $\chi_{\{\psi > 0\}}\partial_\tau \psi = \partial_\tau \psi $ and 
\[
 \int_{\mathbb{S}^{n-1}}\int_{0}^\tau\chi_{\{\psi > 0\}} \partial_\tau \psi = \int_{\mathbb{S}^{n-1}} \big(\psi(\tau,\theta) - \psi(0, \theta)\big)\, d\theta.
\]

In all, plugging back in \eqref{eq.goingbackto} the previous expressions, and using that $\partial_\tau \psi$ and $\nabla_\theta \psi$ are uniformly bounded in $[0, \infty)\times\mathbb{S}^{n-1}$, we deduce that 
\begin{equation}
\label{eq.boundL2partial}
\int_{0}^\infty\int_{\mathbb{S}^{n-1}} (\partial_\tau \psi)^2 =\int_{0}^\infty \|\partial_\tau \psi\|^2_{L^2(\mathbb{S}^{n-1})}  \le C < \infty. 
\end{equation}

To finish, now observe that for any $|s|\le C_*$ fixed and for a sufficiently large $k$ (such that $\tau_k +s \ge 0$), 
\[
\begin{split}
\|\psi(\tau_k + s, \cdot) - \tilde u_0(1, \cdot)\|_{L^2(\mathbb{S}^{n-1})} & \le \|\psi(\tau_k + s, \cdot) - \psi(\tau_k, \cdot)\|_{L^2(\mathbb{S}^{n-1})} \\
& \quad +\|\psi(\tau_k, \cdot) - \tilde u_0(1, \cdot)\|_{L^2(\mathbb{S}^{n-1})}.
\end{split}
\]
The last term goes to zero, by \eqref{eq.goestozerohi}. On the other hand, for the first term and by H\"older's inequality
\[
\begin{split}
\|\psi(\tau_k + s, \cdot) - \psi(\tau_k, \cdot)\|^2_{L^2(\mathbb{S}^{n-1})} & \le \left\|\int_0^s \partial_\tau \psi(\tau_k + \tau, \cdot)\, d\tau\right\|^2_{L^2(\mathbb{S}^{n-1})} \\
& \le C_*\left|\int_{\tau_k}^{\tau_k+s}\|\partial_\tau \psi\|^2_{L^2(\mathbb{S}^{n-1})}\right|\to 0, 
\end{split}
\]
as $k\to \infty$, where we are using \eqref{eq.boundL2partial}. Hence, $\psi(\tau_k + s, \cdot) \to \tilde u_0(1, \cdot)$ in $L^2(\mathbb{S}^{n-1})$ as $k\to \infty$, for any fixed $s\in \R$. On the other hand,
\[
\psi(\tau_k + s, \theta) = e^{2s}r_k^{-2} u(e^{-2}r_k\theta)\to e^{2s}u_0(e^{-s}\theta) = e^{2s}\tilde u_0(e^{-s}, \theta).
\]
 That is, for any $\rho = e^{-s} > 0$, 
\[
\tilde u_0(1, \cdot) = \rho^{-2} \tilde u_0(\rho, \theta),
\]
as we wanted to see. 
\end{proof}

\subsection*{Convexity of blow-ups}

By taking advantage of the fact that we know that blow-ups are 2-homogeneous, we can now give a short (and new) proof of the fact that they are also convex. More precisely, we will prove that 2-homogeneous global solutions to the obstacle problem are convex (and in particular, by Proposition~\ref{cor-Weiss}, blow-ups are convex). 

\begin{thm}
\label{ch4_thm_convexity}
Let $u_0\in C^{1,1}$ be any $2$-homogeneous global solution to 
\[\left\{
\begin{array}{rcll}
u_0&\geq&0&\quad \textrm{in}\ \R^n
\vspace{1mm}
\\
\Delta u_0&=&1&\quad \textrm{in}\ \{u_0>0\} 
\vspace{1mm}
\\
&&& \hspace{-2.0cm} 0\ \textrm{is a free boundary point}.
\end{array}
\right. 
\]
Then, $u_0$ is convex.
\end{thm}

The heuristic idea behind the proof of the previous result is the following: second derivatives $D^2 u_0$ are harmonic in $\{u_0 > 0\}$ and satisfy that $D^2 u_0 \ge 0 $ on $\partial\{u_0 > 0\}$ (since $u_0\ge 0$,  it is ``convex at the free boundary''). Since $D^2u_0$ is also 0-homogeneous, we can apply the maximum principle and conclude that $D^2 u_0 \ge 0$ everywhere. That is, $u_0$ is convex. Let us formalize the previous heuristic idea into an actual proof. 

%
%

We state a short lemma before providing the proof, which says that if $w\ge 0$ is superharmonic in $\{w > 0\}$, then it is superharmonic everywhere. For the sake of generality, we state the lemma for general $H^1$ functions, but we will use it only for functions that are also continuous. 

\begin{lem}
\label{lem.second_convex}
Let $\Lambda\subset B_1$ be closed. Let $w\in H^1(B_1)$ be such that $w \ge 0$ on $\Lambda$ and such that $w$ is superharmonic in the weak sense in $B_1 \setminus \Lambda$. Then $\min\{w, 0\}$ is superharmonic in the weak sense in $B_1$.
\end{lem}
\begin{proof}
Let us start by assuming that $w$ is, furthermore, continuous. In this case, we define $w_\eps = \min\{w, -\eps\}\in H^1(B_1)$. Then notice that (by continuity) in a neighborhood of $\{w = -\eps\}$, $w$ is superharmonic ($\Delta w \le 0$). By Lemma~\ref{lem.pospartSH} (we apply the lemma with $v = -w-\eps$) we have that $\Delta w_\eps \le 0$ in the weak sense, namely, $w_\eps$ is superharmonic. Moreover, they are uniformly in $H^1$, so up to subsequences they converge weakly to $\min\{w, 0\}$. Since the weak limit of weakly superharmonic functions is superharmonic, we deduce  the desired result. 

Finally, to remove the continuity assumption on $w\in H^1(B_1)$, we repeat the proof of Lemma~\ref{lem.pospartSH}. The only thing we need to check is that $F'(v)\eta \in H^1_0(B_1\setminus\Lambda)$, which follows from the fact that such function is in $H^1(B_1)$ and vanishes in $\Lambda$; see for example \cite[Theorem 9.1.3]{AH95}.
\end{proof}

We now give the:

\begin{proof}[Proof of Theorem~\ref{ch4_thm_convexity}]
Let $e\in \mathbb{S}^{n-1}$ and consider the second derivatives $\partial_{ee} u_0$. We define 
\[
w_0 := \min\{\partial_{ee} u_0, 0\}
\]
and we claim that $w_0$ is superharmonic in $\R^n$, in the sense \eqref{eq.superharmonic_integral}.

Indeed, let $\delta_t^2 u_0(x)$ for $t > 0$ be defined by 
\[
\delta_t^2 u_0(x) := \frac{u_0(x+te)+u_0(x-te)-2u_0(x)}{t^2}. 
\]
Now, since $\Delta u_0 = \chi_{\{u_0  >0\}}$, we have that 
\[
\Delta \delta_t^2 u_0 = \frac{1}{t^2} \big(\chi_{\{u_0(\,\cdot\,+te)\}}+\chi_{\{u_0(\,\cdot\,-te)\}}-2\big) \le 0\quad\text{in}\quad \{u_0 > 0\}
\]
in the weak sense. On the other hand, $\delta_t^2 u_0 \ge 0$ in $\{u_0 = 0\}$ and $\delta_t^2 u_0 \in C^{1,1}$. Thus, by Lemma~\ref{lem.second_convex}, $w_t := \min\{\delta_t^2 u_0, 0\}$ is weakly superharmonic, and hence it satisfies \eqref{eq.superharmonic_integral}. Also notice that $\delta_t^2 u_0(x)$ is uniformly bounded independently of $t$, since $u_0\in C^{1,1}$, and therefore $w_{t}$ is uniformly bounded in $t$ and converges pointwise to $w_0$ as $t\downarrow 0$. In particular, by Lemma~\ref{lem.convergence_pointwise} we have that $w_0$ is superharmonic in the sense of \eqref{eq.superharmonic_integral}, as claimed.

Up to changing it in a set of measure 0, $w_0$ is lower semi-continuous by Lemma~\ref{lem.lower_semi}. In particular, since $w_0$ is 0-homogeneous, it must attain its minimum at a point $y_\circ\in B_1$. But since $\ave_{B_r(y_\circ)} w_0$ is non-increasing  for $r > 0$, we must have that $w_0$ is constant. Since it vanishes on the free boundary, we have $w_0 \equiv 0$. That is, for any $e\in \mathbb{S}^{n-1}$ we have that $\partial_{ee} u_0 \ge 0$ and therefore $u_0$ is convex. 
\end{proof}

\begin{rem}[Convexity of blow-ups \emph{\`a la} Caffarelli]
The original proof by Caffarelli on the convexity of blow-ups, \cite{Caf, Caf98},  is more involved than the previous one, but obtains a quantitative estimate on the convexity without using the homogeneity assumption (in particular, it is valid for any global solution). 

More precisely, for any solution $u$ to \eqref{ch4-obst-f=1} in $B_1$
\[\qquad \qquad \partial_{ee}u(x)\geq -\frac{C}{\bigr|\log|x|\bigr|^{\varepsilon}}\qquad  \text{for all}\quad e\in \mathbb{S}^{n-1},~~ x \in B_{1/2}, \]
for some $\varepsilon>0$. Notice that $C\bigr|\log|x|\bigr|^{-\varepsilon}\to0$ as $x\to0$.
Thus, $u$ becomes closer and closer to being convex as we approach to the free boundary. Rescaling this result to $B_R$, and letting $R\to\infty$, this implies that any global solution is convex.
\end{rem}

Finally, we refer to \cite[Theorem 5.1]{PSU} for yet another different proof of the convexity of blow-ups.

\subsection*{Classification of blow-ups}

We next want to classify all possible blow-ups for solutions to the obstacle problem \eqref{ch4-obst-f=1}.
First, we will prove the following.

\begin{prop}\label{ch4-prop-blowups}
Let $u$ be any solution to \eqref{ch4-obst-f=1}, and let 
\[u_r(x):=\frac{u(rx)}{r^2}.\]
Then, for any sequence $r_k\to0$ there is a subsequence $r_{k_j}\to0$ such that
\[u_{r_{k_j}}\longrightarrow u_0\quad\textrm{in}\ C^1_{\rm loc}(\R^n)\]
as $k_j\to\infty$, for some function $u_0$ satisfying 
\[
\left\{\begin{array}{l}
u_0\in C^{1,1}_{\rm loc}(\R^n)
\vspace{1mm}
\\
u_0\geq0\quad \textrm{in}\ B_1 
\vspace{1mm}
\\
\Delta u_0=1\quad \textrm{in}\ \{u_0>0\} 
\vspace{1mm}
\\
0\ \textrm{is a free boundary point}
\vspace{1mm}
\\
u_0\ \textrm{is convex}
\\
u_0\ \textrm{is homogeneous of degree 2}.
\end{array}\right.
\]
\end{prop}

\begin{proof}
By $C^{1,1}$ regularity of $u$, and by nondegeneracy, we have that
\[\frac{1}{C}\leq \sup_{B_1}u_r \leq C\]
for some $C>0$.
Moreover, again by $C^{1,1}$ regularity of $u$, we have
\[\|D^2u_r\|_{L^\infty(B_{1/(2r)})}\leq C.\]

Since the sequence $\{u_{r_k}\}$, for $r_k\to0$, is uniformly bounded in $C^{1,1}(K)$ for each compact set $K\subset \R^n$,   there is a subsequence $r_{k_j}\to0$ such that 
\[u_{r_{k_j}}\longrightarrow u_0\quad\textrm{in}\ C^1_{\rm loc}(\R^n)\]
for some $u_0\in C^{1,1}(K)$.
Moreover, such function $u_0$ satisfies $\|D^2u_0\|_{L^\infty(K)}\leq C$, with $C$ independent of $K$, and clearly $u_0\geq0$ in $K$.

The fact that $\Delta u_0=1$ in $\{u_0>0\}\cap K$ can be checked as follows.
For any smooth function $\eta\in C^\infty_c(\{u_0>0\}\cap K)$ we will have that, for $k_j$ large enough, $u_{r_{k_j}}>0$ in the support of $\eta$, and thus 
\[\int_{\R^n} \nabla u_{r_{k_j}}\cdot \nabla \eta\,dx=-\int_{\R^n} \eta\,dx.\]
Since $u_{r_{k_j}}\to u_0$ in $C^1(K)$, we can take the limit $k_j\to\infty$ to get 
\[\int_{\R^n} \nabla u_0\cdot \nabla \eta\,dx=-\int_{\R^n} \eta\,dx.\]
Since this can be done for any $\eta\in C^\infty_c(\{u>0\}\cap K)$, and for every $K\subset \R^n$, it follows that $\Delta u_0=1$ in $\{u_0>0\}$.

The fact that $0$ is a free boundary point for $u_0$ follows simply by taking limits to $u_{r_{k_j}}(0)=0$ and $\|u_{r_{k_j}}\|_{L^\infty(B_\rho)}\approx \rho^2$ for all $\rho\in (0,1)$.
Finally, the homogeneity and convexity of $u_0$ follow from Proposition~\ref{cor-Weiss} and Theorem~\ref{ch4_thm_convexity}.
\end{proof}

Our next goal is to prove the following.

\begin{thm}[Classification of blow-ups]\label{thm-classification-blowups}\index{Classification of blow-ups}
Let $u$ be any solution to \eqref{ch4-obst-f=1}, and let $u_0$ be any blow-up of $u$ at $0$.
Then,
\begin{itemize}
\item[(a)] either
\[u_0(x)=\frac12(x\cdot e)_+^2\]
for some $e\in \mathbb{S}^{n-1}$.

\item[(b)] or
\[u_0(x)=\frac12x^TAx\]
for some matrix $A\geq0$ with ${\rm tr}\,A=1$.
\end{itemize}
\end{thm}

It is important to remark here that, a priori, different subsequences could lead to different blow-ups $u_0$.

In order to establish Theorem \ref{thm-classification-blowups}, we will need the following.

\begin{lem}\label{lem-cone-homog1}
Let $\Sigma\subset\R^n$ be any closed convex cone with nonempty interior, and with vertex at the origin.
Let $w\in C(\R^n)$ be a function satisfying $\Delta w=0$ in $\Sigma^c$, $w>0$ in $\Sigma^c$, and $w=0$ in $\Sigma$.

Assume in addition that $w$ is homogeneous of degree 1.
Then, $\Sigma$ must be a half-space.
\end{lem}

\begin{proof}
By convexity of $\Sigma$, there exists a half-space $H=\{x\cdot e>0\}$, with $e\in \mathbb{S}^{n-1}$,  such that $H\subset \Sigma^c$.

Let $v(x)=(x\cdot e)_+$, which is harmonic and positive in $H$, and vanishes in~$H^c$.
By the Hopf Lemma (see Lemma~\ref{Hopf}), we have that $w\geq c_\circ d_\Sigma$ in $\Sigma^c\cap B_1$, where $d_\Sigma(x)={\rm dist}(x,\Sigma)$ and $c_\circ $ is a small positive constant.
In particular, since both $w$ and $d_\Sigma$ are homogeneous of degree 1, we deduce that $w\geq c_\circ d_\Sigma$ in all of $\Sigma^c$.
Notice that, in order to apply the Hopf Lemma, we used that --- by convexity of $\Sigma$ --- the domain $\Sigma^c$ satisfies the interior ball condition.

Thus, since $d_\Sigma\geq d_{H^c}= v$, we deduce that $w\geq c_\circ v$, for some $c_\circ >0$.
The idea is now to consider the functions $w$ and $cv$, and let $c>0$ increase until the two functions touch at one point, which will give us a contradiction (recall that two harmonic functions cannot touch at an interior point).
To do this rigorously, define 
\[c_*:=\sup\{c>0\,:\, w\geq cv\quad\textrm{in}\quad\Sigma^c\}.\]
Notice that $c_*\geq c_\circ >0$.
Then, we consider the function $w-c_*v\geq0$. 
Assume that $w-c_*v$ is not identically zero.
Such function is harmonic in $H$ and hence, by the strict maximum principle, $w-c_*v>0$ in $H$.
Then, using  the Hopf Lemma in $H$ (see Lemma~\ref{Hopf}) we deduce that $w-c_*v\geq c_\circ d_{H^c}=c_\circ v$, since $v$ is exactly the distance to $H^c$.
But then we get that $w-(c_*+c_\circ )v\geq0$, a contradiction with the definition of $c_*$.

Therefore, it must be $w-c_*v\equiv0$.
This means that $w$ is a multiple of $v$, and therefore $\Sigma=H^c$, a half-space.
\end{proof}

\begin{rem}[Alternative proof]
An alternative way to argue in the previous lemma could be the following.
Any function $w$ which is harmonic in a cone $\Sigma^c$ and homogeneous of degree $\alpha$ can be written as a function on the sphere, satisfying $\Delta_{\mathbb{S}^{n-1}}w=\mu w$ on $\mathbb{S}^{n-1}\cap \Sigma^c$ with $\mu=\alpha(n+\alpha-2)$ --- in our case $\alpha=1$. (Here, $\Delta_{\mathbb{S}^{n-1}}$ denotes the spherical Laplacian, i.e. the Laplace--Beltrami operator on $\mathbb{S}^{n-1}$.)
In other words, \emph{homogeneous harmonic functions solve an eigenvalue problem on the sphere}.

Using this, we notice that $w>0$ in $\Sigma^c$ and $w=0$ in $\Sigma$  imply that $w$ is the \emph{first} eigenfunction of $\mathbb{S}^{n-1}\cap \Sigma^c$, and that the first eigenvalue is $\mu=n-1$.
But, on the other hand, the same happens for the domain $H=\{x\cdot e>0\}$, since $v(x)=(x\cdot e)_+$ is a positive harmonic function in $H$.
This means that both domains $\mathbb{S}^{n-1}\cap \Sigma^c$ and $\mathbb{S}^{n-1}\cap H$ have the same first eigenvalue $\mu$.
But then, by strict monotonicity of the first eigenvalue with respect to domain inclusions, we deduce that $H\subset \Sigma^c$ implies $H=\Sigma^c$, as desired.
\end{rem}

We will also need the following.

\begin{lem}\label{ch4-removable}
Assume that $\Delta u=1$ in $\R^n\setminus \partial H$, where $\partial H$ is a hyperplane.
If $u\in C^1(\R^n)$, then $\Delta u=1$ in $\R^n$.
\end{lem}

\begin{proof}
Assume  $\partial H=\{x_1=0\}$.
For any ball $B_R\subset\R^n$, we consider the solution to $\Delta w=1$ in $B_R$, $w=u$ on $\partial B_R$, and define $v=u-w$.
Then, we have $\Delta v=0$ in $B_R\setminus \partial H$, and $v=0$ on $\partial B_R$.
We want to show that $u$ coincides with $w$, that is, $v\equiv0$ in $B_R$.

For this, notice that since $v$ is bounded, for $\kappa>0$ large enough we have
\[v(x)\leq \kappa(2R-|x_1|) \quad \textrm{in}\quad B_R,\]
where $2R-|x_1|$ is positive in $B_R$ and harmonic in $B_R\setminus\{x_1=0\}$.
Thus, we may consider $\kappa^*:=\inf\{\kappa\geq0: v(x)\leq \kappa(2R-|x_1|) \quad \textrm{in}\quad B_R\}$.
Assume $\kappa^*>0$.
Since $v$ and $2R-|x_1|$ are continuous in $B_R$, and $v=0$ on $\partial B_R$, we must have a point $p\in B_R$ at which $v(p)=\kappa^* (2R-|p_1|)$.
Moreover, since $v$ is $C^1$, and the function $2R-|x_1|$ has a wedge on $\partial H=\{x_1=0\}$,  we must have $p\in B_R\setminus \partial H$.
However, this is not possible, as two harmonic functions cannot touch tangentially at an interior point $p$.
This means that $\kappa^*=0$, and hence $v\leq0$ in $B_R$.
Repeating the same argument with $-v$ instead of $v$, we deduce that $v\equiv0$ in $B_R$, and thus the lemma is proved.
\end{proof}

Finally, we will use the following basic property of convex functions.

\begin{lem}\label{convex-1D}
Let $u: \R^n \rightarrow \R$ be a convex function such that the set $\{u=0\}$ contains the straight line $\{te'\,:\, t\in \R\}$,  $e'\in \mathbb{S}^{n-1}$.
Then, $u(x+te')=u(x)$ for all $x\in \R^n$ and all $t\in\R$.
\end{lem}

\begin{proof}
After a rotation, we may assume $e'=e_n$.
Then, writing $x=(x',x_n)\in \R^{n-1}\times \R$, we have that $u(0,x_n)=0$ for all $x_n\in \R$, and we want to prove that $u(x',x_n)=u(x',0)$ for all $x'\in \R^{n-1}$ and all $x_n\in \R$.

Now, by convexity, given $x'$ and $x_n$, for every $\varepsilon>0$ and $M\in \R$ we have
\[(1-\varepsilon)u(x',x_n)+\varepsilon u(0,x_n+M)\geq u((1-\varepsilon)x',x_n+\varepsilon M).\]
Since $u(0,x_n+M)=0$, choosing $M=\lambda/\varepsilon$ and letting $\varepsilon\to0$ we deduce that 
\[u(x',x_n)\geq u(x',x_n+\lambda).\]
Since this can be done for any $\lambda\in \R$ and $x_n\in\R$, the result follows.
\end{proof}

We finally establish the classification of blow-ups at regular points.

\begin{proof}[Proof of Theorem \ref{thm-classification-blowups}]
Let $u_0$ be any blow-up of $u$ at $0$.
We already proved that $u_0$ is convex and homogeneous of degree 2.
We divide the proof into two cases.

\noindent{\it \underline{\smash{Case 1}}.}
Assume that $\{u_0=0\}$ has nonempty interior.
Then, we have $\{u_0=0\}=\Sigma$, a closed convex cone with nonempty interior.

For any direction $\tau\in \mathbb{S}^{n-1}$ such that $-\tau\in \mathring{\Sigma}$, we claim that 
\[\partial_\tau u_0\geq0 \quad \textrm{in}\quad \R^n.\]
Indeed, for every $x\in \R^n$ we have that $u_0(x+\tau t)$ is zero for $t\ll -1$, and therefore by convexity of $u_0$ we get that $\partial_t u_0(x+\tau t)$ is monotone non-decreasing in $t$, and zero for $t\ll-1$.
This means that $\partial_tu_0\geq0$, and thus $\partial_\tau u_0\geq0$ in $\R^n$, as claimed.

Now, for any such $\tau$, we define $w:=\partial_\tau u_0\geq0$.
Notice that, at least for some $\tau\in \mathbb{S}^{n-1}$ with $-\tau\in \mathring{\Sigma}$, the function $w$ is not identically zero.
Moreover, since it is harmonic in $\Sigma^c$ --- recall that $\Delta u_0=1$ in $\Sigma^c$ --- then $w>0$ in $\Sigma^c$.

But then, since $w$ is homogeneous of degree 1, we can apply Lemma~\ref{lem-cone-homog1} to deduce that we must necessarily have that $\Sigma$ is a half-space.

By convexity of $u_0$ and Lemma~\ref{convex-1D}, this means that $u_0$ is a one-dimensional function, i.e., $u_0(x)=U(x\cdot e)$ for some $U:\R\to\R$ and some $e\in \mathbb{S}^{n-1}$.
Thus, we have that $U\in C^{1,1}$ solves $U''(t)=1$ for $t>0$, with $U(t)=0$ for $t\leq0$.
We deduce that $U(t)=\frac12 t_+^2$, and therefore $u_0(x)=\frac12(x\cdot e)_+^2$.

\vspace{1mm}

\noindent{\it \underline{\smash{Case 2}}.}
Assume now that $\{u_0=0\}$ has empty interior.
Then, by convexity, $\{u_0=0\}$ is contained in a hyperplane $\partial H$.
Hence, $\Delta u_0=1$ in $\R^n\setminus \partial H$, with $\partial H$ being a hyperplane, and $u_0\in C^{1,1}$.
It follows from Lemma~\ref{ch4-removable} that $\Delta u_0=1$ in all of $\R^n$.
But then all second derivatives of $u_0$ are harmonic and globally bounded in $\R^n$, so they must be constant.
Hence, $u_0$ is a quadratic polynomial.
Finally, since $u_0(0)=0$, $\nabla u_0(0)=0$, and $u_0\geq0$, we deduce that $u_0(x)=\frac12x^TAx$ for some $A\geq0$, and since $\Delta u_0=1$, we have ${\rm tr}\,A=1$.
\end{proof}

\section{Regularity of the free boundary}

The aim of this Section is to prove Theorem \ref{ch4-FB-smooth} below, i.e., that if $u$ is any solution to \eqref{ch4-obst-f=1} satisfying
\begin{equation}\label{ch4-positive-density}
\limsup_{r\to0}\frac{\bigl|\{u=0\}\cap B_r\bigr|}{|B_r|}>0
\end{equation}
(i.e., the contact set has positive density at the origin), 
then the free boundary $\partial\{u>0\}$ is $C^\infty$ in a neighborhood of the origin.

For this, we will use the classification of blow-ups established in the previous Section.

\subsection*{$C^{1,\alpha}$ regularity of the free boundary}

The first step here is to transfer the local information on $u$ given by \eqref{ch4-positive-density} into a blow-up $u_0$.
More precisely, we next show that
\[\eqref{ch4-positive-density}\qquad \Longrightarrow\qquad \begin{array}{cc}
\textrm{The contact set of a blow-up}\ u_0 \\
\textrm{has nonempty interior}.
\end{array}\]

\begin{lem}\label{ch4-lem-positive-density}
Let $u$ be any solution to \eqref{ch4-obst-f=1}, and assume that \eqref{ch4-positive-density} holds.
Then, there is at least one blow-up $u_0$ of $u$ at $0$ such that the contact set  $\{u_0=0\}$ has nonempty interior.
\end{lem}

\begin{proof}
Let $r_k\to0$ be a sequence along which
\[\lim_{r_k\to0}\frac{\bigl|\{u=0\}\cap B_{r_k}\bigr|}{|B_{r_k}|}\geq \theta>0.\]
Such sequence exists (with $\theta>0$ small enough) by assumption \eqref{ch4-positive-density}.

Recall that, thanks to Proposition \ref{ch4-prop-blowups}, there exists a subsequence $r_{k_j}\downarrow 0$ along which $u_{r_{k_j}}\to u_0$ uniformly on compact sets of $\R^n$, where $u_r(x)=r^{-2}u(rx)$ and $u_0$ is convex.

Assume by contradiction that $\{u_0=0\}$ has empty interior.
Then, by convexity, we have that $\{u_0=0\}$ is contained in a hyperplane, say $\{u_0=0\}\subset \{x_1=0\}$.

Since $u_0>0$ in $\{x_1\neq 0\}$ and $u_0$ is continuous, we have that for each $\delta>0$
\[u_0\geq\varepsilon>0\quad \textrm{in}\ \{|x_1|>\delta\}\cap B_1\]
for some $\varepsilon>0$.

Therefore, by uniform convergence of $u_{r_{k_j}}$ to $u_0$ in $B_1$, there is $r_{k_j}>0$ small enough such that
\[u_{r_{k_j}}\geq \frac{\varepsilon}{2}>0\quad \textrm{in}\ \{|x_1|>\delta\}\cap B_1.\]
In particular, the contact set of $u_{r_{k_j}}$ is contained in $\{|x_1|\leq\delta\}\cap B_1$, so
\[\frac{\bigl|\{u_{r_{k_j}}=0\}\cap B_1\bigr|}{|B_1|}\leq \frac{\bigl|\{|x_1|\leq\delta\}\cap B_1\bigr|}{|B_1|}\leq C\delta.\]
Rescaling back to $u$, we find
\[\frac{\bigl|\{u=0\}\cap B_{r_{k_j}}\bigr|}{|B_{r_{k_j}}|} 
=\frac{\bigl|\{u_{r_{k_j}}=0\}\cap B_1\bigr|}{|B_1|}<C\delta.\]
Since we can do this for every $\delta>0$, we find that 
$\lim_{r_{k_j}\to0}\frac{|\{u=0\}\cap B_{r_{k_j}}|}{|B_{r_{k_j}}|}=0$, a contradiction.
Thus, the lemma is proved.
\end{proof}

Combining the previous lemma with the classification of blow-ups from the previous Section, we deduce:

\begin{cor}\label{ch4-lem-one-blowup}
Let $u$ be any solution to \eqref{ch4-obst-f=1}, and assume that \eqref{ch4-positive-density} holds.
Then, there is at least one blow-up of $u$ at $0$ of the form
\[
\qquad\qquad u_0(x)=\frac12(x\cdot e)_+^2,\qquad\qquad e\in \mathbb{S}^{n-1}.
\]
\end{cor}

\begin{proof}
The result follows from Lemma~\ref{ch4-lem-positive-density} and Theorem \ref{thm-classification-blowups}.
\end{proof}

We now want to use this information to show that the free boundary must be smooth in a neighborhood of $0$.
For this, we start with the following.

\begin{prop}\label{ch4-FBreg-prop}
Let $u$ be any solution to \eqref{ch4-obst-f=1}, and assume that \eqref{ch4-positive-density} holds.
Fix any $\varepsilon>0$.
Then, there exist $e\in \mathbb{S}^{n-1}$ and $r_\circ >0$ such that
\[\qquad \bigl|u_{r_\circ }(x) - {\textstyle\frac12}(x\cdot e)_+^2\bigr|\leq \varepsilon\qquad\textrm{in}\quad B_1,\]
and
\[\qquad \bigl|\partial_\tau u_{r_\circ }(x) - (x\cdot e)_+(\tau\cdot e)\bigr|\leq \varepsilon\qquad\textrm{in}\quad B_1\]
for all $\tau\in \mathbb{S}^{n-1}$.
\end{prop}

\begin{proof}
By Corollary \ref{ch4-lem-one-blowup} and Proposition~\ref{ch4-prop-blowups}, we know that there is a subsequence $r_j\to0$ for which $u_{r_j}\to \frac12(x\cdot e)^2_+$ in $C^1_{\rm loc}(\R^n)$, for some $e\in \mathbb{S}^{n-1}$.
In particular, for every $\tau\in \mathbb{S}^{n-1}$ we have $u_{r_j}\to \frac12(x\cdot e)^2_+$ and $\partial_\tau u_{r_j}\to \partial_\tau \bigl[\frac12(x\cdot e)^2_+\bigr]$ uniformly in $B_1$.

This means that, given $\varepsilon>0$, there exists $j_\circ $ such that 
\[\qquad \bigl|u_{r_{j_\circ }}(x) - {\textstyle\frac12}(x\cdot e)_+^2\bigr|\leq \varepsilon\qquad\textrm{in}\quad B_1,\]
and
\[\qquad \left|\partial_\tau u_{r_{j_\circ }}(x) - \partial_\tau \bigl[{\textstyle \frac12}(x\cdot e)^2_+\bigr]\right|\leq \varepsilon\qquad\textrm{in}\quad B_1.\]
Since $\partial_\tau \bigl[{\textstyle \frac12}(x\cdot e)^2_+\bigr]=(x\cdot e)_+(\tau\cdot e)$,  the proposition is proved.
\end{proof}

Now, notice that if $(\tau\cdot e)>0$, then the derivatives $\partial_\tau u_0 = (x\cdot e)_+(\tau\cdot e)$ are \emph{nonnegative}, and strictly positive in $\{x\cdot e>0\}$ (see Figure~\ref{fig.27}).

\begin{figure}
\includegraphics[scale = 1.3]{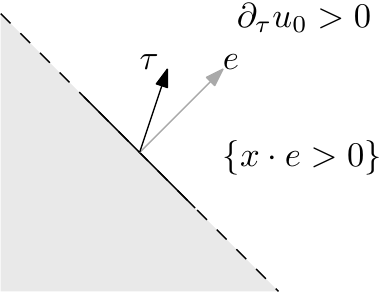}
\caption{Derivatives $\de_\tau u_0$ are nonnegative if $\tau\cdot e \ge \frac12$.}
\label{fig.27}
\end{figure}

We want to transfer this information to $u_{r_\circ }$, and prove that $\partial_\tau u_{r_\circ } \geq0$ in $B_1$ for all $\tau\in \mathbb{S}^{n-1}$ satisfying $\tau\cdot e\geq\frac12$.
For this, we need a lemma.

\begin{lem}\label{ch4-lem-almost-positive}
Let $u$ be any solution to \eqref{ch4-obst-f=1}, and consider $u_{r_\circ }(x)=r_\circ ^{-2}u(r_\circ x)$ and $\Omega=\{u_{r_\circ }>0\}$.

Assume that a function $w\in C(B_1)$ satisfies:
\begin{itemize}
\item[(a)] $w$ is bounded and harmonic in $\Omega\cap B_1$.
\item[(b)] $w=0$ on $\partial\Omega\cap B_1$.
\item[(c)] Denoting $N_\delta:=\{x\in B_1:{\rm dist}(x,\partial\Omega)<\delta\}$, we have
\[w\geq -c_1 \quad\textrm{in}\quad N_\delta\qquad\quad\textrm{and}\quad\qquad w\geq C_2>0\quad\textrm{in}\quad \Omega\setminus N_\delta.\]
\end{itemize}
If $c_1/C_2$ is small enough, and $\delta>0$ is small enough, then $w\geq0$ in $B_{1/2}\cap \Omega$.
\end{lem}

\begin{proof}
Notice that in $\Omega\setminus N_\delta$ we already know that $w>0$.
Let $y_\circ \in N_\delta\cap\Omega\cap B_{1/2}$, and assume by contradiction that $w(y_0)<0$.

Consider, in $B_{1/4}(y_\circ )$, the function
\[v(x)=w(x)-\gamma\left\{u_{r_\circ }(x)-\frac{1}{2n}|x-y_\circ |^2\right\}.\]
Then, $\Delta v=0$ in $B_{1/4}(y_\circ )\cap \Omega$, and $v(y_\circ )<0$.
Thus, $v$ must have a negative minimum in $\partial\bigl(B_{1/4}(y_\circ )\cap \Omega\bigr)$.

However, if $c_1/C_2$ and $\delta$ are small enough, then we reach a contradiction as follows:

On $\partial\Omega$ we have $v\geq0$.
On $\partial B_{1/4}(y_\circ )\cap N_\delta$ we have
\[v\geq -c_1-C_\circ \gamma\delta^2+\frac{\gamma}{2n}\left({\frac14}\right)^2\geq0 \quad \textrm{on}\quad \partial B_{1/4}(y_\circ )\cap N_\delta.\]
On $\partial B_{1/4}(y_\circ )\cap \bigl(\Omega\setminus N_\delta\bigr)$ we have
\[v\geq C_2-C_\circ \gamma\geq0 \quad \textrm{on}\quad \partial B_{1/4}(y_\circ )\cap \bigl(\Omega\setminus N_\delta\bigr).\]
Here, we used that $\|u_{r_\circ }\|_{C^{1,1}(B_1)}\leq C_\circ $, and chose $C_\circ c_1\leq \gamma\leq C_2/C_\circ $.
\end{proof}

Using the previous lemma, we can now show that there is a cone of directions $\tau$ in which the solution is monotone near the origin.

\begin{prop}\label{ch4-monotone-directional}
Let $u$ be any solution to \eqref{ch4-obst-f=1}, and assume that \eqref{ch4-positive-density} holds.
Let $u_r(x)=r^{-2}u(rx)$.
Then, there exist $r_\circ >0$ and $e\in \mathbb{S}^{n-1}$ such that 
\[\partial_\tau u_{r_\circ }\geq0\quad\textrm{in}\quad B_{1/2}\]
for every $\tau\in \mathbb{S}^{n-1}$ satisfying $\tau\cdot e\geq\frac12$.
\end{prop}

\begin{proof}
By Proposition \ref{ch4-FBreg-prop}, for any $\varepsilon>0$ there exist $e\in \mathbb{S}^{n-1}$ and $r_\circ >0$ such that
\begin{equation}\label{ch4-circ-1}
\qquad \bigl|u_{r_\circ }(x) - {\textstyle\frac12}(x\cdot e)_+^2\bigr|\leq \varepsilon\qquad\textrm{in}\quad B_1
\end{equation}
and
\begin{equation}\label{ch4-circ-2}
\qquad \bigl|\partial_\tau u_{r_\circ }(x) - (x\cdot e)_+(\tau\cdot e)\bigr|\leq \varepsilon\qquad\textrm{in}\quad B_1
\end{equation}
for all $\tau\in \mathbb{S}^{n-1}$.

We now want to use Lemma \ref{ch4-lem-almost-positive} to deduce that $\partial_\tau u_{r_\circ }\geq0$ if $\tau\cdot e\geq\frac12$.
First, we claim that 
\[u_{r_\circ }>0\quad \textrm{in}\quad \{x\cdot e>C_\circ \sqrt{\varepsilon}\},\]
\begin{equation}\label{ch4-use-CS}
u_{r_\circ }=0\quad \textrm{in}\quad \{x\cdot e<-C_\circ \sqrt{\varepsilon}\},
\end{equation}
and therefore the free boundary $\partial\Omega=\partial\{u_{r_\circ }>0\}$ is contained in the strip $\{|x\cdot e|\leq C_\circ \sqrt{\varepsilon}\}$, for some $C_\circ $ depending only on $n$ (see Figure~\ref{fig.28}).
To prove this, notice that if $x\cdot e>C_\circ \sqrt{\varepsilon}$ then
\[u_{r_\circ }>\frac12(C_\circ \sqrt{\varepsilon})^2-\varepsilon>0,\]
while if there was a free boundary point $x_\circ $ in $\{x\cdot e<-C_\circ \varepsilon\}$ then by nondegeneracy we would get 
\[\sup_{B_{C_\circ \sqrt{\varepsilon}}(x_\circ )} u_{r_\circ }\geq c(C_\circ \sqrt{\varepsilon})^2>2\varepsilon,\]
a contradiction with \eqref{ch4-circ-1}.

\begin{figure}
\includegraphics{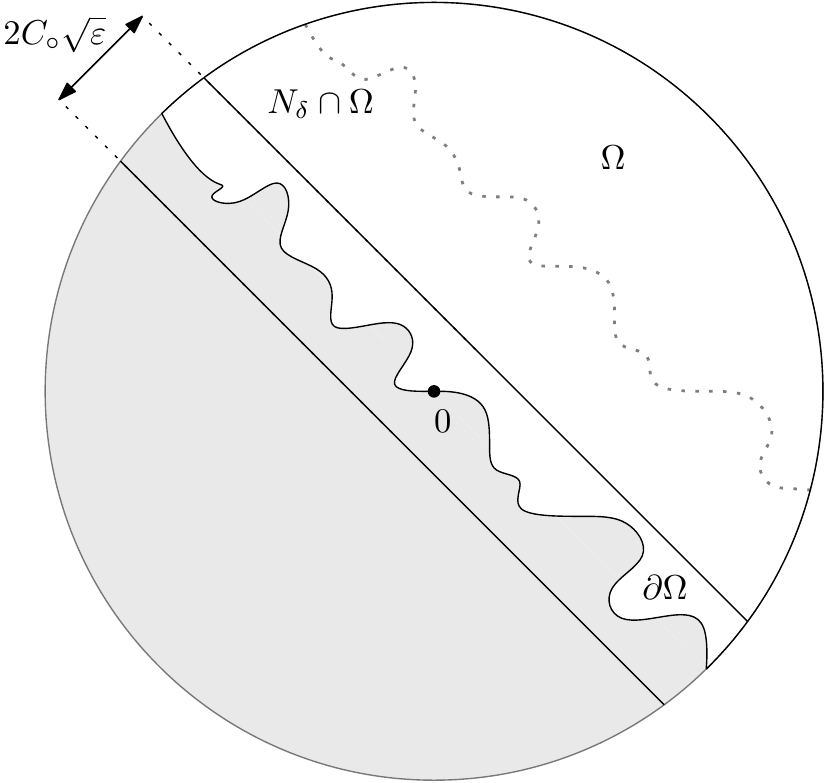}
\caption{The setting in  which we use Lemma~\ref{ch4-lem-almost-positive}.}
\label{fig.28}
\end{figure}

Therefore, we have 
\[
\partial\Omega\subset \{|x\cdot e|\leq C_\circ \sqrt{\varepsilon}\}.
\]
Now, for each $\tau\in \mathbb{S}^{n-1}$ satisfying $\tau\cdot e\geq\frac12$ we define
\[w:=\partial_\tau u_{r_\circ }.\]
In order to use Lemma \ref{ch4-lem-almost-positive}, we notice:
\begin{itemize}
\item[(a)] $w$ is bounded and harmonic in $\Omega\cap B_1$.
\item[(b)] $w=0$ on $\partial\Omega\cap B_1$.
\item[(c)] Thanks to \eqref{ch4-circ-2}, if $\delta\gg\sqrt{\varepsilon}$ then $w$ satisfies
\[w\geq -\varepsilon \quad\textrm{in}\quad N_\delta\]
and
\[w\geq \delta/4>0\quad\textrm{in}\quad (\Omega\setminus N_\delta)\cap B_1.\]
\end{itemize}
(We recall $N_\delta:=\{x\in B_1:{\rm dist}(x,\partial\Omega)<\delta\}$.)

Indeed, to check the last inequality we use that, by \eqref{ch4-use-CS}, we have $\{x\cdot e<\delta-C_\circ \sqrt{\varepsilon}\}\cap \Omega\subset N_\delta$.
Thus, by \eqref{ch4-circ-2}, we get that for all $x\in (\Omega\setminus N_\delta)\cap B_1$
\[w\geq \frac12(x\cdot e)_+-\varepsilon\geq \frac12\delta-\frac12C_\circ \sqrt{\varepsilon}-\varepsilon\geq \frac14\delta,\]
provided that $\delta\gg\sqrt{\varepsilon}$.

Using (a)-(b)-(c), we deduce from Lemma \ref{ch4-lem-almost-positive} that
\[w\geq0\quad \textrm{in}\quad B_{1/2}.\]
Since we can do this for every $\tau\in \mathbb{S}^{n-1}$ with $\tau\cdot e\geq\frac12$, the proposition is proved.
\end{proof}

As a consequence of the previous proposition, we find:

\begin{cor}\label{ch4-FB-Lip}
Let $u$ be any solution to \eqref{ch4-obst-f=1}, and assume that \eqref{ch4-positive-density} holds.
Then, there exists $r_\circ >0$ such that the free boundary $\partial\{u_{r_\circ }>0\}$ is \emph{Lipschitz} in $B_{1/2}$.
In particular, the free boundary of $u$, $\partial\{u>0\}$, is Lipschitz in $B_{r_\circ /2}$.
\end{cor}

\begin{proof}
This follows from the fact that $\partial_\tau u_{r_\circ }\geq0$ in $B_{1/2}$ for all $\tau\in \mathbb{S}^{n-1}$ with $\tau\cdot e\geq \frac12$ (by Proposition~\ref{ch4-monotone-directional}), as explained next.

Let $x_\circ \in B_{1/2}\cap \partial\{u_{r_\circ }>0\}$ be any free boundary point in $B_{1/2}$, and let 
\[\Theta:=\bigl\{\tau\in \mathbb{S}^{n-1}: \tau\cdot e>{\textstyle\frac12}\bigr\},\]
\[\Sigma_1:=\bigl\{x\in B_{1/2}: x=x_\circ -t\tau,\ {\rm with}\ \tau\in \Theta,\ t>0\bigr\},\]
and
\[\Sigma_2:=\bigl\{x\in B_{1/2}: x=x_\circ +t\tau,\ {\rm with}\ \tau\in \Theta,\ t>0\bigr\},\]
see Figure~\ref{fig.29}.

\begin{figure}
\includegraphics[scale = 0.9]{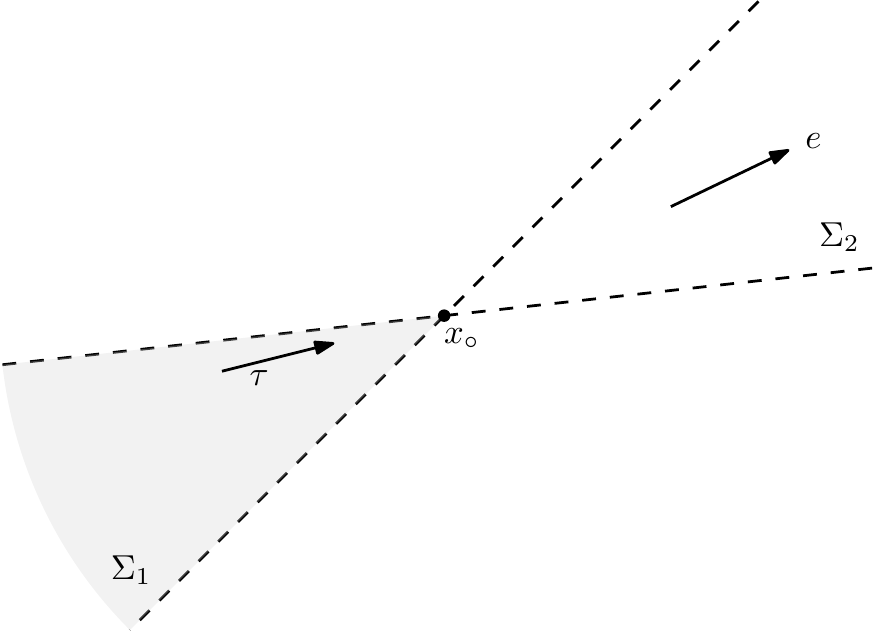}
\caption{Representation of $\Sigma_1$ and $\Sigma_2$.}
\label{fig.29}
\end{figure}

We claim that 
\begin{equation} \label{ch4-FB-Lip-Q}
\left\{
\begin{array}{rcll}
u_{r_\circ }&=&0& \ \textrm{in} \quad\Sigma_1,\\
u_{r_\circ }&>&0& \ \textrm{in} \quad\Sigma_2.
\end{array}
\right.
\end{equation}
Indeed, since $u_{r_\circ}(x_\circ )=0$, it follows from the monotonicity property $\partial_\tau u_{r_\circ }\geq0$ --- and the nonnegativity of $u_{r_\circ }$ --- that $u_{r_\circ }(x_\circ -t\tau)=0$ for all $t>0$ and $\tau\in\Theta$.
In particular, there cannot be any free boundary point in $\Sigma_1$.

On the other hand, by the same argument, if $u_{r_\circ }(x_1)=0$ for some $x_1\in \Sigma_2$ then we would have $u_{r_\circ }=0$ in $\bigl\{x\in B_{1/2}: x=x_1-t\tau,\ {\rm with}\ \tau\in \Theta,\ t>0\bigr\}\ni x_\circ $, and in particular $x_\circ $ would not be a free boundary point.
Thus, $u_{r_\circ }(x_1)>0$ for all $x_1\in \Sigma_2$, and \eqref{ch4-FB-Lip-Q} is proved.

Finally, notice that \eqref{ch4-FB-Lip-Q} yields that the free boundary $\partial\{u_{r_\circ }>0\}\cap B_{1/2}$ satisfies both the interior and exterior cone condition, and thus it is Lipschitz.
\end{proof}

Once we know that the free boundary is Lipschitz, we may assume without loss of generality that $e=e_n$ and that 
\[\partial\{u_{r_\circ }>0\}\cap B_{1/2}=\{x_n=g(x')\}\cap B_{1/2}\]
for a Lipschitz function $g:\R^{n-1}\to\R$.
Here, $x=(x',x_n)$, with $x'\in\R^{n-1}$ and $x_n\in\R$.

Now, we want to prove that Lipschitz free boundaries are $C^{1,\alpha}$.
A key ingredient for this will be the following basic property of harmonic functions (see Figure~\ref{fig.26} for a representation of the setting).

\begin{figure}
\includegraphics{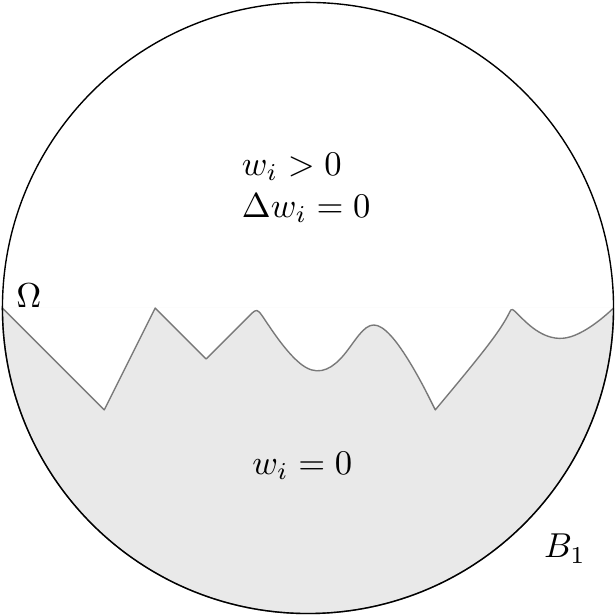}
\caption{Setting of the boundary Harnack.}
\label{fig.26}
\end{figure}

\begin{thm}[Boundary Harnack] \label{boundary-Harnack}\index{Boundary Harnack}
Let $w_1$ and $w_2$ be \emph{positive harmonic} functions in $B_1\cap \Omega$, where $\Omega\subset \R^n$ is any \emph{Lipschitz domain}.

Assume that $w_1$ and $w_2$ vanish on $\partial\Omega\cap B_1$, and $C_\circ ^{-1}\leq \|w_i\|_{L^\infty(B_{1/2})}\leq C_\circ $ for $i = 1,2$.
Then,
\[\frac1C w_2 \leq w_1 \leq Cw_2\qquad \textrm{in}\quad \overline\Omega\cap B_{1/2}.\]
Moreover, 
\[\left\|\frac{w_1}{w_2}\right\|_{C^{0,\alpha}(\overline\Omega\cap B_{1/2})} \leq C\]
for some small $\alpha>0$.
The constants $\alpha$ and $C$ depend only on $n$, $C_\circ $, and~$\Omega$.
\end{thm}

For completeness, we provide in Appendix~\ref{app.D} a proof of this result. 
We refer to \cite{DS-bdryH} for the boundary Harnack for more general operators and to \cite{AS19,RT20} for the boundary Harnack for equations with a right hand side.

\begin{rem}
The main point in Theorem \ref{boundary-Harnack} is that $\Omega$ is allowed to be \emph{Lipschitz}.
If $\Omega$ is smooth (say, $C^2$ or even $C^{1,\alpha}$) then it follows from a simple barrier argument that both $w_1$ and $w_2$ would be comparable to the distance to $\partial\Omega$, i.e., they vanish at a linear rate from $\partial\Omega$.
However, in Lipschitz domains the result cannot be proved with a simple barrier argument, and it is much more delicate to establish.
\end{rem}

The boundary Harnack is a crucial tool in the study of free boundary problems, and in particular in the obstacle problem.
Here, we use it to prove that the free boundary is $C^{1,\alpha}$ for some small $\alpha>0$.

\begin{prop}\label{ch4-FB-C1alpha}\index{Regularity of the free boundary}
Let $u$ be any solution to \eqref{ch4-obst-f=1}, and assume that \eqref{ch4-positive-density} holds.
Then, there exists $r_\circ >0$ such that the free boundary $\partial\{u_{r_\circ }>0\}$ is $C^{1,\alpha}$ in $B_{1/4}$, for some small $\alpha>0$.
In particular, the free boundary of $u$, $\partial\{u>0\}$, is $C^{1,\alpha}$ in $B_{r_\circ /4}$.
\end{prop}

\begin{proof}
Let $\Omega=\{u_{r_\circ }>0\}$.
By Corollary~\ref{ch4-FB-Lip}, if $r_\circ >0$ is small enough then (possibly after a rotation) we have
\[\Omega\cap B_{1/2}=\{x_n\geq g(x')\}\cap B_{1/2}\]
and the free boundary is given by
\[\partial\Omega\cap B_{1/2}=\{x_n= g(x')\}\cap B_{1/2},\]
where $g$ is Lipschitz.

Let 
\[w_2:=\partial_{e_n}u_{r_\circ }\]
and 
\[\qquad w_1:=\partial_{e_i}u_{r_\circ }+\partial_{e_n}u_{r_\circ },\qquad i=1,...,n-1.\]
Since $\partial_\tau u_{r_\circ }\geq0$ in $B_{1/2}$ for all $\tau\in \mathbb{S}^{n-1}$ with $\tau\cdot e_n\geq\frac12$, we have that $w_2\geq0$ in $B_{1/2}$ and $w_1\geq0$ in $B_{1/2}$.

This is because $\partial_{e_i}+\partial_{e_n}=\partial_{e_i+e_n}=\sqrt{2}\partial_{\tau}$, with $\tau\cdot e_n=1/\sqrt{2}>\frac12$.
Notice that we add the term $\partial_{e_n}u_{r_\circ }$ in $w_1$ in order to get a nonnegative function $w_2\geq0$.

Now since $w_1$ and $w_2$ are positive harmonic functions in $\Omega\cap B_{1/2}$, and vanish on $\partial\Omega\cap B_{1/2}$, we can use the boundary Harnack, Theorem~\ref{boundary-Harnack} (or Corollary~\ref{boundary-Harnack_App2}), to get
\[\left\|\frac{w_1}{w_2}\right\|_{C^{0,\alpha}(\overline\Omega\cap B_{1/4})} \leq C\]
for some small $\alpha>0$.
Therefore, since $w_1/w_2=1+\partial_{e_i}u_{r_\circ }/\partial_{e_n}u_{r_\circ }$, we deduce
\begin{equation}\label{ch4-FB-C1alpha-Q}
\left\|\frac{\partial_{e_i}u_{r_\circ }}{\partial_{e_n}u_{r_\circ }}\right\|_{C^{0,\alpha}(\overline\Omega\cap B_{1/4})} \leq C.
\end{equation}
Now, we claim that this implies that the free boundary is $C^{1,\alpha}$ in $B_{1/4}$.
Indeed, if $u_{r_\circ }(x)=t$ then the normal vector to the level set $\{u_{r_\circ }=t\}$ is given by
\[\qquad\qquad \nu^i(x)=\frac{\partial_{e_i}u_{r_\circ }}{|\nabla u_{r_\circ }|}=\frac{\partial_{e_i}u_{r_\circ }/\partial_{e_n}u_{r_\circ }}{\sqrt{1+\sum_{j=1}^{n-1} \left(\partial_{e_j}u_{r_\circ }/\partial_{e_n}u_{r_\circ }\right)^2}},\qquad  i=1,...,n.\]
This is a $C^{0,\alpha}$ function by \eqref{ch4-FB-C1alpha-Q}, and therefore we can take $t\to0$ to find that the free boundary is $C^{1,\alpha}$ (since the normal vector to the free boundary is given by a $C^{0,\alpha}$ function).
\end{proof}

So far we have proved that
\[\left(\begin{array}{c}
\{u=0\}\ \textrm{has positive} \\ \textrm{density at the origin}
\end{array}\right) \Longrightarrow
\left(\begin{array}{c}
\textrm{any blow-up is} \\ u_0={\textstyle\frac12}(x\cdot e)^2_+
\end{array}\right) \Longrightarrow
\left(\begin{array}{c}
\textrm{free boundary} \\ \textrm{is}\ C^{1,\alpha}\ \textrm{near}\ 0
\end{array}\right)\]

As a last step in this section, we will now prove that $C^{1,\alpha}$ free boundaries are actually $C^\infty$.

\subsection*{Higher regularity of the free boundary}\index{Regularity of the free boundary!Higher regularity}

We want to finally prove the smoothness of free boundaries near regular points.

\begin{thm}[Smoothness of the free boundary near regular points]\label{ch4-FB-smooth}
Let $u$ be any solution to \eqref{ch4-obst-f=1}, and assume that \eqref{ch4-positive-density} holds.
Then, the free boundary $\partial\{u>0\}$ is $C^\infty$ in a neighborhood of the origin.
\end{thm}

For this, we need the following result.

\begin{thm}[Higher order boundary Harnack]\label{ch4-higher-bdry-Harnack}\index{Boundary Harnack!Higher order regularity}
Let $\Omega\subset \R^n$ be any $C^{k,\alpha}$ domain, with $k\geq1$ and $\alpha\in(0,1)$.
Let $w_1$, $w_2$ be two solutions of $\Delta w_i=0$ in $B_1\cap \Omega$, $w_i=0$ on $\partial\Omega\cap B_1$, with $w_2>0$ in $\Omega$.

Assume that $C_\circ ^{-1}\leq \|w_i\|_{L^\infty(B_{1/2})}\leq C_\circ $.
Then,
\[\left\|\frac{w_1}{w_2}\right\|_{C^{k,\alpha}(\overline\Omega\cap B_{1/2})} \leq C,\]
where $C$ depends only on $n$, $k$, $\alpha$, $C_\circ $, and $\Omega$.
\end{thm}

Contrary to Theorem~\ref{boundary-Harnack}, the proof of Theorem \ref{ch4-higher-bdry-Harnack} is a perturbative argument, in the spirit of (but much more delicate than) the Schauder estimates from Chapter \ref{ch.2}.
We will not prove the higher order boundary Harnack here; we refer to \cite{DS} for the proof of such result.

Using Theorem \ref{ch4-higher-bdry-Harnack}, we can finally prove Theorem~\ref{ch4-FB-smooth}:

\begin{proof}[Proof of Theorem \ref{ch4-FB-smooth}]
Let $u_{r_\circ }(x)=r_\circ ^{-2}u(r_\circ x)$.
By Proposition \ref{ch4-FB-C1alpha}, we know that if $r_\circ >0$ is small enough then the free boundary $\partial\{u_{r_\circ }>0\}$ is $C^{1,\alpha}$ in $B_1$, and (possibly after a rotation) $\partial_{e_n}u_{r_\circ }>0$ in $\{u_{r_\circ }>0\}\cap B_1$.
Thus, using the higher order boundary Harnack (Theorem \ref{ch4-higher-bdry-Harnack}) with $w_1=\partial_{e_i}u_{r_\circ }$ and $w_2=\partial_{e_n}u_{r_\circ }$, we find that
\[\left\|\frac{\partial_{e_i}u_{r_\circ }}{\partial_{e_n}u_{r_\circ }}\right\|_{C^{1,\alpha}(\overline\Omega\cap B_{1/2})} \leq C.\]
Actually, by a simple covering argument we find that
\begin{equation}\label{ch4-higher-order-subballs}
\left\|\frac{\partial_{e_i}u_{r_\circ }}{\partial_{e_n}u_{r_\circ }}\right\|_{C^{1,\alpha}(\overline\Omega\cap B_{1-\delta})} \leq C_\delta
\end{equation}
for any $\delta>0$.

Now, as in the proof of Proposition \ref{ch4-FB-C1alpha}, we notice that if $u_{r_\circ }(x)=t$ then the normal vector to the level set $\{u_{r_\circ }=t\}$ is given by
\[\qquad\qquad \nu^i(x)=\frac{\partial_{e_i}u_{r_\circ }}{|\nabla u_{r_\circ }|}=\frac{\partial_{e_i}u_{r_\circ }/\partial_{e_n}u_{r_\circ }}{\sqrt{1+\sum_{j=1}^n \left(\partial_{e_j}u_{r_\circ }/\partial_{e_n}u_{r_\circ }\right)^2}},\qquad \qquad i=1,...,n.\]
By \eqref{ch4-higher-order-subballs}, this is a $C^{1,\alpha}$ function in $B_{1-\delta}$ for any $\delta>0$, and therefore we can take $t\to0$ to find that the normal vector to the free boundary is $C^{1,\alpha}$ inside $B_1$.
But this means that the free boundary is actually $C^{2,\alpha}$.

Repeating now the same argument, and using that the free boundary is $C^{2,\alpha}$ in $B_{1-\delta}$ for any $\delta>0$, we find that 
\[\left\|\frac{\partial_{e_i}u_{r_\circ }}{\partial_{e_n}u_{r_\circ }}\right\|_{C^{2,\alpha}(\overline\Omega\cap B_{1-\delta'})} \leq C_{\delta'},\]
which yields that the normal vector is $C^{2,\alpha}$ and thus the free boundary is $C^{3,\alpha}$.
Iterating this argument, we find that the free boundary $\partial\{u_{r_\circ }>0\}$ is $C^\infty$ inside $B_1$, and hence $\partial\{u>0\}$ is $C^\infty$ in a neighborhood of the origin.
\end{proof}

This completes the study of \emph{regular} free boundary points.
It remains to understand what happens at points where the contact set has \emph{density zero} (see e.g. Figure~\ref{fig.21}).
This is the content of the next section.

\section{Singular points}
\label{sec-singular-points}

We finally study the behavior of the free boundary at singular points, i.e., when
\begin{equation}\label{ch4-zero-density-0}
\lim_{r\to0}\frac{\bigl|\{u=0\}\cap B_r\bigr|}{|B_r|}=0.
\end{equation}
For this, we first notice that, as a consequence of the results of the previous Section, we get the following.

\begin{prop}\label{thm-classification-ALL-blowups}
Let $u$ be any solution to \eqref{ch4-obst-f=1}.
Then, we have the following dichotomy:
\begin{itemize}
\item[(a)] Either \eqref{ch4-positive-density} holds and 
all blow-ups of $u$ at $0$ are of the form
\[u_0(x)=\frac12(x\cdot e)_+^2,\]
for some $e\in \mathbb{S}^{n-1}$.

\item[(b)] Or \eqref{ch4-zero-density-0} holds and all blow-ups of $u$ at $0$ are of the form
\[u_0(x)=\frac12x^TAx,\]
for some matrix $A\geq0$ with ${\rm tr}\,A=1$.
\end{itemize}
\end{prop}

Points of   type (a) were studied in the previous Section; they are called \emph{regular} points and the free boundary is $C^\infty$ around them (in particular, the blow-up is unique).
Points of  type (b) are those at which the contact set has zero density, and are called \emph{singular} points.

To prove the result, we need the following:

\begin{lem}\label{ch4-lem-zero-density}
Let $u$ be any solution to \eqref{ch4-obst-f=1}, and assume that \eqref{ch4-zero-density-0} holds.
Then, every blow-up of $u$ at $0$ satisfies $|\{u_0=0\}|=0$.
\end{lem}

\begin{proof}
Let $u_0$ be a blow-up of $u$ at $0$, i.e., $u_{r_k}\to u_0$  in $C^1_{\rm loc}(\R^n)$ along a sequence $r_k\to0$, where $u_r(x)=r^{-2}u(rx)$.

Notice that the functions $u_r$ solve
\[\Delta u_r=\chi_{\{u_r>0\}}\quad \textrm{in}\quad B_1,\]
in the sense that
\begin{equation}\label{ch4-weak-chi}
\int_{B_1}\nabla u_r\cdot \nabla \eta\,dx=\int_{B_1} \chi_{\{u_r>0\}}\eta\,dx\qquad \textrm{for all}\ \eta\in C^\infty_c(B_1).
\end{equation}
Moreover, by assumption \eqref{ch4-zero-density-0}, we have $\bigl|\{u_r=0\}\cap B_1\bigr|\longrightarrow 0$, and thus taking limits $r_k\to0$ in \eqref{ch4-weak-chi} we deduce that $\Delta u_0=1$ in $B_1$.
Since we know that $u_0$ is convex, nonnegative, and homogeneous, this implies that $|\{u_0=0\}|=0$.
\end{proof}

We can now give the:

\begin{proof}[Proof of Theorem \ref{thm-classification-ALL-blowups}]
By the classification of blow-ups (Theorem~\ref{thm-classification-blowups}), the possible blow-ups can only have one of the two forms presented. 
If \eqref{ch4-positive-density} holds for at least one blow-up, thanks to the smoothness of the free boundary (by Proposition~\ref{ch4-FB-C1alpha}), it holds for all blow-ups, and thus, by Corollary~\ref{ch4-lem-one-blowup}, $u_0(x)=\frac12(x\cdot e)_+^2$ (and in fact, the smoothness of the free boundary yields uniqueness of the blow-up in this case). 

If \eqref{ch4-zero-density-0} holds, then by Lemma \ref{ch4-lem-zero-density} the blow-up $u_0$ must satisfy $\bigl|\{u_0=0\}\bigr|=0$, and thus we are in case (b) (see the proof of Theorem~\ref{thm-classification-blowups}).
\end{proof}

In the previous Section we proved that the free boundary is $C^\infty$ in a neighborhood of any regular point.
A natural question then is to understand better the solution $u$ near singular points.
One of the main results in this direction is the following.

\begin{thm}[Uniqueness of blow-ups at singular points]\label{thm-uniqueness-blowups}\index{Uniqueness of blow-ups at singular points}
Let $u$ be any solution to \eqref{ch4-obst-f=1}, and assume that $0$ is a singular free boundary point.

Then, there exists a homogeneous quadratic polynomial $p_2(x)=\frac12x^TAx$, with $A\geq0$ and $\Delta p_2=1$, such that 
\[u_r\longrightarrow p_2\qquad \textrm{in}\quad C^1_{\rm loc}(\R^n).\]
In particular, the blow-up of $u$ at $0$ is unique, and $u(x)=p_2(x)+o(|x|^2)$.
\end{thm}

To prove this, we need the following monotonicity formula due to Monneau.

\begin{thm}[Monneau's monotonicity formula]\label{thm-Monneau}\index{Monneau monotonicity formula}
Let $u$ be any solution to \eqref{ch4-obst-f=1}, and assume that $0$ is a singular free boundary point.

Let $q$ be any homogeneous quadratic polynomial with $q\geq0$, $q(0)=0$, and $\Delta q=1$.
Then, the quantity 
\[M_{u,q}(r):=\frac{1}{r^{n+3}}\int_{\partial B_r}\left(u-q\right)^2\]
is monotone in $r$, that is, $\frac{d}{dr}M_{u,q}(r)\geq 0$.
\end{thm}

\begin{proof}
We sketch the argument here, and refer to \cite[Theorem 7.4]{PSU} for more details.

We first notice that
\[M_{u,q}(r)=\int_{\partial B_1}\frac{(u-q)^2(rx)}{r^4},\]
and hence a direct computation yields
\[\frac{d}{dr}M_{u,q}(r)=\frac{2}{r^{n+4}}\int_{\partial B_r}(u-q)\left\{x\cdot \nabla(u-q)-2(u-q)\right\}.\]

On the other hand, it turns out that 
\[\begin{split}\frac{1}{r^{n+3}}\int_{\partial B_r}(u-q)\left\{x\cdot \nabla(u-q)-2(u-q)\right\}=&\,W_u(r)-W_u(0^+)+\\&+\frac{1}{r^{n+2}}\int_{B_r}(u-q)\Delta(u-q),\end{split}\]
where $W_u(r)$ (as defined in \eqref{weissenergy}) is monotone increasing in $r>0$ thanks to Theorem~\ref{thm-Weiss}.
Thus, we have 
\[\frac{d}{dr}M_{u,q}(r)\geq \frac{2}{r^{n+3}}\int_{B_r}(u-q)\Delta(u-q).\]
But since $\Delta u=\Delta q=1$ in $\{u>0\}$, and $(u-q)\Delta(u-q)=q\geq0$ in $\{u=0\}$,  we have 
\[\frac{d}{dr}M_{u,q}(r)\geq \frac{2}{r^{n+3}}\int_{B_r\cap \{u=0\}}q\geq 0,\]
as wanted.
\end{proof}

We can now give the:

\begin{proof}[Proof of Theorem \ref{thm-uniqueness-blowups}]
By Proposition~\ref{thm-classification-ALL-blowups} (and Proposition~\ref{ch4-prop-blowups}), we know that at any singular point we have a subsequence $r_j\to0$ along which $u_{r_j}\to p$ in $C^1_{\rm loc}(\R^n)$, where $p$ is a $2$-homogeneous quadratic polynomial satisfying $p(0)=0$, $p\geq0$, and $\Delta p=1$.
Thus, we can use Monneau's monotonicity formula with such polynomial $p$ to find that 
\[M_{u,p}(r):=\frac{1}{r^{n+3}}\int_{\partial B_r}\left(u-p\right)^2\]
is monotone increasing in $r>0$.
In particular, the limit $\lim_{r\to0}M_{u,p}(r):=M_{u,p}(0^+)$ exists.

Now, recall that we have a sequence $r_j\to0$ along which $u_{r_j}\to p$.
In particular, $r_j^{-2}\left\{u(r_jx)-p(r_jx)\right\}\longrightarrow 0$ locally uniformly in $\R^n$, i.e.,
\[\frac{1}{r_j^2}\|u-p\|_{L^\infty(B_{r_j})}\longrightarrow 0\]
as $r_j\to0$.
This yields that 
\[M_{u,p}(r_j)\leq \frac{1}{r_j^{n+3}} \int_{\partial B_{r_j}}\|u-p\|_{L^\infty(B_{r_j})}^2 \longrightarrow 0\] 
along the subsequence $r_j\to0$, and therefore $M_{u,p}(0^+)=0$.

Let us show that this implies the uniqueness of blow-ups.
Indeed, if there was another subsequence $r_\ell\to0$ along which $u_{r_\ell}\to q$ in $C^1_{\rm loc}(\R^n)$, for a $2$-homogeneous quadratic polynomial $q$, then we would repeat the argument above to find that $M_{u,q}(0^+)=0$.
But then this yields, by homogeneity of $p$ and $q$,
\[\int_{\partial B_1}\left(p-q\right)^2=\frac{1}{r^{n+3}}\int_{\partial B_r}\left(p-q\right)^2\leq 2M_{u,p}(r)+2M_{u,q}(r)\longrightarrow0,\]
and hence 
\[\int_{\partial B_1}\left(p-q\right)^2=0.\]
This means that $p=q$, and thus the blow-up of $u$ at $0$ is unique.

Let us finally show that $u(x)=p(x)+o(|x|^2)$, i.e., $r^{-2}\|u-p\|_{L^\infty(B_r)}\to0$ as $r\to0$.
Indeed, assume by contradiction that there is a subsequence $r_k\to0$ along which 
\[r_k^{-2}\|u-p\|_{L^\infty(B_{r_k})}\geq c_1>0.\]
Then, there would be a subsequence of $r_{k_i}$ along which $u_{r_{k_i}}\to u_0$ in $C^1_{\rm loc}(\R^n)$, for a certain blow-up $u_0$ satisfying $\|u_0-p\|_{L^\infty(B_1)}\geq c_1>0$.
However, by uniqueness of blow-ups it must be $u_0=p$, and hence we reach a contradiction.
\end{proof}

We refer to \cite{SY19,Bon01} for an alternative approach to the uniqueness of blow-ups at singular points, not based on monotonicity formulas.

Summarizing, we have proved the following result:

\begin{thm}\label{thm-final}
Let $u$ be any solution to \eqref{ch4-obst-f=1}.
Then, we have the following dichotomy:
\begin{itemize}
\item[(a)] Either all blow-ups of $u$ at $0$ are of the form
\[\qquad\qquad u_0(x)=\frac12(x\cdot e)_+^2 \qquad \textrm{for some}\quad e\in \mathbb{S}^{n-1},\]
and the free boundary is $C^\infty$ in a neighborhood of the origin.

\item[(b)] Or there is a homogeneous quadratic polynomial $p$, with $p(0)=0$, $p\geq0$, and $\Delta p=1$, such that 
\[\qquad \qquad \|u-p\|_{L^\infty(B_r)}=o(r^2)\qquad \textrm{as}\quad r\to0.\]
In particular, when this happens we have
\[\lim_{r\to0}\frac{\bigl|\{u=0\}\cap B_r\bigr|}{|B_r|}=0.\]
\end{itemize}
\end{thm}

The last question that remains to be answered is: How large can the set of singular points be?
This is the topic of the following section.

\section{On the size of the singular set}

We finish this chapter with a discussion of more recent results (as well as some open problems) about the set of singular points.

Recall that a free boundary point $x_{\circ}\in \partial\{u>0\}$ is singular whenever
\[
\lim_{r\to0}\frac{\bigl|\{u=0\}\cap B_r(x_{\circ})\bigr|}{|B_r(x_{\circ})|}=0.
\]
The main known result on the size of the singular set reads as follows.

\begin{thm}[\cite{Caf98}]\label{thm-singular-set}\index{Singular points!Size}
Let $u$ be any solution to \eqref{ch4-obst-f=1}.
Let $\Sigma\subset B_1$ be the set of singular points.

Then, $\Sigma\cap B_{1/2}$ is locally contained in a $C^1$ manifold of dimension $n-1$.
\end{thm}

This result is sharp, in the sense that it is not difficult to construct examples in which the singular set is $(n-1)$-dimensional; see \cite{Sch3}.

As explained below, such result essentially follows from the uniqueness of blow-ups at singular points, established in the previous section.

Indeed, given any singular point $x_{\circ}$, let $p_{x_{\circ}}$ be the blow-up of $u$ at $x_{\circ}$ (recall that $p_{x_{\circ}}$ is a nonnegative 2-homogeneous polynomial).
Let $k$ be the dimension of the set $\{p_{x_{\circ}}=0\}$ --- notice that this is a proper linear subspace of $\R^n$, so that $k\in \{0,...,n-1\}$ --- and define 
\begin{equation}\label{Sigma-k}
\Sigma_k:=\big\{x_{\circ}\in\Sigma : {\rm dim}(\{p_{x_{\circ}}=0\})=k\big\}.
\end{equation}
Clearly, $\Sigma=\bigcup_{k=0}^{n-1} \Sigma_k$.

The following result gives a more precise description of the singular set.

\begin{prop}[\cite{Caf98}]\label{prop-singular-set-k}
Let $u$ be any solution to \eqref{ch4-obst-f=1}.
Let $\Sigma_k\subset B_1$ be defined by \eqref{Sigma-k}, $k=1,...,n-1$.
Then, $\Sigma_k$ is locally contained in a $C^1$ manifold of dimension $k$.
\end{prop}

The rough heuristic idea of the proof of this result is as follows.
Assume for simplicity that $n=2$, so that $\Sigma=\Sigma_1\cup \Sigma_0$.

Let us take a point $x_{\circ}\in \Sigma_0$.
Then, by Theorem \ref{thm-final}, we have the expansion
\begin{equation}\label{expansion-final}
u(x)=p_{x_{\circ}}(x-x_{\circ})+o\big(|x-x_{\circ}|^2\big)
\end{equation}
where $p_{x_{\circ}}$ is the blow-up of $u$ at $x_{\circ}$ (recall that this came from the uniqueness of blow-ups at $x_{\circ}$).
By definition of $\Sigma_0$, the polynomial $p_{x_{\circ}}$ must be positive outside the origin, and thus by homogeneity satisfies $p_{x_{\circ}}(x-x_{\circ})\geq c|x-x_{\circ}|^2$, with $c>0$.
This, combined with \eqref{expansion-final}, yields then that $u$ must be positive in a neighborhood of $x_\circ$.
In particular, all points in $\Sigma_0$ are isolated.

On the other hand, let us now take a point $x_{\circ}\in \Sigma_1$.
Then, by definition of $\Sigma_1$ the blow-up must necessarily be of the form $p_{x_{\circ}}(x)=\frac12(x\cdot e_{x_{\circ}})^2$, for some $e_{x_{\circ}}\in \mathbb{S}^{n-1}$.
Again by the expansion \eqref{expansion-final}, we find that $u$ is positive in a region of the form 
\[\big\{x\in B_\rho(x_{\circ}) : \big|(x-x_{\circ})\cdot e_{x_{\circ}}\big|>\omega(|x-x_{\circ}|)\big\},\] 
where $\omega$ is a certain modulus of continuity, and $\rho>0$ is small (see Figure~\ref{fig.30}).

\begin{figure}
\includegraphics[scale = 0.65]{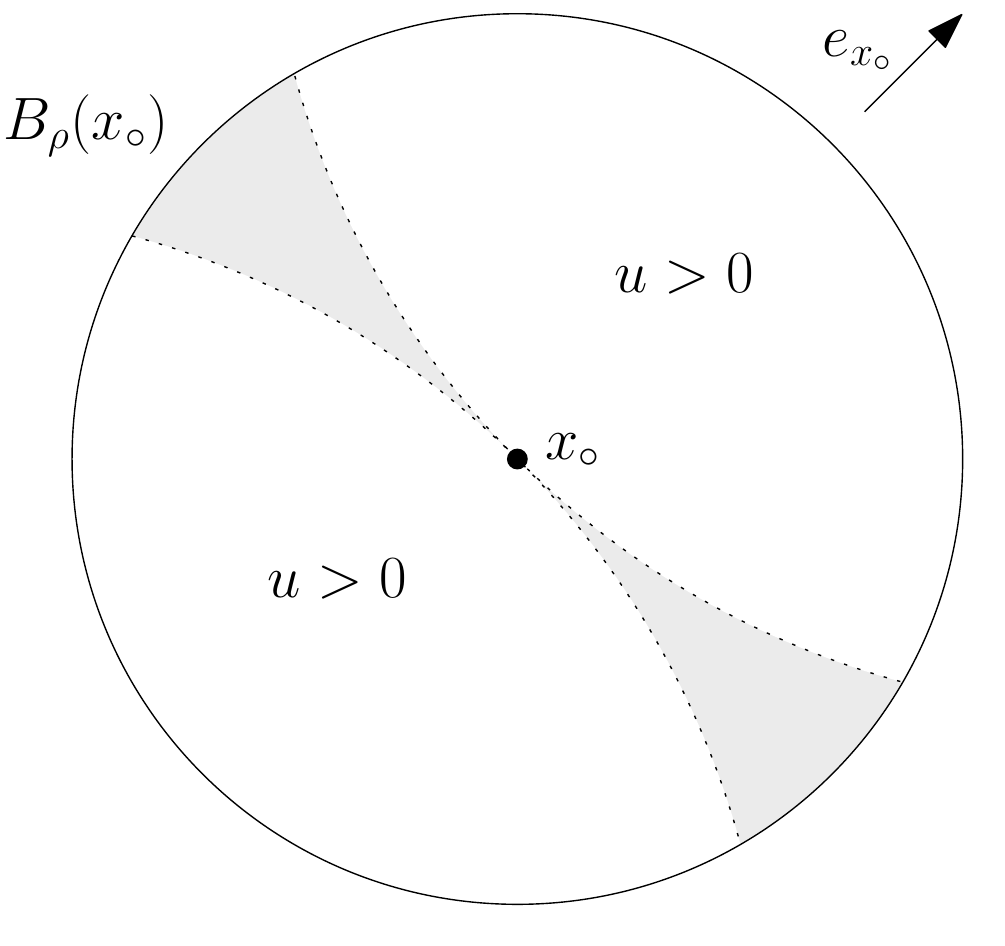}
\caption{$u$ is positive in $\{x\in B_\rho(x_{\circ}) : |(x-x_{\circ})\cdot e_{x_{\circ}}| > \omega(|x-x_{\circ}|)\}$.}
\label{fig.30}
\end{figure}

This is roughly saying that the set $\Sigma_1$ ``has a tangent plane'' at $x_{\circ}$.
Repeating the same at any other point $\tilde x_\circ\in \Sigma_1$ we find that the same happens at every point in $\Sigma_1$ and, moreover, if $\tilde x_\circ$ is close to $x_{\circ}$ then $e_{\tilde x_\circ}$ must be close to $e_{x_{\circ}}$ --- otherwise the expansions \eqref{expansion-final} at $\tilde x_\circ$ and $x_{\circ}$ would not match.
Finally, since the modulus $\omega$ can be made independent of the point (by a compactness argument),   it turns out that the set $\Sigma_1$ is contained in a $C^1$ curve (see Figure~\ref{fig.31}).

\begin{figure}
\includegraphics[scale = 1]{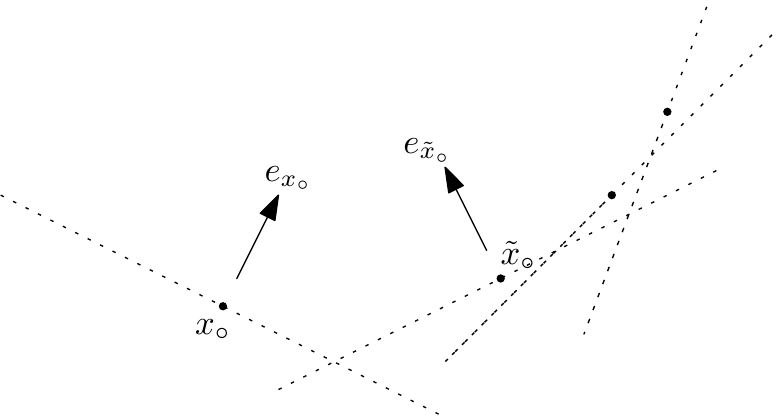}
\caption{Singular points $x_{\circ},\tilde x_\circ \in \Sigma_1$.}
\label{fig.31}
\end{figure}

What we discussed here is just an heuristic argument; the actual proof uses Whitney's extension theorem and can be found for example in \cite{PSU}.
Finally, we refer to \cite{CSV}, \cite{FSerra}, and \cite{FZ21} (and the expository paper \cite{Fig18b}) for some recent finer results about the  set of singular points.

\subsection*{Generic regularity}

\index{Generic regularity}
In PDE problems in which singularities may appear, it is very natural and important to understand whether these singularities appear ``often'', or if instead ``most'' solutions have no singularities.

In the context of the obstacle problem, the key question is to understand the generic regularity of free boundaries.
Explicit examples show that singular points in the obstacle problem can form a very large set, of dimension $n-1$ (as large as the regular set).
Still, singular points are expected to be rare (see \cite{Sch1}):

\vspace{3mm}

\noindent \textbf{Conjecture} \index{Schaeffer conjecture}(Schaeffer, 1974): \ \emph{Generically, the weak solution of the obstacle problem is also a strong solution, in the sense that the free boundary is a $C^\infty$ manifold.}

\vspace{3mm}

In other words, the conjecture states that, generically, the free boundary has \emph{no} singular points.

The first result in this direction was established by Monneau in 2003, who proved the following.

\begin{thm}[\cite{Mon}]\label{Schaeffer-2D}
Schaeffer's conjecture holds in $\R^2$.
\end{thm}

More precisely, Monneau considers a 1-parameter family of solutions $u_\lambda$, with $\lambda\in (0,1)$, such that
\[\left\{\begin{array}{rcll}
\Delta u_\lambda&=&\chi_{\{u_\lambda>0\}} &\quad \textrm{in}\ \Omega \\
u_\lambda&=&g_\lambda & \quad \textrm{on}\ \partial \Omega, 
\end{array}\right.\]
with $g_\lambda=g+\lambda$ and $g\geq0$ on $\partial \Omega$.

Then, the first step is to notice that not only each of the singular sets $\Sigma_\lambda \subset \Omega$ is contained in a $C^1$ manifold of dimension $(n-1)$, but actually the union $\bigcup_{\lambda\in (0,1)}\Sigma_\lambda\subset \Omega$ is still contained in an $(n-1)$-dimensional manifold.

After that, we look at the free boundary as a set in $\Omega\times (0,1)\ni (x,\lambda)$, and notice that it can be written as a graph $\{\lambda=h(x)\}$, for some function~$h$.
A second key step in the proof is to show that $h$ is Lipschitz and, furthermore, it has zero gradient at any singular point.
This, combined with the coarea formula, yields that in $\R^2$ the set of singular points is empty for almost every $\lambda\in(0,1)$, which implies Theorem \ref{Schaeffer-2D}.

Finally, the best known result in this direction was established very recently by Figalli, Serra, and the second author.

\begin{thm}[\cite{FRS}]
Schaeffer's conjecture holds in $\R^3$ and $\R^4$.
\end{thm}

The proof of this result is based on a new and very fine understanding of singular points.
For this, \cite{FRS} combines Geometric Measure Theory tools, PDE estimates, several dimension reduction arguments, and even several new monotonicity formulas.

It remains an open problem to decide whether or not Schaeffer's conjecture holds in dimensions $n\geq5$ or not.

\appendix
%
%
%

\chapter{Some properties of H\"older spaces}
\label{app.A}

In this appendix, we prove the properties \ref{it.H1}-\ref{it.H8} stated in Chapter~\ref{ch.0}.

Recall that, given $\alpha\in(0,1]$, the H\"older space $C^{0,\alpha}(\overline\Omega)$ is the set of functions $u\in C(\overline\Omega)$ such that
\[ [u]_{C^{0,\alpha}(\overline{\Omega})} := \sup_{\substack{x,y\in\overline\Omega\\x\neq y}}\frac{\bigl|u(x)-u(y)\bigr|}{|x-y|^\alpha}<\infty.\]
The H\"older norm is
\[\|u\|_{C^{0,\alpha}(\overline\Omega)}:=\|u\|_{L^\infty(\Omega)}+[u]_{C^{0,\alpha}(\overline{\Omega})}.\]
When $\alpha=1$, this is the usual space of Lipschitz functions.

More generally, given $k\in\mathbb N$ and $\alpha\in(0,1]$, the space $C^{k,\alpha}(\overline\Omega)$ is the set of functions $u\in C^k(\overline\Omega)$ such that the following norm is finite
\begin{align*}
\|u\|_{C^{k,\alpha}(\overline\Omega)}& :=\sum_{j=1}^k\|D^ju\|_{L^\infty(\Omega)}+\sup_{\substack{x,y\in\overline\Omega\\x\neq y}}\frac{\bigl|D^ku(x)-D^ku(y)\bigr|}{|x-y|^\alpha}\\
& =  \|u\|_{C^k(\Omega)} + [D^k u]_{C^{0,\alpha}(\overline{\Omega})}.
\end{align*}

Finally, when $\beta>0$ is \emph{not} an integer, we denote $C^\beta(\overline\Omega):=C^{k,\alpha}(\overline\Omega)$, where $\beta=k+\alpha$, with $k\in\mathbb N$, $\alpha\in(0,1)$.

Next, we give the proofs of the properties of H\"older spaces that we have used throughout the book. 
Unless stated otherwise, in the following statements we assume $\alpha\in (0, 1)$. 

\begin{enumerate}[leftmargin=0pt,align=left,label={\bf (H\arabic*)}]
\setcounter{enumi}{0}
\labelwidth=0pt
\itemindent=0pt
\setlength\itemsep{2mm}\setlength{\leftmargin}{0pt}
\item \label{it.HH1}{\it Assume
\[{\rm osc}_{B_r(x)}u\leq C_{\circ}r^\alpha\qquad \textrm{for all }\, B_r(x)\subset\overline{B_1},\]
where ${\rm osc}_A u:=\sup_A u - \inf_A u$.

Then, $u\in C^{0,\alpha}(\overline{B_1})$ and $[u]_{C^{0,\alpha}(\overline{B_1})}\leq CC_{\circ}$, with $C$ depending only on $n,\alpha$.}

\item[~$\bullet$ \textbf{Proof of \ref{it.HH1}}] We want to prove that $|u(z)-u(x)|\leq CC_{\circ}|z-x|^\alpha$ for all $z,x\in B_1$.
Given $z,x\in B_1$, let $r=|z-x|$.
For this, we may assume $r<1/10$ and distinguish two cases:
\begin{enumerate}[(i)]
\item If $B_r(x)\subset B_1$, then we simply use the assumption to get
\[|u(z)-u(x)|\leq {\rm osc}_{B_r(x)}u\leq C_{\circ}r^\alpha=C_{\circ}|z-x|^\alpha.\]

\item Otherwise, we take $\bar x$ and $\bar z$ on the segments $\overline{0x}$ and $\overline{0z}$, respectively, such that $|x-\bar x|=r$ and $|z-\bar z|=r$. Then, by assumption we have $|u(x)-u(\bar x)|\leq C_{\circ}r^\alpha$, $|u(z)-u(\bar z)|\leq C_{\circ}r^\alpha$, and $|u(\bar x)-u(\bar z)|\leq C_{\circ}r^\alpha$.
The last inequality holds because $|\bar x-\bar z|< r$, which can be easily checked by construction of $\bar x$ and $\bar z$.

Combining the last three inequalities, we deduce that $|u(x)-u(z)|\leq 3C_{\circ}r^\alpha$, as wanted.
\end{enumerate}
\qed
\end{enumerate}

We also state and prove the following slight modification of \ref{it.HH1}, which will be useful in later proofs. 
Notice that the difference with respect to the previous statement is that now, given any ball in $B_1$, we control the oscillation in the ball with half the radius.

\begin{enumerate}[leftmargin=0pt,align=left,label={\bf (H\arabic*{'})}]
\setcounter{enumi}{0}
\labelwidth=0pt
\itemindent=0pt
\setlength\itemsep{2mm}\setlength{\leftmargin}{0pt}
\item \label{it.HH1p}
{\it Assume
\[{\rm osc}_{B_r(x)}u\leq C_{\circ}r^\alpha\qquad \textrm{for all }\, B_{2r}(x)\subset\overline{B_1},\]
where ${\rm osc}_A u:=\sup_A u - \inf_A u$.

 Then, $u\in C^{0,\alpha}(\overline{B_1})$ and $[u]_{C^{0,\alpha}(\overline{B_1})}\leq CC_{\circ}$, with $C$ depending only on $n,\alpha$.}
 
 \item[~$\bullet$ \textbf{Proof of \ref{it.HH1p}}.] We proceed analogously to the proof of \ref{it.HH1}. Let $z,x\in B_1$, and let $r=|z-x|$. We may assume that $r< 1/10$. If $B_{2r}(x)\subset B_1$, the result follows by assumption. 

Otherwise, let us take $\bar x$ and $\bar z$ on the segments $\overline{0x}$ and $\overline{0z}$, respectively,   such that $|x-\bar x|=2r$ and $|z-\bar z|=2r$. Let us define $x_k = (1-2^{-k})x+2^{-k} \bar x$ and $z_k = (1-2^{-k})z+2^{-k} \bar z$. Notice that $|x_{k+1}-x_k|=|z_{k+1}-z_k|= 2^{-k}r$. 
Also, $|x_k| = |x|-2^{-k+1}r$, so that $B_{2|x_{k+1}-x_k|}(x_k)\subset B_1$. That is, we can use our assumption on $x_k$ and $x_{k+1}$ to get that 
\[
|u(x_k)-u(x_{k+1})| \le C_{\circ}|x_k-x_{k+1}|^\alpha = C_{\circ}2^{-k\alpha} r^\alpha. 
\]
(An analogous result holds for $z_k$.) On the other hand, by choice of $\bar x$ and~$\bar z$, they can also be compared in the oscillation of $u$ as 
\[
|u(\bar x)-u(\bar z)|\le C_{\circ}|\bar x-\bar z|^\alpha \le C_{\circ}r^\alpha. 
\]

Putting everything together, we reach that
\begin{align*}
|u(x)-u(z)|& \le \sum_{k \ge 0} |u(x_{k+1})-u(x_k)|+|u(\bar x)-u(\bar z)| + \sum_{k \ge 0} |u(z_{k+1})-u(z_k)|\\
& \le 2\sum_{k \ge 0} C_{\circ} 2^{-k\alpha} r^\alpha + C_{\circ}r^\alpha \le C C_{\circ} r^\alpha,
\end{align*}
for some constant $C$ depending only on $\alpha$. 
\qed
\end{enumerate}

\begin{enumerate}[leftmargin=0pt,align=left,label={\bf (H\arabic*)}]
\setcounter{enumi}{1}
\labelwidth=0pt
\itemindent=0pt
\setlength\itemsep{2mm}\setlength{\leftmargin}{0pt}
\item \label{it.HH2}{\it Let $u_{x,r}:=\ave_{B_r(x)}u$. 
Assume
\[\|u-u_{x,r}\|_{L^\infty(B_r(x))}\leq C_{\circ}r^\alpha\qquad \textrm{for all }\, B_r(x)\subset\overline{B_1}.\]
Then, $u\in C^{0,\alpha}(\overline{B_1})$ and $[u]_{C^{0,\alpha}(\overline{B_1})}\leq CC_{\circ}$, with $C$ depending only on $n,\alpha$.}
\item[~$\bullet$ \textbf{Proof of \ref{it.HH2}}.] By the triangle inequality we have
\[{\rm osc}_{B_r(x)}u\leq 2\|u-u_{x,r}\|_{L^\infty(B_r(x))} \leq 2C_{\circ}r^\alpha,\]
and thus the result follows from \ref{it.HH1}.
\qed

\item \label{it.HH3} {\it Let $u_{x,r}:=\ave_{B_r(x)}u$. Assume
\[\left(\ave_{B_r(x)}|u-u_{x,r}|^2\right)^{1/2}\leq C_{\circ}r^\alpha\qquad \textrm{for all }\, B_r(x)\subset\overline{B_1}.\]
Then, $u\in C^{0,\alpha}(\overline{B_1})$ and $[u]_{C^{0,\alpha}(\overline{B_1})}\leq CC_{\circ}$, with $C$ depending only on $n,\alpha$.}

\item[~$\bullet$ \textbf{Proof of \ref{it.HH3}}.] Notice that, for every $z\in B_1$,
\[\bigl|u_{x,r}-u_{x,\frac{r}{2}}\bigr|^2 \leq 2|u(z)-u_{x,r}|^2+2\bigl|u(z)-u_{x,\frac{r}{2}}\bigr|^2.\]
Thus, integrating in $B_{r/2}(x)$ and using the assumption we deduce
\[\begin{split} |u_{x,r}-u_{x,\frac{r}{2}}|^2 & 
\leq 2\ave_{B_{r/2}(x)}|u-u_{x,r}|^2+2\ave_{B_{r/2}(x)}\bigl|u-u_{x,\frac{r}{2}}\bigr|^2 \\
& \leq
2^{n+1}\ave_{B_r(x)}|u-u_{x,r}|^2+2\ave_{B_{r/2}(x)}\bigl|u-u_{x,\frac{r}{2}}\bigr|^2 \leq CC_{\circ}^2r^{2\alpha}.
\end{split}\]
This means that
\[\bigl|u_{x,r}-u_{x,\frac{r}{2}}\bigr| \leq CC_{\circ}r^\alpha,\]
and summing a geometric series we get
\[|u_{x,r}-u(x)| \leq \sum_{k\geq0} \bigl|u_{x,\frac{r}{2^k}}-u_{x,\frac{r}{2^{k+1}}}\bigr| \leq \sum_{k\geq0} CC_{\circ}\left(\frac{r}{2^k}\right)^\alpha = 2C C_{\circ}r^\alpha.\]
Here we used that, up to redefining $u$ on a set of measure zero, by Lebesgue differentiation theorem (Theorem~\ref{Lebesgue}) we have that $u_{x,r}\to u(x)$ as $r\to0$.

Let now $x,y\in B_1$, $r=2|x-y|$, and assume that $B_r(x)\subset B_1$.
Then, we have
\[\begin{split} |u_{x,r}-u_{y,r}|^2 & 
\leq \ave_{B_{r/2}(x)}|u-u_{x,r}|^2+\ave_{B_{r/2}(x)}\bigl|u-u_{y,r}\bigr|^2 \\
& \leq
2^n\ave_{B_r(x)}|u-u_{x,r}|^2+2^n\ave_{B_{r}(y)}\bigl|u-u_{y,r}\bigr|^2 \leq CC_{\circ}^2r^{2\alpha},
\end{split}\]
and thus 
\[\bigl|u_{x,r}-u_{y,r}\bigr| \leq CC_{\circ}r^\alpha.\]

Combining the previous estimates, we deduce that for every $x,y\in B_1$ such that $B_{2|x-y|}(x)\subset B_1$, we have
\[|u(x)-u(y)| \leq |u(x)-u_{x,r}|+|u_{x,r}-u_{y,r}|+|u_{y,r}-u(y)| \leq 3CC_{\circ}r^\alpha.\]
Once we have this, by \ref{it.HH1p} we are done. 
\qed

\item \label{it.HH4} {\it Assume that for every $x$ there is a constant $C_x$ such that
\[\|u-C_x\|_{L^\infty(B_r(x))}\leq C_{\circ}r^\alpha\qquad \textrm{for all }\, B_r(x)\subset\overline{B_1}.\]
Then, $u\in C^{0,\alpha}(\overline{B_1})$ and $[u]_{C^{0,\alpha}(\overline{B_1})}\leq CC_{\circ}$, with $C$ depending only on $n,\alpha$.

Assume that for every $x$ there is a linear function $\ell_x(y)=a_x+b_x\cdot(y-x)$ such that
\[\|u-\ell_x\|_{L^\infty(B_r(x))}\leq C_{\circ}r^{1+\alpha}\qquad \textrm{for all }\, B_r(x)\subset\overline{B_1}.\]
Then, $u\in C^{1,\alpha}(\overline{B_1})$ and $[Du]_{C^{0,\alpha}(\overline{B_1})}\leq CC_{\circ}$, with $C$ depending only on $n,\alpha$.

Assume that for every $x$ there is a quadratic polynomial $P_x(y)$ such that
\[\|u-P_x\|_{L^\infty(B_r(x))}\leq C_{\circ}r^{2+\alpha}\qquad \textrm{for all }\, B_r(x)\subset\overline{B_1}.\]
Then, $u\in C^{2,\alpha}(\overline{B_1})$ and $[D^2u]_{C^{0,\alpha}(\overline{B_1})}\leq CC_{\circ}$, with $C$ depending only on $n,\alpha$.}

\item[~$\bullet$ \textbf{Proof of \ref{it.HH4}}.]  (i) The first statement --- with the $C^{0,\alpha}$ norm --- follows from \ref{it.HH1}.

(ii) 
Let us sketch the proof of the second statement  --- with the $C^{1,\alpha}$ norm.
Let $x,y\in B_1$ with $y\in B_r(x)\subset B_1$.
Notice that, dividing by $r$ and taking $r\to0$ in the assumption, it follows that $u$ is differentiable at $x$ and that $\ell_x$ must be given by $\ell_x(y)=u(x)+\nabla u(x)\cdot(y-x)$.
Thus, by assumption, we have
\[u(y)=u(x)+\nabla u(x)\cdot (y-x)+O(r^{1+\alpha})\]
and, for every $z\in B_r(x)$ such that $|z-y|\approx |z-x|\approx |y-x|\approx r$,
\[\begin{split}
u(z) & =u(x)+\nabla u(x)\cdot (z-x)+O(r^{1+\alpha})\\ 
&=u(y)+\nabla u(y)\cdot (z-y)+O(r^{1+\alpha})\\ 
&=u(x)+\nabla u(x)\cdot (y-x)+\nabla u(y)\cdot (z-y)+O(r^{1+\alpha}).\end{split}\]
From this, we deduce that 
\[\nabla u(x)\cdot (z-y) = \nabla u(y)\cdot (z-y)+O(r^{1+\alpha}).\]
Taking $z$ such that $z-y$ is parallel to $\nabla u(y)-\nabla u(x)$, we get
\[\nabla u(x)=\nabla u(y)+O(r^{\alpha}),\]
as wanted.

(iii) Let us prove the third statement concerning the $C^{2,\alpha}$ norm --- the following proof is more general and works also in case (ii).

Let $x,y\in B_r(x_{\circ})$ with $|x-y|=r$ and suppose $B_{2r}(x_{\circ})\subset B_1$.
Let us rescale $u$ around $x_{\circ}$, i.e.,  $u_r(z):=u(x_{\circ}+rz)$, so that $|\bar x-\bar y|=1$, where $x_{\circ}+r \bar x=x$ and $x_{\circ}+r \bar y=y$.
Let us define also 
\[P_{x,r}(z):= P_x(x_{\circ}+rz)\qquad \textrm{and} \qquad P_{y,r}:= P_y(x_{\circ}+rz).\]
Then,
\[\|u_r-P_{x,r}\|_{L^\infty(B_1(\bar x))} = \|u-P_{x}\|_{L^\infty(B_{r}(x))} \leq C_{\circ}r^{2+\alpha},\]
\[\|u_r-P_{y,r}\|_{L^\infty(B_1(\bar y))} = \|u-P_{y}\|_{L^\infty(B_{ r}(y))} \leq C_{\circ} r^{2+\alpha}.\]
Hence, if we denote $\bar w  = \frac{\bar x + \bar y}{2}$ then $B_{1/2}(\bar w)\subset B_1(\bar x)\cap B_1(\bar y)$ and 
\[\begin{split}
\|P_{x,r}-P_{y,r}\|_{L^\infty(B_{1/2}(\bar w))} &\leq \|u_r-P_{x,r}\|_{L^\infty(B_1(\bar x))} + \|u_r-P_{y,r}\|_{L^\infty(B_1(\bar y))}\\
& \leq CC_{\circ}r^{2+\alpha}.\end{split}\]
This means that all the coefficients of the polynomial $P_{x,r}-P_{y,r}$ are controlled by $\tilde CC_{\circ} r^{2+\alpha}$.

Now, notice that if we denote $P_x(z)=a_x+b_x\cdot (z-x)+(z-x)^TM_x(z-x)$ then $P_{x,r}(z)=a_x+rb_x\cdot (z-\bar x)+r^2(z-\bar x)^TM_x(z-\bar x)$, and an analogous expression holds for $P_{y,r}$. 
Hence, we can write
\[\begin{split}
P_{x,r}(z)-P_{y,r}(z) = &\bigl(a_x-a_y+rb_x\cdot(\bar y-\bar x)+r^2(\bar y-\bar x)^TM_x(\bar y-\bar x)\bigr) \\
& + r\bigl(b_x-b_y+2r(\bar y-\bar x)^TM_x\bigr)\cdot (z-\bar y) \\
&+r^2(z-\bar y)^T(M_x-M_y)(z-\bar y).\end{split}\]
In particular, by looking at the quadratic and linear coefficients of such polynomial, we have proved that 
\[|M_x-M_y|\leq \tilde CC_{\circ} r^{\alpha}\]
and
\[\bigl|b_x-b_y+2r(\bar y-\bar x)^TM_x\bigr|\leq CC_{\circ} r^{1+\alpha}.\]
Since $r(\bar y-\bar x)=y-x$, this is equivalent to
\[\bigl|b_y-b_x-2(y-x)^TM_x\bigr|\leq CC_{\circ} r^{1+\alpha}.\]

Notice, also, that 
\[\|u-a_x-b_x\cdot (\cdot-x)\|_{L^\infty(B_r(x))}\leq C_{\circ}r^{2+\alpha} + C_xr^2\le 2C_x r^{2}\]
if $r$ small enough, so that, in particular, arguing as in (i), $u$ is differentiable at $x$ and $a_x = u(x)$, $b_x = \nabla u(x)$. 

Thus, using that $r=2|x-y|$, we have
\[\bigl|\nabla u(y)-\nabla u(x)-2(y-x)^TM_x\bigr|\leq CC_{\circ} |y-x|^{1+\alpha},\]
and letting $y\to x$ we deduce that  $\nabla u$ is differentiable at $x$, with $D^2u(x) = 2M_x$. 
An analogous result holds for $M_y$, so that we have shown that, for any $x,y\in B_r(x_{\circ})$ with $|x-y|=r$ and $B_{2r}(x_{\circ})\subset B_1$, 
\[
\bigl|D^2 u(x)-D^2 u(y)\bigr|\leq \tilde CC_{\circ} r^{\alpha}.
\]
The result now follows by \ref{it.HH1p}. 
\qed
\end{enumerate}

\begin{rem*}
Notice that the converse statement to \ref{it.HH4} also holds. For example, when $k = 1$, if $u\in C^{1,\alpha}(B_1)$ then we have
\[\|u-\ell_x\|_{L^\infty(B_r(x))}\leq C_{\circ}r^{1+\alpha}\qquad \textrm{for all }\, B_r(x)\subset\overline{B_1},\]
where $\ell_x(y)=u(x)+\nabla u(x)\cdot (y-x)$.
Indeed, to show this, we use that
\[u(y)=u(x)+\int_0^1 \nabla u(ty+(1-t)x)\cdot (y-x)dt,\]
combined with
\[\bigl|\nabla u\big( ty+(1-t)x\big)-\nabla u(x)\bigr|\leq C_{\circ}| ty+(1-t)x-x|^\alpha
\leq C_{\circ}|y-x|^\alpha,\]
to get 
\[
\bigl|u(y)-u(x)-\nabla u(x)\cdot(y-x)\bigr|\leq \int_0^1 C_{\circ}|y-x|^\alpha |y-x|dt= C_{\circ}|y-x|^{1+\alpha},
\]
as wanted.
\end{rem*}

\begin{enumerate}[leftmargin=0pt,align=left,label={\bf (H\arabic*)}]
\setcounter{enumi}{4}
\labelwidth=0pt
\itemindent=0pt
\setlength\itemsep{2mm}\setlength{\leftmargin}{0pt}
\item \label{it.HH5} {\it Let $\rho_\circ \in (0,1)$. Assume that, for every $x\in B_{1/2}$, there exists a sequence of quadratic polynomials, $(P_k)_{k\in \N}$ such that
\[
\|u-P_k\|_{L^\infty(B_{\rho_\circ^k}(x))}\leq C_{\circ}\rho_\circ^{k(2+\alpha)}\qquad\textrm{for all }\, k\in \N.
\]
Then, $u\in C^{2,\alpha}(B_{1/2})$ and $[D^2u]_{C^{0,\alpha}(B_{1/2})}\le C C_{\circ}$, with $C$ depending only on $n$, $\alpha$, and $\rho_\circ$. }

\item[~$\bullet$ \textbf{Proof of \ref{it.HH5}}.]  Let us take $x=0$.
By hypothesis, we have
\[\begin{split}
\|P_{k-1}-P_k\|_{L^\infty(B_{\rho_\circ^{k}})} &\leq 
\|u-P_{k-1}\|_{L^\infty(B_{\rho_\circ^{k}})} + \|u-P_{k}\|_{L^\infty(B_{\rho_\circ^{k}})} \\
& \leq CC_{\circ}\rho_\circ^{k(2+\alpha)}. 
\end{split}\]

Then, we use the following:

\vspace{2mm}

\noindent \textbf{Claim}. {\it Assume that $P$ is a quadratic polynomial satisfying $\|P\|_{L^\infty(B_r)}\leq \gamma$.
If we denote $P(z)=a+b\cdot z+z^TM z$, then we have that
\[|a|\leq C\gamma,\qquad |b|\leq \frac{C\gamma}{r},\qquad |M|\leq \frac{C\gamma}{r^2},\]
where $C$ is a constant depending only on $n$.}

\vspace{2mm}

To prove the claim, notice that, by rescaling, we have $P_r(z):=P(rz)=a_r+b_r\cdot z+z^TM_r z$, where $a_r=a$, $b_r=rb$, $M_r=r^2M$.
By assumption, we have that $\|P_r\|_{L^\infty(B_1)}\leq \gamma$.
Since the coefficients of polynomials on $B_1$ are controlled by the $L^\infty$ norm, we get that $|a_r|\leq C\gamma$, $|b_r|\leq C\gamma$, and $|M_r|\leq C\gamma$.
This proves the claim.

\vspace{2mm}

Using the previous claim and the bound on $P_{k-1}-P_k$, we deduce that 
\[|a_{k-1}-a_k|\leq CC_{\circ}\rho_\circ^{k(2+\alpha)},\qquad |b_{k-1}-b_k|\leq CC_{\circ}\rho_\circ^{k(1+\alpha)},\]
and 
\[|M_{k-1}-M_k|\leq CC_{\circ}\rho_\circ^{k\alpha},\]
where $P_k(z)=a_k+b_k\cdot z+z^TM_k z$.

It follows that $P_k$ converge uniformly to a polynomial $P(z)=a+b\cdot z+z^TM z$, and that
\[\begin{split}
\|u-P\|_{L^\infty(B_{\rho_\circ^{k}})} &\leq 
\|u-P_k\|_{L^\infty(B_{\rho_\circ^{k}})} + |a_k-a|+\rho_\circ^k|b_k-b|+\rho_\circ^{2k}|M_k-M| \\
& \leq CC_{\circ}\rho_\circ^{k(2+\alpha)}
\end{split}\]
for all $k\geq1$. From this, it follows that for every $r\in(0,1)$ we have
\[ \|u-P\|_{L^\infty(B_r)} \leq  CC_{\circ}r^{2+\alpha}\]
(simply use that for any $r$ we have $\rho_\circ^{k+1} \leq r\leq \rho_\circ^k$ for some $k$).
Thus, since we can do this for every $x\in B_{1/2}$, it follows from \ref{it.HH4} that $[D^2u]_{C^{0,\alpha}(B_{1/2})}\leq CC_{\circ}$.
\qed

We refer to Remark~\ref{rem.zygmund} below for a generalization of property \ref{it.HH5}. 

\item \label{it.HH6} {\it Assume that $\alpha\in(0,1)$, $\|u\|_{L^\infty(B_1)}\leq C_{\circ}$, and
\begin{equation}
\label{eq.H6.1.A}\sup_{\substack{h\in B_1\\ x\in\overline{B_{1-|h|}}}}\frac{\bigl|u(x+h)+u(x-h)-2u(x)\bigr|}{|h|^\alpha}\leq C_{\circ}.
\end{equation}
Then, $u\in C^{0,\alpha}(\overline{B_1})$ and $\|u\|_{C^{0,\alpha}(\overline{B_1})}\leq CC_{\circ}$, with $C$ depending only on $n,\alpha$.

Assume that $\alpha\in(0,1)$, $\|u\|_{L^\infty(B_1)}\leq C_{\circ}$, and
\begin{equation}
\label{eq.H6.2.A}\sup_{\substack{h\in B_1\\ x\in\overline{B_{1-|h|}}}}\frac{\bigl|u(x+h)+u(x-h)-2u(x)\bigr|}{|h|^{1+\alpha}}\leq C_{\circ}.
\end{equation}
Then, $u\in C^{1,\alpha}(\overline{B_1})$ and $\|u\|_{C^{1,\alpha}(\overline{B_1})}\leq CC_{\circ}$, with $C$ depending only on $n,\alpha$. However, such property fails when $\alpha=0$.}

\item[~$\bullet$ \textbf{Proof of \ref{it.HH6}}.]  
\noindent (i) Let us do the case \eqref{eq.H6.2.A} first.

Given $h\in B_1$ and $x\in B_{1-|h|}$, let 
\[w(h):= \frac{u(x+h)-u(x)}{|h|}.\]
Then, by assumption we have
\[\bigl|w(h)-w(h/2)\bigr| = \frac{|u(x+h)+u(x)-2u(x+h/2)|}{|h|} \leq C_{\circ}|h|^{\alpha}.\]
Thus, for every $k\geq0$, 
\[\bigl|w(h/2^k)-w(h/2^{k+1})\bigr| \leq C_{\circ}|h|^{\alpha}2^{-k\alpha}.\]
This implies the existence of the limit $\lim_{t\to0}w(th)$, and by summing a geometric series we get
\[\bigl|w(h)-\lim_{t\to 0}w(th)\bigr|\leq CC_{\circ}|h|^\alpha.\]
Since 
\[\lim_{t\to0}w(th) = \lim_{t\to0} \frac{u(x+th)-u(x)}{t|h|} = \frac{h}{|h|}\cdot\nabla u(x),\]
this leads to
\[\bigl|u(x+h)-u(x)-h\cdot \nabla u(x)\bigr|\leq CC_{\circ}|h|^{1+\alpha}.\]
Using \ref{it.HH4}, we see that the last inequality implies that $[Du]_{C^{0,\alpha}(B_1)}\leq CC_{\circ}$.
Finally, using that $\|u\|_{L^\infty(B_1)}\leq C_{\circ}$, the result follows.

\vspace{1mm}

\noindent (ii) Let us do now the case \eqref{eq.H6.1.A}.

As before, let us define $w(h):= \frac{u(x+h)-u(x)}{|h|}$ and notice that 
\[\bigl|w(h)-w(h/2)\bigr| = \frac{|u(x+h)+u(x)-2u(x+h/2)|}{|h|} \leq C_{\circ}|h|^{\alpha-1}.\]
Then, for every $k\geq0$ we have
\[\bigl|w(2^k h)-w(2^{k+1}h)\bigr| \leq C_{\circ}|h|^{\alpha-1}2^{-k(1-\alpha)}.\]
Take $k_{\circ}\geq0$ such that $2^{k_{\circ}}|h|\approx 1$ (and so that\footnote{Note that this is always possible if $x,y\in B_{9/10}$, for example. 
	If $x,y$ are close to the boundary $\partial B_1$, then this is possible for example when $(x-y)\cdot \frac{x}{|x|}>\frac12 |y-x|$. 
	It is easy to see that we can always reduce to this case.} 
still $x+2^{k_{\circ}}h\in B_1$), and add the previous inequality for all $0\leq k< k_{\circ}$.
Then, by summing a geometric series, we deduce that 
\[\bigl|w(2^{k_{\circ}} h)-w(h)\bigr| \leq CC_{\circ}|h|^{\alpha-1}.\]
Since
\[\bigl|w(2^{k_{\circ}}h)\bigr| \leq C\|u\|_{L^\infty(B_1)}\leq CC_{\circ} \leq CC_{\circ}|h|^{\alpha-1},\]
we finally get
\[\bigl|w(h)\bigr| \leq \bigl|w(2^{k_{\circ}}h)\bigr| + CC_{\circ}|h|^{\alpha-1}\leq CC_{\circ}|h|^{\alpha-1}.\]
Translating back to $u$, this gives the desired result.

\vspace{1mm}

\noindent (iii) Finally, let us prove that the function 
\[u(x)=x\log|x|,\qquad x\in(-1,1),\]
satisfies \eqref{eq.H6.2.A} with $\alpha=0$, but it is not in $C^{0,1}$.

Indeed, let us show that 
\[\frac{\bigl|(x+h)\log|x+h|+(x-h)\log|x-h|-2x\log|x|\bigr|}{|h|}\leq C_{\circ}\]
for all $x,h\in (-1,1)$ and for some $C_{\circ}>0$.
For this, notice that 
\[\begin{split}
&\frac{(x+h)\log|x+h|+(x-h)\log|x-h|-2x\log|x|}{h}=
\\&\hspace{4cm}=
\frac{\left(1+\frac{h}{x}\right)\log\left|1+\frac{h}{x}\right|+\left(1-\frac{h}{x}\right)\log\left|1-\frac{h}{x}\right|}{\frac{h}{x}}\\
&\hspace{4cm}=
\frac{(1+t)\log|1+t|+(1-t)\log|1-t|}{t},\end{split}\]
with $t=h/x$.
Such function of $t$ is smooth in $\R\setminus\{0\}$ and has finite limits at $t=0$ and at $t=\infty$.
Therefore, it is globally bounded in $\R$ by some constant $C_{\circ}$ (actually, $C_{\circ}<2$). 
\qed
\end{enumerate}

\begin{rem*}
We refer to \cite[Section 2]{And97} for higher order versions of the characterization \ref{it.HH6}.
\end{rem*}

\begin{enumerate}[leftmargin=0pt,align=left,label={\bf (H\arabic*)}]
\setcounter{enumi}{6}
\labelwidth=0pt
\itemindent=0pt
\setlength\itemsep{2mm}\setlength{\leftmargin}{0pt}
\item \label{it.HH7}{\it Assume that $\alpha\in (0,1]$, $\|u\|_{L^\infty(B_1)}\leq C_{\circ}$,  and that for every $h\in B_1$ we have
\[
\left\|\frac{u(x+h)-u(x)}{|h|^\alpha}\right\|_{C^\beta(B_{1-|h|})}\leq C_{\circ},
\]
with $C_{\circ}$ independent of $h$.
Assume in addition that $\alpha+\beta$ is not an integer.
Then, $u\in C^{\alpha+\beta}(\overline{B_1})$ and $\|u\|_{C^{\alpha+\beta}(\overline{B_1})}\leq CC_{\circ}$, with $C$ depending only on $n,\alpha,\beta$.

However, such property fails when $\alpha+\beta$ is an integer.}

\item[~$\bullet$ \textbf{Proof of \ref{it.HH7}}.] We prove it in case $\beta\in(0,1]$, the proof for $\beta>1$ is analogous.
Let us define
\[v_h(x)=\frac{u(x+h)-u(x)}{|h|^\alpha}.\]
Then, by assumption we have 
\[\sup_{x,y\in B_{1-|h|}} \frac{|v_h(x)-v_h(y)|}{|x-y|^{\beta}}\leq C_{\circ}.\]
This is equivalent to
\[\sup_{x,y\in B_{1-|h|}}\frac{|u(x+h)-u(x)-u(y+h)+u(y)|}{|h|^{\alpha+\beta}}\leq C_{\circ}.\]
Taking $y=x-h$, this yields
\[\sup_{x\in B_{1-2|h|}}\frac{|u(x+h)+u(x+h)-2u(x)|}{|h|^{\alpha+\beta}}\leq C_{\circ}.\]
By \ref{it.HH6}, we deduce that $\|u\|_{C^{\alpha+\beta}(\overline{B_1})}\leq CC_{\circ}$ --- as long as $\alpha+\beta\neq1$.
\qed

\item \label{it.HH8} {\it Assume that $u_i\to u$ uniformly in $\overline\Omega\subset \R^n$, and that $\|u_i\|_{C^{k,\alpha}(\overline\Omega)}\leq C_{\circ}$, with $\alpha\in (0, 1]$ and for some $C_{\circ}$ independent of $i$.
Then,   $u\in C^{k,\alpha}(\overline\Omega)$, and
\[\|u\|_{C^{k,\alpha}(\overline\Omega)}\leq C_{\circ}.\]}

\item[~$\bullet$ \textbf{Proof of \ref{it.HH8}}.] Assume first $k=0$.
Then, we have that for every $x,y\in \overline\Omega$, $x\neq y$, 
\[\|u_i\|_{L^\infty(\Omega)}+\frac{|u_i(x)-u_i(y)|}{|x-y|^\alpha} \leq C_{\circ}.\]
Taking limits $u_i\to u$, we deduce that the same inequality holds for $u$, and thus $\|u\|_{C^{0,\alpha}(\overline\Omega)}\leq C_{\circ}$, as wanted.

Assume now that $k\geq1$.
Then, it follows from Arzel\`a--Ascoli that $D^m u_i\to D^m u$ uniformly in $\overline\Omega$ for $m \le k$ and thus, as before, taking limits in the inequality
\[\|u_i\|_{C^k(\overline\Omega)}+\frac{|D^ku_i(x)-D^ku_i(y)|}{|x-y|^\alpha} \leq C_{\circ},\]
the result follows.
%
\qed
\end{enumerate}

%

\begin{rem}
\label{rem.zygmund}
In relation with property \ref{it.HH5}, one can define $\mathscr{L}^{\infty,\beta}$ as the set of functions $u:B_1\to \R$ satisfying that, for each $x\in \R$ and each $r \in (0, 1-|x|)$, there exists some polynomial $P_{x, r}$  of degree $\lfloor\beta\rfloor$ such that 
\[
\|u - P_{x, r}\|_{L^\infty(B_r(x))}\le C r^\beta
\] 
for some $C$ universal, and where $\lfloor\beta\rfloor$ denotes the integer part of $\beta$. More generally, one can define\footnote{These spaces are called \emph{Morrey-Campanato spaces} when $p < \infty $ and $\beta < 1$.} $\mathscr{L}^{p,\beta}$  for $p \in [1, \infty]$ as the set of functions $u$ satisfying
\[
r^{-\frac{n}{p}} \|u - P_{x, r}\|_{L^p(B_r)}\le C r^\beta.
\] 

Then,  it turns out that, for any $\beta>  0$ and $p\ge 1$, $\mathscr{L}^{p,\beta}= \mathscr{L}^{\infty,\beta}$; see \cite[Theorem 2]{JTW83}. Moreover, similarly to what we did in \ref{it.HH5}, one can prove that if $\beta = k+\alpha$, then
\[
\mathscr{L}^{p,k+\alpha} = \mathscr{L}^{\infty,k+\alpha} = C^{k,\alpha},\qquad\text{if}\quad \text{$\alpha\in (0,1)$ and $k\in \N$.}
\] 
On the other hand, when $\beta$ is an integer these spaces do not coincide with H\"older spaces. Indeed, for $\beta=1$ we have 
\[
\mathscr{L}^{p,1} = \mathscr{L}^{\infty,1} = \Lambda^{1},
\]
(see \cite[Section 1.6]{JW84}), and for $\beta > 1$, 
\[
u\in \mathscr{L}^{p,\beta}\quad\Longleftrightarrow\quad \nabla u\in \mathscr{L}^{p, \beta-1},
\]
(see \cite[Theorem 3]{JTW83}.) Here, $\Lambda^1$ denotes the \emph{Zygmund space}, i.e. the set of functions $u:B_1 \to \R$ such that
\[
\sup_{\substack{h\in B_1\\ x\in\overline{B_{1-|h|}}}}\frac{\bigl|u(x+h)+u(x-h)-2u(x)\bigr|}{|h|}\leq C,
\]
for some universal $C$. Finally, when $\beta = 0$ we have 
\[
\mathscr{L}^{p, 0} = \mathscr{L}^{1, 0} = {\rm BMO},\qquad\text{if}\quad p \in [1, \infty),
\]
where ${\rm BMO}$ denotes the space of \emph{bounded mean oscillation} functions, see \cite{JN61, JW84}. Notice also that $\nabla u \in {\rm BMO}$ implies $u \in \Lambda^1$, but the opposite implication does not hold, see \cite[Theorem 3.4]{Str80}.
\end{rem}

%
%
%

\chapter{Proof of the boundary Harnack inequality}
\label{app.D}

The goal of this appendix is to prove the boundary Harnack inequality for Lipschitz domains, Theorem~\ref{boundary-Harnack}. 
The proof we present here is due to De Silva and Savin \cite{DS-bdryH}, and is different to the one given in the book \cite{CS}.

For simplicity, we consider domains $\Omega$ such that 
\begin{equation}
\label{eq.g1}
\begin{array}{c}
\Omega\cap B_1\ \mbox{is given by a Lipschitz graph in the $e_n$ direction,} \vspace{1mm} \\ \mbox{with Lipschitz norm $\leq 1$, and with $0\in \partial\Omega$.}\end{array}
\end{equation}
In other words, we consider $(x', x_n)\in \R^{n-1}\times \R$, and let 
\begin{equation}
\label{eq.g2}
\begin{array}{c}
g:\R^{n-1}\to \R,\qquad [g]_{C^{0,1}(\R^{n-1})}\le 1,\qquad g(0) = 0, \vspace{3mm} \\
\Omega := \{x\in \R^n : x_n > g(x')\}. \end{array}
\end{equation}

The boundary Harnack inequality in Lipschitz domains is the following. (See Figure~\ref{fig.B_1} for a depiction of the setting in the theorem.)

\begin{thm}[Boundary Harnack] \label{boundary-Harnack_App}\index{Boundary Harnack}
Let $w_1$ and $w_2$ be \emph{positive harmonic} functions in $B_1\cap \Omega$, where $\Omega\subset \R^n$ is a \emph{Lipschitz domain} as in \eqref{eq.g1}-\eqref{eq.g2}.

Assume that $w_1$ and $w_2$ vanish continuously on $\partial\Omega\cap B_1$, and $C_\circ ^{-1}\leq \|w_i\|_{L^\infty(B_{1/2})}\leq C_\circ $ for $i = 1,2$.
Then,
\[
C^{-1} w_2 \leq w_1 \leq Cw_2\qquad \textrm{in}\quad \overline\Omega\cap B_{1/2}.
\]
The constant $C$ depends only on $n$ and $C_\circ$.
\end{thm}

\begin{figure}
\includegraphics[scale = 1]{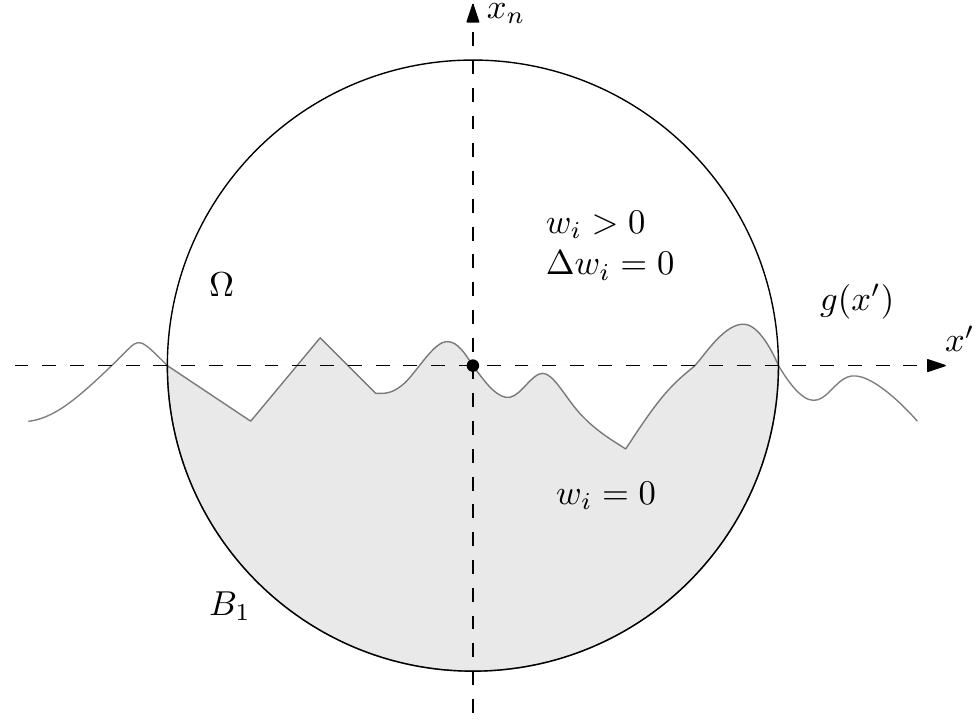}
\caption{Depiction of the setting in Theorem~\ref{boundary-Harnack_App} and Corollary~\ref{boundary-Harnack_App2}.}
\label{fig.B_1}
\end{figure}

Moreover, an appropriate iteration of the previous result gives the following.

\begin{cor} \label{boundary-Harnack_App2}
Let $w_1$ and $w_2$ be as in Theorem \ref{boundary-Harnack_App}.
Then,
\[
\left\|\frac{w_1}{w_2}\right\|_{C^{0,\alpha}(\overline\Omega\cap B_{1/2})} \leq C
\]
for some small $\alpha>0$.
The constants $\alpha$ and $C$ depend only on $n$ and $C_\circ$.
\end{cor}

\begin{rem}
Notice that, for simplicity, we deal with Lipschitz domains with Lipschitz constant bounded by 1 and, as a consequence, none of the constants appearing in Theorem~\ref{boundary-Harnack_App} depend on the domain $\Omega$. 
The same proof presented here can be adapted to the case of general Lipschitz domains. 

The reasons we consider domains with Lipschitz constant bounded by~1 are to avoid introducing more notation  and so that the domain $\Omega$ in $B_1$ has a single connected component. 
Note, moreover, that when we apply the boundary Harnack in Proposition~\ref{ch4-FB-C1alpha}, we are doing so to a Lipschitz domain with Lipschitz constant smaller than 1 (therefore, we can directly apply Corollary~\ref{boundary-Harnack_App2}).
\end{rem}

The following two (well-known) lemmas for sub- and superharmonic functions will be used. 
Notice that these are interior regularity properties. 

\begin{lem}[Weak Harnack Inequality for supersolutions]
\label{lem.prop1}
Let $u\in C(B_1)$. 
Then,
\[
\left\{
\begin{array}{rcll}
-\Delta u & \ge & 0 &\text{in } B_1\\
u &\geq& 0 &\text{in } B_1
\end{array}\right. \quad \Longrightarrow\quad 
\inf_{B_{1/2}} u \ge c\,\|u\|_{L^1(B_{1/2})},
\]
for some $c>0$ depending only on $n$.
\end{lem}

\begin{proof}
By the mean value property of the Laplace equation, for any $x_\circ\in B_{1/3}$  we have
\[
u(x_\circ) \ge \frac{1}{|B_{2/3}|}\int_{B_{2/3}(x_\circ)} u = c \|u\|_{L^1(B_{2/3})(x_\circ)}\ge c\|u\|_{L^1(B_{1/3})},
\]
with $c$ a dimensional constant, so that we have proved the property in a ball of radius $1/3$. Take now any $\bar x_\circ\in \partial B_{1/3}$ and consider the ball $B_{1/6}(\bar x_\circ)$. Notice that we can repeat the previous steps to derive
\[
\inf_{B_{1/6}(\bar x_\circ)} u \ge c \|u\|_{L^1(B_{1/6})(\bar x_\circ)}.
\]
Moreover, if we denote $\mathcal{B} := B_{1/3}\cap B_{1/6}(\bar x_\circ)$, then 
\[
\|u\|_{L^1(B_{1/6})(\bar x_\circ)} \ge \int_{\mathcal{B}} u\ge |\mathcal{B}| \inf_{\mathcal{B}} u  \ge c\inf_{B_{1/3}} u.
\]
From the first result in this proof, we can conclude
\[
\inf_{B_{1/2}} u \ge c_1 \inf_{B_{1/3}} u \ge c_2 \|u\|_{L^1(B_{1/3})} \ge c_3\|u\|_{L^1(B_{1/2})}
\]
for some dimensional constant $c_3$. 
In the last step we have used the monotonicity of averages with respect to the radius for superharmonic functions; see for example \eqref{Laplacian-radially_2}.
\end{proof}

The second lemma reads as follows.

\begin{lem}
[$L^\infty$ bound for subsolutions]
\label{lem.prop2}
Let $u\in C(B_1)$. 
Then,
\[
-\Delta u  \le  0 \quad \text{in} \quad  B_1
\quad \Rightarrow\quad 
\sup_{B_{1/2}} u \le C_\eps \|u\|_{L^\eps(B_{1})},
\]
for any $\eps > 0$, and for some $C_\eps$ depending only on $n$ and $\eps$. 
\end{lem}

\begin{proof}
Again, by the mean value property we have that, for any $r > 0$,
\[
\|u \|_{L^\infty(B_{r/2})} \le C \ave_{B_r}u \le C \|u \|_{L^\infty(B_{r})}^{1-\eps}\ave_{B_r} |u|^{\eps}.
\]
We now want to use an interpolation inequality. Notice that, for any $\delta > 0$, there exists some $C_\delta$ (depending only on $\delta$ and $\eps$) such that $\xi^{1-\eps} \le \delta\xi + C_\delta$ for all $\xi \ge 0$. Taking $\xi = \frac{A}{B}$ with 
\[
A = \|u \|_{L^\infty(B_{r})} ,\qquad B = \left(C \ave_{B_r} |u|^{\eps}\right)^{\frac{1}{\eps}}
\]
we deduce that, for any $\delta > 0$, there exists some $C_\delta$ such that 
\[
\|u \|_{L^\infty(B_{r/2})}  \le C \|u \|_{L^\infty(B_{r})}^{1-\eps}\ave_{B_r} |u|^{\eps}\le \delta\, \|u\|_{L^\infty(B_r)} + C_\delta \left(C \ave_{B_r} |u|^{\eps}\right)^{\frac{1}{\eps}}.
\]

In particular, 
\[
\|u \|_{L^\infty(B_{r/2})}  \le  \delta \|u\|_{L^\infty(B_r)} + C_\delta r^{-\frac{n}{\eps}} \|u\|_{L^\eps(B_1)}.
\]

We are now in position to apply Lemma~\ref{lem.SAL} with $S(A) = \|u\|_{L^\infty(A)}$, $k = \frac{n}{\eps}$ and $\gamma = C_\delta \|u\|_{L^\eps(B_1)}$, to deduce that 
\[
\|u\|_{L^\infty(B_{1/2})}\le C \|u\|_{L^\eps(B_1)},
\]
for some constant $C$ depending only on $n$ and $\eps$, as wanted.
\end{proof}

As a consequence of the previous lemmas we obtain the following two useful results, which are partial steps towards the proof of Theorem \ref{boundary-Harnack_App}.
The first one gives an $L^\infty$ bound for $u$ in terms of the value of the function at an interior point in $\Omega$.

\begin{lem}
\label{lem.bhbounded}
Let $u\in C(B_1)$ be a positive harmonic function in $B_1\cap \Omega$ with $u = 0$ on $B_1\setminus \Omega$, where $\Omega\subset \R^n$ is a Lipschitz domain as in \eqref{eq.g1}-\eqref{eq.g2}. Assume, moreover, that $u(\frac12 e_n) = 1$. Then,
\[
\|u\|_{L^\infty(B_{1/2})}\le C,
\]
for some constant $C$ depending only on $n$. 
\end{lem}

\begin{proof}
Notice that since $u \ge 0$ is harmonic whenever $u > 0$, and it is continuous, we have $\Delta u \ge 0$ in $B_1$ in the viscosity sense.

\begin{figure}
\includegraphics[scale = 1]{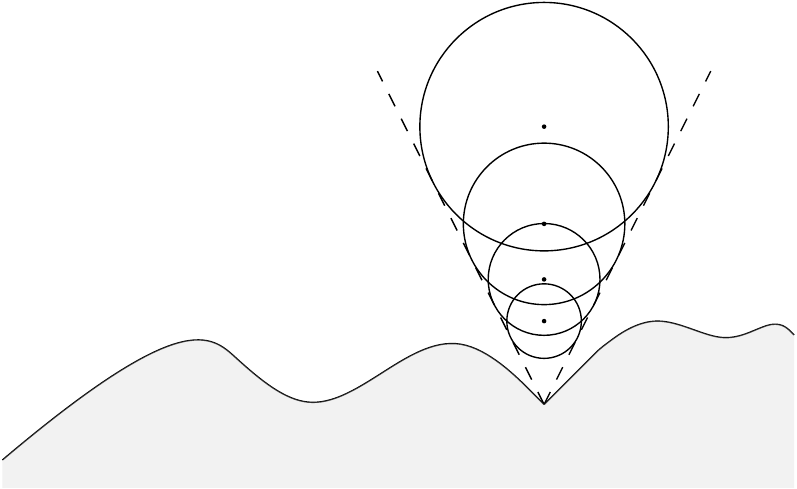}
\caption{A chain of balls to apply the Harnack inequality sequentially.}
\label{fig.B_2}
\end{figure}

On the other hand, since $g$ in \eqref{eq.g1} has Lipschitz constant bounded by 1, we have $B_{\varrho}(\frac12 e_n)\subset \{\Delta u = 0\}$, with $\varrho = \frac{1}{2\sqrt{2}}$. In particular, by Harnack's inequality (see \eqref{eq.Harnack_rho}) we have that $u \le C_n$ in $B_{1/4}(\frac12 e_n)$. That is, $u(0, x_n) \le C_n$ for $x_n\in \left(\frac14, \frac12\right)$. Repeating iteratively, we get $u(0, x_n) \le C_n^k$ for $x_n\in \left(2^{-k-1}, 2^{-k}\right)$  (see~Figure~\ref{fig.B_2} for a sketch of this chain of inequalities), so that $u(0, t) \le  t^{-K}$ for $t\in \left(0, \frac12\right)$, for some large dimensional constant $K$. We can repeat the same procedure at all points in $B_{1/2}$ by iterating successive Harnack inequalities, to deduce that 
\[
u\le  d^{-K}\qquad\text{in}\quad B_{1/2},\qquad\mbox{where}\quad d(x):={\rm dist}(x,\Omega^c). 
\] 
In particular, for $\eps>0$ small enough we have
\[
\int_{B_{1/2}} |u|^{\eps} \le C.
\]
By Lemma~\ref{lem.prop2}, we deduce that $\|u\|_{L^\infty(B_{1/4})}\leq C$, and the result in $B_{1/2}$ follows from a simple covering argument. 
\end{proof}

The second lemma reads as follows.

\begin{lem}
\label{lem.kdelta}
Let $\delta > 0$ be small, let $\Omega\subset \R^n$ be a Lipschitz domain as in \eqref{eq.g1}-\eqref{eq.g2}, and let $\Omega_\delta:=\{x\in \Omega : {\rm dist}(x,\Omega^c)\geq \delta\}$. 
Let $u\in C(B_1)$ satisfy
\[
\left\{
\begin{array}{rcll}
\Delta u & = & 0 &\text{in } \Omega\cap B_1\\
u &= & 0 &\text{on } \partial \Omega\cap B_1
\end{array}\right.
\qquad\text{and}\qquad 
\left\{
\begin{array}{rcll}
 u & \ge & 1 &\text{in } B_1\cap \Omega_\delta\\
u &\ge  & -\delta &\text{in }  B_1.
\end{array}\right.
\]
Then, for all $k\in \N$ such that $k\delta \le \frac34$, we have
\[
u \ge -\delta (1-c_\circ)^k \quad\text{in}\quad B_{1-k\delta}
\]
for some constant $c_\circ$ depending only on $n$. 
\end{lem}

\begin{proof}
Let $u^- = \min\{u, 0\}$. 
Notice that $u^-$ is superharmonic (in the viscosity sense) since $\Delta u^- = 0$ when $u^- < 0$, and $u^- \le 0$, so we have $\Delta u^-\le 0$. Let $w = u^- + \delta$. 
By assumption, $w \ge 0$ and $\Delta w \le 0$. 

Let $x_\circ\in \partial \Omega\cap B_{1-2\delta}$. 
Let us apply Lemma~\ref{lem.prop1} to a ball of radius $2\delta$ around $x_\circ$, so that (after scaling) we deduce
\[
\inf_{B_{\delta}(x_\circ)} w \ge c \delta^{-n} \|w \|_{L^1(B_{\delta}(x_\circ) )}.
\]
Notice, now, that since the domain is Lipschitz and $w\ge \delta$ in $\Omega^c$, we can bound $\|w\|_{L^1(B_{\delta}(x_\circ))} \ge \delta |\{w \ge \delta\}\cap B_{\delta}(x_\circ)|\ge c \delta^{n+1}$ for some $c$ (see~Figure~\ref{fig.B_3}) depending only on~$n$. Thus, 
\[
\inf_{B_{\delta}(x_\circ)} w \ge c_\circ\delta.
\]
In particular, since $w \ge \delta$ in $B_1\cap \Omega_\delta$ we have $w \ge c_\circ\delta$ in $B_{1-\delta}$ and therefore $u \ge -\delta(1-c_\circ)$ in $B_{1-\delta}$. Applying iteratively  this inequality for balls of radius $1-2\delta$, $1-3\delta$, ..., we obtain the desired result. 
\end{proof}

\begin{figure}
\includegraphics[scale = 1.3]{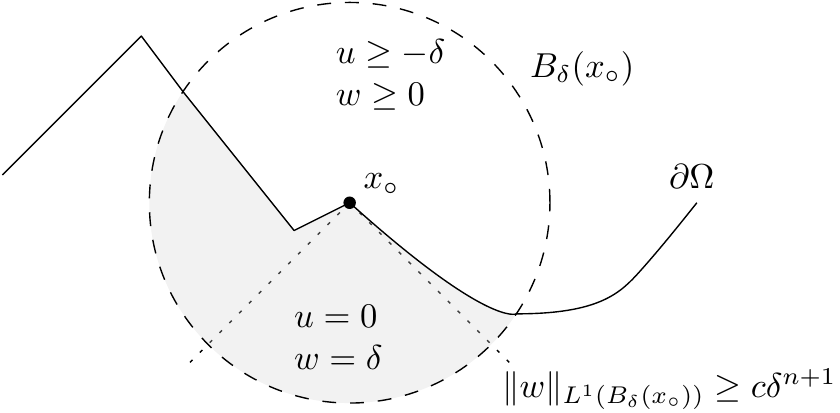}
\caption{The fact that the domain is Lipschitz allows us to bound the $L^1$ norm of $w$ in $B_\delta(x_\circ)$ from below.}
\label{fig.B_3}
\end{figure}

We can now show the following result, which is a key step in the proof of Theorem~\ref{boundary-Harnack_App}.

\begin{prop}
\label{prop.bh}
There exists $\delta>0$, depending only on $n$, such that the following holds.

Let $\Omega\subset \R^n$ be a Lipschitz domain as in \eqref{eq.g1}-\eqref{eq.g2}, and let $\Omega_\delta:=\{x\in \Omega : {\rm dist}(x,\Omega^c)\geq \delta\}$. 
Assume that $u\in C(B_1)$ satisfies
\[
\left\{
\begin{array}{rcll}
\Delta u & = & 0 &\text{in } \Omega\cap B_1\\
u &= & 0 &\text{on } \partial \Omega\cap B_1
\end{array}\right.
\quad\,\text{and}\quad \,
\left\{
\begin{array}{rcll}
 u & \ge & 1 &\text{in } B_1\cap \Omega_\delta \\
u &\ge  & -\delta &\text{in }  B_1.
\end{array}\right.
\]
Then, $u \ge 0$ in $B_{1/2}$. 
\end{prop}

\begin{proof}
It is enough to show that, for some $a > 0$, we have
\begin{equation}
\label{eq.ua}
\left\{
\begin{array}{rcll}
 u & \ge & a &\text{in } B_{1/2}\cap \Omega_{\delta/2}\\
u &\ge  & -\delta a &\text{in }  B_{1/2}.
\end{array}\right.
\end{equation}
Indeed, iterating \eqref{eq.ua} at all scales, and at all points $z\in \partial \Omega\cap B_{1/2}$, we obtain 
\[
\left\{
\begin{array}{rcll}
 u & \ge & a^k &\text{in } B_{2^{-k}}(z)\cap \Omega_{2^{-k}\delta}\\
u &\ge  & -\delta a^k &\text{in }  B_{2^{-k}}(z)
\end{array}\right.
\]
for all $k\in \N$. 
In particular, the first inequality yields that $u(z+t e_n)\ge 0$ for $z\in \partial \Omega\cap B_{1/2}$ and $t > 0$, and therefore $u \ge 0$ in $B_{1/2}$.

Let us show \eqref{eq.ua}. 
We start with the first inequality. 
Let $x_\circ\in B_{1/2}\cap \Omega_{\delta / 2}$, and let us suppose that $\frac{\delta}{2}\le {\rm dist}(x_\circ,\Omega^c) < \delta$ (otherwise, we are done by assumption). 
Consider the function $w = u + \delta$, which satisfies $w\ge 0 $ in $\Omega$ by assumption. 

Notice that we can connect the points $x_\circ$ and $x_\circ+\frac{1}{2}\delta e_n$ with a sequence of (three) overlapping balls in $\Omega$, so that we can apply Harnack's inequality to $w$ to deduce 
\[
w(x_\circ) \ge \frac{1}{C}\,w\big(x_\circ+{\textstyle \frac{1}{2}}\delta e_n\big) \ge \frac{1}{C},
\]
for some dimensional constant $C$, where in the last step we are using that $w\left(x_\circ+\frac{1}{2}\delta e_n\right)\ge 1+\delta$ by assumption. 
In particular, by taking $\delta > 0$ smaller than $\frac{1}{2C}$, we get 
\[
u(x_\circ) \ge \frac{1}{C} - \delta \ge \frac{1}{2C}\qquad \mbox{for all}\quad x_\circ\in B_{1/2}\cap \Omega_{\delta / 2}.
\]

On the other hand, by Lemma~\ref{lem.kdelta} we know that $u \ge - \delta(1-c_\circ)^k$ in $B_{1-k\delta}$ as long as $k\delta \le \frac34$. If we take $k= \frac{1}{2\delta}$, we deduce 
\[
u \ge - \delta(1-c_\circ)^{\frac{1}{2\delta}}\quad \text{in}\quad B_{1/2},
\]
and taking $\delta$ small enough such that $(1-c_\circ)^{\frac{1}{2\delta}}\le \frac{1}{2C}$ we are done. 
\end{proof}

\begin{rem}[Proposition~\ref{prop.bh} for small Lipschitz constants]
The proofs of Lemma~\ref{lem.kdelta} and 
Proposition~\ref{prop.bh} can be simplified a lot in the case of a domain with small Lipschitz constant. 

Indeed, let us assume that the hypotheses of Proposition~\ref{prop.bh} hold, where the domain $\Omega$ satisfies \eqref{eq.g1}-\eqref{eq.g2} but with Lipschitz constant $L < \frac{1}{n-1}$, and let us consider the harmonic function 
\[
\varphi(x) = x_n^2 - \frac{1}{n-1}\left(x_1^2+x_2^2+\dots+x_{n-1}^2\right).
\]
Then, for $\delta$ small enough, $\varphi \le u$ on $\partial B_{1/2}\cap \Omega$, and by assumption on the Lipschitz constant of the domain we have that $\varphi\le 0$ on $\partial\Omega\cap B_1$. In all, the maximum principle gives $\varphi \le u$ in $B_{1/2}\cap \Omega$, which implies that $u(t e_n) \ge 0$ for $t\in\left[0, \frac12\right]$. By repeating the same argument at all boundary points in $\partial\Omega\cap B_{1/2}$ we reach that $u \ge 0$ in $B_{1/2}$. 
%
\end{rem}

We can now give the proof of Theorem~\ref{boundary-Harnack_App}.

\begin{proof}[Proof of Theorem~\ref{boundary-Harnack_App}]
Thanks to Lemma \ref{lem.bhbounded}, up to a constant depending on $C_\circ$, we may assume $w_1(\frac12 e_n) = w_2(\frac12 e_n) = 1$. Then, let us define 
\[
v = M w_1 - \eps w_2
\]
for some constants $M$ (large) and $\eps$ (small) to be chosen. 
Let $\delta >0$ be given by Proposition~\ref{prop.bh}.
Then, since $w_2$ is bounded,
\[
v \ge -\eps w_2 \ge -\delta \quad \text{in } \quad B_{1/2}
\]
for $\eps>0$ small enough. 
On the other hand, by the interior Harnack  inequality, we can take $M$ large enough so that $M w_1 \ge 1+\delta$ in $B_{1/2}\cap \Omega_\delta$, where we recall that $\Omega_\delta = \{x \in \Omega : {\rm dist}(x,\Omega^c) \ge \delta\}$. 
That is, 
\[
v = Mw_1 - \eps w_2 \ge 1\quad \text{in } \quad B_{1/2}\cap \Omega_\delta,
\]
for $M$ large enough depending only on $n$. 
Thus, the hypotheses of Proposition~\ref{prop.bh} are satisfied, and therefore we deduce that $v \ge 0$ in $B_{1/2}$. 

This means that, $w_2 \le C w_1$ in $B_{1/4}$ for some constant $C$ depending only on $n$. The inequality in $B_{1/2}$ follows by a covering argument.
Finally, reversing the roles of $w_1$ and $w_2$, we obtain the desired result.  
\end{proof}

Finally, we give the:

\begin{proof}[Proof of Corollary~\ref{boundary-Harnack_App2}]
Let us denote
\[W := \frac{w_1}{w_2},\]
so that we have to prove H\"older regularity for $W$ in $\overline{\Omega}\cap B_{1/2}$. 

Notice that, by Theorem \ref{boundary-Harnack_App}, we know that 
\[
\frac{1}{C}\le W \le C\quad\text{in}\quad B_{1/2}\cap \Omega,
\]
for some $C$ depending only on $n$. 
We start by claiming that, for some  $\theta > 0$ and all $k\in \mathbb N$, we have
\begin{equation}
\label{eq.claimbh}
\osc_{B_{2^{-k-1}}} W \le (1-\theta) \osc_{B_{2^{-k}}} W.
\end{equation}
Indeed, let 
\[
a_k := \sup_{B_{2^{-k}}} W \qquad\text{and}\qquad b_k := \inf_{B_{2^{-k}}}W.
\]
If we denote $p_k = \frac{1}{2^{k+1}} e_n$, then either $W(p_k) \ge \frac{1}{2}(a_k +b_k)$ or $W(p_k) \le \frac{1}{2}(a_k +b_k)$.

Suppose first that $W(p_k) \ge \frac{1}{2}(a_k +b_k)$, and let us define 
\[
v:= \frac{w_1 - b_k w_2}{a_k - b_k}. 
\]
Notice that, by assumption, 
\[
\frac12 w_2(p_k) \le v(p_k) \le w_2(p_k). 
\]
In particular, we can apply Theorem \ref{boundary-Harnack_App} to the pair of functions $v$ and $w_2$ in the ball $B_{2^{-k}}$, to deduce that $v \ge \frac{1}{C} w_2$ in $B_{2^{-k-1}}$, that is, 
\[
\frac{w_1 - b_k w_2}{a_k - b_k}\ge \frac{1}{C} w_2 \quad\text{in}\quad B_{2^{-k-1}}\quad\Longleftrightarrow\quad \inf_{B_{2^{-k-1}}} W \ge \frac{1}{C} (a_k - b_k)  + b_k. 
\]
Since $\sup_{B_{2^{-k-1}}} W \le \sup_{B_{2^{-k}}} W \le a_k$, we deduce that 
\[
\osc_{B_{2^{-k-1}}} W \le a_k - \frac{1}{C}(a_k - b_k) - b_k = \left(1-\frac{1}{C}\right) (a_k - b_k) = (1-\theta) \osc_{B_{2^{-k}}} W,
\]
with $ \theta = \frac{1}{C}$, as wanted. 

If we assume instead that $W(p_k) \le \frac{1}{2}(a_k +b_k)$, then the argument is similar taking $v := (a_k w_2 - w_1)/(a_k - b_k)$ instead. In all, \eqref{eq.claimbh} holds. 

In particular, we have shown that, for some small $\alpha$ depending only on $n$, we have
\begin{equation}
\label{eq.oscbh}
\osc_{B_r(x_\circ)}W \le C r^\alpha\quad\text{for all}\quad r\in (0,{\textstyle \frac14}) \quad\text{and}\quad x_\circ\in \partial \Omega\cap B_{1/2},
\end{equation}
(compare with the proof of Corollary~\ref{cor.Holder_regularity_1}). 
We now need to combine \eqref{eq.oscbh} with interior estimates for harmonic functions to deduce our desired result. 

Indeed, letting $x, y\in \overline{\Omega}\cap B_{1/2}$,   we want to show that 
\begin{equation}
\label{eq.toshowbh}
\left|W (x) - W(y)\right| \le C |x-y|^\alpha,
\end{equation}
for some constant $C$ depending only on $n$. 

Let $2r = {\rm dist}(x, \partial \Omega) = |x-x_*|$, with $x_*\in \partial \Omega$. We consider two cases:

\vspace{2mm}

\noindent $\bullet$  If $|x - y |\ge \frac{r}{2}$, then we apply \eqref{eq.oscbh}  in a ball $B_\rho(x_*)$ with radius $\rho = 2r + |x-y|$ to deduce that 
\[
|W(x) - W(y) |\le \osc_{B_{\rho}(x_*)}W \le C (2r+|x-y|)^\alpha\le C |x-y|^{\alpha}. 
\]

\vspace{2mm}

\noindent $\bullet$  If $|x - y |\le \frac{r}{2}$, then by \eqref{eq.oscbh} we know that $\osc_{B_r(x)}W \le C r^\alpha$. In particular, if we denote $c_* := W(x)$, then 
\[
\|w_1 - c_* w_2\|_{L^\infty(B_r(x))} = \|w_2\left(W - c_*\right)\|_{L^\infty(B_r(x))}\le C r^\alpha \|w_2\|_{L^\infty(B_r(x))}.
\]
On the other hand, since $w_1 - c_* w_2$ is harmonic in $B_r(x)$, by  Corollary~\ref{cor.Holder_regularity_1} (rescaled) we know that 
\[
[w_1 - c_* w_2]_{C^{0,\alpha}(B_{r/2}(x))} \le \frac{C}{r^\alpha} \|w_1 - c_* w_2\|_{L^\infty(B_{r}(x))}  \le C \|w_2\|_{L^\infty(B_r(x))}.
\]
Hence, 
\[
|W (y) - W(x) |  = \left|\frac{w_1(y)- c_*w_2(y)}{w_2(y)}\right|\le C|x-y|^\alpha\frac{\|w_2\|_{L^\infty(B_r(x))}}{w_2(y)}.
\]
We finish by noticing that, by Harnack's inequality applied to $w_2$ in $B_{2r}(x)$, we have $\|w_2\|_{L^\infty(B_r(x))} \le Cw_2(y)$ for some $C$ depending only on $n$.

\vspace{2mm}

With these two cases, we have shown \eqref{eq.toshowbh}. This proves the result. 
\end{proof}

\begin{rem}
\label{rem.general}
As said above, the proofs in this Appendix have been carried out in case that $\Omega$ is a Lipschitz domain as in \eqref{eq.g1}, with Lipschitz constant bounded by 1.
This slightly simplifies the notation, and we have that $\Omega\cap B_1$ has only one connected component.

In case of general Lipschitz domains (with Lipschitz constant bounded by $L$), the same proofs can be carried out, provided that one is slightly more careful with the underlying geometry.
A simple way to do this is to prove all the results with $B_{1/2}$ replaced by $B_\rho$, with $\rho>0$ small depending on $L$.
An alternative way to do this is to work with cylinders, rather than balls, as in \cite{DS-bdryH}.
\end{rem}

%
%
%

\chapter{Probabilistic interpretation of fully nonlinear equations}
\label{app.B}

In this appendix, we heuristically  describe the probabilistic interpretation of fully nonlinear elliptic PDEs. This extends the discussion from Section~\ref{sec.prob_interp} in the context of the Laplace operator.

We start by recalling the following probabilistic interpretation of harmonic functions from Chapter 1: 

We have a Brownian motion $X_t^x$, starting at $x\in \Omega$, and a  payoff function $g: \de\Omega\to \R$. When we hit the boundary $\de\Omega$ (for the first time) at a point $z\in \de\Omega$, we get a paid $g(z)$. 
The question is then:
\[\textrm{What is the expected payoff?}\]

It turns out that 
\[
u(x) := 
\left\{
\begin{array}{c}
\textrm{expected}\\
\textrm{payoff}
\end{array}
\right\} = \E\big[g\left(X_\tau^x\right)\big]
\quad\textrm{satisfies}\quad
\left\{
\begin{array}{rcll}
\Delta u & = & 0 & \textrm{ in  }~\Omega\\
u & = & g & \textrm{ on  }~\de\Omega,
\end{array}
\right.
\]
where $\tau$ is the first time at which $X_t^x$ hits $\de\Omega$.

We already saw this in Chapter~\ref{ch.0}. Now, we will see more general ``probabilistic games'' that lead to more general elliptic PDEs.

\subsection*{Stochastic processes}

A stochastic process $X_t$ is a collection of random variables indexed by a parameter, that for us is going to be $t \ge 0$, taking values in a state space, that for us is going to be $\R^n$. 
One can think of them as simply a ``particle'' moving randomly in $\R^n$, with $t \ge 0$ being the time. 

The most famous and important stochastic process is the \emph{Brownian motion}, that we already introduced in Section~\ref{sec.prob_interp}. We recall that it is characterized by the following properties:
\begin{enumerate}[(1)]
\item \label{it2.1} $X_0 = 0$ almost surely.
\item \label{it2.2} $X_t$ has \emph{no memory} (is independent of the past, or it has independent increments). 
\item \label{it2.3} $X_t$ has \emph{stationary increments}: $X_{t+s}-X_s$ is equal in distribution to $X_{t}$.
\item \label{it2.4} $X_t$ has \emph{continuous paths} ($t\mapsto X_t$ is continuous) almost surely. 
\item \label{it2.5} $X_t$ is \emph{isotropic}, i.e., it is rotationally symmetric in distribution. 
\end{enumerate}

A more general class of stochastic processes is obtained by removing the assumption (5).

\subsection*{Infinitesimal generator}
\index{Infinitesimal generator}
The infinitesimal generator of a stochastic process $X_t$ is an operator $L$ defined to act on functions $u: \R^n\to \R$ by 
\begin{equation}
\label{eq.infgen}
L u(x) := \lim_{t\downarrow 0} \frac{\E\left[u\left(x+X_t\right)\right] - u(x)}{t}.
\end{equation}
It takes $C^2$ functions $u$, and gives $Lu$. 

For the Brownian motion, we have that $L$ is the Laplacian $\Delta$. 

More generally, under the assumptions \ref{it2.1}-\ref{it2.2}-\ref{it2.3}-\ref{it2.4}, the infinitesimal generator $L$ will be a second order elliptic operator of the form
\[
Lu = \sum_{i,j = 1}^n a_{ij}\de_{ij} u + \sum_{i = 1}^n b_i\de_i u + cu. 
\]
Why is this infinitesimal generator useful? 

The infinitesimal generator of a stochastic process encodes all the information of such process. Indeed, it is a classical fact that the definition of $L$ leads to the formula
\begin{equation}
\label{eq.classL}
\E\left[u(x+X_t)\right] = u(x) +\E\left[\int_0^t L u(x+X_s)\, ds\right].
\end{equation}
(This is analogous to the fundamental theorem of Calculus!)

We can come back to the ``expected payoff'' problem:

\begin{center} 
 \fbox{
\begin{minipage}{0.95\textwidth}
\vspace{1.5mm}

\noindent \index{Expected payoff} Let $\Omega\subset\R^n$ be a fixed domain, and consider a stochastic process $(x+X_t)$ starting at $x\in \Omega$, satisfying \ref{it2.2}-\ref{it2.3}-\ref{it2.4} above. Given a payoff function $g: \de\Omega \to \R$, we have the following: when $X_t^x$ hits the boundary $\de\Omega$ for the first time at $z\in \de\Omega$, we get a payoff $g(z)$. (See Figure~\ref{fig.3.2}.)
\[
\textrm{\it What is the expected payoff? }
\]

\vspace{2mm}

\end{minipage}
}

\end{center}

\begin{figure}
\includegraphics[scale = 1.4]{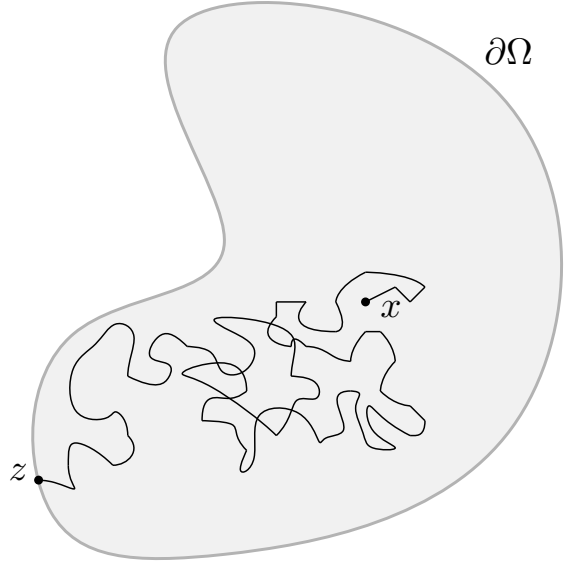}
\caption{A stochastic process $X_t^x$ defined in $\Omega$ starting at $x$ until it hits the first point on the boundary $z\in \de\Omega$.}
\label{fig.3.2}
\end{figure}

Of course, the expected payoff will depend on $x\in \Omega$. For the Brownian motion, we defined $u(x)$ to be the expected payoff when starting at~$x$, $\E\left[g(X_\tau^x)\right]$, where $\tau$ is the first time we hit $\de\Omega$. 
Then, we observed that, since the Brownian motion is isotropic,  $u$ must satisfy the mean value property, and thus $u$ is harmonic: $\Delta u = 0$ in $\Omega$. 

Now, for more general stochastic processes, we must use \eqref{eq.classL}. Indeed, we define $u$ as before (expected payoff), and notice that if $t > 0$ is small enough, then $x+X_t$ will still be inside $\Omega$, and therefore, the expected payoff is simply equal to $\E\left[ u(x+X_t)\right]$ (up to a small error), i.e,
\[
u(x) = \E\left[ u(x+X_t)\right] + o(t)\qquad \textrm{ (for $t  >0$ small enough)}.
\]
where the term $o(t)$ is due to the fact that $x+X_t$ could potentially lie outside of $\Omega$, even for arbitrarily small times $t>0$. 

Now, using the definition of infinitesimal generator, \eqref{eq.infgen}, we obtain that 
\[
Lu(x)  = \lim_{t\downarrow 0} \frac{\E\left[u(x+X_t)\right]-u(x)}{t} 
= \lim_{t\downarrow 0} \frac{o(t)}{t} = 0. 
\]

Therefore, for every $x\in \Omega$, we get $Lu(x) = 0$. 
We clearly have $u= g$ on $\de\Omega$, thus, $u$ must be the solution of 
\[
\left\{
\begin{array}{rcll}
Lu & = & 0 & \textrm{ in }~\Omega\\
u & = & g & \textrm{ on }~\de\Omega.
\end{array}
\right.
\]

Summarizing:
\[
\left\{
\begin{array}{c}
\textrm{Expected}\\
\textrm{payoff for $X_t$}
\end{array}
\right\}
\quad
\longleftrightarrow
\quad
\left\{
\begin{array}{c}
\textrm{Dirichlet problem for $L$}\\
\textrm{(infinitesimal generator)}
\end{array}
\right\}.
\]

Something similar can be done to solve other probabilistic problems related to $X_t$:

\begin{itemize}\index{Expected exit time}
\item[--] What is the expected time it will take to exit $\Omega$ if we start at $x$?
\[
\left\{
\begin{array}{rcll}
-Lu & = & 1 & \textrm{ in }~\Omega\\
u & = & 0 & \textrm{ on }~\de\Omega.
\end{array}
\right.
\]
\item[--] What is the probability density $p(x, t)$ of $X_t$ in $\R^n$?
\[
\left\{
\begin{array}{rcll}
\de_t p - L p & = & 0 & \textrm{ in }~\R^n\times (0,\infty)\\
p(\cdot,  0) & = & \delta_{\{ x = 0\}} & \textrm{ on }~\de\Omega.
\end{array}
\right.
\]
\end{itemize}

We next see what happens when we have a control, or a two-player game. In that case, we get nonlinear PDEs. 

\subsection*{Optimal stopping}\index{Optimal stopping}
We start with the optimal stopping problem. This kind of problem appears very often in Mathematical Finance, for example.

Given a process $X_t$ in $\R^n$, we can decide at each instant of time whether to stop it or not. When we stop, we get a payoff $\varphi$ (which depends on the point we stopped at). The goal is to discover what is the optimal strategy so that we maximize the payoff. 

Let us consider the process $x+X_t$ (starting at $x\in \R^n$), and a payoff $\varphi\in C^\infty_c(\R^n)$. For any stopping time $\theta$, we get a payoff $\E\left[\varphi(x+X_\theta)\right]$, and therefore we want to maximize
\[
u(x) := \max_\theta\E\left[\varphi(x+X_\theta)\right] 
\]
among all possible stopping times $\theta$ (notice that a stopping time $\theta$ is actually a random variable; see \cite{EvaS} for more details). 

Can we find a PDE for $u(x)$?

Roughly speaking, the only important thing to decide here is: 

If we are at $x$, is it better to stop and get $\varphi(x)$, or to continue and hope for a better payoff later?

Let us find the PDE for $u$:
\begin{itemize}
\item[--] First, since we can always stop (take $\theta = 0$), we have $u(x) \ge \varphi(x)$ for every $x\in \R^n$. 
\item[--] Second, since we can always continue for some time (take $\theta\ge t_\circ > 0)$, we have that $u(x) \ge \E\left[u(x+X_t)\right]$ for $t \le t_\circ$. This, combined with \eqref{eq.classL} (or with the definition \eqref{eq.infgen}), gives 
\[
Lu(x) \le 0\quad\textrm{ for every $x\in \R^n$}.
\]
\item[--] Third, at those points where we have $u(x) > \varphi(x)$,  we are clearly not stopping there, so we have $u(x) = \E\left[u(x+X_t)\right]+o(t)$ for $t$ very small, and thus $Lu(x) = 0$ whenever $u(x) > \varphi(x)$.
\end{itemize}

The PDE for $u$ is 
\[
\left\{
\begin{array}{rcll}
u & \ge & \varphi & \textrm{ in }~\R^n,\\
-Lu & \ge&  0 & \textrm{ in }~\R^n,\\
Lu & = & 0& \textrm{ in }~\{ u > \varphi\},
\end{array}
\right.
~~
\longleftrightarrow
~~
\min\{-Lu, u-\varphi\} = 0\quad\textrm{in}\quad \R^n.
\]
This is the \emph{obstacle problem} in $\R^n$ from Chapter~\ref{ch.4}. (See Figure~\ref{fig.11}.)

\begin{figure}
\includegraphics[width = \textwidth]{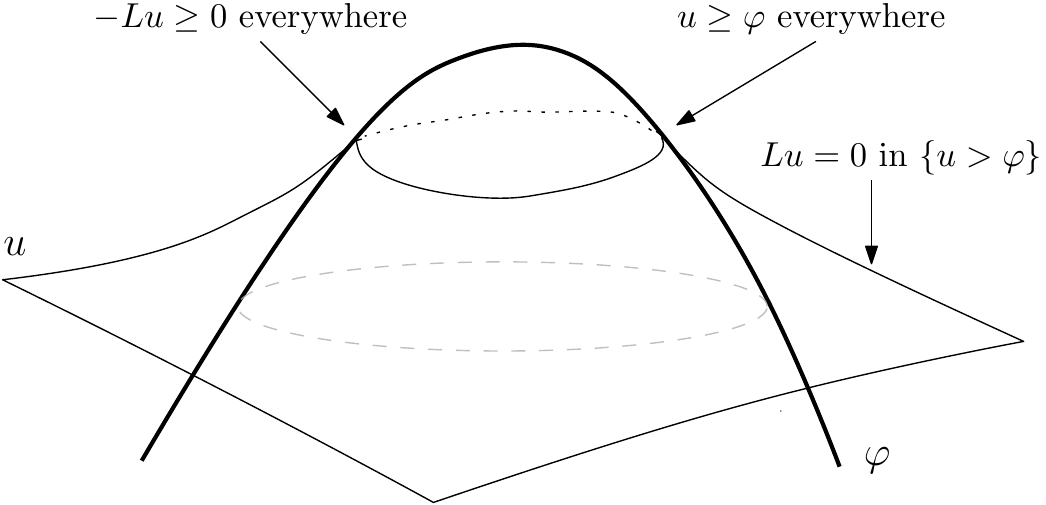}
\caption{The obstacle problem.}
\label{fig.11}
\end{figure}

Notice that once we know $u$,  we know the sets $\{u = \varphi\}$ and $\{u > \varphi\}$, so we have the optimal strategy!

\subsection*{Controlled diffusion}
\index{Controlled diffusion}
Let us now take a different problem, that nonetheless is quite similar to the optimal stopping. 

Consider two stochastic processes, $X_t^{(1)}$ and $X_t^{(2)}$, with infinitesimal generators $L_1$ and $L_2$ respectively. Let $\Omega\subset\R^n$ be a domain, and let $g:\de\Omega\to \R$ be a payoff. We have the same ``game'' as before (we get a payoff when we hit the boundary), but now we have a \emph{control}: for every $x\in \Omega$, we can choose to move according to $X_t^{(1)}$ or $X_t^{(2)}$. 

The question is then:
\[
\textrm{What is the optimal strategy if we want to maximize the payoff?}
\]

Notice that now the strategy consists of choosing between $X_t^{(1)}$ and $X_t^{(2)}$ for every $x\in \Omega$.  As before, we define 
\[
u(x) := \max_{\substack{\textrm{ all possible choices } \\ \textrm{of $a: \Omega\to \{1, 2\}$}}} \E\left[ g(X_\tau^a)\right]
\]
(where $\tau$ is the time we hit the boundary  $\de\Omega$).  Notice that for every $a: \Omega\to \{1, 2\}$ we have $X_t^a$, a process which could change from point to point. 

Is there any PDE for $u$?

The optimality conditions are:
\begin{itemize}
\item[--] First, when we are at $x$ we can simply decide to continue with $X_t^{(1)}$, and therefore, $u(x) \ge \E\left[u(x+X_t^{(1)})\right]$ for every $x\in \Omega$. This yields $L_1 u(x) \le 0$ for every $x\in \Omega$. 
\item[--] Similarly, we can do the same for $X_t^{(2)}$, and get $L_2 u(x) \le 0$ for every $x\in\Omega$. 
\item[--] Finally, it turns out that either
\[
 u(x) = \lim_{t\downarrow 0}\E\left[u(x+X_t^{(1)})\right]\quad \text{or}\quad u(x) = \lim_{t\downarrow 0}\E\left[u(x+X_t^{(2)})\right],\]
 since close to $x$ we are taking either $X_t^{(1)}$ or $X_t^{(2)}$. 
This means that either $L_1 u(x) = 0$ or $L_2 u(x) = 0$, for every $x\in \Omega$. 
\end{itemize}

Therefore, $u$ satisfies
\[
\left\{
\begin{array}{c}
  \begin{array}{rcll}
  -L_1 u & \ge & 0& \textrm{ in }~\Omega,\\
  -L_2 u & \ge&  0& \textrm{ in }~\Omega,
  \end{array}
  \\
  \textrm{either $L_1 u = 0$ or $L_2 u =0$ in $\Omega$}
\end{array}
\right.
\quad
\longleftrightarrow
\quad
\max\{L_1 u, L_2 u\} = 0\quad\textrm{in}\quad \Omega.
\]

More generally, if we have a family of processes $X_t^\alpha$, with $\alpha\in \mathcal{A}$, then the PDE for $u$ becomes \index{Bellman equation}
\begin{equation}
\label{eq.maxPDE}
\max_{\alpha\in \mathcal{A}}\left\{L_\alpha u\right\} = 0\quad\textrm{in}\quad \Omega.
\end{equation}

\index{Two-players game}Even more generally, we could have two players, one that wants to maximize the payoff and the other one that wants to minimize the payoff. 
They have two parameters, $X_t^{\alpha \beta}$, $\alpha\in \mathcal{A}$, $\beta\in \mathcal{B}$, and each player controls one parameter. Then, the optimal payoff solves the PDE \index{Isaacs equation}
\begin{equation}
\label{eq.minmaxPDE}
\min_{\beta\in \mathcal{B}}\max_{\alpha\in \mathcal{A}}\left\{L_{\alpha\beta}u\right\} = 0\quad\textrm{in}\quad \Omega.
\end{equation}

Equation~\eqref{eq.maxPDE} above is called the \underline{\smash{Bellman equation}} (stochastic control). 

Equation~\eqref{eq.minmaxPDE} above is called the \underline{\smash{Isaacs equation}} (differential games). 

These two equations are \emph{fully nonlinear elliptic equations}!

Indeed, assume that we have \eqref{eq.maxPDE}, and that the infinitesimal generators $L_\alpha u$ are of the form 
\[
L_\alpha u = \sum_{i,j =1}^n a_{ij}^{(\alpha)} \de_{ij} u,\qquad (\alpha\in \mathcal{A})
\]
with $a_{ij}^{(\alpha)}$ uniformly elliptic: $0 < \lambda{\rm Id}\le (a_{ij}^{(\alpha)})_{ij}\le \Lambda{\rm Id}$. Then, the equation \eqref{eq.maxPDE}  is 
\[
\max_{\alpha\in \mathcal{A}}\left\{\sum_{i,j = 1}^n a_{ij}^{(\alpha)} \de_{ij} u\right\} = 0\quad\textrm{in}\quad \Omega. 
\]
This is a nonlinear function of the Hessian $D^2 u$: 
\[
F(D^2 u) = 0\quad\textrm{in}\quad\Omega,\qquad\textrm{with}\qquad F(M) := \max_{\alpha\in \mathcal{A}}\left\{\sum_{i,j = 1}^n a_{ij}^{(\alpha)} M_{ij} \right\}.
\]
The function $F:\R^{n\times n}\to \R$ is the maximum of linear functions. In particular, $F$ is \emph{convex}. 

Moreover, $F$ is uniformly elliptic: 
\begin{align*}
0< \lambda\|N\| \le \min_{\alpha\in \mathcal{A}}\left\{\sum_{i,j = 1}^n a_{ij}^{(\alpha)} N_{ij} \right\} &\le F(M+N)-F(M) \leq \\
& \le \max_{\alpha\in \mathcal{A}}\left\{\sum_{i,j = 1}^n a_{ij}^{(\alpha)} N_{ij} \right\}\le \Lambda\|N\|
\end{align*}
for any symmetric matrix $N\ge 0$ (we are using here that $\max f + \min g\le \max(f+g) \le \max f + \max g$). 

Furthermore, any convex function can be written as the maximum of linear functions (see Figure~\ref{fig.12}), and thus:
\begin{rem}
Any $F:\R^{n\times n}\to \R$ which is uniformly elliptic and convex can be written as 
\[
F(M) = \max_{\alpha\in \mathcal{A}}\left\{{\rm tr}\, (A^{(\alpha)} M) + c_\alpha\right\}\qquad\textrm{(where $c_\alpha$ are constants)}. 
\]
(If $F$ is homogeneous of degree 1, then we do not need the $c_\alpha$.)

In particular, every fully nonlinear uniformly elliptic equation
\[
F(D^2 u) = 0\quad\textrm{in}\quad\Omega,
\]
with $F$ being convex, can be written as a Bellman equation
\[
\max_{\alpha\in \mathcal{A}}\left\{L_\alpha u \right\} = 0\quad\textrm{in}\quad \Omega\quad\text{with}\quad \scalebox{1}{$ L_\alpha u = \sum_{i,j= 1}^n a_{ij}^{(\alpha)}\de_{ij} u+c_\alpha $}.
\]
\end{rem}

\begin{figure}
\includegraphics[scale = 1.5]{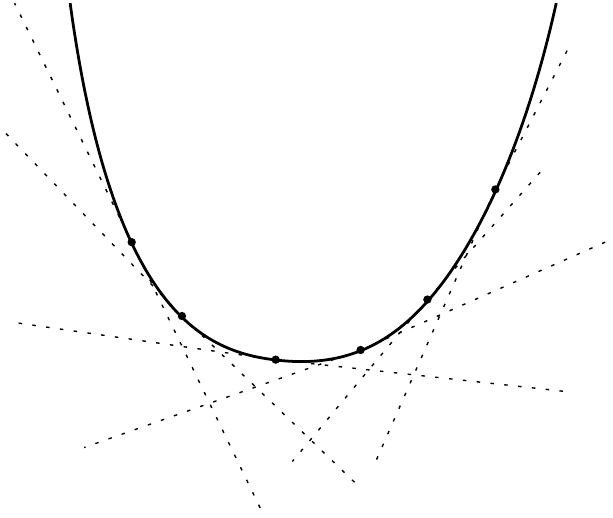}
\caption{Convex function as the maximum of linear functions.}
\label{fig.12}
\end{figure}

Finally, for non-convex $F$ it turns out that: 
\begin{obs}
Any $F: \R^{n\times n }\to \R$ which is uniformly elliptic (not necessarily convex), can be written as 
\[
F(M) = \min_{\beta\in \mathcal{B}}\max_{\alpha\in \mathcal{A}}\left\{
\sum_{i,j= 1}^n a_{ij}^{(\alpha\beta)}M_{ij}+c_{\alpha\beta}
\right\} = \min_{\beta\in \mathcal{B}}\max_{\alpha\in \mathcal{A}}\left\{
{\rm tr}\left(A^{(\alpha,\beta)}M\right) + c_{\alpha\beta}\right\}.
\]
This is because any Lipschitz function $F$ can be written as the minimum of convex functions, and convex functions can be written as the maximum of linear functions. 

In particular, every fully nonlinear uniformly elliptic equation
\[
F(D^2 u) = 0\quad\textrm{in}\quad \Omega
\]
can be written as an Isaacs equation
\[
\min_{\beta\in \mathcal{B}}\max_{\alpha\in \mathcal{A}}\left\{
L_{\alpha\beta} u\right\} = 0\quad\textrm{in}\quad \Omega
\quad\text{with}\quad \scalebox{1}{$ L_{\alpha\beta} u = \sum_{i,j= 1}^n a_{ij}^{(\alpha,\beta)}\de_{ij} u+c_{\alpha\beta} $}.
\]
\end{obs}

\underline{\smash{Summary}}: Every fully nonlinear elliptic PDE has an interpretation in terms of a probabilistic game!

\newpage

\subsection*{Probabilistic interpretation of PDEs}

\ \index{Probabilistic interpretation of PDEs}

\vspace{3mm}

\begin{center} 
 \fbox{
\begin{minipage}{0.93\textwidth}
\vspace{1mm}
\[
\begin{array}{ccl}

\textrm{Expected payoff} 
& 
~~\longleftrightarrow~~
&
\begin{array}{l}
\textrm{\underline{\smash{Dirichlet problem}}}
\\[0.1cm]
\left\{
\begin{array}{rcll}
Lu & = & 0 & \textrm{ in }~\Omega\\
u & = & g & \textrm{ on }~\de\Omega.
\end{array}
\right.
\end{array}

\\[1cm]

\begin{array}{c}
\textrm{Expected exit time}\\
\textrm{\tiny (or running costs/ }\\
\textrm{\tiny non-homogeneous environments)}
\end{array}
& 
~~\longleftrightarrow~~
&
\begin{array}{l}
\textrm{\underline{\smash{Dirichlet problem}}}
\\[0.1cm]

\left\{
\begin{array}{rcll}
 -Lu & = & f & \textrm{ in }~\Omega\\
u & = & 0 & \textrm{ on }~\de\Omega.
\end{array}
\right.

\end{array}

\\[1cm]

\textrm{Distribution of the process} 
& 
~~\longleftrightarrow~~
&
\begin{array}{l}
\textrm{\underline{\smash{Heat equation}}}
\\[0.1cm]
\quad\de_t u - Lu~~=~~0.

\end{array}

\\[1cm]

\textrm{Optimal stopping} 
& 
~~\longleftrightarrow~~
&
\begin{array}{l}
\textrm{\underline{\smash{Obstacle problem}}}
\\[0.1cm]
\quad\min\{-L u, u-\varphi\}~~=~~0.
\end{array}

\\[1cm]

\textrm{Controlled diffusion} 
& 
~~\longleftrightarrow~~
&
\begin{array}{l}
\hspace*{-0.2cm}
\begin{array}{l}
\textrm{\underline{\smash{Fully nonlinear equation}}}
\end{array}
\\[0.1cm]
\hspace*{-0.2cm}
\begin{array}{l}
\quad F(D^2 u)~~=~~0,\quad\textrm{$F$ convex.}
\end{array}
\end{array}

\\[1cm]

\textrm{Two-player games} 
& 
~~\longleftrightarrow~~
&
\begin{array}{l}
\textrm{\underline{\smash{Fully nonlinear equation}}}
\\[0.1cm]
\quad F(D^2 u)~~=~~0.
\end{array}

\end{array}
\]

\vspace{1mm}

\end{minipage}
}

\end{center}

\vspace{2mm}

One could even consider the equations with $x$-dependence, or with lower order terms. All equations studied in Chapters~\ref{ch.3} and \ref{ch.4} have a probabilistic interpretation.


%
%
%

\chapter{Motivations and applications for the obstacle problem}
\label{app.C}

Here, we give a brief overview of the motivations and applications for the obstacle problem listed in Chapter~\ref{ch.4}.
We refer to the books \cite{DL,KS, Rod87,Fri,PSU} for more details, as well as for further applications of obstacle-type problems.

\subsection*{Fluid filtration}

\index{Fluid filtration}
Consider two reservoirs of water at different heights separated by a porous dam. 
For simplicity, we will assume a flat dam, with rectangular cross section, which yields a problem in $\R^2$. 
Alternatively, one could consider variable cross sections, which would yield an analogous obstacle problem in $\R^3$ instead. 

The dam is permeable to the water, except in the base. 
Thus, there is some flow of fluid between the two reservoirs across the dam, and some wet part of the cross section depending only on the relative distance to each of the two water sources. 

Let us assume one reservoir has water at height 1, and the other has water at height $0 <h < 1$. 
Let us denote by $\varphi(x)$ the profile of the water through the dam cross section. 
See Figure~\ref{fig.dam} for a representation of the situation. 

\begin{figure}
\includegraphics{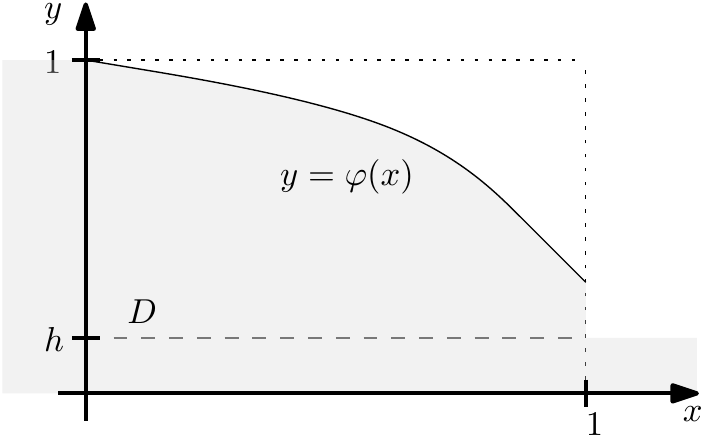}
\caption{Graphic representation of the cross section of a porous dam.}
\label{fig.dam}
\end{figure}

Let us denote by $u = u(x, y):[0,1]\times[0,1]\to \R_+$ the hydraulic piezometric head of the fluid, given by the sum between the pressure $p(x, y) $ and the elevation head (i.e., the potential energy of the fluid):
\[
u(x, y) = y +\frac{1}{\gamma} p(x, y),
\]
where $\gamma$ is a constant depending on the fluid. 
The hydraulic head is defined where there is fluid, namely, in 
\[
D := \big\{(x, y) \in (0, 1)\times(0, 1) : y < \vartheta(x)  \big\},
\]
and is such that $u(0, y) = 1$ for $0 \le y \le 1$, and $u(1, y) = h$ for $0\le y \le h$ and $u(1, y) = y$ for $h \le y \le \vartheta(1)$.

Here, $u$ itself is an unknown, but $D$ is also to be determined (and therefore, $\vartheta$). 
In these circumstances we have that $u (x, y) \ge y$ in $D$, and if we define 
\[
w(x, y) := \int_y^{\varphi(x)} \big(u(x, \zeta)-\zeta\big)\,d\zeta\quad\textrm{for}\quad (x, y) \in D,
\]
and $w(x, y) \equiv 0$ for $(x, y) \in [0,1]\times[0,1]\setminus D$, then $w$ fulfils the equation 
\[
\Delta w = \chi_{\{w > 0\}} = \chi_D \quad\textrm{in}\quad [0,1]\times[0,1]. 
\]
That is, $w$ is a solution to the obstacle problem (see \eqref{eq.ELOP}) with $f \equiv 1$. 

We refer to \cite{Baiocchi} and the references therein for more details about the Dam problem.

\subsection*{Phase transitions}
\index{Phase transitions}

The Stefan problem, dating back to the 19th century, is the most classical and important free boundary problem.
It aims to describe the temperature distribution in a homogeneous medium undergoing a phase change, such as ice melting to water. 

We denote by $\theta(x,t)$ the temperature (at position $x$ and time $t$), and assume $\theta\geq0$.
The function $\theta$ satisfies the heat equation $\partial_t\theta-\Delta \theta=0$ in the region $\{\theta>0\}$, while the evolution of the free boundary $\partial\{\theta>0\}$ is dictated by the Stefan condition $\partial_t\theta=|\nabla_x\theta|^2$ on $\partial\{\theta>0\}$ --- where the gradient is computed from inside $\{\theta>0\}$.

After the transformation $u(x,t):=\int_0^t\theta(x,\tau)d\tau$ (see \cite{Duv,Fig18}), the problem is locally equivalent to
\[
\left\{
\begin{array}{rcll}
\partial_t u-\Delta u &=& -\chi_{\{u>0\}}&\quad \textrm{in}\quad B_1\times (0,T)\subset\R^3\times \R \\
u & \geq &0&  \\
\partial_t u & \geq& 0.&
\end{array}
\right.\]
This is the parabolic version of the obstacle problem $\Delta u = \chi_{\{u>0\}}$ in $B_1$.

\subsection*{Hele-Shaw flow}

\index{Hele-Shaw flow}
This model, dating back to 1898, describes a fluid flow between two flat parallel plates separated by a very thin gap. 
Various problems in fluid mechanics can be approximated to Hele-Shaw flows, and that is why understanding these flows is important.

A Hele-Shaw cell is an experimental device in which a viscous fluid
is sandwiched in a narrow gap between two parallel plates. 
In certain regions, the gap is filled with fluid while in others the gap is
filled with air. 
When liquid is injected inside the device through some sinks (e.g. through a
small hole on the top plate) the region filled with liquid grows.

We denote by $p(x,t)$ the pressure of the fluid (at position $x$ and time $t$).
By definition, $\{p>0\}$ is the region filled with liquid, while in $\{p=0\}$  there is just air.
The pressure $p$ is harmonic in $\{p>0\}$, and the evolution of the free boundary $\partial\{p>0\}$ is dictated by $\partial_tp=|\nabla_xp|^2$ on $\partial\{p>0\}$ --- where the gradient is computed from inside $\{p>0\}$.
Notice the striking similarity to the Stefan problem --- the only important difference here is that $p$ is harmonic (and not caloric) in the region where it is positive.

After the transformation $u(x,t)=\int_0^tp(x,\tau)d\tau$, it turns out that $u$ solves locally (i.e., outside the region where liquid is injected)
\[
\left\{
\begin{array}{rcll}
\Delta u &=&\chi_{\{u>0\}}&\quad \textrm{in}\quad B_1\times (0,T)\subset\R^2\times \R \\
u & \geq&0 &\\
\partial_t u & \geq &0.&
\end{array}
\right.\]
This means that, for each fixed time $t$, $u(\cdot,t)$ is a solution to the (stationary) obstacle problem.

\subsection*{Optimal stopping, finance}
\index{Optimal stopping}

As explained in Appendix~\ref{app.B}, the obstacle problem appears when considering optimal stopping problems for stochastic processes.

A typical example is the Black--Scholes model for pricing of American options. 
An American option is a contract that entitles its owner to buy some financial asset (typically a share of some company) at some specified price (the ``strike price'') at any time --- often before some specified date. 
This option has some value, since in case that the always fluctuating market
price of the asset goes higher than the strike price then the option can be ``exercised'' to buy the asset at the lower price.
The Black-Sholes model aims to calculate the rational price $u = u(x, t)$ of an option at
any time~$t$ prior to the maturity date and depending on the current price $x$ of the financial asset. 
Since the option can be exercised at any time, determining the ``exercise region'' (i.e. the region in which it is better to exercise the option) is a part of the problem. 
Interestingly, this problem leads to an obstacle problem (often parabolic) posed in $\R^n$, where the dimension $n$ is the number of assets.

We refer to \cite{LS} and the references therein for more details about such kind of models.

\subsection*{Interacting particle systems}
\index{Interacting particle systems}
Large systems of interacting particles arise in several models in the natural sciences (one can think of physical particles in Physics or Biology, for example). 
In such systems the discrete energy can be well approximated by the continuum interacting energy. 
We denote $\mu$ the (probability) measure representing the particle density. 

In several models the particles attract
each other when they are far, but experience a repulsive force when they are close \cite{CDM16}. 
Then, the interaction energy $E$ associated to the interaction potential $W\in L^1_{\rm loc}(\R^3)$, is given by 
\[
E [\mu ] := \frac12 \int_{\R^3}\int_{\R^3} W(x-y) d\mu(x) \, d\mu(y). 
\]
In general, the interaction potential can have very different structures. 
It is common to assume a repulsive behaviour for particles that are very close (blowing up at zero distance), and attractive behaviour when they are far. 
A typical assumption is to have $W(z)\sim |z|^{-1}$ near the origin.

In other models in statistical mechanics, the particles (e.g. electrons) repel with a
Coulomb force and one wants to understand their behaviour in presence of some external
field that confines them \cite{Serfaty}.
In that case, the interaction energy associated with the system is given by 
\[
E [\mu] := \frac12 \int_{\R^3}\int_{\R^3} \frac{d\mu(x) \, d\mu(y)}{|x-y|} + \int_{\R^3}Vd\mu. 
\]

One of the main questions when dealing with these systems is to understand the ``equilibrium configurations'', that is, minimizers of the energy~$E$.

It turns out that, in both cases, any minimizer $\mu_\circ$ is given by $\mu_\circ = -\Delta u$, with $u$ satisfying (locally) the obstacle problem 
\[
\min\{-\Delta u,\ u -\varphi\} = 0,
\]
for some obstacle $\varphi$ that depends on $W$ (or on $V$). 
The free boundary corresponds to the boundary of the region in which the particles concentrate.
 
We refer to \cite{CDM16, Serfaty} and the references therein for a thorough study of these problems.

\subsection*{Quasi-Steady Electrochemical Shaping} \index{Quasi-Steady Electrochemical Shaping} 

Electrochemical Machining (ECM) is an electrochemical method to remove metals (electroconductive) by placing the material inside an electrolytic call as an anode, surrounded by a fixed cathode. Then an electric potential is applied between a cathode and an anode, which is submerged in an appropriate electrolyte, thus producing a chemical reaction that removes the metal from the anode and gives rise to a \emph{moving boundary}. This method is used to shape extremely hard materials, to produce complicated shapes which are otherwise very difficult to obtain. 

Let us suppose we have cylindrical symmetry (that is, both anode and cathode are long cylindrical materials), so that we can work with the cross section and thus in two dimensions. A similar approach works in the three-dimensional case. 

Let $\Omega\subset \R^2$ denote the domain enclosed by the cathode, and $\Lambda(0)\subset \Omega$ denote the anode at time $t = 0$ (an electric potential is applied between $\de\Omega$ and $\de\Lambda(0)$, where the region $\Omega\setminus \Lambda(0)$ contains the electrolyte). Then, the metal starts to be removed, so that after a time $t \ge 0$, we denote by $\Lambda(t)$ the set defining the anode. By this process we have that $\Lambda(t) \subset \Lambda(t')$ if $t \ge t'$. The boundary $\Gamma(t) = \de\Lambda(t)$ is unknown, it is a free boundary, which we assume is represented by a function $\gamma:\Omega\to \R$ as 
\[
\Gamma(t) = \{(x, y) \in \Omega : \gamma(x, y) = t\},
\]
for some function $\gamma$ to be determined. We assume that $\gamma(x, y) = 0$ in $\Omega\setminus \Lambda(0)$. If we denote by $\pi = \pi(t)> 0$ the potential difference at time $t > 0$ between anode and cathode, then the ECM problem is concerned with finding a function $\eta(t, x, y)$ that solves 
\[
\begin{split}
&\Delta \eta(t, x, y) = 0\quad\textrm{in}\quad \Omega\setminus\Gamma(t),\qquad \eta(t, x, y) = 0\quad\textrm{on}\quad \{t > 0 \}\times \de\Omega,\\
&\eta(t, x, y) = \pi(t),\quad \nabla \eta(t, x, y) \cdot\nabla \gamma(x, y) = \lambda\quad \textrm{on}\quad \{t > 0\}\times \Gamma(t)
\end{split}
\]
(with the convention that the gradient and the Laplacian are only taken in the spatial variables), for some constant $\lambda > 0$ (the ECM constant). Notice that $0\le \eta(t, x, y) \le \pi(t)$ in $\Lambda(t)$ by the maximum principle, and let us extend $\eta$ to $\Omega$ as $\eta(t, x, y) = \pi(t)$ in $\Lambda(t)$. Now, if we define 
\[
u(t, x, y) = \int_0^t \left(\pi(s) -\eta(s, x, y) \right)\, ds,
\]
then $u \ge 0$ and in $\Lambda(0)$, $u$ fulfils 
\[
\Delta u(t, \cdot, \cdot)  = \lambda\chi_{\{u(t, \cdot, \cdot) > 0\}}\quad\textrm{for any}\quad t > 0. 
\]
That is, $u$ fulfils an obstacle problem (compare with \eqref{eq.ELOP}) with $f \equiv \lambda$, for each time $t > 0$. 
We refer to  \cite{Rod87} for more details.

\subsection*{Heat control}\index{Heat control}

Given a domain $\Omega$ and a temperature $T_\circ$, we have heating devices evenly distributed on $\Omega$ that need to ensure that the temperature $u(x)$, $x\in \Omega$, is as close as possible to $T_\circ$, by injecting flux proportional to the distance between $u(x)$ and $T_\circ$. Due to the limited power of the devices, the heat flux generated by them needs to remain in the interval $(-q, 0]$ for $q \ge 0$. 

Thus, the heat flux injected is 
\[
\Phi(u) = \max\{C(u-T_\circ)_-, -q\}
\]
for some constant $C>0$. In equilibrium, the temperature satisfies
\[
\Delta u = \Phi(u)\quad\textrm{in}\quad \Omega,
\]
In particular, letting $C \to \infty$, the previous equation becomes 
\[
\Delta u = - q \chi_{\{u < T_\circ\}}\quad\textrm{in}\quad \Omega. 
\]

Notice that this structure is almost the same as for the obstacle problem (upside down). That is, if we define $w = T_\circ - u$ then the previous equation becomes
\[
\Delta w =  q \chi_{\{w >0\}}\quad\textrm{in}\quad \Omega,
\]
(see the parallelism to \eqref{eq.ELOP} with $f \equiv q>0$). If $w \ge 0$ (that is, $u \le T_\circ$) then this is exactly the obstacle problem. This can be obtained by putting Dirichlet boundary conditions on $\de\Omega$ that are $u|_{\de\Omega}\le T_\circ$ (for example, in a room with lateral walls without thermal insulation). 
We refer to \cite{DL} for more details.

\subsection*{Elasticity}\index{Elasticity}

We finish with probably the most intuitive physical interpretation of the obstacle problem: the deformation of a thin membrane in elasticity theory. 

Let us consider an elastic membrane represented by a function in $\R^2$, $u:\R^2\to \R$, so that $u(x, y)$ represents the vertical displacement with respect to the $xy$-plane. Given a domain $\Omega\subset \R^2$, we suppose that the membrane has a fixed boundary, that is, we prescribe the value of $u$ on $\de\Omega$, by some (say continuous) function $g:\de\Omega\to \R$. We assume an homogeneous membrane equally stretched in all directions, whose shape is determined by the surface tension. For simplicity we also assume lack of external forces. 

In this setting, the shape of the membrane will be such that the total area is minimized, among all possible configurations with the same boundary values. Namely, the following functional 
\[
\int_\Omega \sqrt{1+|\nabla w|^2}\, dx\, dy
\]
is minimized among functions $w\in H^1(\Omega)$ such that $w|_{\de\Omega} = g$. This yields the classical Plateau's problem. The Dirichlet energy appears as a lower order approximation of the previous functional. Namely, if we assume that the vertical displacements are not \emph{large} (say, the membrane is \emph{rather} flat), then a Taylor expansion of the functional yields
\[
\int_\Omega \sqrt{1+|\nabla w|^2}\, dx\, dy \sim \int_\Omega \left(1+\frac12 |\nabla w|^2\right)\, dx\, dy,
\]
so that the minimization of the area is \emph{roughly} a minimization of the Dirichlet energy. 

The obstacle problem is concerned with finding the membrane that minimizes the Dirichlet energy (thus, approximately the area) among those with prescribed boundary, that lie above a given obstacle $\varphi:\R^2\to \R$.

\backmatter

\chapter*{Notation}

Let us introduce some of the notation be used throughout the book. 
\\[0.4cm]
\noindent {\bf Matrix notation.}\\[0.3cm]
\begin{tabular}{ l m{10cm} }
 $A = (a_{ij})_{ij}$ & Matrix with $(i, j)-{\rm th}$ entry denoted by $a_{ij}$.   \\[0.15cm] 
 $\mathcal{M}_n$ & Space of matrices of size $n\times n$.  \\[0.15cm]  
 ${\rm Id}$& Identity matrix.  \\[0.15cm]
 ${\rm tr}\, A$& Trace of the matrix $A$, ${\rm tr}\, A = a_{11}+\dots+a_{nn}$.  \\[0.2cm]
 ${\rm det}\, A$& Determinant of the matrix $A$. \\[0.15cm]
 $A^T$& Transpose of the matrix $A$. \\[0.2cm]
\end{tabular}

\noindent {\bf Geometric notation.}\\[0.3cm]
\begin{tabular}{ l  m{11cm} }
 $\R^n$, $\mathbb{S}^n$ & $n$-dimensional Euclidean space, $n$-sphere.  \\[0.15cm] 
 $e_i\in \mathbb{S}^{n-1}$ & $i-{\rm th}$ element of the base, $e_i = (0,\dots,0, \stackrel{(i)}{1},0,\dots 0)$. \\[0.15cm]  
 $x \in \R^n$& Typical point $x = (x_1, \dots, x_n)$.  \\[0.15cm]
  $|x| $& Modulus of the point $x$, $|x| = \sqrt{x_1^2 +\dots + x_n^2}$.  \\[0.15cm] 
    $|U| $& $n$-dimensional Lebesgue measure of a set $U\subset\R^n$.  \\[0.15cm] 
 $\R^n_+$& $\{x = (x_1,\dots, x_n)\in \R^n : x_n > 0\}$.  \\[0.15cm]
 $\partial U$& Boundary of the set $U\subset \R^n$. \\[0.15cm]
 $V\subset\subset U$& The set $V$ is compactly contained in $U$, that is $\overline{V}\subset U$. \\[0.15cm]
  $B_r(x)$& Ball of radius $r$ centered at $x$, $B_r(x) := \{y \in \R^n : |x-y|< r\}$. \\[0.15cm]
  $x\cdot y$ & For $x, y\in \R^n$, scalar product of $x$ and $y$, $x\cdot y = x_1y_1+\dots+x_n y_n$. 
\end{tabular}
\\[1cm]
\noindent {\bf Functional notation.}\\[0.3cm]
\begin{tabular}{ l m{11cm} }
 $u$ & In general, $u$ denotes a function $u:\R^n\to \R$ (unless stated otherwise).   \\[0.2cm] 
  $u^+,u^-$ & Positive and negative part of a function, $u^+ = \max\{u, 0\}$, $u^- = \max\{-u, 0\}$.   \\[0.2cm]
    $\chi_E$ & Characteristic function of the set $E$, $\chi_E(x) = 1$ for $x\in E$, and $\chi_E(x) = 0$ for $x\notin E$.  \\[0.2cm]
    ${\rm supp}\, u$ & Support of $u$, ${\rm supp}\,u = \overline{\{x : u(x) \neq 0\}}$.          \\[0.2cm]
    $\ave_A$ & Average integral over the positive measure set $A$, $\ave_A f := \frac{1}{|A|}\int_A f$.   
\end{tabular}
\\[0.5cm]
\noindent {\bf Function spaces.} Let $U\subset\R^n$ be an open set. \\[0.3cm]
\begin{tabular}{ l m{9.5cm} }
 $C(U), C^0(U)$ & Space of continuous functions $u: U\to \R$.   \\[0.2cm] 
  $C(\overline{U}), C^0(\overline{U})$ & Functions $u\in C(U)$ continuous up to the boundary.   \\[0.2cm] 
  $C^k(U), C^k(\overline{U})$ & Space of functions $k$ times continuously differentiable.   \\[0.2cm]
    $C^{k,\alpha}(U)$ & H\"older spaces, see Section \ref{sec.hs}.   \\[0.2cm]
   $C^\infty(U), C^\infty(\overline{U})$  & Set of functions in $C^k(U)$ or $C^k(\overline{U})$ for all $k\ge 1$.\\[0.2cm]
   $C_c(U), C^k_c(U)$  & Set of functions with compact support in $U$.\\[0.2cm]
   $C_0(U), C^k_0(U)$  & Set of functions with $u = 0$ on $\partial U$.\\[0.2cm]
    $L^p$ & $L^p$ space, see Section \ref{sec.hs}.\\[0.2cm]
    $L^\infty$ & $L^\infty$ space, see Section \ref{sec.hs} (see ${\rm esssup}_\Omega u$ below).\\[0.2cm]
    ${\rm esssup}_\Omega u$ & {Essential supremum of $u$ in $\Omega$:  infimum of the essential upper bounds, ${\rm esssup}_\Omega u := \inf\{b > 0 : |\{u > b\}| = 0\}$.}\\[0.2cm]
    $W^{1,p}, W^{1,p}_0$ & Sobolev spaces, see Section \ref{sec.hs} and \ref{it.S6}.\\[0.2cm]
    $H^{1}, H^{1}_0$ & Sobolev spaces with $p = 2$, see Section \ref{sec.hs} and \ref{it.S6}.\\[0.2cm]
    $\|\cdot\|_{\mathcal{F}}$ & Norm in the functional space $\mathcal{F} \in \{C^0, C^k, L^p, \dots\}$, defined when used for the first times.
\end{tabular}
\\[0.4cm]
\newpage
\noindent {\bf Differential notation.} Let $u:U \to \R$ be a function.\\[0.3cm]
\begin{tabular}{ l m{9cm} }
 $\de_i u, \de_{x_i} u , u_{x_i}$ & Partial derivative in the $e_i$ direction, $\frac{\de u}{\de x_i}$. \\[0.15cm] 
  $\de_e u$ & Derivative in the $e\in \mathbb{S}^{n-1}$ direction. \\[0.15cm] 
  $\nabla u, Du$ & Gradient, $\nabla u = (\de_1 u, \dots, \de_n u)$. \\[0.15cm] 
   $\de_{ij} u, \de_{x_i x_j} u , u_{x_i x_j}$ & Second partial derivatives in the directions $e_i$ and $e_j$, $\frac{\de^2 u}{\de x_i\de x_j }$. \\[0.15cm] 
    $D^2 u$ & Hessian, $D^2 u  = (\de_{ij} u)_{ij} \in \mathcal{M}_n$. \\[0.15cm] 
    $D^k u$ & Higher derivatives forms, $D^k u := (\de_{i_1}\dots\de_{i_k} u)_{i_1,\dots ,i_k}$. \\[0.15cm] 
    $|D^k u(x)|$ & Norm of $D^k u(x)$ (any equivalent norm). \\[0.15cm] 
    $\|D^k u(x)\|_{\mathcal{F}}$ & Norm of $D^k u$, $\| |D^k u|\|_{\mathcal{F}}$.\\[0.15cm] 
    $\Delta u$ & Laplacian of $u$, $\Delta u = \de_{11} u + \dots + \de_{nn} u$.
\end{tabular}
\\[0.4cm]
\label{domainnotation}
\noindent {\bf Domains.} We say that $\Omega\subset\R^n$ is a domain if it is an open connected set.\\[0.2cm]
A domain $\Omega$ is said to be $C^{k, \alpha}$ (resp. $C^k$) if $\partial\Omega$ can be written locally as the graph of a $C^{k,\alpha}$ (resp. $C^k$) function.



\bibliographystyle{amsalpha}



\printindex

\end{document}